\documentclass[12pt, a4paper, parskip=half, abstracton]{scrartcl}
\pdfoutput=1

\usepackage{array}
\usepackage{marginnote}
\usepackage{xcolor}
\usepackage{amscd,amssymb,amsfonts,amsmath,latexsym,amsthm}
\usepackage{hyperref}
\usepackage[all,cmtip]{xy}
\usepackage{mathrsfs}
\usepackage{tikz}
\usepackage{tikz-cd}
	\tikzcdset{row sep/normal=1.3em}
	\tikzcdset{column sep/normal=1.3em}
\usepackage{graphicx}
\usepackage{bm}

\usepackage{etoolbox}

\def\hB{\hspace*{\fill}$\qed$}

\usepackage[nottoc]{tocbibind}

\usepackage{defs_pp1}
\usepackage{slashed}
\usepackage[utf8]{inputenc}
\usepackage{microtype}
\usepackage[english]{babel}
\usepackage{mathtools}
\usepackage{bm}
\usepackage{esvect}
\usepackage{rotating}

\title{Paschke duality and assembly maps}
\author{
Ulrich Bunke\thanks{Fakult{\"a}t f{\"u}r Mathematik,
Universit{\"a}t Regensburg,
93040 Regensburg,
GERMANY\newline
\href{mailto:ulrich.bunke@mathematik.uni-regensburg.de}{ulrich.bunke@mathematik.uni-regensburg.de}}
\and Alexander Engel\thanks{Institut f{\"u}r Mathematik und Informatik, Universit{\"a}t Greifswald, 17489 Greifswald, GERMANY\newline
\href{mailto:alexander.engel@uni-greifswald.de}{alexander.engel@uni-greifswald.de}}
\and Markus Land\thanks{Mathematisches Institut, Ludwig-Maximilians-Universit\"at M\"unchen, 80333 M\"unchen, GERMANY\newline \href{mailto:markus.land@math.lmu.de}{markus.land@math.lmu.de}}
}

\numberwithin{equation}{section}
\newtheorem{theorem}{Theorem}[section] 
\newtheorem{prop}[theorem]{Proposition}
\newtheorem{lem}[theorem]{Lemma}

\newtheorem{ddd}[theorem]{Definition}
\newtheorem{kor}[theorem]{Corollary}
\newtheorem{ass}[theorem]{Assumption}

\theoremstyle{remark}
\theoremstyle{definition}
\newtheorem{ex}[theorem]{Example}
\newtheorem{rem}[theorem]{Remark}

\newcommand{\EE}{\mathrm{EE}}
\newcommand{\eadd}{\mathrm{eadd}}
\newcommand{\ndeg}{\mathrm{ndeg}}
\newcommand{\fg}{\mathrm{fg}}
\newcommand{\hfin}{\mathrm{hfin}}
\newcommand{\cf}{\mathrm{cf}}
\newcommand{\LF}{\mathrm{LF}}
\newcommand{\pc}{\mathrm{pc}}

\newcommand{\kkA}{\mathrm{kk}_{C^{*}\mathbf{Cat}}}
\newcommand{\kkGA}{\mathrm{kk}_{C^{*}\mathbf{Cat}}^{G}}
\newcommand{\kkHA}{\mathrm{kk}_{C^{*}\mathbf{Cat}}^{H}}

\newcommand{\CAT}{\mathbf{CAT}}
\newcommand{\ctc}{\mathrm{ctc}}
\newcommand{\comm}{\mathrm{comm}}

\renewcommand{\Met}{\mathbf{Met}}
\newcommand{\KKHs}{\mathrm{KK}_{\mathrm{sep}}^{H}}
\newcommand{\kks}{\mathrm{kk}_{\mathrm{sep}}}

\newcommand{\Tw}{\mathbf{Tw}}
\newcommand{\kk}{\mathrm{kk}}

\newcommand{\scale}{\mathrm{prop}}

\renewcommand{\Ind}{\mathrm{Ind}}

\newcommand{\std}{\mathrm{std}}

 \newcommand{\bCtsmc}{ \mathbf{\bar C}_{\mathrm{lf}}^{\mathrm{ctr}}}
  \newcommand{\btCtsmc}{ \mathbf{\widetilde{\bar C}}_{\mathrm{lf}}^{\mathrm{ctr}}}
  \newcommand{\bCgtsmc}{ \mathbf{\bar C}_{\mathrm{lf}}^{ G,\mathrm{ctr}}}

\newcommand{\topp}{\mathrm{top}}
\newcommand{\cw}{\mathrm{CW}}

\newcommand{\cofib}{\mathrm{cofib}}
\newcommand{\Homol}{\mathrm{Hg}}

\newcommand{\UBC}{\mathbf{UBC}}

\newcommand{\cJ}{\mathcal{J}}

\newcommand{\yo}{\mathrm{yo}}

\newcommand{\Res}{\mathrm{Res}}

\newcommand{\Orb}{\mathbf{Orb}}

\newcommand{\Hilb}{\mathbf{Hilb}}
\newcommand{\cR}{\mathcal{R}}
\newcommand{\inter}{\mathrm{int}}

\newcommand{\bQ}{\mathbf{Q}}

\newcommand{\BC}{\mathbf{BC}}

\newcommand{\cQ}{\mathcal{Q}}
\newcommand{\Fin}{\mathbf{Fin}}

\newcommand{\Ob}{\mathrm{Ob}}

\newcommand{\bB}{{\mathbf{B}}}

\newcommand{\incl}{\mathrm{incl}}

\newcommand{\An}{\mathrm{An}}

\newcommand{\bL}{\mathbf{L}}
\newcommand{\bM}{\mathbf{M}}

\newcommand{\cP}{\mathcal{P}}

\newcommand{\bF}{{\mathbf{F}}}

\newcommand{\E}{\mathbb{E}}

\newcommand{\cI}{{\mathcal{I}}}

\newcommand{\cZ}{{\mathcal{Z}}}

\newcommand{\PSh}{{\mathbf{PSh}}}

\newcommand{\bA}{{\mathbf{A}}}

\newcommand{\bK}{{\mathbf{K}}}
\newcommand{\const}{{\mathtt{const}}}

\newcommand{\cO}{{\mathcal{O}}}
\newcommand{\cU}{{\mathcal{U}}}
\newcommand{\cY}{{\mathcal{Y}}}
\newcommand{\Var}{\mathrm{Var}}

\newcommand{\cD}{{\mathcal{D}}}

 \newcommand{\Cone}{{\mathtt{Cone}}}

\newcommand{\cE}{{\mathcal{E}}}

\DeclareMathOperator{\proj}{Proj}
\renewcommand{\proj}{\mathrm{proj}}

\newcommand{\lf}{\mathrm{lf}}

\newcommand{\lto}{\longrightarrow}

\newcommand{\Idem}{\mathrm{Idem}}

\newcommand{\Born}{\mathbf{Born}}

\newcommand{\Spc}{\mathbf{Spc}}

\newcommand{\Ccat}{{\mathbf{C}^{\ast}\mathbf{Cat}}}
\newcommand{\Calg}{{\mathbf{C}^{\ast}\mathbf{Alg}}}

\newcommand{\Kasp}{\mathrm{Kasp}}
\newcommand{\MN}{\mathrm{MN}}
\newcommand{\DL}{\mathrm{DL}}

\newcommand{\bd}{\mathrm{bd}}
\newcommand{\op}{\mathrm{op}}

\newcommand{\add}{\mathrm{add}}

\newcommand{\nCcat}{C^{*}\mathbf{Cat}^{\mathrm{nu}}}
\renewcommand{\Ccat}{C^{*}\mathbf{Cat}}
\newcommand{\alg}{\mathrm{alg}}

\renewcommand{\Calg}{C^{*}\mathbf{Alg}}
\newcommand{\nCalg}{C^{*}\mathbf{Alg}^{\mathrm{nu}}}

\newcommand{\Kcat}{K^{C^{*}\mathrm{Cat}}}

\newcommand{\Ass}{\mathrm{Asmbl}}
\newcommand{\Kast}{K^{C^{*}\mathrm{Alg}}}
\newcommand{\an}{\mathrm{an}}

\newcommand{\tCglf}{{\mathbf{C}}^{(G)}_{\mathrm{lf}}}
\newcommand{\Cglf}{\mathbf{C}^{G}_{\mathrm{lf}}}
\newcommand{\sepa}{\mathrm{sep}}

\newcommand{\kkG}{\mathrm{kk}^{G}}

\newcommand{\kkH}{\mathrm{kk}^{H}}
\newcommand{\KKG}{\mathrm{KK}^{G}}

\newcommand{\KKH}{\mathrm{KK}^{H}}
\newcommand{\KK}{\mathrm{KK}}
\newcommand{\KKs}{\mathrm{KK}_{\sepa}}
\newcommand{\KKGs}{\mathrm{KK}_{\sepa}^{G}}

\newcommand{\kkGs}{\mathrm{kk}_{\sepa}^{G}}
\newcommand{\proper}{\mathrm{prop}}
\newcommand{\KKth}{K\!K}

\newcommand{\la}{\mathrm{}}
 
\newcommand{\Simpl}{\mathbf{Simpl}}

\newcommand{\cc}{\mathcal{CC}}
\newcommand{\ci}{\mathcal{CI}}
 
\DeclareMathOperator{\hatotimes}{\hat{\otimes}}

\newcommand{\ppGTop}{G\mathbf{LCH}^{\mathrm{prop}}_{+}}

\newcommand{\ppGTopf}{G\mathbf{LCH}^{\mathrm{prop,hfin}}_{+}}
\newcommand{\ppGTopo}{G\mathbf{LCH}^{\mathrm{prop, \sigma hfin}}_{\mathrm{2nd},+}}

\newcommand{\pGTopc}{G\mathbf{LCH}_{+,\mathrm{pc}}^{\mathrm{prop}}}

\newcommand{\uliemph}{}
\begin{document}
	
\maketitle

\begin{abstract}
We construct a  natural transformation between two versions of   {$G$-equivariant $K$-homology} with coefficients in a $G$-$C^{*}$-category  {for a  {countable} discrete  group $G$}. Its domain is a coarse geometric $K$-homology and its target is the usual analytic $K$-homology. Following classical terminology, we call this transformation the Paschke transformation. We show that under certain finiteness assumptions on a $G$-space $X$, the Paschke transformation is an equivalence on $X$.
As an application, we {provide a direct comparison} of the  homotopy theoretic Davis--L\"uck {assembly map} with {Kasparov's} analytic assembly map appearing in the Baum--Connes conjecture.
\end{abstract}

\tableofcontents
\setcounter{tocdepth}{5}

\paragraph{Acknowledgements.}

Ulrich Bunke was supported by the SFB 1085 (Higher Invariants) funded by the Deutsche Forschungsgemeinschaft (DFG).

Alexander Engel acknowledges financial support by the Deutsche Forschungsgemeinschaft (DFG, German Research Foundation) through the Priority Programme SPP 2026 ``Geometry at Infinity'' (EN 1163/5-1, project number 441426261, Macroscopic invariants of manifolds) and through Germany's Excellence Strategy EXC 2044-390685587, Mathematics Münster: Dynamics -- Geometry -- Structure.

Markus Land was supported by the research fellowship DFG 424239956, and by the Danish National Research Foundation through the Copenhagen Centre for Geometry and Topology (DNRF151).

\section{Introduction and statements}

The main result of the present paper is the construction of a natural transformation 
\begin{equation}\label{wefewrfwewefwerfwef}
K^{G,\cX}_{{\bC}}\to K^{G,\An}_{{\bC}}
\end{equation} 
between two versions of spectrum-valued equivariant {$K$-homology} functors, {where $G$ is a  {countable} discrete  group.  The evaluation of this transformation
 on $G$-finite ${G}$-simplicial complexes with finite stabilzers is an equivalence.}  Following the classical terminology, we call this transformation the \uliemph{Paschke transformation}.
{The} functor  $K^{G,\cX}_{{\bC}}$ in the domain is called the equivariant local $K$-homology and is derived from an equivariant  coarse $K$-homology functor using {coarse geometric constructions}, 
while the target $ K^{G,\An}_{{\bC}}$ is a spectrum-valued version of the classical equivariant analytic $K$-homology. {In both versions {the subscript 
indicates}  a natural  dependence    {on} 
a coefficient $G$-$C^{*}$-category {$\bC$.} 

The Paschke transformation \eqref{wefewrfwewefwerfwef} will be used to compare the domains   of the Davis--L\"uck type assembly map  and  the  Baum--Connes type assembly map. 
Our   second main result  is Theorem \ref{wtoiguwegwergergregwe}
showing that these two assembly maps are equal 
on the level of homotopy groups.

 
In the following we give {an 
 informal description of the construction of the two homology theories   entering \eqref{wefewrfwewefwerfwef}. 
Starting from  classical Paschke duality we further explain 
 the development of ideas leading to  {the}   construction of the map
 in \eqref{wefewrfwewefwerfwef}. We then} state 
  the precise version {of} our  Paschke duality  result as Theorem \ref{qreoigjoergegqrgqerqfewf}, {and finally discuss the comparison of assembly maps.}

{We emphasize that this paper is not the first to treat the topic of equivariant Paschke duality and comparisons of assembly maps, most current are the papers \cite{Benameur:2020aa} and \cite{kranz}. We explain more about this in Remarks~\ref{Remark:previous-work} and \ref{Remark:comparison-assembly}.}

\subsubsection*{Constructions with the coefficients}

For facts about $C^{*}$-categories and their $K$-theory we will generally refer to
\cite{crosscat} and \cite{cank} which were written to provide the necessary background for the present paper,  \cite{coarsek} and \cite{KKG}.
  Both $K$-homology  functors  occuring in \eqref{wefewrfwewefwerfwef} 
depend on the choice of a $G$-$C^{*}$-category 
{$\bC$, i.e., an object of $\Fun(BG,\nCcat)$ (see \cite[Sec. 3]{crosscat} or \cite[Def. 2.6]{cank} for $\nCcat$). We use the symbol $\bM\bC$ in order to denote the multiplier category of $\bC$ \cite[Def. 3.1]{cank}.}
{  {In Definition~\ref{qeroighjoreqgreqfefwfqwef}  we  describe}   an exact sequence
$$0\to \bC_{\std}^{(G)}\to \bM\bC_{\std}^{(G)}\to \bQ_{\std}^{(G)}\to 0$$ of $G$-$C^{*}$-categories (see \cite[Def. 8.5]{crosscat} or \cite[Def. 13.2]{cank} for the notion of an exact sequence)
defining the  
\uliemph{Calkin}} $G$-$C^{*}$-category $\bQ^{(G)}_{\std}$. These constructions depend functorially on $\bC$ for non-degenerate morphisms.} 
\begin{ex} In
the case  of trivial coefficients     $\bC$ is the    $G$-$C^{*}$-category  $\Hilb_{c}(\C)$ of Hilbert spaces 
  and compact 
operators with trivial $G$-action.
The multiplier category  of  $\Hilb_{c}(\C)$ can be identified with  the category  $\Hilb(\C)$ of Hilbert spaces and all bounded operators \cite[Lem. 8.1]{cank}. 
By specializing  Definition~\ref{qeroighjoreqgreqfefwfqwef} the $G$-$C^{*}$-category  $ {\bC^{(G)}_{\std} }$  turns out to be  the category  $\Hilb(\C)_{\std}^{(G)}$ of all pairs $(H,\rho)$ of a Hilbert space $ {H}$ with a unitary $G$-representation $\rho$ that are isomorphic to $(L^{2}(G)\otimes H' ,\lambda\otimes \id_{H'})$, where $\lambda$ is the left-regular representation and $H'$ is some auxiliary Hilbert space. 
 The morphisms $(H_{0},\rho_{0})\to (H_{1},\rho_{1})$  in $\Hilb_{c}(\C)^{(G)}_{\std}$   are all compact operators $H_{0}\to H_{1}$.  
 
The $G$-$C^{*}$-category   $\bM\bC^{(G)}_{\std}$ is the category $\Hilb(\C)_{\std}^{(G)}$ 
{which}  has the same objects, but its morphism spaces are  the bigger spaces of  all bounded linear operators. In both cases the $G$-action fixes objects and acts by conjugation on the morphism spaces.
The $G$-$C^{*}$-category  $\bQ^{(G)}_{\std}$ is the Calkin category $\Hilb(\C)^{(G)}_{\std}/ \Hilb_{c}(\C)^{(G)}_{\std}$. Its objects are the 
objects   $(H,\rho)$ of {$\Hilb(\C)_{\std}^{(G)}$}, and its morphism spaces  are the quotient spaces  of bounded operators by compact operators with the induced $G$-action. In particular, the endomorphism algebra of each object $(H,\rho)$ is the usual Calkin algebra $Q(H)$ of $H$ with the $G$-action by conjugation, hence the name. \hB \end{ex}

\begin{ex}
More generally, for a  $G$-$C^*$-algebra $A$ we consider 
the $G$-$C^{*}$-category {$\bC=\Hilb_{c}(A)$ of Hilbert $A$-modules and compact operators. 
Its multiplier category is the category $\Hilb(A)$ of Hilbert $A$-modules and all adjointable operators \cite[Lem. 8.1]{cank}. The $G$-action on both categories is described explicitly in \cite[Ex. 2.10]{cank}.}

 {{If $A$ is unital, } then the associated $G$-$C^{*}$-category $\Hilb_{c}(A)^{(G)}_{\std}$
  consists of pairs $(H,\rho)$
of a Hilbert $A$-module together with a unitary   $G$-action $\rho$ such that 
$H$  is  isomorphic to an orthogonal sum of a family   of  finitely generated projective $A$-modules indexed by a free $G$-set.
Since $G$ acts non-trivially on $\Hilb_{c}(A)$ the details are slightly more complicated   to describe, see  Definition~\ref{qeroighjoreqgreqfefwfqwef}.}
\hB   \end{ex}

\subsubsection*{Analytic $K$-homology}
The construction of 
the equivariant  {analytic} $K$-homology functor $K^{G,\An}_{{\bC}}$ with coefficients in ${\bC}$
 employs the
  $\infty$-categorical version  $$\kkG \colon \Fun(BG,\nCalg)\to \KKG$$ of the $KK$-functor from \cite[Def.\ 1.8]{KKG} and its extension to $C^{*}$-categories
 \begin{equation}\label{werferfrewgrerefw}\xymatrix{ \Fun(BG,\nCalg)\ar[dr]_-{\incl}\ar[rr]^{\kkG}&&\KKG\\& \Fun(BG,\nCcat)\ar[ur]_-{\kkG_{\Ccat}}&}
\end{equation}
 introduced in \cite[Def. 1.29]{KKG}, where $\incl$ interprets a $G$-$C^{*}$-algebra as a $G$-$C^{*}$-category with a single object. The mapping spectrum functor of the stable $\infty$-category 
 $\KKG$ will be denoted by
 $$\KKG(-,-) \colon {\KKG}^{\op}\times \KKG\to \Sp\, .$$ In order to simplify the notation we drop the symbols $\kkG$ or $\kkG_{\Ccat}$  when we express the value of a functor $F$ defined on $\KKG$ on a $G$-$C^{*}$-algebra $A$ or a $G$-$C^{*}$-category $\bC$.
By \cite[Prop. 3.5]{KKG}, if $A$ is a  separable  $G$-$C^{*}$-algebra and $B$ is $\sigma$-unital, then the  homotopy groups
$ \pi_{*}\KKG( A , B)$  are canonically isomorphic to the classical equivariant $\KK^{G}$-groups  of Kasparov  \cite{kasparovinvent} associated to $A,B$.

The equivariant analytic $K$-homology functor $K^{G,\An}_{{\bC}}$ is  defined by the formula
\begin{equation}\label{vfdsvsvvsvsdvdsadsvdsva}
K^{G,\An}_{{\bC}} \colon \ppGTop \to \Sp^{\la}\, , \quad X\mapsto \KKG(C_{0}(X),\bQ_{\std}^{(G)})\, .
\end{equation} 
The domain of this functor is the category $\ppGTop$  of  locally compact {Hausdorff} $G$-spaces with partially {defined proper} maps. {Equivalently, $\ppGTop$ is the} Gelfand dual of the category $G\nCalg_{\comm}$ of non-unital  commutative $G$-$C^{*}$-algebras. 
 {The connection with the notation from \cite[Def.\ 1.15]{KKG} is given by
\begin{equation}\label{fbsfvssfdvsfdvdscsdf}
K^{G,\An}_{{\bC}}=K^{G,\an}_{ \bQ^{(G)}_{\std}}\, ,
\end{equation} 
  {In particular, $K^{G,\An}_{{\bC}}$ is different from $K^{G,\an}_{{\bC}}$ ---  we apologize for this notational inconvenience.  

{In view of \eqref{fbsfvssfdvsfdvdscsdf} the basic properties of $K^{G,\an}$ listed in   \cite[Thm.\ 1.15]{KKG}  imply corresponding properties of $K^{G,\An}_{\bC}$.} In particular, the functor $K^{G,\An}_{\bC}$ is homotopy invariant, is excisive for closed decompositions of second countable spaces with proper action (this restriction is due to the usage of \cite[Prop. 1.12.1]{KKG}), and it annihilates spaces of the form $[0,\infty)\times X$. 

\begin{ex}\label{wtrhiotrhrthtrhdtrh}
Let us consider the  coefficients $ \bC=\Hilb_{c}(A)$ for a unital $G$-$C^{*}$-algebra $A$.
{For a $G$-space $X$ which is homotopy equivalent to a $G$-finite CW-complex with finite stabilizers,   Proposition~\ref{wegojeogrregregwergwegre}}  {provides} a natural isomorphism
\begin{equation}\label{qefqwefwefewdq}
\pi_{*}K^{G,\An}_{\bC}(X)\cong \KKth^{G}_{*-1}(C_{0}(X),{A}) \, .
\end{equation}
This isomorphism {identifies our definition of equivariant {analytic}  $K$-homology with the classical definition {given} by 
 the right hand side of \eqref{qefqwefwefewdq}, up to a shift of degrees.} \hB
 \end{ex}

In order to deal correctly with non-$G$-compact spaces in $\ppGTop$ we will   consider
the  {locally finite} version $K^{G,\An,\mathrm{lf}}_{\bC}$ of $K^{G,\An}_{\bC}$ which is defined as follows. If $X$ is in $ \ppGTop$ and $U$ is an open $G$-invariant subset  of $X$ with $G$-compact closure, then we have a morphism
$X\to U$ in $ \ppGTop$ given by the partially defined map
$X\supset U\stackrel{\id_{U}}{\to} U$   {which corresponds to} the extension-by-zero homomorphism $C_{0}(U)\to C_{0}(X)$ {on the level} of commutative $G$-$C^{*}$-algebras. We define
\begin{equation}\label{wergeefeferfewrferf}
K^{G,\An,\mathrm{lf}}_{\bC}(X)\coloneqq\lim_{U\subseteq X}   K^{G,\An}_{\bC}(U)\, ,
\end{equation}  where the limit runs over all open subsets $U$ of $X$ with $G$-compact closure.
Using right Kan extensions, 
one can turn this prescription into the definition of a functor
%
\begin{equation}\label{regrwefefefewf}
K^{G,\An,\mathrm{lf}}_{\bC}\colon \ppGTop\to \Sp\, ,
\end{equation} 
see \cite[Sec. 7.1.2]{buen} for a similar construction.
We have a natural transformation \begin{equation}\label{qwefwefwqoiudjoqwdwedqwedweqdewd}
c \colon K^{G,\An}_{\bC}\to K^{G,\An,\mathrm{lf}}_{\bC}
\end{equation}
of functors from $\ppGTop$ to $\Sp$.
 The functor $K^{G,\An,\mathrm{lf}}_{\bC}$ is homotopy invariant. Its restriction 
to second countable spaces with proper $G$-action  is excisive for closed decompositions. Finally, it sends countable disjoint unions to products.
If $X$ is $G$-compact, then the canonical map
$c_{X} \colon K^{G,\An}_{\bC}(X)\to K^{G,\An,\mathrm{lf}}_{\bC}(X)$ is an equivalence. We refer to 
Proposition \ref{werigowergrefwerfwe} for a calculation of the  values of $K^{G,\An,\mathrm{lf}}_{\bC}$ on more general spaces.

{The functors $K^{G,\An}_{\bC}$ and $ K^{G,\An,\mathrm{lf}}_{\bC}$ depend functorially on the coefficient  $G$-$C^{*}$-category $\bC$  for non-degenerate morphisms.}
 
\begin{rem}
Using the equivariant $E$-theory functor \cite[Def. 3.22]{budu} one could define a  version of
analytic $K$-homology 
$$E^{G,\An}_{\bC}\colon \ppGTop\to \Sp\ , \quad X\mapsto \EE^{G}(C_{0}(X),\bQ_{\std}^{(G)})$$
with better formal properties.
Since the $E$-theory functor sends all exact sequences of $C^{*}$-categories to fibre sequences, in the case of $\bC=\Hilb_{c}(A)$ for a unital $G$-$C^{*}$-algebra $A$
we have the analogue of \eqref{qefqwefwefewdq}
$$E^{G,\An}_{\bC}(X)\simeq \Sigma \EE^{G}(C_{0}(X),A)$$   without any restriction on $X$.
Furthermore, the functor $ E^{G,\An}_{\bC}$ is excisive for arbitrary invariant closed decompositions, i.e., we can 
drop the assumptions of properness of the $G$-action and second countability needed for $K^{G,\An}_{\bC}$.
Finally, since the $E$-theory functor preserves filtered colimits of $G$-$C^{*}$-algebras, the functor
$E^{G,\An}_{\bC}$ is already locally finite, i.e., the analogue $E^{G,\An}_{\bC}\to E^{G,\An,\lf}_{\bC}$ of the comparison map
 \eqref{qwefwefwqoiudjoqwdwedqwedweqdewd} is an equivalence (see \cite[Prop. 3.30]{Bunke:2024aa} for an analogous statement).

The comparison functor $\KK^{G}\to \EE^{G}$ induces a transformation
$K^{G,\An}_{\bC}\to E^{G,\An}_{\bC}$ which is an equivalence  
on spaces which are homotopy equivalent to $G$-finite $G$-simplicial complexes with finite stabilizers.
Composing the Paschke morphism \eqref{Paschke-morphisms} below with this comparison map
we get a Paschke morphism with target  $E^{G,\An}_{\bC}$. Furthermore, 
 our main Theorem \ref{qreoigjoergegqrgqerqfewf} on the Paschke equivalence implies a similar  result 
involving $E^{G,\An}_{\bC}$. 

Here are our  three reasons to prefer  $K^{G,\An}_{\bC}$.
First of all this is the  analytic $K$-homology functor considered in the classical literature. 
Secondly,  working with  $K^{G,\An}_{\bC}$ provides a finer result. 
Finally, and this is our main reason, in the application to assembly maps 
we need reduced crossed products with $G$ which descend to 
 equivariant $KK$-theory, but not  to equivariant $E$-theory by the lack of exactness of $-\rtimes_{r}G$.
\hB
 \end{rem}

\subsubsection*{Coarse $K$-homology}
{
We now turn to a brief description of the equivariant   local  $K$-homology functor $K^{G,\cX}_{\bC}$. 
For our purposes, the functor $K^{G,\cX}_{\bC}$ is most naturally defined on {the category $G\UBC$ of} $G$-uniform bornological coarse spaces \cite[Def. 9.9]{equicoarse}.
{This category comes with a} cone-at-$\infty$ functor $\cO^\infty \colon G\UBC \to G\BC$ (see Definition \ref{wepogjpweggreewg}),  {where $G\BC$ denotes the category of $G$-bornological coarse spaces} \cite[Def. 2.1]{equicoarse}.  {We  define our equivariant local $K$-homology as the composition of $\cO^\infty$ with the}
equivariant coarse homology theory  $K\bC\cX^G_{G_{can,max}}\colon G\BC \to \Sp$. This functor is the twist (see Definition \ref{regoiergowregrwergwreg}) of  the equivariant coarse $K$-homology $
K\bC\cX^G\colon G\BC \to \Sp$ constructed in 
  \cite{coarsek} (see also Definition \ref{qrogijeqoifefewfefewqffe}) by the object $G_{can,max}$ in $G\BC$. 
  
  {In order to construct
 $K\cX\bC^{G}$ we must assume that  {the} coefficient  {$G$-$C^{*}$-category  $\bC$  {satisfies further axioms, namely} {that it is}
 effectively additive and admits countable AV-sums (see Definitions \ref{wigjosgbsgfgsdfgs91} and \ref{wigjosgbsgfgsdfgs9}).}
 {The coefficient category $\Hilb_c(A)$ for a $G$-$C^*$-algebra $A$ satisfies these axioms by  {\cite[Lem. 7.9]{cank} since it admits all small AV-sums}.  

We   define  the equivariant   local  $K$-homology functor by
\begin{equation}\label{qewfoijfoiqwefqewfewfqewfd}
K^{G,\cX}_{\bC} \coloneqq K\bC\cX^G_{G_{can,max}}\circ \cO^\infty \colon G\UBC \to \Sp\,.
\end{equation}
This composition is an equivariant local homology theory, i.e.\ it is homotopy invariant, excisive for closed decompositions, {$u$-continuous,} and vanishes on spaces of the form $[0,\infty) \otimes X$, see Proposition \ref{qroigjqowfewfqwefqw}. 

{The functor $K\bC\cX^{G}$ and therefore also $K^{G,\cX}_{\bC}$  depend also functorially
 on the  coefficient category $\bC$ for non-degenerate morphisms.}

\subsubsection*{A common domain for $K^{G,\An}_\bC$ and $K^{G,\cX}_\bC$}
By now, the functors $K^{G,\An}_{\bC}$ and $K^{G,\cX}_\bC$ can not be compared. They are invariants of different objects: locally compact Hausdorff $G$-spaces on the one hand{,} and $G$-uniform bornological coarse spaces on the other hand.  In order to compare their domains we  consider the functor   $$\iota^\topp:G\UBC  \to \ppGTop $$   from  \eqref{q3oigfherqoifqrfvffqwefqeeqwfqeeqwfqefqefeeqfeqfewffefqewqewfqfqewf}. It is uniquely characterized by the equalities
\begin{equation}\label{werfwefwerffffrf}C_{0}(X)=C_{0}(\iota^{\topp}(X))
\end{equation}
for all 
 $X$ in $G\UBC$, where the $C^{*}$-algebra $C_{0}(X)$ on the left-hand side consists of  the bounded uniformly continuous  functions
which become arbitrary small outside of sufficiently large bounded subsets. The symbol $C_{0}(\iota^{\topp}(X))$ has the usual meaning.

We let  $G\Simpl_{\mathrm{fin}}^\proper$   denote the category  of $G$-finite $G$-simplicial complexes with finite stabilizers and equivariant proper simplicial maps. Equipping $G$-simplicial complexes with the spherical path metric provides a functor
$$G\Simpl_{\mathrm{fin}}^\proper\to G\UBC\ .$$

We can summarize our first main result, slightly informally, by the following diagram.
\[\begin{tikzcd}
	& G\Simpl_{\mathrm{fin}}^{\proper}  \ar[dl] \ar[dr] & \\
	G\UBC \ar[ddr,"{K^{G,\cX}_\bC}"', bend right]    \ar[rr,"\iota^\topp"] && \ppGTop \ar[ddl,"K^{G,\An}_{\bC}", bend left] \\
	 & \stackrel{p}{\Longrightarrow} & \\
	 & \Sp &
\end{tikzcd}\]
The Paschke transformation $p$ will be constructed as a natural transformation filling the lower triangle.  Equivalently, naturality  of $p$ can be stated by saying that it makes the lower square \uliemph{lax-commutative}. We then show that the Paschke transformation renders 
 the large square commutative. In other words, the Paschke transformation becomes a natural equivalence when 
 restricted to $G$-finite   $G$-simplicial complexes with finite stabilizers. {In addition, the Paschke transformation is natural in the coefficient category $\bC$ for non-degenerate morphisms.} We will state our main theorem more formally as Theorem~\ref{qreoigjoergegqrgqerqfewf} below.

\subsubsection*{A review of classical Paschke duality}

{In order to motivate the definitions involved in the above diagram, we now review some aspects of classical Paschke duality.}
Based on the seminal work of Paschke  \cite{paschke}, the general theme of Paschke duality is to express the {analytic} 
  $K$-homology   $$K_{*}^{\mathrm{an}}(X)\coloneqq \KK_{*}(C_{0}(X),\C)$$ in terms of the $K$-theory
  of a $C^{*}$-algebra naturally associated to $X$, which is then often referred to as the \uliemph{Paschke dual algebra} of $X$.

Classically, this is implemented as follows.
Let $X$ be a proper metric space and $\phi \colon C_{0}(X)\to B(H)$ be  a homomorphism of $C^{*}$-algebras, where  $H$ is a separable Hilbert space. To this data one associates  an exact sequence of $C^{*}$-algebras 
\begin{equation}\label{wergpowkgpowergregregwegrg}
0\to C(H,\phi)\to D( H,\phi)\to Q (H,\phi)\to 0
\end{equation}
where $D(H,\phi)$ is the $C^*$-subalgebra of $B(H)$ generated by the controlled and pseudolocal operators and $C(H,\phi)$,  called the \uliemph{Roe algebra}, is its ideal generated by the operators which are in addition locally compact.

{If  $(H,\phi)$  is 
 sufficiently large ({very ample 
 in classical terminology or} 
 absorbing in the sense of Definition \ref{werogihetgergfdsf}) {and non-degenerate (meaning that $\overline{\phi(C_{0}(X))H}=H$)},  then the $K$-theory of $Q(H,\phi)$ is a well-behaved invariant of $X$. More precisely, for a proper map $f \colon X \to X'$ and absorbing {non-degenerate}  representations $(H,\phi)$ and $(H',\phi')$ for $X$ and $X'$ respectively, there exists a unitary, controlled and pseudolocal {isometry} $(H',\phi')\cong (H,\phi\circ f^{*})$ {called a \uliemph{covering}, which is} unique up to conjugation by unitaries {in $D(H',\phi')$}. This covering induces  a homomorphism  $D(H,\phi)\to D(H',\phi')$ preserving the respective Roe algebras and therefore a homomorphism
 $Q(H,\phi)\to Q(H',\phi')$ between the quotients. For $f=\id_X$, this shows that the $K$-theory of $Q(H,\phi)$ is independent of the choice of an absorbing representation $(H,\phi)$. {We recall here that Voiculescu's Theorem grants the existence of such absorbing representations.}  {Furthermore,}  setting }
\[K_{*}^{\cX}(X) \coloneqq \Kast_{*}(Q(H,\phi))\]
{for {any choice of an} absorbing {non-degenerate} representation $(H,\phi)$, one obtains a functor}
\[K_{*}^{\cX}(-) \colon \Met^{\proper}\to \Ab^{\Z}\]
defined on the category of proper metric spaces and proper maps and taking values in graded abelian groups.  The superscript $\cX$  indicates the coarse geometric origin of the construction, whose implementation was
 initiated by Roe \cite{roe_exotic_cohomology}.
The functor $K_*^{\cX}(-)$ is homotopy invariant {and} admits Mayer--Vietoris sequences. In addition, there is a natural  Paschke duality isomorphism 
\begin{equation}\label{qwefqoihfoiqwefewfqwefef}
K_*^{\cX}(X) \cong K^\an_{*-1}(X)
\end{equation}
given by a concrete cycle level construction, see \cite{higson_roe} for details. 
So up to suspension $Q(H,\phi)$ is the Paschke dual of $C_{0}(X)$.

\subsubsection*{The Paschke {transformation} following Quiao--Roe}

The paper \cite{quro} discusses a systematic approach to the isomorphism \eqref{qwefqoihfoiqwefewfqwefef}, whose basic idea we now adapt to the equivariant situation. We continue to assume that the $G$-space $X$ is equipped with an absorbing {non-degenerate} representation $\phi\colon C_0(X) \to B(H,\rho)$ where $H$ is a separable Hilbert space equipped with a unitary $G$-action $\rho$.
The idea is to derive the isomorphism in \eqref{qwefqoihfoiqwefewfqwefef} from 
a multiplication map
\begin{equation}\label{efqwefewdq}
\mu_{X} \colon C_{0}(X)\otimes Q^{G}(X)\to Q(H ) \, ,
\end{equation}
 $$Q^{G}(X) \coloneqq Q^{G}(H,\rho,\phi) \coloneqq D^{G}(H,\rho,\phi)/C^{G}(H,\rho,\phi)\, ,$$
 where $D^{G}(H,\rho,\phi)$  and $C^{G}(H,\rho,\phi)$ are defined as in the non-equivariant case by just adding the condition that the controlled generators are {$G$}-invariant. Furthermore 
  $Q(H) = Q(H,\rho)$ is the Calkin algebra of $(H,\rho)$ with the induced $G$-action.
Using the multiplication map \eqref{efqwefewdq}, one may define a Paschke {morphism} as the composition 
\begin{equation}\label{rfqwpofkjpqowefewfeqwfqef}
 p^{(H,\rho,\phi)}_{X} \colon \KK (\C,Q^{G}(X))\xrightarrow{\delta_{X}} \KK^{G}  (C_{0}(X),C_{0}(X)\otimes Q^{G}(X))\xrightarrow{\mu_{X}} \KK^{G} (C_{0}(X),Q(H) )  \, .
\end{equation}
The map  $\delta_{X} \coloneqq C_{0}(X)\otimes -$ is the exterior product in equivariant KK-theory and is called the diagonal morphism. We note that the algebras $Q^{G}(X )$ and $Q (H)$ are not separable, which is the reason  why $E$-theory instead of $\KKth$-theory is used in \cite{quro}. However, the equivariant KK-theory of \cite{KKG} is well-defined for all $G$-$C^{*}$-algebras, so we can safely work with this version rather than with $E$-theory. 

With this more abstract definition, how can one show that the Paschke {morphism} induces an isomorphism on $K$-groups, at least for suitable spaces $X$?  Our strategy to answer this question is as follows. 
Suppose one could show that the maps $p^{(H,\rho,\phi)}_{X}$ in \eqref{rfqwpofkjpqowefewfeqwfqef} were the components of a natural transformation of functors with values in the $\infty$-category of spectra,
and that both the domain and target of the  Paschke {transformation} are homotopy invariant and excisive\footnote{This is the spectrum analogue of the property of admitting Mayer--Vietoris sequences  for group-valued functors} as functors in $X$.
Then for {$G$}-finite $G$-CW-complexes $X$, by induction over the number of $G$-cells, one can reduce the verification  that $p_X^{(H,\rho,\phi)}$ is an equivalence 
to the  cases of $G$-orbits, i.e., of spaces of the form $G/H$, where $H$ runs over the  subgroups of $G$ appearing as   stabilizer of the $G$-action on $X$. {While in the non-equivariant case only the trivial case $X=\ast$ is to be treated, the verification that the Paschke maps are equivalences on general $G$-orbits is a non-trivial matter.}

{The above strategy} will indeed be the essential idea of the proof of our main Theorem \ref{qreoigjoergegqrgqerqfewf} below. The first difficulty to overcome is to show that the Paschke maps $p^{(H,\rho,\phi)}_{X}$ are indeed the components of a natural transformation, in particular, to show that the spectrum $\KK(\C,Q^{G}(X))$ appearing in the domain of the Paschke map, is a homotopy invariant and excisive functor in $X$ (at the moment  is not even a functor in any obvious manner).
The origin of the problem  is that in order to define $Q^{G}(X)=Q^{G}(H,\rho,\phi)$, one has to \emph{choose} an absorbing {non-degenerate} representation $(H,\rho,\phi)$, and for a morphism $X\to X'$ one  has to \emph{choose}  a covering in order to define the map 
$\KK (\C,Q^{G}(X))\to \KK (\C,Q^{G}(X'))$. 
Defined in this way, the resulting  map of  spectra depends on these choices and is, at best, unique up to an unspecified homotopy, which is not sufficient for our purposes.

\subsubsection*{The Paschke {transformation} in our setup}

Our key idea to overcome these functoriality issues is to work with the category of all representations.  In fact, the categories  of such representations themselves depend on the space in a strictly functorial manner. Their use hence circumvents the need to find absorbing representations.
The idea  to work with the whole category of representations is not new; it has first been exploited in \cite{buen} in order to define a spectrum-valued coarse $K$-homology functor $K\!\cX$. 

In the present paper, as indicated earlier, we work with its equivariant generalization,  the equivariant  coarse $K$-homology functor 
\[K\bC\cX^{G} \colon G\BC\to \Sp\] 
introduced in \cite{coarsek}. 
Again, the symbol $\bC$ refers to its dependence on a coefficient  {$G$-$C^*$-category $\bC$}. 
In the case of trivial coefficients it is shown in  \cite[Thm.\ 6.1]{indexclass} that this functor is equivalent to the classical definition of equivariant coarse $K$-homology in terms of Roe algebras. More precisely, if   the    $G$-space $X$ is nice, and $C^G(X) \coloneqq C^G(H,\rho,\phi)$ with $(H,\rho,\phi)$ ample,
we have a natural equivalence
\[ K\bC\cX^{G}(X) \simeq \Kast(C^{G}(X)) \, .\]

By construction, see Definition~\ref{qrogijeqoifefewfefewqffe}, for $X$ in $G\BC$  we have $$K\bC\cX^{G}(X)=\KK(\C,\bCgtsmc(X)) \ ,$$ where  
$\bCgtsmc(X)$ is a $C^{*}$-category  of  equivariant locally finite $X$-controlled objects in {$\bC$}, see Definition~\ref{rfquhwfiuqwhfiufewqefqwefqwefwefqwef} for the details.
The endomorphism algebras of the objects of $ \bCgtsmc(X)$  are natural analogues
of the Roe algebras $C(H,\rho,\phi)$.

We now indicate the relation between the functor $X \mapsto K^{G,\cX}_\bC(X)$ and the association $X \mapsto \KK(\C,Q^G(X))$ appearing in the source of the Paschke {morphism} \eqref{rfqwpofkjpqowefewfeqwfqef}. Recall from \eqref{qewfoijfoiqwefqewfewfqewfd} that $K^{G,\cX}_\bC$ is defined as a composition of $K\bC\cX^{G}$ with the functor $\cO^\infty(-)\otimes G_{can,max}$ on $G$-uniform bornological coarse spaces.

 If $X$ is in $G\UBC$, then the cone-at-$\infty$ $\cO^{\infty}(X)$ is the $G$-set
 $\R\times X$ with a certain $G$-bornological coarse structure described in  Definition \ref{qriugiqregergwerg}. 
 It contains the underlying $G$-bornological coarse space of $X$ as the subspace $\{0\}\times X$. We  further consider the cone
 $\cO(X)$ in $G\BC$ defined as the subset $[0,\infty)\times X$ with the induced structures.
The inclusion $X\to \cO(X)$ induces an inclusion of categories \begin{equation}\label{r3foifjowqefewqfqewfef}
\bCgtsmc(X\otimes G_{can,max})\to \bCgtsmc(\cO(X)\otimes G_{can,max})\, 
\end{equation}
to be thought of as the analog of the inclusion $C^G(X) \to D^G(X)$ in the classical situation, see Section \ref{weogiwjiergreggwrgwerg} for more details.
  The resulting quotient $C^{*}$-category  $\bQ(X)$
 is then our version of the algebra $Q^{G}(X)$, and we have natural equivalence
\begin{equation}\label{gepojgregergegwgwergrweg}
 K^{G,\cX}_{\bC}(X)\simeq \KK(\C, \bQ(X))\, .
\end{equation}   
   We refer to 
   Lemma \ref{wrtoihgjoergergegwgrgwerg} for more details and necessary additions.  We construct
   a
 multiplication map (see \eqref{oreihoijfvoisfdvervfdsvdfvsvv})
 \[ \mu_{X} \colon C_0(X) \otimes \bQ(X) \to \bQ^{(G)}_{\std}.\]
In complete analogy to the earlier described Paschke morphism \eqref{rfqwpofkjpqowefewfeqwfqef}, we define our version of the Paschke morphism as the composition:
\begin{equation}\label{Paschke-morphisms}
	p_{X} \colon \KK(\C,\bQ(X)) \stackrel{\delta_X}{\lto} \KKG(C_0(X),C_0(X)\otimes \bQ(X)) \stackrel{\mu_{X}}{\lto} \KKG(C_0(X),\bQ^{(G)}_{\std})\,.
\end{equation}

 The main result of this paper is then the following theorem. 

 \begin{theorem}\label{qreoigjoergegqrgqerqfewf}
 {We assume that $\bC$  is effectively additive and admits countable AV-sums.} \begin{enumerate}
 \item\label{weiufhiqwefewwfqewfewf}
 
The morphisms in \eqref{Paschke-morphisms} assemble into a natural transformation of  spectrum-valued functors on $G\UBC $
\begin{equation}\label{qwoeifjoifjewfqwefwfqewfq}
p \colon K_{\bC}^{G,\cX} \to K_{\bC}^{G,\An}\circ \iota^{\topp}\,
\end{equation}
that   is natural in the coefficient category $\bC$ for non-degenerate morphisms.
\item\label{qregiojqwfewfqwfqewf} If 
$X$ is in $G\UBC $ and homotopy equivalent to a $G$-finite
  $G$-simplicial complex with finite {stabilizers, 
 then} 
 $$p_{X} \colon K_{\bC}^{G,\cX}(X)\to K_{\bC}^{G,\An}(\iota^{\topp}(X))$$
 is an equivalence.
 \item \label{wrtigjiogowrefwerfewrf}
 {If  $\bC$ admits all very small AV-sums,} $G$ is finite, 
$X$ is in $G\UBC $ and homotopy equivalent to a countable finite-dimensional
  $G$-simplicial {complex, 
 then} 
 $$  p^{\mathrm{lf}}_{X} \colon K_{\bC}^{G,\cX}(X)\to K_{\bC}^{G,\An,{\mathrm{lf}}}(\iota^{\topp}(X))$$
 is an equivalence.
\end{enumerate}
 \end{theorem} 
{We  refer  {again} to Definitions~\ref{wigjosgbsgfgsdfgs9} and \ref{wigjosgbsgfgsdfgs91} for the conditions  on $\bC$ appearing in the statement above,}  {and recall   that the coefficient category $\Hilb_c(A)$, for $A$ a $G$-$C^*$-algebra, satisfies these conditions.}
 {In Assertion \ref{qreoigjoergegqrgqerqfewf}.\ref{wrtigjiogowrefwerfewrf}  we use the transformation  $c \colon K^{G,\An}_{\bC}\to K^{G,\An,\mathrm{lf}}_{\bC}$ from \eqref{qwefwefwqoiudjoqwdwedqwedweqdewd}
 and set $p^{\mathrm{lf}} \coloneqq c\circ  p$.}
 \begin{ddd}\label{troigjwrtoijgergrweg}
The transformation $p$ in \eqref{qwoeifjoifjewfqwefwfqewfq} is called the Paschke {transformation}.
 \end{ddd}
 
The proof of Assertion  \ref{qreoigjoergegqrgqerqfewf}.\ref{weiufhiqwefewwfqewfewf} will be finished in Section \ref{eroigheifogegegergewgerg}, and  the proof of Assertions   \ref{qreoigjoergegqrgqerqfewf}.\ref{qregiojqwfewfqwfqewf}  and   \ref{qreoigjoergegqrgqerqfewf}.\ref{wrtigjiogowrefwerfewrf}   will be completed in Section \ref{wefgiojwoegergregegw}.
Once $p$ is constructed, which is not at all trivial, the verification that it is an equivalence under additional conditions follows the route described above, i.e.\ by reducing it to the case of orbits.  The verification that $p$ is indeed an equivalence on $G$-orbits with finite stabilizers also turns out to be quite involved and uses a lot of the properties of the $K$-theory functor for $C^{*}$-categories obtained in \cite{cank}.

In the case of  trivial coefficients and under the assumption of the existence of an absorbing representation $(H,\rho,\phi)$ we can compare the version of the Paschke morphism $p^{(H,\rho,\phi)}_{X}$ from \eqref{rfqwpofkjpqowefewfeqwfqef} with the newly defined Paschke morphism $p_{X}$ from \eqref{Paschke-morphisms} (in particular their domains): Indeed, in Proposition \ref{iuhugiergwergergwg} we show that there is a commutative diagram
\[ \begin{tikzcd}
	K_{\bC}^{G,\cX}(X) \ar[r,"\gamma"] \ar[d,"{p_{X}}"] & \KK(\C,Q^{G}(X))\ar[d,"{p_{X}^{(H,\rho,\phi)}}"]  \\
	K_{\bC}^{G,\An} (X) & \ar[l,"\simeq"]  \KK^{G}(C_{0}(X) ,Q(H))  
\end{tikzcd}\]
so that, under the assumption that $p_{X}$ is an equivalence, $\gamma$ is an equivalence if and only if $p_X^{(H,\rho,\phi)}$ is.

\subsection*{Assembly maps}

Our original motivation to show the Paschke duality theorem above
was the wish to write out a complete proof for the fact the  homotopy theoretic assembly map  of Davis--L\"uck \cite{davis_lueck}
and the analytic assembly map appearing in the Baum--Connes conjecture 
  are equivalent. 
 Such an equivalence was asserted in \cite{hamped}, but the details of the proof given in this reference
 remained sparse. {While we were preparing this paper, a comparison of the two assembly maps was recently also carried out by Kranz \cite{kranz} with methods different from ours, see Remark~\ref{Remark:comparison-assembly}.}
 
 Homotopy theoretic assembly maps are generally defined for any equivariant homology theory $G\Orb\to \bM$ with 
 cocomplete target $\bM$ and a family $\cF$ of subgroups, see  Definition~\ref{woiejtgoiwegegrgergwer}. Our comparison concerns
 the functor  
 \begin{equation}\label{asdvijqoirgfvfsdvsva}
K\bC^{G} \colon G\Orb\to \Sp^{\la} \, , \quad S\mapsto  K\bC\cX^{G}_{G_{can,min}}(S_{min,max})\, ,
\end{equation}
see Definition \ref{qroeigjoqergerqfewewfqewfeqwf}. 
 Note that the twist is different from the one used in the Definition~\eqref{qewfoijfoiqwefqewfewfqewfd} of $K^{G,\cX}_{\bC}$, namely {it is  $G_{can,min}$ rather than  $G_{can,max}$}.
 {For appropriate choice of coefficients {$\bC$}, the  functor $K\bC^{G}$ is  equivalent to the functor introduced by Davis--L\"uck, see
 Remark~\ref{oiejgoigsvfdvfvs}.}

The equivariant homology theory $K\bC^{G}$  canonically extends to a functor
$$K\bC^{G} \colon G\Top\to \Sp^{\la}$$ denoted by the same symbol, see Definition \ref{weiogwegerffs}. 
For any family of subgroups $\cF$ of $G$  the homotopy theoretic  assembly map {can be described as the map}
\[\Ass^{h}_{{\bC,\cF}} \colon K\bC^{G}(E_{\cF}G^\cw)\to K\bC^{G}(*)\] 
induced by the projection $E_{\cF}G^\cw\to *$, where $E_{\cF}G^\cw$ is a $G$-CW-complex representing the homotopy type of the classifying space of $G$ with respect to the family $\cF$.   

{For the following we assume that  $\cF\subseteq \Fin$.
We define}  
\[ RK^{G,\An}_{\bC} (E_{\cF}G^\cw) \coloneqq \colim_{W\subseteq E_{\cF}G^\cw } K^{G,\An}_{\bC}(W),\]
where the colimit runs over the $G$-finite subcomplexes of $E_{\cF}G^\cw$. {In Definition \ref{wergoijowergerreggwgw} we} construct an 
 analytic assembly map \begin{equation}\label{eqwfqwefewfwdqwedqwq}
 \Ass^{\an}_{{\bC},\cF} \colon RK^{G,\An}_{\bC} (E_{\cF}G^\cw) \to  \Sigma \KK(\C,  {\bC}^{(G)}_{\std}\rtimes_{r} G)\, ,
\end{equation}
 where the $C^{*}$-category  ${\bC}^{(G)}_{\std}$ is defined in Definition \ref{qeroighjoreqgreqfefwfqwef}  and the reduced crossed product for $C^{*}$-categories is as introduced in \cite[Thm. 12.1]{cank}. 

 The assembly maps  $\Ass^{h}_{{\bC,\cF}}$ and   $\Ass^{\an}_{{\bC},\cF}$ depend naturally
 on  the coefficient category $\bC$  for non-degenerate morphisms.

{In  Definition \ref{wergoijowergerrrfrfrfrfeggwgw} we construct
a spectrum-valued version of  the classical Kasparov assembly map \begin{equation}\label{rerwwerfefefwef}
\mu^{\Kasp}_{A,\cF} \colon RK^{G,\an}_{A}(E_{\cF}G^\cw)\to \KK(\C,A\rtimes_{r}G)
\end{equation}
which functorially depends on $A$ in $\KKG$.}
{We consider the spectrum-valued refinement {\eqref{rerwwerfefefwef}} of Kasparov's assembly map  as an interesting result in its own right. {In view of the definition of the domain, one has to construct  a family of such  assembly maps  indexed by the $G$-finite subcomplexes $W$
 of $E_{\cF}G^\cw$ which is {compatible}   with inclusions}. While it is easy to lift Kasparov's construction to a map of spectra for each such $W$ individually, and it is {also} easy to obtain the {required} compatibility on the level of homotopy groups, it is a non-trivial matter {to enhance the   compatibility to}  the spectrum level. {We obtain this enhancement in the form of the natural transformation \eqref{rgfewrfwweferf}.}
 }



{For a $G$-$C^{*}$-category $\bC$ let $\bC^{u}$ denote the full unital $G$-$C^{*}$-subcategory of unital objects. In} Proposition \ref{weiojgwoerferww} we show the following comparison 
result.}
 \begin{prop} \label{weokjgpwefrewfwerf}We have an equivalence between the assembly maps 
$ \Ass^{\an}_{{\bC},\cF} $ from \eqref{eqwfqwefewfwdqwedqwq}
and ${\Sigma}\mu^{\Kasp}_{ (\bC^{u})^{(G)},\cF}$  from \eqref{rerwwerfefefwef}. 
\end{prop}

\begin{ex}
In the case of a unital $G$-$C^{*}$-algebra $A$ and for $\bC:=\Hilb_{c}(A)$ it follows from \eqref{fdvsfdvvsfdvav} and Proposition \ref{weokjgpwefrewfwerf} that 
 the assembly map  $\Ass^{\an}_{{\bC},\cF}$ is equivalent to 
${\Sigma}\mu^{\Kasp}_{A,\cF}$.
\hB
\end{ex}

 The following theorem (whose proof will be completed at the end of Section \ref{wtogpwgregwegreg})
now provides a comparison of the Davis--L\"uck and  Baum--Connes assembly maps on the level of homotopy groups.  As indicated earlier, a version of this result has recently been shown also by \cite{kranz} with completely different methods.

 \begin{theorem}\label{wtoiguwegwergergregwe}{We assume that $\bC$ is effectively additive and admits countable AV-sums.} 
{We} have  a commutative square
\begin{equation}\label{avvoijfoivjoafvasddsvdvasdv}
\xymatrix{
K\bC_{*}^{G}(E_{\cF}G^\cw)\ar[rr]^-{\pi_{*}\Ass_{{\bC},\cF}^{h}} && K\bC_{*}^{G}(*)\ar[d]^{\cong}\\
RK^{G,\An}_{\bC,*+1} (E_{\cF}G^\cw)\ar@{-}[u]^{\cong}\ar[rr]^{\pi_{*+1}\Ass^{\an}_{{\bC},\cF}} && \KK_{*}(\C,   {\bC}^{(G)}_{\std}\rtimes_{r} G)
}
\end{equation}
in which all terms are natural in $\bC$  for non-degenerate morphisms.
   \end{theorem}      The left vertical equivalence  in \eqref{avvoijfoivjoafvasddsvdvasdv} is, in a non-obvious manner, a consequence of our Paschke Duality Theorem~\ref{qreoigjoergegqrgqerqfewf}. 
{{If $A$ is a  $G$-$C^{*}$-algebra, then   $\bC:=\Hilb_{c}(A)$ 
admits all small AV-sums (this follows from \cite[Thm. 8.4]{cank}) 
 and} hence satisfies the assumption of Theorem \ref{wtoiguwegwergergregwe}.}

{We believe that our method can be upgraded to provide a commutative diagram on the spectrum level, but carrying this out would involve to control further large coherence diagrams. We refrain from doing this additional step at this point, but emphasize that the passage to a statement about homotopy groups is really only in the very final step where one filters $E_{\cF}G^\cw$ through $G$-finite subcomplexes. For any $G$-finite $X$ in place of $E_{\cF}G^\cw$, the diagram in Theorem~\ref{wtoiguwegwergergregwe} commutes already before applying homotopy groups. {In particular,  the square in \eqref{avvoijfoivjoafvasddsvdvasdv}} 
commutes before applying homotopy groups when there is a $G$-finite model of $E_\cF G^\cw$. It is just that we have not worked out that the homotopies for varying $X$ can be obtained in a compatible way. This problem is not visible to homotopy groups, and hence one obtains Theorem~\ref{wtoiguwegwergergregwe} irrespective of this issue.}

We note that it is important to consider the reduced crossed product in the target for the approach presented here. 
While the construction of the analytic assembly map easily lifts to the maximal crossed product our method unfortunately does not generalize to produce the corresponding comparison of assembly maps also for the maximal crossed product.

 \subsection*{Further remarks}
   
Finally, we explain some relations to previous works on (equivariant) Paschke duality and the analytical assembly map.
We begin with  Paschke duality.

  \begin{rem} 
 {Valette established} a non-commutative generalization of the classical Paschke duality \cite{Val83}  {whose statement we briefly recall here}. 
 We  consider  a $C^{*}$-algebra $B$ with a strictly positive element. Then we have an exact sequence
 $$0\to B\otimes K(\ell^{2})\to \cM^{s}(B)\stackrel{\pi}{\to}  \cQ^{s}(B)\to 0\, ,$$
 where $ \cM^{s}(B)$ is the stable multiplier algebra and the stable Calkin algebra   $\cQ^{s}(B)$ is defined as the quotient.
  In place of $\phi\colon C_{0}(X)\to B(H)$ above we now consider a
 unital separable nuclear $C^{*}$-algebra $A$ with a representation
  $\tau\colon A \to  B(\ell^{2})$ such that $\tau(A)\cap K(\ell^{2})=\{0\}$ and set
  $\phi\colon A\xrightarrow{1\otimes \tau}  \cM^{s}(B)\stackrel{\pi}{\to} \cQ^{s}(B)$.
  We further  replace $Q(H,\phi)$ from above by  the commutant 
  $Q(A,\phi,B)\coloneqq\phi(A)'$ of the image of $\phi$.  The proof of the following result employs Kasparov's generalization of Voiculescu's Theorem.
  \begin{prop}[{\cite[Prop. 3]{Val83}}]
  We have an isomorphism
  $$\KK_{*}(\C,Q(A,\phi,B))\cong \KK_{*-1}(A,B)$$
  which is natural in $A$ and $B$.
  \end{prop}
 In this statement  $\KK_{*}$ denote Kasparov's $\KK$-groups. Note  that  
  the right-hand side in the  original statement of Valette is   expressed  in terms of $Ext$-groups
  which are isomorphic to the $\KK_{*}$-groups under the given assumptions on $A$ and $B$.
 If $B$ is in addition  $\sigma$-unital, then by \cite[Prop.\ 1.20]{KKG} the $\KK$-group on the right-hand side   coincides with the 
 $KK$-group  obtained from the spectrum-valued $\KK$-theory constructed in \cite{KKG}.
 
See also \cite[Thm. 3.2]{thabs} for a related result.
\hB
\end{rem}

\begin{rem}\label{Remark:previous-work}
Our Theorem \ref{qreoigjoergegqrgqerqfewf}  is similar in spirit to  \cite[Thm.\ 1.5]{Benameur:2020aa}.   
But while Theorem \ref{qreoigjoergegqrgqerqfewf}   produces a natural transformation between spectrum-valued functors which becomes an equivalence when evaluated on spaces satisfying suitable finiteness conditions,
 \cite[Thm.\ 1.5]{Benameur:2020aa} states an isomorphism between $K$-theory groups for a fixed space. While the class of spaces to which 
  \cite[Thm.\ 1.5]{Benameur:2020aa} applies is larger than  the class of spaces for which 
   Theorem \ref{qreoigjoergegqrgqerqfewf} provides an equivalence, our theorem allows to treat more general coefficients.

But even in the case where 
both theorems are applicable the   technical details of their statements are quite different so that at the moment it is difficult to compare them   in a precise way.  
In the following we explain this problem in greater detail.

The space $X$ in \cite[Thm.\ 1.5]{Benameur:2020aa} (denoted by $Z$ in the reference) is a  metric space with an isometric proper  cocompact action of $G$.
In order {to} fit into our theorem we must require   that it is homotopy equivalent to a $G$-finite $G$-simplicial complex.
  The domain of the Paschke map  in    \cite[Thm.\ 1.5]{Benameur:2020aa} is the $K$-theory of a certain $C^{*}$-algebra $Q^{G}(H,\rho,\phi)$, where $H$ is a sufficiently large  Hilbert $C^{*}$-module over a commutative unital $C^{*}$-algebra $A$.  In order to compare with our theorem
  we would restrict the coefficients to the special case {$\bC= \Hilb_{c}(A)$.}
  We then could ask whether {we have}
  $$\Kcat_{*}(\bQ(X))\cong \Kast_{*}(Q(H,\rho,\phi))\, ,$$
  see \eqref{gepojgregergegwgwergrweg}. 
  The construction of a comparison map could proceed similarly as the construction
  of the map $\gamma$ in Proposition \ref{iuhugiergwergergwg} once we know that 
  $(H,\rho,\phi)$ is absorbing in the sense of   the natural generalization of Definition \ref{werogihetgergfdsf} to controlled Hilbert $A$-modules.
 
 On the positive side, in the case {$\bC= \Hilb_{c}(A)$},
 the targets  of the two Paschke duality   
   maps in    \cite[Thm.\ 1.5]{Benameur:2020aa} and {Theorem} \ref{qreoigjoergegqrgqerqfewf}
   are equivalent in view of 
   $$K^{G,\An}_{\bC}(X)\simeq \KKG(C_{0}(X),\bQ^{(G)}_{\std})
 \stackrel{\text{Prop.\ }\ref{wegojeogrregregwergwegre}}{\simeq} \Sigma \KK^{G}(C_{0}(X),A)
$$
provided $X$ is homotopy equivalent to a $G$-finite $G$-CW-complex.
   \hB
\end{rem}

\begin{rem}\label{Remark:comparison-assembly}
As mentioned earlier, in \cite{kranz} Kranz also provides an identification of the Davis--L\"uck assembly map and the Kasparov assembly map. In fact, the contribution of Kranz is a comparison of    the Davis--L\"uck assembly map  with the version of the assembly map  introduced by    Meyer--Nest \cite{MR2193334}.  The latter is compared in  \cite{MR2193334} with Kasparov's assembly map employing work of 
Chabert--Echterhoff \cite{MR1836047}.  In Section \ref{thelast} 
 we will give a detailed account of   the argument of Kranz  using the  $\infty$-categorical language of equivariant $\mathrm{KK}$-theory developed in \cite{KKG}. 
 As an application, {in   Theorem \ref{wtkogwegerfwerf}  we give an argument}   {(which is independent of Chabert--Echterhoff \cite{MR1836047})} that the  Kasparov assembly map   is an equivalence for compactly induced coefficient categories or algebras.
\hB
\end{rem}
 
\section{Constructions with \texorpdfstring{$\boldsymbol{C^{*}}$}{Cstar}-categories}\label{section:constructions}

In order to fix size issues we choose a sequence of four Grothendieck universes whose sets will be called very small, small, large,  and very large, respectively.  The group $G$, bornological coarse spaces or $G$-topological spaces  belong to  
the very small universe.  The categories of these objects,   the coefficient  $C^{*}$-categories,  the categories of controlled objects, and the 
values {of}  the $K$-theory functor $\Kcat$ will
belong to the small universe. The categories of spectra $\Sp$ and $\KKG$
are large, but locally small. They are objects of a category of stable $\infty$-categories $\CAT^{ex}_{\infty}$ which is itself very large.

{We let $\Fun(BG,\nCcat)$ denote the category of small  {not necessarily unital} $C^{*}$-categories with $G$-action  and equivariant functors, and $\Fun(BG,\Ccat)$ be the subcategory of unital $C^{*}$-categories and  functors  preserving  units.}  Both versions of  {$K$-homology} considered in the present paper depend on the choice of  {a coefficient $C^{*}$-category
$\bC$  in $\Fun(BG,\nCcat)$.} 


\begin{ex}\label{weriogwergw} 
{We let $\Fun(BG,\nCalg)$ be the full subcategory of $\Fun(BG,\nCcat)$ of
$C^{*}$-algebras with $G$-action considered as single object categories.  We furthermore set $$\Fun(BG,\Calg):=\Fun(BG,\nCalg)\cap \Fun(BG,\Ccat)\ .$$}Our basic example  {of a coefficient category}  is the {category $\bC=\Hilb_{c}(A)$ of Hilbert $A$-modules  {and compact operators}
for $A$ in $\Fun(BG,\nCalg)$, see Example   \ref{wtrhiotrhrthtrhdtrh}.} 
\hB
\end{ex}

 
{Below we will consider  conditions on $\bC$ in $ \nCcat$ which  involve 
 orthogonal sums of possibly  infinite families 
  $(C_{i})_{i\in I}$   of objects of $\bC$. Let  $(C,(e_{i})_{i\in I})$ be a pair of an object of $\bC$ and a family of  {mutually orthogonal } isometries
$e_{i} \colon C_{i}\to C$ in {the multiplier category} $\bM\bC$ {of $\bC$.} 
\begin{ddd}[{\cite[Def.~3.1]{cank}}]\label{wigjosgbsgfgsdfgs9}
We say that $(C,(e_{i})_{i\in I})$ represents 
an   AV-sum  of the family $(C_{i})_{i\in I}$ if
$\sum_{i\in I} e_{i}e_{i}^{*}$ converges strictly to $\id_{C}$ in $\bM\bC$.
\end{ddd}}

\phantomsection\label{qroigjqoifjmefmewf89uqf8wefqwefqf}
Let $p$ be an orthogonal projection on an object $C$ in a $C^{*}$-category, i.e., an endomorphism of $C$ satisfying $p^{*}=p$ and $p^{2}=p$.
A morphism $u \colon C'\to C$ represents the image of $p$ if $u$ is an isometry, i.e.,    $u^{*}u=\id_{C'}$,  and $uu^{*}=p$.
We say that $p$ is effective if it admits an image.\label{effective_projection} In the present paper we will {only consider}  orthogonal projections, and therefore we will omit the word orthogonal from now on. {We refer to \cite[2.16-2.19]{cank} for more details.}

{\begin{ddd} [{\cite[Def. 3.12]{coarsek}}] \label{wigjosgbsgfgsdfgs91}
We   say that $\bC$ is effectively additive  
if for every object $C$ of $\bC$ and mutually orthogonal family of effective projections $(p_{i})_{i\in I}$ on $C$ in $\bM\bC$ such that $\sum_{i\in I} p_{i}$ converges strictly  to a projection $p$ in $\bM\bC$, the latter  is also effective in $\bM\bC$. \end{ddd}
 If $\bC$ admits all small AV-sums or is  idempotent complete, then it is effectively additive.}  
 {If $\bC$ is in $\Fun(BG,\nCcat)$, then we will apply the notions introduced above to the underlying $C^{*}$-category   obtained by forgetting the $G$-action.}
 
In general the category  $  \bC$ in $\Fun(BG,\nCcat)$   may contain objects which admit an identity morphism. These objects are called unital.  {We note that automorphisms of $\bC$ preserve unital objects.}
\begin{ddd}\label{wefbsfopkfvdfsv} {For $\bC$ in $\Fun(BG,\nCcat)$, }{we let $  {\bC}^{u}$ in $\Fun(BG,\Ccat)$} 
denote  the full subcategory of unital objects in $  {\bC}$.
\end{ddd}
  
  \begin{ex}
  Let $A$ {be in $\Fun(BG,\Calg)$} and  $\bC = \Hilb_{c}(A)$ as in Example \ref{weriogwergw}. Then
 ${\bC^{u}}=\Hilb(A)^{\proj, \fg}$ is the full subcategory of $\Hilb(A)$ of finitely generated projective Hilbert $A$-modules.
  \hB
  \end{ex}

{For the moment, let $\bD$ be in $\Fun(BG,\Ccat)$. Our main example will  be   the multiplier category  $\bM\bC$ of $\bC$ in $\Fun(BG,\nCcat)$.}
We fix the  {following} notation convention concerning the $G$-action on $  \bD$.
If $D$ is an object of $  \bD$ and $g$ is in $G$, then we let ${gD}$ denote the object 
obtained by applying $g$ to {$D$}. 
Similarly, if $A$ is a morphism in $  \bD$, then we write
${gA}$ for the morphism obtained by applying $g$ to~$A$.

\begin{ddd}\label{tgijrogwegweffsfv}
A $G$-object in $  \bD$ is a pair $(D,\rho)$ of an object in $\bD$ and a family $\rho=(\rho_{g})_{g\in G}$ of unitaries $\rho_{g} \colon D\to {gD}$ such that $g\rho_{h}\ \rho_{g}=\rho_{gh}$ for all $h,g$ in $G$.
\end{ddd} 

\begin{ex}
If $G$ acts trivially on $\bD$, then the datum of a $G$-object $(D,\rho)$ in $\bD$ is the same as an object  $D$ of $\bD$ together with a homomorphism $\rho \colon G\to \Aut_{\bD}(D)$, $g\mapsto \rho_{g}^{{-1}}$, such that
$\rho_{g^{-1}}=\rho_{g}^{*}$.
\hB
\end{ex}

\begin{ddd}\label{qegijqgiofjfqewfqwef}
The category of $G$-objects in $  \bD$ is 
the $C^{*}$-category with $G$-action   $  \bD^{(G)}$ in $\Fun(BG,\Ccat)$ defined as follows:
  \begin{enumerate}
  \item objects: The objects of   $  \bD^{(G)}$ are    the $G$-objects in $  \bD$.
  \item morphisms: The morphisms in $  \bD^{(G)}$ are given by 
  \begin{equation}\label{vewvejvovwrwervwevwe}
  \Hom_{\bD^{(G)}}((D,\rho),(D',\rho')) \coloneqq \Hom_{\bD}(D,D')\, . 
\end{equation}
  \item composition and involution: The composition and involution are inherited from $\bD$.
  \item $G$-action:
  \begin{enumerate}
  \item objects: $G$ fixes the objects of $  \bD^{(G)}$.
  \item morphisms: $g$ in $G$ acts on a morphism   $A \colon (D,\rho)\to (D',\rho')$ by
 \begin{equation}\label{eqwfijoifqjowiefjwoefewqfewqf}
g\cdot A \coloneqq \rho_{g}^{\prime, -1} \, gA \, \rho_{g}\, .
\end{equation}
  \end{enumerate}
 \end{enumerate}
 \end{ddd}
 
 Note that we use the notation $g\cdot-$ in order to denote the $G$-action on morphisms between $G$-objects which should not be confused with the original action denoted by $g-$.

Associated to   $\bC$ in $\Fun(BG,\nCcat)$ we have   two derived objects $\bC^{u}$ and $(\bC^{u})^{(G)}$ in 
$\Fun(BG,\Ccat)$. In the following we will show that they are related by a canonical zig-zag of  
fully faithful functors.
To this end we construct a third object $\hat \bC^{u,(G)}$ in $\Fun(BG,\Ccat)$.
\begin{enumerate}
\item objects: The $G$-set of objects of $\hat \bC^{u,(G)}$ is the union of the $G$-sets of objects
of $\bC^{u}$ and $(\bC^{u})^{(G)}$.
\item morphisms: The morphism spaces of $\hat \bC^{u,(G)}$ are defined such that
$\bC^{u}$ and $(\bC^{u})^{(G)}$ are fully faithfully embedded. If $C$ is in $\bC^{u}$ and $(C',\rho')$ is in
$(\bC^{u})^{(G)}$, then we define
$\Hom_{\hat \bC^{u,(G)}}(C,(C',\rho')):=\Hom_{\bC}(C,C')$ and
$\Hom_{\hat \bC^{u,(G)}}((C',\rho'),C):=\Hom_{\bC}(C',C)$.
\item The composition and the involution are inherited from $\bC$.
\item $G$-action: The $G$-action  is defined such that both the inclusions $\bC^{u}\to \hat \bC^{u,(G)}$ and
$(\bC^{u})^{(G)}\to \hat \bC^{u,(G)}$ are $G$-equivariant. 

If $\hat f \colon C\to (C',\rho')$ is a morphism
in $\Hom_{\hat \bC^{u,(G)}}(C,(C',\rho'))$ given by $f \colon C\to C'$ in $\bC$, then  $g \hat f \colon gC\to (C',\rho')$ is given by 
$\rho^{\prime,-1}_{g}\circ gf \colon gC\to C'$. Similarly, if $\hat h \colon (C',\rho')\to C$ is a morphism in $\Hom_{\hat \bC^{u,(G)}}((C',\rho'),C)$ given  by $h \colon C'\to C$, then  $g\hat h \colon C'\to gC$ is given by 
$gh\circ \rho'_{g} \colon C'\to gC$.
\end{enumerate}

\begin{ddd}\label{werkgerwgreferfwref}
We say that $G$ weakly fixes the objects of $\bC^{u}$ if for every object $C$ of $\bC^{u}$ there exists a refinement $(C,\rho)$ to an object of $(\bC^{u})^{(G)}$.
\end{ddd}
 
In other words, $G$ weakly fixes the objects of $\bC^{u}$ if and only if the canonical functor
$${\lim}^{\Ccat_{2,1}}_{BG}\bC^{u}\to \Res^{G}(\bC^{u})$$
from the $2$-categorical $G$-fixed points of $\bC^{u}$ to $\bC^{u}$  with $G$-action forgotten is essentially surjective. 

\begin{lem}\label{wekgowerfrefwerf}\mbox{}
\begin{enumerate}
\item
The inclusion
$\bC^{u}\to \hat \bC^{u,(G)}$
is a unitary equivalence. \item  If $G$ weakly fixes the objects of $\bC^{u}$, then the inclusion  
$(\bC^{u})^{(G)}\to  \hat \bC^{u,(G)}$
is a unitary equivalence.
\end{enumerate}
\end{lem}
\begin{proof}
By construction both inclusion functors are fully faithful. We now argue that they are essentially surjective. 
We start with the inclusion of $\bC^{u}$. We consider an object $(C,\rho)$ in $(\bC^{u})^{(G)}$. Then $C$ is in $\bC^{u}$ and  $\id_{C}$  gives a unitary isomorphism 
$C\to (C,\rho)$ in $\hat \bC^{u,(G)}$.

We now consider the inclusion of $(\bC^{u})^{(G)}$.
Let $C$ be an object of $\bC^{u}$. By assumption there exists an object $(C,\rho)$ in $(\bC^{u})^{(G)}$
and again  $\id_{C}$  gives a unitary isomorphism 
$C\to (C,\rho)$ in $\hat \bC^{u,(G)}$.
\end{proof}

%
%
  {For a $G$-$C^*$-category $\bC$ and a $G$-bornological space  $X$  we will  introduce  the notion of $X$-controlled $G$-objects in $\bC$}. To this end, we recall that
a $G$-bornology  on a $G$-set $X$ 
is a $G$-invariant  subset of the power set $\cP_{X}$ of $X$ which is closed 
under forming finite unions, subsets{, and which} contains all one-point subsets. 
 A $G$-bornological space
 is a pair $(X,\cB)$  of a $G$-set  $X$ with   a $G$-bornology $\cB$ whose elements will be called the bounded subsets of $X$. If $ (X,\cB)$ and $ (X',\cB')$ are $G$-bornological spaces
  and  $f \colon X\to X'$ is an equivariant map of underlying $G$-sets, then $f$ is called  proper
 if  $f^{-1}(\cB')\subseteq \cB$.
  By  $G\Born$ denote the category of very small $G$-bornological spaces and proper maps.
We refer to \cite{equicoarse} for more details. We will usually use the shorter  notation $X$ for $G$-bornological spaces. To any $G$-set $S$ we can associate the  following objects in $G\Born$.
\begin{enumerate}
\item $S_{min}$ is $S$ equipped with the minimal bornology consisting of the finite subsets.
The map $S\mapsto S_{min}$ is functorial for morphisms of $G$-sets with finite fibres.
\item $S_{max}$ is $S$ equipped with the maximal bornology consisting of all subsets of $S$. We   have a functor $G\Set\to G\Born$ given on objects by $S\mapsto S_{max}$. 
\end{enumerate}

Let $X$ be in $G\Born$. 
 
 \begin{ddd}\label{oiguoeigqregregwergregwgrg}
 A subset $L$ of $X$ is called locally finite if $B\cap L$ is finite for every bounded subset in $X$.
 \end{ddd}

     

 The following definition is {an expanded version of} \cite[Def.\ 4.6]{coarsek}. {Let $X$ be in $G\Born$.}

 \begin{ddd}\label{qeroigergeggweerg} A locally finite $X$-controlled $G$-object in $  \bC$ is a triple
$(C,\rho,  \mu)$, where:
\begin{enumerate}
\item  $(C,\rho)$ is an object in $  {\bM}\bC^{(G)}$.
\item\label{eoigeogerwgrwerg}  $ \mu$ is an invariant, finitely additive measure on $X$ with values {in projections in} $\End_{{\bM}\bC}(C)$ such that the following properties hold:
\begin{enumerate}
\item  $\mu(X)=\id_{C}$.
\item\label{ogureoiguogr9g02034} {$\mu(\{x\})$ is effective and belongs to  $ \bC$} for all $x$ in $X$.
\item\label{qeroighjoergfqewfqf} $C$ {is the} {orthogonal} {AV}-sum of the images of the family of projections $(\mu(\{x\}))_{x\in X}$. 
\item\label{toiwjjkfmfdfds89w} {The subset $\supp(\mu)$ of $X$} is locally finite.
\end{enumerate}
\end{enumerate}
\end{ddd} 

\begin{rem} \label{rqfqiofioewjeowfewfqewfqwefe} {In this remark we explain   Condition  \ref{eoigeogerwgrwerg} in more detail.}  It first of all says that 
 $\mu$ is a function from the power set $\cP_{X}$ of $X$ to the set of  projections in $\End_{{\bM}\bC}(C)$ such that for all   $Y,Z$  in $\cP_{X}$ with $Y\subseteq Z$  we have $\mu( Z)=\mu(Y)+\mu(Z\setminus Y) $. 
     {The} invariance condition of $\mu$ means that
\begin{equation}\label{wfoizeqwfu98u9e8wfqewf}
g\cdot \mu(Y)=\mu(gY)
\end{equation}  for all $g$ in $G$ and subsets $Y$ of $X$.  

{Condition \ref{qeroighjoergfqewfqf} says {that 
$\sum_{x\in X}\mu(\{x\})$} converges strictly to $\id_{C}$.}

The support of $\mu$ is the subset
$$\supp(\mu) \coloneqq \{x\in X\:|\: \mu(\{x\})\not=0\}$$ of $X$.  

The Conditions  \ref{ogureoiguogr9g02034} 
{and}  \ref{toiwjjkfmfdfds89w}    together imply that $\mu(B)$ belongs to the ideal ${\bC}$ of $\bM\bC$ for every bounded subset $B$ of $X$.
 \hB
%
\end{rem}

{Let $\bC$ be in $\Fun(BG,\nCcat)$ and  $X$ be in $G\Born$.}

   \begin{ddd}\label{qrohwoifewfqwefqwfewqf}
   We define  {$\tCglf(X)$  in $\Fun(BG,\Ccat)$}
   as follows:
   \begin{enumerate}
   \item objects: The objects of  $\tCglf(X)$  are the  locally finite $X$-controlled $G$-objects in $  \bC$.
   \item morphisms: The morphisms in  $\tCglf(X)$ are given by $$\Hom_{\tCglf(X)}((C,\rho,  \mu),(C',\rho',  \mu')) \coloneqq \Hom_{ {\bM} \bC^{(G)}}((C,\rho),(C',\rho'))\, .$$
   \item composition, involution and $G$-action: The composition, involution and  the $G$-action are induced from $  {\bM}\bC^{(G)}$. 
   \end{enumerate}
   \end{ddd}

 We have a fully faithful forgetful functor
 \begin{equation}\label{qewfqwefefqewfqfef}
\cF \colon \tCglf(X)\to   {\bM}\bC^{(G)}\, , \quad (C,\rho,\mu) \mapsto (C,\rho)\, .
\end{equation}
 
\begin{ddd}\label{qeroighjoreqgreqfefwfqwef}\mbox{}
\begin{enumerate}
\item We define 
{$ {\bM} \bC_{\std}^{(G)}$  in $\Fun(BG,\Ccat)$}
as the full subcategory of $ {\bM} \bC^{(G)}$
of objects which  {are} isomorphic to objects of the form $\cF( (C,\rho,\mu))$ for some object
$(C,\rho,\mu)$ in $\tCglf(Y_{min})$ for some free $G$-set $Y$.
\item \label{wgiowjegoreregwg}We   let $  {\bC}_{\std}^{(G)} $ {in $\Fun(BG,\nCcat)$}  denote the $G$-invariant  ideal   of ${\bM}  \bC_{\std}^{(G)}$ of morphisms belonging to ${  \bC}$. 
\item \label{wrtogwprtgwergwergfreferf}We define {the quotient} 
\begin{equation}\label{wefwfwfwefflkjlqwepofp}
  \bQ^{(G)}_{\std} \coloneqq \frac{ {\bM} \bC_{\std}^{(G)}}{  {\bC}_{\std}^{(G)}}
\end{equation}
{in $\Fun(BG,\Ccat)$.} 
\end{enumerate}
\end{ddd}

\begin{rem}
Let us assume for simplicity that $\bC$ is effectively additive.
Applying Definition \ref{wigjosgbsgfgsdfgs91} to the empty family of projections  on an object $C$ shows that
$\bC$ admits zero objects since the zero projection on $C$ must be  effective.
It can happen that $\bC^{u}$  only consists of zero objects.
In this case $ \tCglf(X)$ consists of zero objects for any $X$ in $G\Born$.
Furthermore, the categories $  \bC^{(G)}_{\std}$,   $ \bM\bC^{(G)}_{\std}$, and $ \bQ^{(G)}_{\std} $
consist of zero objects.  \hB
\end{rem}

{
 \begin{lem}\label{witjodgdgndhghdghh}
The inclusion $ \bC^{(G)}_{\std}\to \bM\bC^{(G)}_{\std}$ presents  $\bM\bC^{(G)}_{\std}$
 as the multiplier category of $ \bC^{(G)}_{\std}$.
\end{lem}
\begin{proof}
We have a fully faithful forgetful functor $ \bC^{(G)}_{\std}\to \bC$ which sends $(C,\rho)$ to $C$.
It induces a fully faithful functor $\bM( \bC^{(G)}_{\std})\to \bM\bC$. This functor has an obvious factorization  $\bM( \bC^{(G)}_{\std})\to  \bM\bC^{(G)}_{\std}\to \bM\bC $, where the first functor is the identity on objects. Since the composition and the second functor are fully faithful, so is the first which is therefore an isomorphism.
\end{proof}
}

%
%

For $A$ in   $\Fun(BG,{\nCalg})$   we consider
$\bC:=\Hilb_{c}(A)$ in $\Fun(BG,\nCcat)$.
The following constructions will be used later to compare $K$-theoretic constructions involving, e.g., $\bQ^{(G)}_{\std}$
with constructions involving $A$ directly.

For  $\bC$  in $\Fun(BG,\nCcat)$
we let $\bM\bC^{(G)}_{\std,+}$ denote the full subcategory of $\bM\bC^{(G)}$ of objects
$(C,\rho)$ which  belong to  $\bM\bC^{(G)}_{\std}$ or  $(\bC^{u})^{(G)}$. We furthermore let $\bC^{(G)}_{\std,+}:=\bM\bC^{(G)}_{\std,+}\cap \bC^{(G)}$.

\begin{ex}\label{wteklgwrtegwergfwrefrfwferfer}
For  $A$ in $\Fun(BG,{\nCalg})$  and  
$\bC:=\Hilb_{c}(A)$ in $\Fun(BG,\nCcat)$
 we let  $\hat A$ be the object of $\bC$ given by 
  $A$ with the right-multiplication and the scalar product $\langle a,b\rangle_{\hat A}=a^{*}b$.   
  Left multiplication {identifies $A$ with $\End_{\bC}(\hat A)$}.
 For $g$ in $G$ we have a $\C$-linear map
$\kappa_{g}\colon \hat A\to \hat A$ given by the  action of $g^{-1}$ on $A$, i.e., $\kappa_{g}(a):={}^{g^{-1}}a$. This map is a unitary {multiplier}  isomorphism  $\hat A \to g \hat A$  in $\bC$. 
The family $\kappa:=(\kappa_{g})_{g\in G}$  refines $\hat A$ to an object $(\hat A,\kappa)$ of $\bC^{(G)}$. {Moreover}, the identification
$A\cong \End_{\bC^{(G)}}((\hat A,\kappa))$ is equivariant.

{If  $A$ is unital,   then the object  $(\hat A,\kappa)$ belongs to
$ (\bC^{u})^{(G)}$ and hence to $\bM\bC^{(G)}_{\std,+}$. In this case we}
have a zig-zag of equivariant inclusions
$$A\to \bM\bC^{(G)}_{\std,+} \leftarrow  \bM\bC^{(G)}_{\std}\, ,\quad  A\to \bC^{(G)}_{\std,+} \leftarrow  \bC^{(G)}_{\std}\, .$$
 The left functors sends
$A$  to the object $(\hat A,\kappa)$ and identify $A$ with $\End_{\bM\bC^{(G)}}((\hat A,\kappa))$  
 or  $\End_{\bC^{(G)}}((\hat A,\kappa))$, respectively.
 \hB
\end{ex}
 
Recall the definitions of a Morita equivalence \cite[16.7]{cank}, of a relative Morita equivalence \cite[Def. 17.1]{cank}, and   of a weak Morita equivalence  \cite[Def. 18.3]{cank} between $C^{*}$-categories.  {In the equivariant case, an equivariant  functor is a Morita equivalence or   weak Morita equivalence if it has the respective property after forgetting the $G$-action. In addition we will need in the following  a stronger version of the notion of a relative Morita equivalence which we call a split relative Morita equivalence. Let $\phi \colon \bD\to \bE$ in $\Fun(BG,\nCcat)$. 
\begin{ddd}
 We say that $\phi$ is a split relative Morita equivalence if there exists a diagram
\begin{equation}\label{csdioucsadoicuoasdcdsacasdcasc} \xymatrix{0\ar[r]& \bD \ar[d]^{\phi}\ar[r] & \bD'\ar[d]\ar[r]^-{p}&\bD'/\bD\ar[d]\ar[r]&0\\
0\ar[r]& \bE\ar[r]&\bE' \ar[r]^-{q}&  \bE'/\bE\ar[r]&0}\end{equation} 
 in $\Fun(BG,\nCcat)$ with horizontal exact sequences such that the two
right vertical functors are Morita equivalences between unital $C^{*}$-categories and
the functors $p$ and $q$ admit right-inverses. 
\end{ddd}
 }

 
 {Let  $\bC$ be in $\Fun(BG,\nCcat)$.}
\begin{lem}\label{eriughwierewfwerfwf}\mbox{}
\begin{enumerate}
\item \label{wergjiergjosfdgdfg}     $\bM\bC^{(G)}_{\std}\to  \bM\bC^{(G)}_{\std,+}$   is a Morita equivalence.
\item \label{wergjiergjosfdgdfg1}     $\bC^{(G)}_{\std}\to    \bC^{(G)}_{\std,+}$ is a {split} relative Morita equivalence.
\item\label{wergjiergjosfdgdfg2}   If $A$ is in $\Fun(BG,\Calg)$ and $\bC=\Hilb_{c}(A)$, then $A\to \bC^{(G)}_{\std,+}$ 
  has a factorization 
into the Morita equivalence $A\to  (\bC^{u})^{(G)} $ followed by the weak Morita equivalence $ (\bC^{u})^{(G)} \to  \bC^{(G)}_{\std,+}$.
\end{enumerate}
\end{lem}
\begin{proof}
We start with the Assertion \ref{wergjiergjosfdgdfg}.
The inclusion  $\bM\bC^{(G)}_{\std}\to  \bM\bC^{(G)}_{\std,+}$   is fully faithful.
We will show that every object of $ \bM\bC^{(G)}_{\std,+}$ is a summand of an object of $\bM\bC^{(G)}_{\std}$.  It suffices to show this for objects of $  (\bC^{u})^{(G)}$. Thus 
let $(C',\rho')$ be an object of
$  (\bC^{u})^{(G)}$. Then  using the fact that $\bC$ admits countable AV-sums one can construct an object $(C,\rho,\mu)$ in $ \tCglf(G_{min})$
such that there exists an isometry $u:C'\to C$ in $\bM\bC$  representing an image of
$\mu(\{e\})$. For $C$ we must take an   AV-sum of the family $(g C')_{g\in G}$. 
We consider  $u$ as an isometry $u:(C',\rho')\to (C,\rho)$ in $\bM\bC^{(G)}_{\std,+}$ with $(C,\rho)\in \Ob(\bM\bC^{(G)}_{\std})$. It realizes  $(C',\rho')$ as a summand of the 
 object    $ (C,\rho)$  of $\bM\bC^{(G)}_{\std}$.
This finishes the proof of Assertion \ref{wergjiergjosfdgdfg}.


Let $ \bC^{(G),\sharp}_{\std}$ and $ \bC^{(G),\sharp}_{\std,+}$ be the $C^{*}$-categories obtained from
$ \bC^{(G)}_{\std}$ and $ \bC^{(G)}_{\std,+}$ by adjoining units to all non-unital objects.
We then have a diagram of exact sequences
$$\xymatrix{0\ar[r]& \bC^{(G)}_{\std}\ar[d]\ar[r] &\bC^{(G),\sharp}_{\std}\ar[d]\ar[r]^-{p}&\bC^{(G),\sharp}_{\std}/\bC^{(G)}_{\std}\ar[d]\ar[r]&0\\
0\ar[r]& \bC^{(G)}_{\std,+}\ar[r]&\bC^{(G),\sharp}_{\std,+}\ar[r]^-{p_{+}}& \bC^{(G),\sharp}_{\std,+}/\bC^{(G)}_{\std,+}\ar[r]&0}$$
Since the objects of $(\bC^{u})^{(G)}$ are unital 
they represent zero objects in  $\bC^{(G),\sharp}_{\std,+}/\bC^{(G)}_{\std,+}$.
We conclude that the right vertical morphism is a Morita equivalence.
Since the morphisms $u:(C',\rho')\to (C,\rho)$ from the argument  for  Assertion \ref{wergjiergjosfdgdfg} actually
belong  to $\bC^{(G),\sharp}_{\std,+}$  we  conclude that the middle arrow is a Morita equivalence, too.
The projections $p$ and $p_{+}$ have obvious splits. 

 In order to show Assertion   \ref{wergjiergjosfdgdfg2} 
 first note that  if $(C,\rho)$ is an object of $ (\bC^{u})^{(G)}$, then
 $C$ is a finitely generated projective $A$-module and hence a summand 
 of a finite sum of copies of $A$. This implies that $A\to  (\bC^{u})^{(G)}$ is a Morita equivalence. In order to show that the second morphism  
 $ (\bC^{u})^{(G)} \to  \bC^{(G)}_{\std,+}$
 is a weak Morita equivalence we first observe that it is fully faithful. We then use that the morphisms in   $\bC^{(G)}_{\std,+}$ are compact operators between Hilbert $C^{*}$-modules. A compact operator can be approximated arbitrary well by an operator which factorizes over a finitely generated projective $A$-module, i.e., an object of $\bC^{u}$. 
 This implies that the set of objects of $(\bC^{u})^{(G)} $ is weakly generating in $ \bC^{(G)}_{\std,+}$. 
\end{proof}

Recall the definition of flasque $G$-$C^{*}$-categories \cite[Def. 11.3]{cank}.
  
 \begin{lem}\label{wrtijhoerhgrtgertg}
If $\bC$ admits    {countable AV-sums, then
$ \bM \bC^{(G)}_{\std}$ is flasque.}\end{lem}
\begin{proof}{We  claim that  
 $  \bC^{(G)}_{\std}$ also  admits countable AV-sums. Then $\bM(\bC^{(G)}_{\std})$  is flasque 
 by \cite[Ex.\ {11.5}]{cank}. We finally use Lemma 
  \ref{witjodgdgndhghdghh} in order to conclude that  $ \bM \bC^{(G)}_{\std}$ is flasque.}
  
 We show the claim.
 We consider a countable family $(C_{i},\rho_{i})_{i\in I}$ of objects in $ \bC^{(G)}_{\std}$. For every $g$ in $G$ we can choose an AV-sum $(C_{g},(e^{gC_{i}}_{i})_{i\in I})$ of the family $( gC_{i})_{i\in I}$ in $\bC$. We set $C \coloneqq C_{e}$ and let $u_{g} \colon C_{g}\to g C $ be the canonical multiplier unitary such that $g(e^{C_{i},*}_{i}) u_{g} e^{gC_{i}}_{i}=\id_{gC_{i}}$ {for all $i$ in $I$}.
 Then $\rho \coloneqq (u_{g}\circ \oplus_{i\in I} \rho_{i})_{g\in G}$ defines a multiplier cocycle on $C$ such that we have $(C,\rho)\in \bC^{(G)}$. We now show that $ (C,\rho)\in\bC^{(G)}_{\std}$. 
By assumption, for every $i$  in $I$  we can refine the pair $(C_{i},\rho_{i})$ to an object
$(C_{i},\rho_{i},\mu_{i})$ in $  \tCglf(X_{i})$ for some free $G$-set $X_{i}$.  
Then $(C,\rho,\mu)$ belongs to $  \tCglf(X)$, where $X=\bigsqcup_{i\in I}X_{i}$
and the measure $\mu$ is given by
$\mu(Y) \coloneqq \oplus_{i\in I} \mu_{i}(Y\cap X_{i})$ for all subsets $Y$ of $X$. Since $X$ is again a free $G$-set we conclude that $(C,\rho)$ belongs to  $ \bC^{(G)}_{\std}$. 

{By construction, the sum  $\sum_{i\in I} e^{C_{i}}_{i}e^{C_{i},*}_{i}$ strictly  converges  to 
$\id_{(C,\rho)}$ in $\bM\bC^{(G)}_{\std}$. 
By 
  Lemma  \ref{witjodgdgndhghdghh} it also strictly converges in $\bM(\bC^{(G)}_{\std})$. Therefore 
the pair   $(C,\rho)$ represents the AV-sum of the family 
$(C_{i},\rho_{i})_{i\in I}$  in $\bC^{(G)}_{\std}$.}
 \end{proof}

If $\bK$ is in $\Fun(BG,\nCcat)$, then we can form the reduced crossed product
$\bK\rtimes_{r}G$ introduced in \cite[Thm. 12.1]{cank}. We use  the explicit description of   the  algebraic crossed product $\bK\rtimes^{\alg}G$ and the notation   introduced in  \cite[Def. 5.1]{crosscat}.  Recall that the maximal crossed product is defined in
\cite[Def. 5.9]{crosscat} as the completion of the pre-$C^{*}$-category $\bK\rtimes^{\alg}G$. In contrast,
the reduced crossed product $\bK\rtimes_{r}G$ is defined in \cite[Def. 12.9]{cank} as the completion of  $\bK\rtimes^{\alg}G$ in the norm induced by a specific representation on a $W^{*}$-category $\bL^{2}(G,\bW\bM\bK)$ \cite[Def. 12.2]{cank}, where $\bW\bM\bK$ is the universal 
$W^{*}$-envelope of the multiplier category $\bM\bK$ defined in \cite[Def. 2.33]{cank}. In order to define $\bL^{2}(G,\bW\bM\bK)$  we must assume that $\bK$ admits countable $AV$-sums. The $W^{*}$-category $\bL^{2}(G,\bW\bM\bK)$ has the same objects as $\bK$, and the morphisms are given by
\begin{equation}\label{gerferfewfreferwfw}\Hom_{\bL^{2}(G,\bW\bM\bK)}(K,K')\cong \Hom_{\bW\bM\bK}(\bigoplus_{g\in G} gK, \bigoplus_{g\in G} gK')\ .
\end{equation} 
Let $(e^{K}_{h})_{h\in G}$ be the family of isometries $e^{K}_{h}:hK\to \bigoplus_{g\in G} gK$ witnessing the sum $\bigoplus_{g\in G}gK$.  On generators  
the representation $\bK\rtimes^{\alg}G\to \bL^{2}(G,\bW\bM\bK)$ is then defined according to \cite[(12.8)]{cank}
by \begin{equation}\label{rwegwerfrefwerf}(f,g)\mapsto \sum_{h\in G} e_{hg^{-1}}^{K'} \: hf \: e^{K,*}_{h}
\end{equation}
(note that $f:K\to g^{-1}K'$), where the sum converges strictly.

 In the present paper we  in particular  need the 
  reduced crossed product $\bC^{(G)}_{\std}\rtimes_{r}G$ 
  for $\bC$ in $\Fun(BG,\nCcat)$. In the following, by specializing the general description above,  we describe
  this crossed product and a 
  part of its multipliers explicitly, thereby introducing notation which will be employed later in the paper.
We assume that $\bC$  is effectively additive and admits countable AV-sums. 
 In the proof of Lemma \ref{wrtijhoerhgrtgertg} we saw  that   $\bC^{(G)}_{\std}$   also admits countable AV-sums.  
The objects of $ \bC_{\std}^{(G)}\rtimes_{r}G$ 
 are the objects of ${\bC}_{\std}^{(G)}$. 
 The $C^{*}$-category
 $ \bC_{\std}^{(G)}\rtimes_{r}G$ is  the completion of the image the  functor 
  $\sigma \colon \bC^{(G)}_{\std}\rtimes^{\alg}G\to \bL^{2}(G,\bW\bM\bC^{(G)}_{\std})$.
  The $W^{*}$-category $\bL^{2}(G,\bW\bM\bC^{(G)}_{\std})$ has  the same objects as $\bC^{(G)}_{\std}$. Since the functor $\bW \bM\bC^{(G)}_{\std}\to \bW\bM\bC$ induced by $\bC^{(G)}_{\std}\to \bC$ is fully faithful and that $G$ fixes the objects of      $\bC^{(G)}_{\std}$, by 
   specializing \eqref{gerferfewfreferwfw}  can identify the morphism spaces of $\bL^{2}(G,\bW\bM\bC^{(G)}_{\std})$ with $$\Hom_{\bL^{2}(G,\bW\bM\bC^{(G)}_{\std})}((C,\rho),(C',\rho'))\cong \Hom_{\bW\bM\bC}(\bigoplus_{g\in G} C,\bigoplus_{g\in G} C')\, ,$$
   where $( \bigoplus_{g\in G} C, (e_{l})_{l\in G})$   and 
$( \bigoplus_{g\in G} C', (e'_{l})_{l\in G})$
represent    AV-sums of the constant families  $(C)_{g\in G}$ and $(C')_{g\in G}$, respectively.  

 We can now 
 describe the functor $\sigma$ explicitly specializing \eqref{rwegwerfrefwerf} where we use 
  that the $G$-action in morphisms in $\bC^{(G)}_{\std}$ is given by $(h,f)\mapsto h\cdot f$.
On objects $\sigma$ acts as the identity.  Furthermore,   $\sigma$ sends   
  the morphism $(f,g) \colon (C,\rho)\to (C',\rho')$ in $\bC^{(G)}_{\std}\rtimes^{\alg}G$  to 
  \begin{equation}\label{ewfefaewfeda}
\sigma(f,h) \coloneqq \sum_{l\in G} e'_{lh^{-1}} \:l\cdot f \:e_{l}^{*}\colon\bigoplus_{g\in G}C\to \bigoplus_{g\in G} C'\, .
\end{equation}

{If $\bL$ is a closed wide subcategory of a $C^{*}$-category  $\mathbf{H}$, then  the idealizer of $\bL$  in $\mathbf{H}$ is the maximal wide subcategory of $\mathbf{H}$ containing $\bL$ as an ideal. It consists of all morphisms of $\mathbf{H}$ which preserve $\bL$ by left- and right composition.}

\begin{ddd}\label{witgjierojgoewrfre}
We define $\bU$ to be the idealizer  of $ \bC_{\std}^{(G)}\rtimes_{r}G$ in $\bL^{2}(G,\bW\bM\bC^{(G)}_{\std})$.  \end{ddd}

We will understand $ \Idem({\bC}_{\std}^{(G)}\rtimes_{r}G)$ as the 
  idempotent completion relative to $\bU$, see \cite[Def. 17.5]{cank}.
  Therefore objects in $ \Idem({\bC}_{\std}^{(G)}\rtimes_{r}G)$ are triples $(C,\rho,p)$, where $p$ is a projection on $(C,\rho)$ in $\bU$.   

Using formula \eqref{ewfefaewfeda}, we see that $\sigma$ extends canonically to a functor
$\sigma\colon\bM\bC^{(G)}_{\std}\rtimes^{\alg} G\to \bU$
given by the same formula. By the universal property of the maximal crossed product it further extends to 
a morphism 
\begin{equation}\label{asvasvadvadsvadvdvoihjio}
\sigma\colon \bM\bC^{(G)}_{\std}\rtimes  G\to \bU \, .
\end{equation}

%

Let $\phi:\bC\to \bC'$ be a morphism in $\Fun(BG,\nCcat)$.   \begin{ddd}[{\cite[Def. 3.11]{cank}
}]
 The morphism $\phi$ is called non-degenerate if for every two objects $C_{0},C_{1}$ of $\bC$ the  linear subspaces 
 $\phi(\End_{\bC}(C_{1}))\Hom_{\bC'}(\phi(C_{0}),\phi(C_{1}))$ and
 $\Hom_{\bC'}(\phi(C_{0}),\phi(C_{1}))\phi(\End_{\bC}(C_{0}))$ are dense in $\Hom_{\bC'}(\phi(C_{0}),\phi(C_{1}))$.
 \end{ddd}
 
We will consider the chain of subcategories
\begin{equation}\label{eq_defn_functoriality_category}
\nCcat_{\ndeg,\add} \subseteq \nCcat_{\ndeg,\omega\add,\eadd} \subseteq
 \nCcat_{\ndeg}\subseteq  \nCcat\ ,
\end{equation}
 where
 \begin{enumerate}
 \item $\nCcat_{\ndeg}$ is the wide subcategory of $ \nCcat$
 of non-degenerate morphisms,
 \item $\nCcat_{\ndeg,\omega\add,\eadd}$ is full  subcategory of $ \nCcat_{\ndeg}$
 of  effectively additive objects which admit countable AV-sums,
 \item $\nCcat_{\ndeg, \add}$ is full  subcategory of $ \nCcat_{\ndeg}$ of objects which admit all small AV-sums.
 \end{enumerate}

By \cite[Prop. 3.16]{cank} a non-degenerate morphism $\phi \colon \bC\to \bC'$ naturally induces a morphism
$\bM\phi \colon \bM\bC\to \bM\bC'$ of the associated multiplier categories and, {again by non-degeneracy,} it restricts to a unital morphism
$\phi^{u} \colon \bC^{u}\to \bC^{\prime,u}$ of full subcategories of unital objects. This implies that the constructions of  $\bC^{(G)}_{\std}$, 
$\bM\bC^{(G)}_{\std}$, $\bQ^{(G)}_{\std}$, {$\bC^{u}$,}  $(\bC^{u})^{(G)}$ and $\bC^{(G)}_{\lf}$ extend to functors on $\Fun(BG,\nCcat_{\ndeg,\eadd,\omega\add})$. {Further, $\phi$ induces a morphism $\bL^{2}(G,\bW\bM\bC_{\std}^{(G)})\to \bL^{2}(G,\bW\bM\bC_{\std}^{\prime, (G)})$ (see the proof of \cite[Lem. 12.10]{cank}) and hence $\bU$ and $\bC^{(G)}_{\std}\rtimes_{r}G$ also extend to such functors.}

\section{\texorpdfstring{$\boldsymbol{G}$}{G}-bornological coarse spaces and  \texorpdfstring{$\boldsymbol{K\bC\cX^{G}}$}{KCXGc}}

  We fix    {$   \bC$ in $\Fun(BG,\nCcat)$.}
 In {the present} section we recall the construction 
 of the equivariant coarse homology theory 
$$K\bC\cX^{G} \colon G\BC\to \Sp^{\la}$$ introduced in \cite{coarsek} (see Definition \ref{qrogijeqoifefewfefewqffe}) which will give rise to the equivariant local   $K$-homology  {$K_\bC^{G,\cX}$ {described in} Definition \ref{qriofjqofewfefqeqff}.

In order to define the functor $K\bC\cX^{G}$ the {coefficient category  $   \bC $ must   be effectively additive (Definition \ref{wigjosgbsgfgsdfgs91}).  
If $\bC$ also admits countable AV-sums (Definition \ref{wigjosgbsgfgsdfgs9}), then  $K\bC\cX^{G}$ is an equivariant coarse homology theory. Finally, in order to ensure strong additivity of $K\bC\cX^{G}$ {by \cite[Thm.11.1]{coarsek}} we must assume the existence of all very small AV-sums.}

{
\begin{ex}
For  $    A$  in $\Fun(BG,\nCalg)$ the category $\Hilb_{c}(A)$
in $\Fun(BG,\nCcat)$ admits all small AV-sums and is idempotent complete, hence is in particular effectively additive. It therefore satisfies all {the} conditions listed above. \hB\end{ex}}
 
Let $X$ be a set.
Subsets of $X\times X$ will be called entourages on $X$.
The set $\cP_{X\times X}$  of all entourages is a monoid with involution, where
the composition of {the} entourages $U$ and $V$ is  the entourage $$U\circ V \coloneqq \pr_{14}[(U\times V)\cap (X\times \diag(X)\times X)]\, ,$$
the unit is the entourage $\diag(X)$, and the involution is given by the formula $$U^{*} \coloneqq \{(y,x)\mid (x,y)\in U\}\, .$$
The monoid $\cP_{X\times X} $ acts on $\cP_{X}$ by
\begin{equation}\label{wfwfwefsdfsdfd}
(U,Y)\mapsto U[Y] \coloneqq \pr_{1}[U\cap (X\times Y)]\, .
\end{equation} 
A $G$-coarse structure $\cC$ on a $G$-set $X$ is by definition  a $G$-invariant  submonoid of $\cP_{X\times X}$ which is  
closed under taking subsets, {applying the involution,} and forming finite  unions, and in which the subset of $G$-invariant entourages $\cC^{G}$ is cofinal with respect to the inclusion relation.
A $G$-coarse space is a pair $(X,\cC)$ of a $G$-set and a $G$-coarse structure. If $ (X,\cC)$  and $ (X',\cC')$ are two $G$-coarse spaces and $f \colon X\to X'$ is an equivariant map of the  underlying $G$-sets, then $f$ is controlled if $(f\times f)(\cC)\subseteq \cC'$. Finally, 
a coarse structure $\cC$ is compatible with a bornology $\cB$ if
$\cC[\cB]\subseteq \cB$.

The category $G\BC$ of $G$-bornological coarse spaces was introduced in \cite[Def. 2.1]{equicoarse}. Its objects are triples $(X,\cC,\cB)$ of a very small 
$G$-set $X$ with a $G$-coarse structure $\cC$     and a $G$-bornology $\cB$ which is compatible with $\cC$. Morphisms are maps of $G$-sets which are controlled   and proper.
 We usually use the shorter  notation $X$ for $G$-bornological coarse spaces.

Let $X$ be in $G\BC$. Then we can consider the category \begin{equation}
\Cglf(X) \coloneqq \lim_{BG} \tCglf(X)
\end{equation}
in $\Ccat$, {where $ \tCglf(X)$  in $\Fun(BG,\Ccat)$  {is as}  introduced  in Definition~\ref{qrohwoifewfqwefqwfewqf}. Explicity, $ \Cglf(X) $}
 is the wide subcategory of $ \tCglf(X)$ consisting of the $G$-invariant morphisms, i.e., morphisms $A$ satisfying $g\cdot A=A$ for all $g$ in $G$, {where the $G$-action is given by  {formula}} 
  \eqref{eqwfijoifqjowiefjwoefewqfewqf}.
 Note that this construction does not  use the coarse structure yet, but this will be the case in the following. 
  
If $Y,Y'$ are two subsets of $X$ and $U$ is an entourage of $X$, then we say that $Y'$ is $U$-separated from $Y$ if $Y'\cap U[Y]=\emptyset$, see \eqref{wfwfwefsdfsdfd} for the definition of the $U$-thickening $U[Y]$ of $Y$. 
We say that  a morphism $A \colon (C,\rho,\mu)\to (C',\rho',\mu')$ in $ \Cglf(X)$ is $U$-controlled if
$\mu'(Y')A\mu(Y)=0$ for all pairs of subsets $Y',Y$ of $X$ such that $Y'$ is $U$-separated from $Y$.

\begin{ddd}\label{rfquhwfiuqwhfiufewqefqwefqwefwefqwef}
We define  $\bCgtsmc(X)$ in $\Ccat$ as follows:
\begin{enumerate}
\item objects: The objects of $ \bCgtsmc(X)$ are the objects of $\Cglf(X)$.
\item  \label{weoigjowegergegwrggwergw222} morphisms: The space of morphisms $\Hom_{ \bCgtsmc(X)}((C,\rho,\mu),(C',\rho',\mu'))$
is the closed subspace of $\Hom_{\Cglf(X)}((C,\rho,\mu),(C',\rho',\mu'))$ generated by {those} 
morphisms which are $U$-controlled for some coarse entourage $U$ of $X$.
\item\label{weoigjowegergegwrggwergw} composition and involution: The composition and the involution of $ \bCgtsmc(X)$ {are} inherited from $\Cglf(X)$. 
\end{enumerate}
\end{ddd}
{One must check that the composition defined in 
  Point \ref{weoigjowegergegwrggwergw}  preserves the morphism spaces 
  defined in Point \ref{weoigjowegergegwrggwergw222}. We refer to 
  \cite[Sec.~4]{coarsek} for the argument.}
 
{Let $\bC$ in $\Fun(BG,\nCcat)$ be effectively additive.}
\begin{ddd}\label{wtogjoergrwegreggwege}
We define a functor
 $$\bCgtsmc \colon G\BC\to \Ccat$$  as follows:
 \begin{enumerate}
 \item objects: The functor $\bCgtsmc$ sends $X$ in $G\BC$ to $\bCgtsmc(X)$ in $\Ccat$.
 \item morphisms: The functor $\bCgtsmc$ sends a morphism $f \colon X\to X'$ {in $G\BC$} to the functor
 $f_{*} \colon \bCgtsmc(X)\to \bCgtsmc(X')$ defined as follows:
 \begin{enumerate}
 \item objects: $f_{*}(C,\rho,\mu) \coloneqq (C,\rho,f_{*}\mu)$.
 \item \label{oiwejgiowegregwr}morphisms: $f_{*}(A) \coloneqq A$.
  \end{enumerate}
\end{enumerate}
\end{ddd}
For the verification that $f_{*}$ is well-defined  
we again refer to  \cite[Sec.~4]{coarsek}. It is at this point where we need the assumption  {that $\bC$ is effectively additive.} 

Using the functors from  \eqref{werferfrewgrerefw} for the trivial group we define 
 the topological $K$-theory functor for $C^{*}$-categories as the composition
\begin{equation}\label{rewgwergwergegwergrewgregwregegwre}
\Kcat \colon  \nCcat \stackrel{\kk_{\Ccat}}{\to} \KK \stackrel{\KK(\C,-)}{\to}  \Sp^{\la}\, .
\end{equation} 
The functor  \eqref{rewgwergwergegwergrewgregwregegwre} is equivalent to the functors  considered in  \cite{joachimcat}, \cite[Sec.\ 8.5]{buen}, \cite[Sec.\ {14}]{cank}. 
Note that here we consider $C^{*}$-algebras like $\C$ as $C^{*}$-categories with a single object.

{Let $\bC$ be in $\Fun(BG,\nCcat)$ be effectively additive.}

\begin{ddd}\label{qrogijeqoifefewfefewqffe}
We define the functor $K\bC\cX^{G}$ as the composition
$$K\bC\cX^{G} \colon G\BC\xrightarrow{\bCgtsmc} \Ccat\xrightarrow{\Kcat} \Sp^{\la}\, .$$
\end{ddd}

For the definition of the notion of 
an  equivariant coarse homology theory we refer to   \cite[Def. 3.10]{equicoarse}. References for additional properties 
are:
\begin{enumerate}
\item strongly additive:  \cite[Def. 3.12]{equicoarse}
\item strongness:   \cite[Def. 4.19]{equicoarse}
\item continuity:   \cite[Def. 5.15]{equicoarse}\ .
\end{enumerate}
 The following  theorem is shown in \cite[Sec.~6]{coarsek} (and  \cite[Sec.~11]{coarsek} for strong additivity).
 
\begin{theorem}\label{qeroigqwefqewfefewfq} {If  $\bC$    in $\Fun(BG,\nCcat)$ is effectively additive and admits   countable AV-sums, then} 
 $K\bC\cX^{G}$ is an equivariant coarse homology theory which is in addition strong and continuous. {If $\bC$ admits all very small AV-sums, then $K\bC\cX^{G}$ is}  strongly additive.
 \end{theorem}
 
{By construction the functors $\bCgtsmc$ and $K\bC\cX^{G}$  depend functorially on the coefficient category $\bC$  in $\Fun(BG,\nCcat_{\ndeg,\eadd,\omega\add})$.}

\section{\texorpdfstring{$\boldsymbol{G}$}{G}-uniform bornological coarse spaces, cones and \texorpdfstring{$\boldsymbol{K_{\bC}^{G,\cX}}$}{KGXC}}\label{woigjowteggwergewrgerg}

A $G$-uniform structure on $X$ is a $G$-invariant subset $\cU$ of $\cP_{X\times X}$ consisting of entourages  containing the diagonal,  which
 is closed under taking supersets, finite intersections, compositions, and the involution, and which has the property that every $U$ in $\cU$ contains a $G$-invariant element of $\cU$ and  admits $V$ in $\cU$ with $V\circ V\subseteq U$.  A $G$-uniform space is a pair  $(X,\cU)$ of a $G$-set and a $G$-uniform structure.
 Let $(X,\cU)$ and  $(X',\cU')$ be $G$-uniform spaces and $f \colon X\to X'$ be a $G$-invariant map of the underlying sets. Then $f$ is uniform  if $(f\times f)^{-1}(\cU')\subseteq \cU$.
  A uniform structure $\cU$ is compatible with a coarse structure if $\cU\cap \cC\not=\emptyset$.

Let $G\UBC$ denote  the category of $G$-uniform bornological coarse spaces introduced in \cite[Def. 9.9]{equicoarse}.  Objects are tuples $(X,\cC,\cB,\cU)$ such that $(X,\cC,\cB)$ is a $G$-bornological coarse space and $\cU$ is a $G$-uniform structure  compatible with $\cC$. Morphisms are morphisms of $G$-bornological coarse spaces which are in addition uniform.  We will usually use the shorter  notation $X$ for $G$-uniform bornological coarse spaces.
We have  canonical forgetful functors 
\begin{equation}\label{qweflkjoqwfeqewfwefeefeqwf}
G\UBC\to G\BC\ , \quad G\UBC\to G\Top
\end{equation} 
 which forget the uniform structure or take the underlying $G$-topological space, respectively.

If not said differently we will consider all subsets of $\R^{n}$ 
{as} objects of $G\UBC$ with the 
  trivial $G$-action and the structures induced by the standard metric.

   The categories $G\BC$ and $G\UBC$ have monoidal structures $\otimes$ which  {are} the cartesian structure on the underlying $G$-uniform and $G$-coarse spaces (see   \cite[Ex.\ 2.17]{equicoarse} for the case of $G\BC$) such that the forgetful functor 
   $G\UBC\to G\BC$ is symmetric monoidal in the canonical way. 
   The bornology on {$X \otimes X'$} is generated by the subsets $B\times B'$ for all bounded subsets $B$ of $X$ and $B'$ of $X'$, respectively.

Let $X$ be in $G\UBC$.    
\begin{ddd}
$X$ is flasque if  it is a retract of $[0,\infty)\otimes X$. 
\end{ddd}

Note that this definition is a little more restrictive than the definition given in \cite[Text before Def.\ {3.10}]{ass}. 
The same argument as for   \cite[Lem.\ 3.28]{buen}  in {the} non-equivariant case shows {that} 
  the  underlying $G$-bornological coarse space of $X$ is flasque in the generalized sense.
   
The notion of homotopy in the category $G\UBC$ is defined in the usual   manner using the interval 
functor $X\mapsto [0,1]\otimes X$. 
   
Recall the definitions of uniformly or coarsely excisive pairs from   \cite[Def.\ 3.3]{ass}  and  \cite[Def.\ 3.5]{ass}.

Let $E \colon G\UBC \to \bM$ be a functor whose target is a stable $\infty$-category.
\begin{ddd}\label{qeroigjoerwgwergwrgwe9111}
\mbox{}
\begin{enumerate}
\item $E$ is homotopy invariant if it sends the projection $[0,1]\otimes X\to X$ to an equivalence for every $X$ in $G\UBC$.
\item $E$ satisfies closed excision if  $E(\emptyset)\simeq 0$ and for every uniformly   and coarsely excisive pair $(Y,Z)$ of invariant closed subsets of some $X$ in $G\UBC$
such that $X=Y\cup Z$ the square
$$\xymatrix{E(Y\cap Z)\ar[r]\ar[d]&E(Y)\ar[d]\\E(Z)\ar[r]&E(X)}$$
is a push-out square.
\item $E$ vanishes on flasques if $E(X)\simeq 0$ for any flasque $X$ in $G\UBC$.
\item $E$ is $u$-continuous if for every $X$ in $G\UBC$ we have $ \colim_{V} E(X_{V})\simeq  E(X) $, where $V$ runs over $\cC^{G}\cap \cU$,  
and $X_{V}$ is obtained from $X$  by  replacing its coarse structure $\cC$   on $X$   by the coarse structure   generated by $V$.
\end{enumerate}
\end{ddd}


Let $X$ be in $G\UBC$ with uniform structure $\cU$.  
Note that $\cU$ and ${\cP_{X\times X}}$ are posets with respect to the inclusion relation.

\begin{ddd} \label{wfiuhiufqwefeewqfqwfewfqewf}
A    scale for $X$ is a non-increasing function $\psi \colon \R\to \cP(X\times X)^{G}$  with the following properties:
\begin{enumerate}
\item If $t$ is in $(-\infty,0]$, then $\psi(t)=X\times X$.
\item \label{gpowgergerlgewrgwregerg} For every   $V$  in $\cU$ there exists $t_{0}$ in $\R$ such that $\psi(t)
 \subseteq V$ for all $t$ in $[t_{0},\infty)$.
\end{enumerate}
\end{ddd}

\begin{ddd}\label{qriugiqregergwerg}
We define the geometric cone-at-$\infty$ of $X$   to be the object $\cO^{\infty}(X)$ in  $G\BC$ given as follows:
\begin{enumerate}
\item The underlying $G$-set of $\cO^{\infty}(X)$ is $\R\times X$.
\item The bornology of $\cO^{\infty}(X)$ is generated by the subsets $[-r,r]\times  B$ for all $r$ in $(0,\infty)$ and bounded subsets $B$ of $X$.
\item\label{qoirgjoqierfewfqfewfw} The coarse structure is  generated by the  entourages $U\cap U_{\psi}$ for all scales  $\psi$, where   $U$ is a coarse entourage of $\R\otimes X$ and
\begin{align}
\mathclap{
U_{\psi} \coloneqq \{((s,x),(t,y))\in (\R\times X)\times ( \R\times X)\:|\:   (x,y)\in \psi(\max\{s,t\})\}\,.
}\notag\\
\label{fqewfpokopqkfqef}
\end{align}
 \end{enumerate}
We furthermore define the cone $\cO(X)$  of $X$ to be the subset $[0,\infty)\times X$  {of $\cO^{\infty}(X)$} with the  induced  structures.
\end{ddd}


\begin{ddd}\label{wepogjpweggreewg}
We define functors $$\cO^{\infty},\cO \colon G\UBC\to G\BC$$ as follows:
\begin{enumerate}
\item objects: The functors send $X$ in $G\UBC$ to $\cO^{\infty}(X)$ or $\cO(X)$, respectively.
\item \label{qwoifhoirfqfefqwef} morphisms: The functors send a morphism $f \colon X\to X'$ in $G\UBC$ to the morphism
$\cO^{\infty}(X)\to \cO^{\infty}(X')$ or $\cO (X)\to \cO (X')$ given by
$\id_{\R}\times f$ or $\id_{[0,\infty)}\times f$, respectively.
\end{enumerate}
\end{ddd}

The definition of the functors for morphisms in Point \ref{qwoifhoirfqfefqwef} needs a justification which is given e.g.\ by {a specialization of the argument for} \cite[Lem.\ 5.15]{buen}.

    For $X$ in $G\UBC$ we have a natural sequence of maps in $G\UBC$ \begin{equation}\label{ecelkmlwqcqwecwecqcqwecq}
X\to \cO(X)\to \cO^{\infty}(X)\to \R\otimes X
\end{equation}
 called the cone sequence. Here the first map is given by $x\mapsto (0,x)$, the second map is the inclusion, and the third map is the identity on the underlying sets.

 Let $E \colon G\BC\to \bM$ be a functor with target a stable $\infty$-category.
Then we consider the functors
\begin{align}\label{wqefqwefewfqewf}
E\cO^{\infty} &\coloneqq E\circ \cO^{\infty} \colon G\UBC\to \bM \\
 E\cO& \coloneqq E\circ \cO \colon G\UBC\to \bM \, . \nonumber
\end{align}
\begin{prop}\label{qroigjqowfewfqwefqw}
We assume that $E$ is a coarse homology {theory} which is in addition strong.
Then the functors $E\cO^{\infty}$ and $E\cO$ have the following properties:
\begin{enumerate}
\item homotopy invariance,
\item closed excision,
\item vanishing on flasques,
\item $u$-continuous.
\end{enumerate}
Moreover, the cone sequence  \eqref{ecelkmlwqcqwecwecqcqwecq} induces a  fibre sequence of functors
\begin{equation}\label{wefqwefqewewfwdqewdewdqede}
E\to E\cO\to E\cO^{\infty}\xrightarrow{\partial^{\Cone}} \Sigma E\, .
\end{equation}
\end{prop}
This proposition follows from the results stated in \cite[Sec.~9]{ass} (which are stated there in the non-equivariant case, but the same proof applies here). In particular, the list of properties of the functors is given by  \cite[Lem.~9.6]{ass} and 
the cone sequence follows from  \cite[(9.1)]{ass}. 
Note that we consider $E$ {in \eqref{wefqwefqewewfwdqewdewdqede}} as a functor on $G\UBC$ by using the {first} forgetful functor {in} \eqref{qweflkjoqwfeqewfwefeefeqwf}.  

  Let $Y$ be in $G\BC$ and $E \colon G\BC\to \bM$  be some functor.  \begin{ddd}[{{\cite[(10.17)]{equicoarse}}}]\label{regoiergowregrwergwreg} We define the twist $E_{Y}$ of $E$ by $Y$ as the functor
 $$E_{Y} \colon G\BC\to \bM\, , \quad E_{Y}(X) \coloneqq E(X\otimes Y)\, .$$
 \end{ddd} 
 
 The following has been shown in \cite[{Lem.\ 4.17 \& 11.25}]{equicoarse}:
 \begin{lem}
 If $E$ is a coarse homology theory, then so is its twist $E_{Y}$.
 If $E$ is strong, then so is $E_{Y}$.
 \end{lem}
 
 We apply this construction to the equivariant coarse homology theory $K\bC\cX^{G}$ from  Definition \ref{qrogijeqoifefewfefewqffe}. 
  The group $G$ gives rise to the $G$-bornological coarse spaces
$G_{can,min}$    \cite[Ex.\ 2.4]{equicoarse} and also  $G_{can,max}$.
Here $min$ and $max$ refer to the minimal (finite subsets) and maximal (all subsets) bornologies, and the canonical coarse structure {$can$}
is the minimal $G$-coarse structure such that $G_{can}$ is a connected $G$-coarse space. It is generated by the entourages
$\{(g,h)\}$ for all {$(g,h)$ in $G \times G$.} 
Later we will in particular consider the coarse homology theories
$K\bC\cX^{G}_{G_{can,max}}$ and $K\bC\cX^{G}_{G_{can,min}}$ obtained from 
$K\bC\cX^{G}$ by {twisting 
with} $G_{can,max}$ and $G_{can,min}$, {respectively}.

 Let $\bC$ be in $\Fun(BG,\nCcat)$ be effectively additive.   \begin{ddd}\label{qriofjqofewfefqeqff} 
We define the equivariant local  $K$-homology functor
$$K_{\bC}^{G,\cX} \colon G\UBC\to \Sp^{\la}$$
as the composition
$$ K_{\bC}^{G,\cX} \colon G\UBC\xrightarrow{\cO^{\infty}} G\BC \xrightarrow{K\bC\cX^{G}_{G_{can,max}}} \Sp^{\la}\,.$$
\end{ddd}

The following proposition   lists the properties of the functor $K_{\bC}^{G,\cX}$.
 It is a consequence of Theorem \ref{qeroigqwefqewfefewfq} and Proposition \ref{qroigjqowfewfqwefqw}.
 \begin{prop}\label{oirgjoerwgwegregw9}
{If $\bC$ is effectively additive and admits all countable AV-sums, then}
the functor $K_{\bC}^{G,\cX}$ has the following properties:
\begin{enumerate}
\item closed excision,
\item homotopy invariant,
\item $u$-continuous,
\item vanishing on flasques.
\end{enumerate}
\end{prop}
  
 {The functor  $K_{\bC}^{G,\cX}$   depends functorially on coefficient category $\bC$ in $\Fun(BG,\nCcat_{\ndeg,\eadd,\omega\add})$.}

\section{Locality and pseudolocality}\label{qriugfhiergfedfqewfqewf9}

 For a set $X$ we let $\ell^{\infty}(X)$ denote the $C^{*}$-algebra of all bounded functions $X\to \C$
with the supremum norm $\|f\|:=\sup_{x\in X} |f(x)|$.

 For an entourage $U$  on $X$  and a subset $W$   we define the $U$-variation on $W$ of  a function $f:X\to \C$  by
\[
\Var_{U}(f,W):=\sup_{(x,y)\in U\cap (W\times W)} |f(x)-f(y)|\ .
\]

Let $\cY$ be a filtered family of subsets in $X$, ordered by inclusion.
 
 \begin{ddd}
 	
\label{gjsoepgergseffs} \mbox{}
\begin{enumerate}
\item	
The $C^{*}$-algebra $\ell^{\infty}(\cY)$ of functions  {vanishing} away from $\cY$ is defined as the sub-$C^*$-algebra of $\ell^{\infty}(X)$ of functions $f$ satisfying 
 $$\lim_{Y\in \cY} \|f_{X\setminus Y}\|=0\ .$$ 
 \item
 \label{wkotpgegfergweg} 
For a coarse space $X$ with coarse structure $\cC$ we define the algebra of bounded functions with vanishing variation away from $\cY$  as 
\[
\ell^{\infty}_{\cY}(X):=\{f\in \ell^{\infty}(X)\mid  { \forall U\in \cC : {\lim_{Y\in \cY}}\Var_{U}(f,X\setminus Y ) =0 }\}\ .
\]
\end{enumerate}
 \end{ddd}
 If $X$ is a coarse space, then $\cY$ is a big family if for every $Y$ in $\cY$ and coarse entourage $U$ of $X$
 the thickening $U[Y]$ is again contained in a member  of $\cY$ \cite[Def. 3.2]{buen}.
 If $\cY$ is a big family, then 
 we have  $\ell^{\infty}(\cY)\subseteq \ell^{\infty}_{\cY}(X)$.

 {For  $\bC$   in $\Fun(BG,\nCcat)$ and}
  $X$  in {$G\Born$} we consider  the $G$-$C^{*}$-category 
$\tCglf(X)$ introduced  in Definition \ref{qrohwoifewfqwefqwfewqf}.  
 Let $(C,\rho,\mu)$ be an object in 
 $\tCglf(X)$.
We  then extend the projection-valued measure $\mu$ to a homomorphism of $C^{*}$-algebras $$\mu \colon  \ell^{\infty}(X)\to \End_{{\bM}\bC}(C)$$
  which sends $f$ in $\ell^{\infty}( X)$ to
\begin{equation}\label{ewqfpojlqkrmeflkwefqef1}
\mu(f) \coloneqq \int_{ X}  f  d\mu\, .
\end{equation}
 \begin{rem}
This integral   can be interpreted as follows. 
For every $x$ in $X$ we can choose a representative
$u_{x} \colon C_{x}\to C$ of the image  in $\bM\bC$ of the projection $\mu(\{x\})$ on $C$. By Definition  \ref{qeroigergeggweerg}.\ref{qeroighjoergfqewfqf}
\begin{equation}\label{feqwfpokjfpowefkofqwpeofkoqkkopk}
(C,(u_{x})_{x  \in X})
\end{equation}
represents the AV-sum of the family $(C_{x})_{x\in X }$. 
Using that $f$ is bounded and that the family $(u_{x})_{x\in  X }$ is mutually orthogonal we conclude using  \cite[{Lem.\ 7.8}]{cank} that the sum
$$\mu(f):=\sum_{x\in   X } u_{x}f(x) u_{x}^{*}$$   {strictly converges   in $\bM\bC$.} 
\hB
\end{rem}
 One checks  that $\mu$ is a homomorphism of $C^{*}$-algebras and that
 $\mu(\chi_{Y})=\mu(Y)$ for the characteristic function $\chi_{Y}$ of a subset $Y$ of $X$. Furthermore,  using the equivariance \eqref{wfoizeqwfu98u9e8wfqewf} of $\mu$,  one checks that $\phi$ is equivariant in the sense that 
\begin{equation}\label{gwwgegdfsfvfdvsdvsfv}
g^{-1}\cdot \mu(f)=\mu(g^{*}f)
\end{equation}
for all $g$ in $G$, see \eqref{eqwfijoifqjowiefjwoefewqfewqf} for notation.

  Let $X$ be in $G\BC$  and {$\cY$} be a big family on $X$. 
Let $(C,\rho,\mu)$, $(C',\rho',\mu')$ be objects of    $\tCglf(X)$  and
  $A:(C,\rho,\mu)\to(C',\rho',\mu') $ be a  morphism in this $C^{*}$-category.   
    The argument  for the following commutator estimate is taken from \cite{quro}, see also \cite[Lemma 3.9]{Bunke:2024aa}.

  \begin{lem}
  \label{kohpkertphokgpertge} 
  If $f$  is in $\ell^{\infty}_{\cY}(X)$ and $A$ is $U$-controlled for some coarse entourage $U$, then 
 \[
 {\lim_{Y \in \cY} \|\mu'(X\setminus Y)\:(\mu'(f)A-A\mu(f))\:\mu(X\setminus Y)\|=0\ .}
 \]
  \end{lem}
  \begin{proof} 
Let $\epsilon$ in $(0,\infty)$ be given and set $\eta := \epsilon /4 \|A\|$. 
We then choose $Y$ in $\cY$ such that $\Var_{U}(f, X\setminus Y')\leq \eta$ for each $Y'$ in $\cY$ with $Y \subseteq Y'$. 
We define the partition $(S_{k})_{k\in \Z}$  of $X \setminus Y$ by\ \begin{equation*}
  S_k := \{ x \in X \setminus Y \mid (k-1) \eta \leq f(x) < k\eta\}\ .
\end{equation*}
Since $f$ is bounded, only finitely many of these sets are non-empty.
If $k,l$ are in $\Z$, then  $x\in S_k$ and $y \in S_l$ implies $|f(x) - f(y)| \geq (|k-l|-1)\eta$.
Since  the $U$-variation of $f$ on   ${X \setminus} Y =\bigcup_{k\in \Z}S_{k}$ is bounded by $\eta$, the condition $|k-l|\geq 2$ implies that $S_k \cap U[S_l] = U[S_k] \cap S_l = \emptyset$.
Since $A$ is $U$-controlled we can conclude that $\mu'(S_k) A \mu(S_l) = 0$.

We set
\begin{equation*}
 \tilde{f} := \chi_{Y}\cdot f + \eta \sum_{k \in \Z} k \cdot \chi_{S_k}\ .	
\end{equation*}
Then by construction $\|\tilde{f} - f\ \| \leq \eta$ and hence
\begin{equation}
\label{FirstComparison}
  \|(\mu^\prime(f) A - A \mu(f)) - (\mu^\prime(\tilde{f}) A - A \mu(\tilde{f})) \| \leq 2\eta\|A\|	= \frac{\epsilon}{2}\ .
\end{equation}
Since $A$ is $U$-controlled, we have
\begin{equation}
\label{okwgpergrefwf}	
\begin{aligned}
&\mu'(X\setminus U[Y])(\mu'(\tilde{f})A-A\mu(\tilde{f}))\mu(X\setminus U[Y]) \\
&\qquad\qquad\qquad  = \eta\sum_{k \in\Z} k \cdot \mu'(X\setminus U[Y])(\mu^\prime(S_k) A - A \mu(S_k)) \mu(X \setminus U[Y]) \ .
\end{aligned} 
\end{equation}
Inserting the identities $\mu(X \setminus Y) = \sum_{k \in \Z} \mu(S_k)$ and $\mu'(X \setminus Y) = \sum_{k \in \Z} \mu'(S_k)$ and   using that $\mu'(S_k) A \mu(S_l) = 0$ whenever $|k-l| \geq 2$, we get
\begin{equation*}
  \sum_{k \in \Z} k (\mu^\prime(S_k) A - A \mu(S_k))  = \mu'(X \setminus Y)  \sum_{k \in \Z} (\mu^\prime(S_k) A \mu(S_{k-1}) - \mu^\prime(S_k) A \mu(S_{k+1}))\mu(X \setminus Y)\ .
\end{equation*}
The right-hand side is an operator  with norm bounded by $2\|A\|$.
Using $\mu(X\setminus U[Y])=\mu(X\setminus U[Y])\mu(X\setminus Y)$ and plugging the above equality into  \eqref{okwgpergrefwf}, we get
\begin{equation*}
\|\mu'(X\setminus U[Y])(\mu'(\tilde{f})A-A\mu(\tilde{f}))\mu(X\setminus U[Y])\| \leq 2\eta\|A\|= \frac{\epsilon}{2}.
\end{equation*}
Combining this with \eqref{FirstComparison}, we see that  
\begin{equation*}
\|\mu'(X\setminus Y')(\mu'(f)A-A\mu(f))\mu(X\setminus Y')\| \leq \varepsilon
\end{equation*} 
for all $Y'$ in $\cY$ with $U[Y] \subseteq Y'$.
   \end{proof}

 Recall Definition \ref{rfquhwfiuqwhfiufewqefqwefqwefwefqwef} of the $C^{*}$-category 
 $ \bCgtsmc(X)$.
   \begin{kor}\label{vfjfvjpfvosdfvs} For a morphism $A:(C,\rho,\mu)\to (C',\rho',\mu')$  in $ \bCgtsmc(X)$ and $f$ in $\ell^{\infty}_{\cY}(X)$ we have
    $$\lim_{{Y \in \cY}}\| \mu'(X\setminus Y)(\mu'(f)A-A\mu(f))\mu (X\setminus Y)\|=0\ .$$
    \end{kor}
    \begin{proof}
    We use that $A$ can be approximated in norm by $U$-controlled equivariant morphisms $A'$
    and apply Lemma \ref{kohpkertphokgpertge} to the approximants $A'$.
    \end{proof}

{If} $Y$ is an invariant subset  of  $X$, then we  define the wide subcategory
$\bCgtsmc (Y\subseteq X)$ of $\bCgtsmc(X)$ (see \cite[Def.\ 5.5]{coarsek}) such that 
 for objects $(C,\rho,\mu)$ and $(C',\rho',\mu')$ in  
$\bCgtsmc
(Y\subseteq X)$  
$$\Hom_{\bCgtsmc
(Y\subseteq X)}((C,\rho,\mu),(C',\rho',\mu')) \coloneqq \mu'(Y) \Hom_{\bCgtsmc
(  X)}((C,\rho,\mu),(C',\rho',\mu'))\mu(Y)\, .$$
Similarly, for an invariant big family $\cY=(Y_{i})_{I\in I}$ on $X$ (see \cite[Def. 3.5]{equicoarse}) we have  the wide subcategory  \begin{equation}\label{f23r8z89fz2893zf892334234f2}
\bCgtsmc
(\cY\subseteq X) \coloneqq \overline{\bigcup_{i\in I}\bCgtsmc
(Y_{i}\subseteq X)}
\end{equation}
  of $\bCgtsmc(X)$ (the union and closure are both taken on the level of morphisms). By \cite[Lem.~5.9]{coarsek}  we know that $\bCgtsmc
(\cY\subseteq X)$
 is an ideal  in $\bCgtsmc
(  X)$.

    \begin{kor}
   \label{gokpergergsgre}
    For a morphism $A:(C,\rho,\mu)\to (C',\rho',\mu')$  in $ \bCgtsmc(X)$ and $f$ in $\ell^{\infty}_{\cY}(X)$ for an invariant big family $\cY$  on $X$ we have
 we have
$\mu'(f)A-A\mu(f)\in \bCgtsmc
(\cY\subseteq X)$.  
\end{kor}

Let $X$ be in $G\UBC$ and   $\cB$ denote  the bornology of $X$.
\begin{ddd} \label{kohperthrtgertg}\mbox{} \begin{enumerate}
\item  We let $  C_{u}(X)\subseteq  \ell^{\infty}(X)$ be the sub-algebra of uniformly continuous functions on $X$.
\item \label{kohethertgtergelabel} We   set
 $  C_{0}(X):= C_{u}(X)\cap  \ell^{\infty}(\cB)$.
\end{enumerate}
 \end{ddd}
 
Note the discussion \cite[3.13]{Bunke:2024aa} about the difference of  $C_{0}(X)$
and the possibly smaller $C^{*}$-algebra $C_{u}(\cB)$ generated by uniformly continuous  functions supported on bounded 
 subsets.
 
Recall the cone construction $\cO:G\UBC\to G\BC$ introduced in Definition \ref{wepogjpweggreewg}.
For $X$ in $G\UBC$ we consider  $\cO(X)\otimes G_{can,max}$ in $G\BC$. The underlying $G$-set of  this $G$-bornological coarse space is $[0,\infty)\times X\times G$. We let $\pi:
 [0,\infty)\times X\times G\to X$ be the projection.
 It induces a homomorphism
 $\pi^{*}:\ell^{\infty}(X)\to \ell^{\infty}(\cO(X)\otimes G_{can,max})$.
 
 In the following $\cB$ denotes the bornology of $\cO(X)\otimes G_{can,max}$.
\begin{lem}\label{kohprethegrtgertg}
The homomorphism $\pi^{*}$ restricts to a homomorphism
$$\pi^{*}:C_{0}(X)\to \ell^{\infty}_{\cB}(\cO(X)\otimes G_{can,max}) \ .$$
  \end{lem}
\begin{proof}
Let $f$ be in $C_{0}(X)$ and $V$ be a coarse entourage of $\cO(X)\otimes G_{can,max}$.
For every $\epsilon$ in $(0,\infty)$ we must find a bounded subset $A$ of
$\cO(X)\otimes G_{can,max}$ such that $\Var_{V}(\pi^{*}f,X\setminus A)\le \epsilon$.

We can find a bounded subset
$B$  of $X$ such that $\|\chi_{X\setminus B} f\|\le \frac{\epsilon}{2}$.
By uniform continuity we can further find a uniform entourage $U$ of $X$ such that
$\Var_{U}(f,X)\le \epsilon$.
   There exists $t$ in  $(0,\infty)$  such that
 $((s,x,g),(s',x',g'))\in V$ and $s\ge t$ or $s'\ge t$ implies $(x,x')\in U$.
 It follows that $\Var_{V}(\pi^{*}f,Y_{t})\le   \epsilon $, where  $Y_{t}:=[t,\infty)\times  X\times G$.

 We also have $\|    \chi_{ \pi^{-1}(X\setminus B)}\pi^{*}f \|\le  \frac{\epsilon}{2}$ so that actually
 $\Var_{V}(\pi^{*}f,  Y_{t}\cup  \pi^{-1}(X\setminus B) \|\le   \epsilon$.   
   Finally note that $A:= (\cO(X)\otimes G_{can,max})\setminus (Y_{t}\cup  \pi^{-1}(X\setminus B))$
  is a bounded subset of $ \cO(X)\otimes G_{can,max}$.
   \end{proof}

Let $(C,\rho,\mu)$ be an object of  $\bCgtsmc(\cO(X)\otimes  G_{can,max})$. Using \eqref{ewqfpojlqkrmeflkwefqef1} we define
the homomorphism 
  \begin{equation}\label{ewqfpojlqkrmeflkwefqef}\phi \colon \ell^{\infty}(X)\to \End_{{\bM}\bC}(C)\ , \quad f\mapsto   \phi(f) \coloneqq  \mu(\pi^{*}f)\ .
 \end{equation}
  
Let $A \colon (C,\rho,\mu)\to (C',\rho',\mu')$ be a morphism in  $\bCgtsmc(\cO(X)\otimes  G_{can,max})$.  
{Recall that $A$ is in particular   a  multiplier morphism from $C$ to $C'$.}
 Our next result states that $A$ is pseudolocal {(in the sense of \cite[Def.~12.3.1]{higson_roe} {{if} one replaces the ideal of compact operators  in all bounded operators by the ideal $\bC$ in the multiplier category $\bM\bC$}} and  we consider the objects of $\bCgtsmc(\cO(X)\otimes  G_{can,max}) $ as $X$-controlled via \eqref{ewqfpojlqkrmeflkwefqef}). Let  $\phi'$ be defined as in \eqref{ewqfpojlqkrmeflkwefqef}, but  for the object $(C',\rho',\mu')$.

\begin{lem}\label{wtioghjwergfgrefwref} For    $f$  in $C_{0}(X)$  
the difference $A\phi(f)-\phi'(f)A$ belongs to   {$\bC$}. 
 \end{lem}
\begin{proof}Recall that  $\cB$ denotes the bornology of $\cO(X)\otimes G_{can,max}$.
By Lemma \ref{kohprethegrtgertg} we have
$$ \pi^{*}f\in \ell_{\cB}(\cO(X)\otimes  G_{can,max})\ .$$
By Corollary \ref{gokpergergsgre} we have 
$$A\phi(f)-\phi'(f)A\in \bCgtsmc(\cB\subseteq \cO(X)\otimes  G_{can,max})\ .$$
 By local finiteness of the objects of  $ \bCgtsmc(  \cO(X)\otimes  G_{can,max})$  we conclude that
 $$A\phi(f)-\phi'(f)A:C\to C'$$ is a morphism in $\bC$.
 \end{proof}

We consider the big family   
\begin{equation}\label{fqewqewfkqwpefowqefewfef}
\cZ \coloneqq (Z_{n})_{n\in \nat}\, , \quad 
Z_{n} \coloneqq [0,n]\times X\times G\, .
\end{equation}
on $\cO(X)\otimes G_{can,max}$.

Let $A \colon (C,\rho,\mu)\to (C',\rho',\mu')$ be a morphism in  $\bCgtsmc(\cZ\subseteq  \cO_{\tau}(X)\otimes G_{can,max})$, see \eqref{f23r8z89fz2893zf892334234f2}.  Our next result shows that it locally belongs to ${\bC}$.   
 Let $\cB_{X}$ denote the bornology of $X$.
  
\begin{lem}\label{ewiogjoerwgfdsfdgsfg}
For $f$ in $\ell^{\infty}(\cB_{X})$
we have $\phi'(f)A\in {\bC}$ and $A\phi(f)\in {\bC}$.
\end{lem}
\begin{proof}
It suffices to show that  $\phi'(f)A\in {\bC}$. 
In order to deduce $A\phi(f)\in {\bC}$ we {then} use  the involution.
 
 We fix $\epsilon$ in $(0,\infty)$. Then we can find
 $A'$ in $\bCgtsmc(\cZ\subseteq  \cO(X)\otimes G_{can,max})$  and $n$ in $\nat$ such that
 $\|A-A'\|\le \frac{\epsilon}{2\|f\|}$ and
 $\mu(Z_{n})A' \mu(Z_{n})=A'$.
 We can furthermore find a bounded subset $B$ of $X$ such that $\| \chi_{X\setminus B}f\|\le  \frac{\epsilon}{2\|A\|}$.
 We set  $f':=
 \chi_{ B} f$.
 Then $\|\phi'(f)A-\phi'(f')A'\|\le \epsilon$.
 Since $\epsilon$ can be taken arbitrary small and $\bC$ is closed in $\bM\bC$
 it suffices to show that $\phi'(f')A'\in \bC$.
 But $\phi'(f')A'$ is supported on the bounded set
 $[0,n]\times B\times G$ of $ \cO(X)\otimes G_{can,max}$.
 Hence  $\phi'(f')A'\in \bC$ by local finiteness of $(C',\rho',\mu')$.
%
\end{proof}

\section{Construction of the Paschke morphism}
 

To $X$ in $G\UBC$ we can associate the commutative $G$-$C^{*}$-algebra 
$C_{0}(X)$ introduced in Definition \ref{kohperthrtgertg}.
Since a morphism $f:X\to X'$   in $G\UBC$ is uniform and proper  it  induces a homomorphism
$f^{*}:C_{0}(X')\to C_{0}(X)$ given by pre-composition. We therefore get a functor
$$C_{0}:G\UBC\to (G\nCalg_{\comm})^{\op}\ , \qquad X\mapsto C_{0}(X)\ .$$
Using Gelfand duality $(G\nCalg_{\comm})^{\op}\simeq \ppGTop$
we thus get a functor
\begin{equation}\label{q3oigfherqoifqrfvffqwefqeeqwfqeeqwfqefqefeeqfeqfewffefqewqewfqfqewf}
\iota^{\topp}:G\UBC\to \ppGTop
\end{equation}
uniquely characterized by the equality \eqref{werfwefwerffffrf}.

The main result of the present section is the description of the Paschke morphism for a given space $X $ in $G\UBC$.  
The general idea for its construction via a multiplication map like $\mu_{X}$ as below, {but} with completely different technical details otherwise, has been used at various places, see e.g.\ \cite[Sec.~6.5]{willett_yu_book} or \cite[Sec.~6.4]{wulff_axioms}.
In the Section \ref{wtgiojioergerfwerfewfwerf} we will provide a refinement of this construction to a natural transformation of functors defined on $G\UBC^{\scale}$.

 We start with a description of the following intermediate constructions which go into the  construction of the Paschke morphism:
 \begin{enumerate}
 \item The functor $X\mapsto \bQ(X)$ from $G\UBC $ to $\nCcat$,
 \item the tensor product $C_{0}(X)\otimes \bQ(X)$,
 \item the multiplication morphism $\mu_{X} \colon C_{0}(X)\otimes \bQ(X)\to  {\bQ^{(G)}_{\std}}$,
 \item the diagonal morphism $\delta_{X} \colon \KK(\C,\bQ(X))\to \KKG(C_{0}(X),C_{0}(X)\otimes \bQ(X))$.
 \end{enumerate}

Using  the cone functor $\cO$    introduced in Definition \ref{wepogjpweggreewg} we define the functor \begin{equation}\label{adsfasffafd}
G\UBC \to G\BC\, , \quad X \mapsto \cO(X)\otimes G_{can,max}\,  .
\end{equation}  
 {For {an effectively additive} $\bC$ in $\Fun(BG,\nCcat)$,  composing  \eqref{adsfasffafd}} with $\bCgtsmc$ from Definition \ref{wtogjoergrwegreggwege} we get a functor 
 \begin{equation}\label{rgfqfewfqewf1}
G\UBC \to \Ccat\, ,\quad   X \mapsto \bD (X) \coloneqq  \bCgtsmc(  \cO (X)\otimes G_{can,max})\, .
\end{equation} 
 We furthermore have the subfunctor  
 \begin{equation}\label{rgfqfewfqewf}
G\UBC \to \nCcat\, , \quad   X \mapsto  \bC (X) \coloneqq \bCgtsmc(\cZ\subseteq \cO (X)\otimes G_{can,max})
\end{equation}
(see \eqref{f23r8z89fz2893zf892334234f2}  {and \eqref{fqewqewfkqwpefowqefewfef}} for notation)
  such that $\bC (X)$ is a closed ideal in $\bD (X)$.
 Note that $\bC (X)$ is our replacement for
 $\bCgtsmc( X\otimes G_{can,max})$ which can be considered as a subcategory of
 $\bCgtsmc(  \cO (X)\otimes G_{can,max})$ of objects {which are} supported on $\{0\}\times X\times G$, but which is not an ideal {(these two $C^*$-categories actually have the same $K$-theory as will be used and also explained further below in Diagram \eqref{qfwefeeqfeqfqwfqewwfeqwf}).}
 Our choice of  notation  $ \bC (X)$ and $\bD (X)$ should indicate that
 these $C^{*}$-categories are our versions of the Roe algebra and the algebra of pseudolocal operators. We refer to Section \ref{weogiwjiergreggwrgwerg} for more details.  
 By forming quotients of $C^{*}$-categories we finally define the functor 
  \begin{equation}\label{qwefoujujfqew09ufewewfewfqfqef}
G\UBC \to \nCcat\,, \quad X \mapsto\bQ (X) \coloneqq  \frac{\bD (X) }{  \bC (X)}\, .
\end{equation}
{The functors $\bC $, $\bD $ and $\bQ $ depend functorially on $\bC$ in $\Fun(BG,\nCcat_{\ndeg,\eadd,\omega\add})$   since $ \bCgtsmc$ has this property. }
 
 Recall the  functor $K_{\bC}^{G,\cX}$ from Definition~\ref{qriofjqofewfefqeqff}, and 
  the $K$-theory functor $\Kcat$ for $C^{*}$-categories from \eqref{rewgwergwergegwergrewgregwregegwre}.}
 
  {We assume that $\bC$  in $\Fun(BG,\nCcat)$ is effectively additive and admits countable AV-sums.}
 \begin{lem}\label{wrtoihgjoergergegwgrgwerg}
 We have a canonical  equivalence of functors
 \begin{equation}\label{regvoijoifevfdvafvfdsv}
K_{\bC}^{G,\cX} \simeq \Kcat\circ \bQ \colon  G\UBC  \to \Sp^{\la}\, .
\end{equation} 
 \end{lem}
 \begin{proof}
  We have a  natural (naturality here and below  refers to   $X$ in $G\UBC $)   commutative diagram of  $C^{*}$-categories
\begin{equation}\label{qfqwfqewfeqedw}
\xymatrix{
& \bCgtsmc( X\otimes G_{can,max})
\ar[r]
\ar[d]& \bCgtsmc( \cO(X)\otimes G_{can,max})\ar@{=}[d]
&
&
\\0\ar[r]&\bC(X)\ar[r]&\bD(X)\ar[r]&\bQ(X)\ar[r]&0}
\end{equation}
where the {top} horizontal and left vertical morphisms are induced from canonical inclusions of bornological coarse spaces. 
We  apply $\Kcat$ to Diagram \eqref{qfqwfqewfeqedw}. Since 
  $\Kcat$   sends {exact} sequences in $\nCcat$ to fibre sequences in $\Sp^{\la}$ (\cite[Thm.~1.32.5]{KKG} or \cite[Prop.\ 14.7]{cank}) we get  a natural  morphism of fibre sequences
\begin{align}\label{qfwefeeqfeqfqwfqewwfeqwf}  
\\
\mathclap{\xymatrix{ 
 \Kcat( \bCgtsmc( X\otimes G_{can,max}))\ar[r]\ar[d]^{\simeq}&
\Kcat( \bCgtsmc( \cO (X)\otimes G_{can,max}))\ar@{=}[d]\ar[r]&
P
\ar[d]^{\simeq}\\ 
\Kcat (\bC (X))\ar[r]&\Kcat(\bD (X))\ar[r]&\Kcat(\bQ (X))
}}\notag 
\end{align}
in $\Sp^{\la}$, where $P$ is  defined as the cofibre of the left upper horizontal morphism.
In order to see that the left vertical morphism is an equivalence we argue as in the proof of  \cite[Thm.\ 7.2]{coarsek}.
 For every $n$ in $\nat$ the
  inclusion $X\otimes G_{can,max}\to Z_{n}$ (see  \eqref{fqewqewfkqwpefowqefewfef}) is a coarse equivalence and hence induces an equivalence
$$  \Kcat( \bCgtsmc( X\otimes G_{can,max})) \stackrel{def}{=} K\bC\cX^{G}( X\otimes G_{can,max})\stackrel{\simeq}{\to}   K\bC\cX^{G}(Z_{n})\ .$$ The inclusion
$$\bCgtsmc(Z_{n})\to  \bCgtsmc(Z_{n}\subseteq \cO(X)\otimes G_{can,max})$$ is a unitary equivalence by  \cite[Lem.\ 6.10(2)]{coarsek} and therefore induces an equivalence
$$   K\bC\cX^{G}(Z_{n})\stackrel{def}{=} \Kcat(\bCgtsmc(Z_{n}))\stackrel{\simeq}{\to}  \Kcat(\bCgtsmc(Z_{n}\subseteq \cO(X)\otimes G_{can,max}))\ .$$
We therefore get an equivalence
$$  \Kcat( \bCgtsmc( X\otimes G_{can,max}))  \stackrel{\simeq}{\to} \colim_{n\in \nat}\Kcat(\bCgtsmc(Z_{n}\subseteq \cO(X)\otimes G_{can,max}))\ .$$
Finally using  \eqref{f23r8z89fz2893zf892334234f2}, \eqref{rgfqfewfqewf}
and the fact that $\Kcat$ preserves filtered colimits  (see  \cite[Thm.\ 14.4]{cank})
we get the  equivalence
$$ \Kcat( \bCgtsmc( X\otimes G_{can,max})) \stackrel{\simeq}{\to}\Kcat (\bC (X))$$ appearing as the left vertical arrow in 
\eqref{regvoijoifevfdvafvfdsv}. 

It follows that the right vertical morphism is an equivalence, too.

Using the Definition \ref{qrogijeqoifefewfefewqffe} of $K\bC\cX^{G}$
 we get a natural morphism of fibre sequences
\begin{align}\label{qfwefeeqfeqfqwfqewwffeqwf}  
\\
\mathclap{ \xymatrix{ 
\Kcat( \bCgtsmc( X\otimes G_{can,max})) \ar[r]\ar@{=}[d]&\Kcat( \bCgtsmc( \cO(X)\otimes G_{can,max})) \ar@{=}[d] \ar[r] &\ar[d]^{\simeq}P
 \\     K\bC\cX_{G_{can,max}}^{G} ( X ) \ar[r]&  K\bC\cX_{G_{can,max}}^{G} ( \cO(X) )\ar[r]&K\bC\cX_{G_{can,max}}^{G} ( \cO^{\infty}(X) )
 }}\notag
 \end{align}
 where, by inserting definitions,  we have rewritten the lower sequence as an instance  of the cone sequence
   \eqref{wefqwefqewewfwdqewdewdqede} applied to $E \coloneqq K\bC\cX_{G_{can,max}}^{G}$.   
   
  Composing the inverse of the  right vertical equivalence  in \eqref{qfwefeeqfeqfqwfqewwffeqwf} with the right vertical equivalence in 
 \eqref{qfwefeeqfeqfqwfqewwfeqwf} and invoking   Definition~\ref{qriofjqofewfefqeqff} yields the natural equivalence\begin{equation}\label{eqwfjoiwefqwfewfqewf}
K_{\bC}^{G,\cX}(X)\simeq \Kcat(\bQ(X))\, .
\end{equation}
as desired.
\end{proof}

 In the present paper $\otimes$ denotes the maximal tensor product of $C^{*}$-categories \cite[Def.~7.2]{KKG}. By \cite[Prop.~1.21]{KKG} the stable $\infty$-category category $\KKG$ has a presentably symmetric monoidal 
 structure induced by the maximal tensor product of $C^{*}$-algebras, and by \cite[Thm.\ 1.35]{KKG} the functor
  $\kkGA$ 
   has a symmetric monoidal refinement.   We define  the functor
 \begin{equation}\label{saCWDQKOIOJ1} 
 - \hatotimes - \colon
\Fun(BG,\nCalg ) \times    \KK\xrightarrow{\kkG\times \Res^{\{1\}}_{G}}\KKG\times \KKG  \xrightarrow{\otimes} \KKG\, ,
\end{equation}
where $\otimes $ is structure map of  the symmetric monoidal structure of $\KKG$  
and $\Res^{\{1\}}_{G}$ is the restriction induced by the projection $G\to \{1\}$ from  \cite[Thm.~1.22]{KKG}  {(on $C^*$-algebras $\Res^{\{1\}}_{G}$ is given by equipping a $C^*$-algebra with the trivial $G$-action).}  
Using that  $\KKG$ is   presentably symmetric monoidal  category and $\Res^{\{1\}}_{G}$ preserves
small colimits we see that   $\hatotimes$ preserves
small colimits in its second variable.

Let   $    A$ be in $\Fun(BG,\Calg ) $ and $\bQ$ be in $\nCcat$.
\begin{lem}\label{glkbijowerfvvfevsdfvsfdv} 
 We have an  equivalence
 $$     A \hatotimes \kkA(\bQ)\simeq  \kkGA (    A\otimes \Res^{\{1\}}_{G}(\bQ))$$
 which is natural in $A$ and $\bQ$.
\end{lem}
\begin{proof}
The chain of natural equivalences
\begin{eqnarray*}
     A \hatotimes \kkA( \bQ)&\stackrel{\text{def.}}{\simeq}&
\kkG(     A)    \otimes   \Res^{\{1\}}_{G}( \kkA  (\bQ))\\
&\stackrel{\scriptsize{\cite[\text{Thm.\,1.22}]{KKG}}}{\simeq}&
\kkGA   ( A)    \otimes  \kkGA (\Res^{\{1\}}_{G} (\bQ))\\
&\stackrel{\scriptsize{\cite[\text{Thm.\,1.35}]{KKG}}}{\simeq}&
 \kkGA(  A\otimes \Res^{\{1\}}_{G} (\bQ))
\end{eqnarray*}
gives the desired equivalence, where in the last two lines we implicitly consider $A$ as a $G$-$C^{*}$-category with a single object.
\end{proof}

From now on, in order to simplify the notation, we will write  $\bQ$ instead of $\Res^{\{1\}}_{G} (\bQ)$.

 For $ X $ in $G\UBC $ we have  {the objects} $C_{0}(X)$ in $\Fun(BG, {\nCalg})$ and 
    $\bQ(X)$ in $ \nCcat$ and can thus define     $C_{0}(X)\otimes \bQ (X)$ in $\Fun(BG,\nCcat)$, where  consider the left tensor factor as a $C^{*}$-category. The objects of this category are
    the objects of $ \bQ(X)$, and the morphism spaces are certain completions of the algebraic tensor products of the 
 morphism spaces of $ \bQ(X)$ with $C_{0}(X)$. 
 {For concreteness, we will work with the maximal tensor product \cite[Def. 7.2]{KKG}.} 

    Recall the Definition \ref{qeroighjoreqgreqfefwfqwef}.\ref{wrtogwprtgwergwergfreferf}   of $   \bQ^{(G)}_{\std}$ in $\Fun(BG,\nCcat)$.
    We  define the multiplication  morphism     
  \begin{equation}\label{oreihoijfvoisfdvervfdsvdfvsvv}
\mu_{X } \colon C_{0}(X)\otimes \bQ (X)\to    \bQ^{(G)}_{\std}
\end{equation} 
  in $\Fun(BG,\nCcat)$ as follows.
    \begin{enumerate} \item
    objects: The morphism $\mu_{X  }$ sends the object $(C,\rho,\mu)$ to the object $(C,\rho)$ of
    $  \bQ^{(G)}_{\std}$. 
    Note that 
    $(C,\rho)$ belongs to  $  \bQ^{(G)}_{\std}$ since the underlying $G$-set of $\cO (X)\otimes G_{can,max}$ is a free $G$-set {(see \eqref{rgfqfewfqewf1}, \eqref{rgfqfewfqewf} and  \eqref{qwefoujujfqew09ufewewfewfqfqef})}.
    \item \label{vefvvdfvweewfvfdsv} morphisms: The morphism $\mu_{X}$ is defined on morphisms uniquely by the universal property of the maximal tensor product 
    of $C^{*}$-categories 
    such that it sends the morphism $f\otimes  [A]$    
    in $C_{0}(X)\otimes \bQ(X)$ with 
    $A \colon (C',\rho',\mu')\to (C,\rho,\mu)$ to the morphism $[\phi(f)A]$ in $ { \bQ^{(G)}_{\std}}$. Here the brackets {$[-]$} indicate  classes in the respective   quotients \eqref{qwefoujujfqew09ufewewfewfqfqef} and \eqref{wefwfwfwefflkjlqwepofp}, and $\phi(f)$ is defined in \eqref{ewqfpojlqkrmeflkwefqef}.
    
    To see that this map is well-defined note that if $A$ is in $\bC(X)$, then
    $\phi(f)A\in  {\bC}^{(G)}_{\std}$ by Lemma \ref{ewiogjoerwgfdsfdgsfg}.
    Further, by  Lemma \ref{wtioghjwergfgrefwref} we have 
  $[\phi(f)A]=[A\phi'(f)]$ which  
    implies that
  this prescription is compatible with the composition and the involution.
        \end{enumerate}

  Finally, we define the diagonal morphism $\delta_{X}$ as the composition
 \begin{eqnarray}\label{rgergergegwegerg}
\delta_{X} \colon \KK(\C,\bQ (X))&\stackrel{}{\simeq}&\KK(\kkA(\C),\kkA(\bQ(X))\\
&\stackrel{C_{0}(X) \hatotimes -}{\to}& \KKG(C_{0}(X) \hatotimes\kkA(\C),C_{0}(X) \hatotimes \kkA(\bQ (X)))\nonumber\\&\stackrel{!}{\simeq} & \KKG(\kkGA(C_{0}(X)\otimes \C) ,\kkGA( C_{0}(X) \otimes \bQ (X)))\nonumber\\&\simeq& 
 \KKG(C_{0}(X) ,C_{0}(X) \otimes \bQ (X)) \, .\nonumber
\end{eqnarray}
The last  equivalence  is given by  the identification $C_{0}(X)\otimes \C\cong C_{0}(X)$, and the equivalence marked by $!$ uses Lemma  \ref{glkbijowerfvvfevsdfvsfdv}

 We now define the Paschke morphism whose existence was claimed in Theorem~\ref{qreoigjoergegqrgqerqfewf}.\ref{weiufhiqwefewwfqewfewf}.  {We assume that $\bC$  in $\Fun(BG,\nCcat)$   is effectively additive and admits countable AV-sums.}
 \begin{ddd}\label{qeorigjoqfeqwfqfewf} 
The Paschke morphism {for $X$} in $G\UBC$ is defined as the composition
\begin{align}\label{nlkkmlmvfdvsdva}
p_{X } \colon K_{\bC}^{G,\cX}( X ) \qquad & \mathclap{\stackrel{\eqref{regvoijoifevfdvafvfdsv}, \eqref{rewgwergwergegwergrewgregwregegwre}}{\simeq}}  \qquad \KK(\C,\bQ (X))\\
& \mathclap{\stackrel{\delta_{X }}{\to}} \qquad \KKG(C_{0}(X),C_{0}(X)\otimes \bQ (X))\notag\\
& \mathclap{\stackrel{\mu_{X }}{\to}} \qquad \KKG(C_{0}(X),   \bQ^{(G)}_{\std})\notag\\
& \mathclap{\stackrel{\eqref{vfdsvsvvsvsdvdsadsvdsva}}{\simeq}} \qquad K_{\bC}^{G,\An}(\iota^{\topp}(X))\,.\notag
\end{align}
\end{ddd}

Note that from this definition is not clear that the Paschke morphism is natural in $X$. The naturality will be discussed in the next Section  \ref{wtgiojioergerfwerfewfwerf}.

\section{Naturality of the Paschke morphism}\label{wtgiojioergerfwerfewfwerf}

In this subsection   {we} discuss the naturality of the Paschke morphism {from} Definition \ref{qeorigjoqfeqwfqfewf}.
More precisely, we will construct a natural transformation whose component on $X$ in $G\UBC$ is the Paschke morphism of Definition \ref{qeorigjoqfeqwfqfewf}.
 Note that naturality in the $\infty$-categorical sense is more {than} the existence of a filler for the square 
\begin{equation}\label{fqwefoiwjfwoqefewfwqfe}
\xymatrix{K_{\bC}^{G,\cX}( X )\ar[r]^{f_{*}}\ar[d]^{p_{X}}&K_{\bC}^{G,\cX}( X')\ar[d]^{p_{X', }}\\ K_{\bC}^{G,\An}(\iota^{\topp}(X ))\ar[r]^{f_{*}}& K_{\bC}^{G,\An}(\iota^{\topp}(X'))}
\end{equation}  
for all  morphisms $f \colon X \to X', $ in $G\UBC $.
 The existence of such a filler  can indeed be easily seen by considering the big diagram \eqref{fqweflkfnlfewfqewfwfqfw} below.
 In  order to produce the data of a natural transformation we  must reformulate the construction 
 of the Paschke morphisms appropriately. 
 The main problem is that $\KKG(C_{0}(X),C_{0}(X)\otimes \bQ(X))$ is not a functor on $X$ so that  $\delta_{X}$ and $\mu_{X}$ can not be interpreted as natural transformations separately.

We {assume that $\bC$  in $\Fun(BG,\nCcat)$    is effectively additive and admits countable AV-sums.}
In order to get an idea  what we have to do  {to} get the existence of a filler of \eqref{fqwefoiwjfwoqefewfwqfe} we first consider the diagram
\begin{equation}\label{fqweflkfnlfewfqewfwfqfw}
\xymatrix{\KK(\C,\bQ(X))\ar[ddd]_{\KKG(-,\bQ(f))}\ar[r]^-{\delta_{X }}\ar@/_1.5pc/[ddr]^-(0.3){\delta_{X'}}&\KKG(C_{0}(X),C_{0}(X)\otimes \bQ(X))\ar[d]^{\KKG(f^{*},-)}\ar[r]^-{\mu_{X}}&\KKG(C_{0}(X),  \bQ_{\std}^{(G)})\ar[ddd]^{\KKG(f^{*},-)}\\
&\KKG(C_{0}(X'),C_{0}(X)\otimes \bQ (X))\ar@/^1.5pc/[ddr]^-(0.7){\mu_{ X }}&\\
&\KKG(C_{0}(X'),C_{0}(X')\otimes \bQ (X))\ar[u]_{\KKG(-,f^{*})}\ar[d]^{\KKG(-,\bQ(f)}&\\
\KK(\C,\bQ (X'))\ar[r]^-{\delta_{ X }}&\KKG(C_{0}(X'),C_{0}(X')\otimes \bQ (X'))\ar[r]^-{\mu_{ X }}&\KKG(C_{0}(X'),  \bQ_{\std}^{(G)})}
\end{equation}
all of whose  cells have essentially obvious fillers.  This  already  implies that the Paschke morphism is natural on the level of homotopy categories. 

\begin{rem}
Our idea for showing that the Paschke morphism is an equivalence is to reduce this by homotopy invariance to $G$-simplicial complexes, and then by excision to $G$-orbits where it can be verified by an explicit calculation. 
The excision step  requires a natural transformation on the spectrum level. 
If one is only interested in homotopy groups, then it would be sufficient to know the compatibility
of the Paschke map with the Mayer--Vietoris boundary maps which is an immediate consequence of the
spectrum-valued naturality.
So even if we were   finally  {only} interested in the Paschke isomorphism on the level of  homotopy groups
we would  {still} need the spectrum level natural transformation for the proof that it is an isomorphism.  

For similar reasons, the spectrum-valued version is also crucial in the proof of our second Theorem  \ref{wtoiguwegwergergregwe}
comparing the two assembly maps,  {though} the latter is indeed a statement on the level of homotopy groups.
\hB
\end{rem}

In the following remarks  about general $\infty$-categorical constructions we prepare the actual construction of the natural Paschke  transformation.

  \begin{rem} \label{wegijowergfwerfgwrf}For a category $\cC$ let $\Tw(\cC)$ denote the twisted arrow category.
 Objects are morphisms $C\to C'$ in $\cC$, and morphisms $(C_{0}\to C_{0}')\to (C_{1}\to C_{1}')$ are commut{ative} diagrams 
 \begin{equation}\label{qewfoijoiqfewfqewfqewfqef}
\xymatrix{C_{0}\ar[r]&C_{0}'\ar[d]\\C_{1}\ar[r]\ar[u]&C_{1}'}
\end{equation}
 We have a canonical functor
 $$(\ev,\ev') \colon \Tw(\cC)\to \cC^{\op}\times \cC\, , \quad (C\to C')\mapsto (C,C')\, .$$
 If $F,G \colon \cC\to \cD$ are two functors to a stable $\infty$-category, then we can express the spectrum of natural transformations 
 between $F$ and $G$ {as}
 \begin{equation}\label{wetliogjwoiergregregef}
\mathrm{nat}(F,G)\simeq \lim_{\Tw(\cC)} \map_{\cD}(F\circ \ev,G\circ \ev')\, .
\end{equation}
 We refer to \cite{Gepner:2015aa,Glasman:2014aa} where this is discussed even  in the more general  case  of $\cC$ being an $\infty$-category.
 \hB
  \end{rem}

  \begin{rem}\label{qiuhqriegergrggregwergwegwerg}
  Recall that our universe in which we do homotopy theory is the one of small sets. 
  The corresponding categories then belong to the  large universe. 
  A   locally small,  large presentable {stable} $\infty$-category  $\cC$  
 is enriched and tensored over $\Sp$. We thus
  a functor
  \begin{equation}\label{sdvjqr0evj0fvsfvvfv}
\cC\times \Sp^{\la}\to \cC\, ,\quad (C,E)\mapsto C\wedge E
\end{equation} 
preserving small colimits in both variables and such that \begin{equation}\label{qwfqewfqewfqewfqewffefefqewefefdfqef}-\wedge S\simeq \id_{\cC}\ .
\end{equation} Furthermore,  for every object 
  $C_{0}$ in $\cC$  we  have an adjunction
  \begin{equation}\label{ewfoijqoiwejfoqwefqewqefef}
  C_{0}\wedge -:\Sp^{\la}\leftrightarrows \cC: \map_{\cC}(C_{0},-)\, .
\end{equation} 
  
The counit of the adjunction in \eqref{ewfoijqoiwejfoqwefqewqefef} is a natural transformation 
  \begin{equation}\label{qroujioqrfqefefeqwf}
    C_{0}\wedge \map_{\cC}(C_{0},-)\to \id_{\cC} (-) 
\end{equation} of endofunctors of $\cC$.
\hB
    \end{rem}

  \begin{rem}\label{weoigjwoigregwegwergrewgw}
  Let $\cC,\cD,\cE$ be $\infty$-categories and $- \hatotimes - \colon \cC\times \cD\to \cE$ be a functor. We consider  $\infty$-categories $\cI$, $\cJ$  and  
   natural transformations of functors $ {(F\stackrel{\alpha}{\to}  F')}\colon \cI\to \cC$ and ${(G\stackrel{\beta}{\to} G')}\colon\cJ\to\cD$. 
  Then we get a natural transformation of functors
  \[
  {(F\times G \stackrel{\alpha\times \beta}{\to} F'\times G')} \colon \cI\times \cJ\to \cC\times \cD\,,
  \]
  and by composition with $- \hatotimes -$ a natural transformation  \begin{equation}\label{fwrefwpofkwerpoffewrfwerf}
 {(F \hatotimes G \stackrel{\alpha \hatotimes \beta}{\to} F' \hatotimes G')} \colon \cI\times \cJ\to \cE\, ,
\end{equation}
 where we write
  $F \hatotimes G$ {for} $(- \hatotimes -)\circ (F\times G) $.
  \hB
    \end{rem}

 Applying    \eqref{sdvjqr0evj0fvsfvvfv}  to $\cC=\KK$ we get a functor 
 $$(B,E)\mapsto B\wedge E \colon \KK\times \Sp^{\la}\to \KK\, .$$
 In the following we specialize $B$ to $\kk(\C)$.
 We then have a functor $ (A,E)\mapsto A\wedge E$  given as the composition  
 \begin{equation}\label{qewfpojfopqwfqewfeqf}
  \Fun(BG,\nCalg)\times \Sp^{\la}\xrightarrow{\id\times {( \kk(\C)\wedge -)}} \Fun(BG,\nCalg)\times \KK \xrightarrow{- \hatotimes -} \KKG\, ,
\end{equation}
where $\hat \otimes$ is as in \eqref{saCWDQKOIOJ1}.
Note that $$A\wedge S\stackrel{\eqref{qewfpojfopqwfqewfeqf}}{\simeq} A\hatotimes {(\kk(\C)\wedge S)}\stackrel{\eqref{qwfqewfqewfqewfqewffefefqewefefdfqef}}{\simeq}    A \hatotimes \kkA(\C) \stackrel{Lem. \ref{glkbijowerfvvfevsdfvsfdv}}{\simeq} \kkG(A\otimes {\Res^{G}_{\{1\}}(\C)})\simeq \kkG(A)\, .$$
{Since} the functor $-\hat \otimes-$ in  \eqref{saCWDQKOIOJ1} preserves small colimits in its second variable, the functor in \eqref{qewfpojfopqwfqewfeqf} is essentially uniquely determined by
the equivalence
$A\wedge S\simeq \kkG(A)$ and the fact that it preserves small colimits in the second variable.
Furthermore, by the adjunction \eqref{ewfoijqoiwejfoqwefqewqefef} we have a natural equivalence 
\begin{equation}\label{wergoihweriogergrwfere}
\map_{\Sp^{\la}}(E,\KKG(A,B))\simeq \KKG(A\wedge E,B)
\end{equation}
for $E$ in $\Sp^{\la}$, $A$ in $\Fun(BG,\nCalg)$, and $B$ in $\KK^{G}$.

   We consider the functor 
\begin{equation}\label{ewfiuhqiowefqwefewfqewfw}
F \colon G\UBC^{\op}\times \KK\xrightarrow{C_{0}(-)\times   \KK(\C,-)} \Fun(BG,\nCalg) \times \Sp^{\la}\xrightarrow{ -\wedge -,\eqref{qewfpojfopqwfqewfeqf}} \KKG
\end{equation}
written as 
$$(X ,B)\mapsto C_{0}(X)\wedge     \KK(\C,B)\, .$$
We further consider the functor 
$$H \colon G\UBC ^{\op}\times \KK\xrightarrow{C_{0}(-)\times \id(-)}  \Fun(BG,\nCalg) \times \KK \xrightarrow{- \hatotimes -} \KKG$$ written as
$$(X ,B)\mapsto C_{0}(X) \hatotimes B\, .$$

We now  construct the diagonal
transformation 
\begin{equation}\label{qewfoijfiofwfwfqwfef}
{(F \stackrel{\tilde \delta}{\to} H)} \colon   G\UBC^{\op}\times \KK\to \KKG\, .
\end{equation}
 Its specialization at $X $ in $ G\UBC $ and $B$ in $\KK$ is a morphism 
 \begin{equation} 
\tilde \delta_{X,B} \colon C_{0}(X) \wedge \KK(\C,B)\to C_{0}(X) \hatotimes B
\end{equation}
in $\KKG$.  Inserting   \eqref{qewfpojfopqwfqewfeqf} into the definition \eqref{ewfiuhqiowefqwefewfqewfw} of $F$
we get 
 $$F=(- \hatotimes -)\circ (C_{0}(-)\times \kk(\C)\wedge   \KK(\C,-))\, .$$
We now obtain $\tilde \delta$ in \eqref{qewfoijfiofwfwfqwfef} by specializing \eqref{fwrefwpofkwerpoffewrfwerf} to the transformations
$$ {(C_{0}(-) \stackrel{\id}{\to} C_{0}(-))} \colon  G\UBC^{\op}\to  \Fun(BG,\nCalg )$$ 
and 
$$ (\kk(\C)\wedge   \KK(\C,-)\to  \id(-)) \colon \KK\to \KK$$  given by  \eqref{qroujioqrfqefefeqwf}.

We define the functor 
\begin{equation}\label{wregpkpwergwgreegergw}
Q \colon G\UBC  \to \KK\, , \quad   X \mapsto Q (X) \coloneqq \kkA (\bQ (X))\,,
\end{equation}
 {see \eqref{qwefoujujfqew09ufewewfewfqfqef} for $\bQ (X)$.} 
 Then we consider the functor
 \begin{equation}\label{wqefqwefwefqflkmklqf}
\Tw(G\UBC)^{\op}\to G\UBC^{\op}\times \KK\, , \quad  (X \to  X' )\mapsto (X',   Q (X))\, .
\end{equation}
 The pull-back of $\tilde \delta$ in \eqref{qewfoijfiofwfwfqwfef} along \eqref{wqefqwefwefqflkmklqf} yields 
 a natural transformation 
 \begin{equation}\label{roiegoijviojfeqwfqewfewfq}
 (\delta \colon C_{0}(-')\wedge \KK(\C,Q(-))\to  C_{0}(-') \hatotimes Q(-)):\Tw(G\UBC)^{\op}\to \KKG
\end{equation} 
 whose evaluation at {an object} $f \colon  X \to X'  $ in $\Tw(G\UBC )^{\op}$ is a morphism
 \begin{equation}\label{ewqoifjeiqjfioqewfqewfwef}
  \delta_{f} \colon C_{0}(X')\wedge \KK(\C,\bQ (X))\to C_{0}(X') \hatotimes Q (X)
\end{equation}
in $\KKG$.
 This is our version of the diagonal \eqref{rgergergegwegerg} as a natural transformation. In fact,  under the  canonical  equivalence
  \begin{align}\label{eq_identification_deltas_X_nat}
\KKG & (C_{0}(X)\wedge \KK(\C,\bQ (X)), C_{0}( {X}) \hatotimes Q (X))\\
& \stackrel{\eqref{wergoihweriogergrwfere}}{\simeq} \map( \KK(\C,\bQ (X)),\KKG(C_{0}(X),C_{0}(X) \otimes \bQ (X))\notag
\end{align}
the map $\delta_{\id_{X}}$ in \eqref{ewqoifjeiqjfioqewfqewfwef} corresponds to $\delta_{X}$ from \eqref{rgergergegwegerg}.

We now construct {the} refinement \eqref{ewfiuhiu34gferqgfqfewfeqwfqe} of the family of multiplication maps $\mu_{X}$ from \eqref{oreihoijfvoisfdvervfdsvdfvsvv}  for all $X $ in 
$G\UBC $. We start with 
the functor
\[\begin{tikzcd}[column sep=large, row sep=tiny]
	 \Tw(G\UBC)^{\op} \ar[r,"{C_{0}(-')\otimes \bQ (-)}"] & \Fun(BG,\nCcat) \\ 
	 (X \to  X' ) \ar[r, mapsto] & C_{0}(X'){\otimes}  \bQ (X) \,.
\end{tikzcd}\]

We also consider ${\bQ}^{(G)}_{\std}
$ as a constant functor from $\Tw(G\UBC^{\scale})^{\op}$ to $  \Fun(BG,\nCcat)$.  
{We first} construct 
  a natural transformation 
  \begin{equation}\label{frvjvoijoqvvvqwvwevqw}
(\tilde \mu \colon C_{0}(-')\otimes \bQ (-)\to   \bQ^{(G)}_{\std}) \colon \Tw(G\UBC )^{\op}\to \Fun(BG,\nCcat)\, .
\end{equation} 
 For  every object $f \colon  X \to   X' $ in $\Tw(G\UBC )^{\op}$ we must define a functor 
 \begin{equation}\label{eq_identification_mu_X_nat}
 \tilde \mu_{f} \colon C_{0}(X')\otimes \bQ (X)\to   \bQ^{(G)}_{\std}\, .
 \end{equation}
 This construction  extends the construction of ${\mu_{X }}$ in \eqref{oreihoijfvoisfdvervfdsvdfvsvv} which {will be} 
 recovered as $\mu_{X }=\tilde \mu_{\id_{X }}$.
 \begin{enumerate}
 \item objects: The functor $\tilde \mu_{f}$ sends the object $(C,\rho,\mu)$ in $C_{0}({X'})\otimes \bQ (X)$ (hence an object of  $\bQ (X)$) to the object $(C,\rho)$ in $   \bQ^{(G)}_{\std}$.
 \item morphisms: If $[A] \colon (C',\rho',\mu')\to (C,\rho,\mu)$ is a morphism in $\bQ (X)$ and $h$ is in $C_{0}({X'})$, then 
 $\tilde \mu_{f}(h\otimes [A]) \coloneqq [\phi(f^{*}h)A]${, see \eqref{ewqfpojlqkrmeflkwefqef} for the definition of $\phi$.}
 \end{enumerate}
The argument that the functor $\tilde \mu_{f}$ is well-defined is  the same  as for  ${\mu_{ X }}$. 
We now check that $\tilde \mu \coloneqq ({\tilde{\mu}}_{f})_{f\in \Tw(G\UBC)^{\op}}$ is a natural transformation.
We consider a morphism $f\to g$ in $ \Tw(G\UBC )^{\op}$, see \eqref{qewfoijoiqfewfqewfqewfqef}. Since we work with the opposite of the twisted arrow category, it is given by a commutative diagram \begin{equation}\label{ewqfoijhqoiwefqwefewqfqewfef}
\xymatrix{X\ar[r]^{f}\ar[d]^{\alpha}&X' \\Y\ar[r]^{g}&Y'\ar[u]^{\beta}}
\end{equation}
 We must show that 
 $$\xymatrix{ C_{0}(X')\otimes \bQ (X)\ar[rr]^{\beta^{*}\otimes \bQ(\alpha)}\ar[dr]_{{\tilde{\mu}}_{f}}&&\ar[dl]^{{\tilde{\mu}}_{g}}C_{0}(Y')\otimes \bQ (Y) \\& \bQ^{(G)}_{\std} &}$$
 commutes.
 \begin{enumerate}
 \item objects:  Let   $(C,\rho,\mu)$ be an object  in $C_{0}(X')\otimes \bQ (X)$. Then we have the equality
   $${\tilde{\mu}}_{g}((\beta^{*}\otimes \bQ(\alpha))(C,\rho,\mu))= {\tilde{\mu}}_{g}(C,\rho, {\alpha}_{*}\mu)=(C,\rho)= {\tilde{\mu}}_{f}(C,\rho,\mu)\, .$$
   \item morphisms:  Let $[A] \colon (C',\rho',\mu')\to (C,\rho,\mu)$ be  a morphism in $\bQ (X)$ and  $h$ {be} in $C_{0}(X')$.
   Then we have the equality
   $$ {\tilde{\mu}}_{g}((\beta^{*}\otimes \bQ(\alpha))( h\otimes [A]))= {\tilde{\mu}}_{g}(  \beta^{*}h\otimes [\alpha_{*}A]) =  {[\phi(g^*(\beta^*(h))) \alpha_*A]} = [(\alpha_{*}\phi)(g^{*}\beta^{*}h)A]\, .$$
   On the other hand,
   $$ {\tilde{\mu}}_{f}( h\otimes [A])=[\phi( {f^*}h)A]\, .$$
   The desired equality
   $$[\phi( {f^*} h)A]=[(\alpha_{*}\phi)( {g^*}\beta^{*}h)A]$$ now follows from the identity
  $$(\alpha_{*}\phi)(g^{*}\beta^{*}h)=\phi(\alpha^{*}g^{*}\beta^{*}h)=\phi(f^{*}h)$$ since
   $\alpha^{*}g^{*}\beta^{*}h=f^{*}h$ by the commutativity of \eqref{ewqfoijhqoiwefqwefewqfqewfef}.
 \end{enumerate}

 We post-compose the transformation in \eqref{frvjvoijoqvvvqwvwevqw} with the functor $\kkGA$ 
  and get   a natural transformation
 \begin{equation}\label{r9u89fqfewewqfwefqwefqwefqwef}
(\kkG(\tilde \mu) \colon \kkGA(C_{0}(-')\otimes \bQ (-))\to  Q^{(G)}_{\std}) \colon \Tw(G\UBC )^{\op}\to\KKG\, ,
\end{equation}
    where we use the  abbreviation 
 \[
  Q^{(G)}_{\std} \coloneqq \kkGA( \bQ^{(G)}_{\std})\,.
 \]
 Composing the transformation   \eqref{r9u89fqfewewqfwefqwefqwefqwef}
 with the   equivalence $$C_{0}(-') \hatotimes     Q (-)\simeq \kkGA(C_{0}(-')\otimes \bQ(-))$$    given by Lemma \ref{glkbijowerfvvfevsdfvsfdv}  
 (see \eqref{wregpkpwergwgreegergw} for {the} notation $Q(-)$) we get a natural transformation   \begin{equation}\label{ewfiuhiu34gferqgfqfewfeqwfqe}
(  \mu \colon C_{0}(-') \hatotimes Q(-)\to \tilde Q^{(G)}_{\std}) \colon \Tw(G\UBC )^{\op}\to \KKG\, .
\end{equation}

The composition of \eqref{roiegoijviojfeqwfqewfewfq} and  \eqref{ewfiuhiu34gferqgfqfewfeqwfqe}
then gives a natural transformation
$$  (\mu\circ   \delta \colon C_{0}(-')\wedge \KK(\C,Q(-))\to {C_0(-')\otimes Q(-)} \to      Q^{(G)}_{\std}) \colon \Tw(G\UBC )^{\op} \to \KKG$$
 whose value at the object  $f \colon X \to  X' $ is the morphism 
$$\mu_{f}\circ \delta_{f} \colon C_{0}(X')\wedge \KK(\C,Q  (X)) \to  {C_0(X') \otimes Q (X)} \to     Q^{(G)}_{\std}\, .$$
Equivalently, by \eqref{wetliogjwoiergregregef} and since the target functor is constant we can interpret 
this as a map   {of  spectra}
\begin{equation}\label{gpojeopkgwpegregrgergw}
S\to \KK^{G}(\colim_{\Tw(G\UBC )^{\op}} C_{0}(-')\wedge \KK(\C,Q(-)),      Q^{(G)}_{\std})\, .
\end{equation}
Note that $\Tw(G\UBC )^{\op}$ is small    and  the  presentable category $\KKG$ admits all small colimits.
We now use the chain of canonical equivalences  
\begin{eqnarray*}\lefteqn{
 \KK^{G}(\colim_{\Tw(G\UBC )^{\op}} C_{0}(-')\wedge \KK(\C,Q(-)),      Q^{(G)}_{\std})}&&\\&\simeq& \lim_{\Tw(G\UBC )}
 \KK^{G}(  C_{0}(-')\wedge \KK(\C,Q(-)),      Q^{(G)}_{\std})\\&\stackrel{\eqref{wergoihweriogergrwfere}}{\simeq}& \lim_{\Tw(G\UBC )}
 \map( \KK(\C,Q(-)) , \KK^{G}(  C_{0}(-') ,     Q^{(G)}_{\std}))
  \\
 &\stackrel{\eqref{wetliogjwoiergregregef}}{\simeq}&\mathrm{nat}(   \KK(\C,Q(-)),  \KK^{G}(  C_{0}(-) ,     Q^{(G)}_{\std}  ))\, ,
\end{eqnarray*}
where $\mathrm{nat}$ denotes the spectrum of natural transformations between functors from 
$G\UBC $ to $\Sp^{\la}$.
Therefore \eqref{gpojeopkgwpegregrgergw} provides a map 
$$S\to \mathrm{nat}(   \KK(\C,Q(-)),  \KK^{G}(  C_{0}(-) ,     Q^{(G)}_{\std}  ))
\, .$$
 This is the desired natural transformation 
  \begin{equation}\label{ekjbjkbejkefeqwfqwefe}
p \colon \KK(\C,Q(-))\to  \KKG(C_{0}(-), Q^{(G)}_{\std})
\end{equation} 
of functors from 
$G\UBC $ to $\Sp^{\la}$.
It follows from the identifications of $\delta_{\id_{X }}$ {with $\delta_{  X }$ by \eqref{eq_identification_deltas_X_nat}} and  {of} ${\tilde{\mu}}_{\id_{ X }}$ {with $\mu_{ X }$ stated after \eqref{eq_identification_mu_X_nat}} that the evaluation of $p$ at $X$ in $ G\UBC $ is equivalent to the morphism
$p_{ X }$ from \eqref{nlkkmlmvfdvsdva}.

    Recall that we use the notation 
    $$ \KK(\C,Q (X))\simeq  \KK(\C,\bQ (X))\simeq K_{\bC}^{G,\cX}( X ) \, , $$ and  $$  \KKG(C_{0}(X),  Q^{(G)}_{\std})\simeq   \KKG(C_{0}(X),  \bQ^{(G)}_{\std})\simeq K_{\bC}^{G,\An}(\iota^{\topp}(X ))\, .$$
  Therefore \eqref{ekjbjkbejkefeqwfqwefe}  is the desired Paschke transformation
$$p \colon  K_{\bC}^{G,\cX} \to K_{\bC}^{G,\An}\circ \iota^{\topp}\, .$$
By construction, we see that the   Paschke transformation is natural in the coefficient category $\bC$  in $\Fun(BG,\nCcat_{\ndeg,\eadd,\omega\add})$.
This finishes the proof of  Theorem \ref{qreoigjoergegqrgqerqfewf}.\ref{weiufhiqwefewwfqewfewf}.

 \phantomsection \label{eroigheifogegegergewgerg}
\section{ Reduction to  \texorpdfstring{$\boldsymbol{G}$}{G}-orbits}\label{ewoijegoewrgrewfe}

In {this} section we reduce the verification of  the Assertions   \ref{qreoigjoergegqrgqerqfewf}.\ref{qregiojqwfewfqwfqewf} and   \ref{qreoigjoergegqrgqerqfewf}.\ref{wrtigjiogowrefwerfewrf} to the case of $G$-orbits.
A discrete $G$-uniform bornological coarse space is a $G$-set   with
  the minimal coarse and bornological {structures} and the discrete uniform structure.
An object  $Y$ of $G\Set$ can canonically be considered as a discrete object  in $G\UBC$ which we will also denote by~$Y$. Alternatively we may use the more informative, but lengthier notation $Y_{min,min,disc}$, where the first $min$ indicates the minimal coarse structure, the second $ min$ the minimal bornology, and finally $disc$ the discrete uniform structure.
 Note that the construction $Y\mapsto Y_{min,min,disc}$ is functorial only for maps between $G$-sets with finite fibres.

Let $\cF$ denote a family of subgroups of $G$. We will be mainly interested in the family $\Fin$ of finite subgroups{, but the {following} proposition  {is valid} for any family $\cF$.} We let $G_{\cF}\Set$ be the category of very small $G$-sets with stabilizers in $\cF$. 

Let $X $ be in $G\UBC $.    {We assume that $\bC$ in $\Fun(BG,\nCcat)$ is effectively additive and admits countable AV-sums and} recall the  Definition \ref{qeorigjoqfeqwfqfewf} of the Paschke morphism.
\begin{prop}\label{eoirgjwergerregwgreg} 
Assume: 
\begin{enumerate}
\item \label{ergoiwejrgrgrgwrgwergw}
The Paschke morphism {for $S$} is an equivalence for every
 $S$ in $ G_{\cF}\Orb$.
\item \label{weliorgjwgwreegwre9} $X$ is homotopy equivalent to a  $G$-finite  $G$-simplicial complex with stabilizers in $\cF$  {and} with  structures induced by its spherical path metrics.
\end{enumerate}
Then the Paschke morphism for $X $ is an equivalence.
\end{prop}
\begin{proof}
We argue by induction on the dimension  $n$ of the $G$-simplicial complex in  Assumption  {\ref{eoirgjwergerregwgreg}.}\ref{weliorgjwgwreegwre9}.  
In order to simplify the notation we drop the functor $\iota^{\topp}$ from the notation if we apply 
  $K_{\bC}^{G,\An}$ to an object of $G\UBC$.
  
 Assume that  $n=0$ and that  $K  $ is in $G\UBC $
 such that    $K$ is a $0$-dimensional  $G$-finite $G$-simplicial complex with stabilizers in $\cF$.   For every orbit  $S$ in $G\backslash K$  we consider the closed  invariant partition  $(S,K\setminus S)$ of $K$.
 Applying excision for the functors $K_{\bC}^{G,\cX}$ and $K_{\bC}^{G,\An}$  
we get the respective projections 
$q^{\cX}_{S } \colon K_{\bC}^{G,\cX}(K)\to K_{\bC}^{G,\cX}(S)$
and
$q^{{\An}}_{ S} \colon K_{\bC}^{G,\An}(\iota^{\topp}(K))\to K_{\bC}^{G,\An}(\iota^{\topp}(S))$ for all $S$ in $G\backslash K$. 
We have a {commutative} square 
$$\xymatrix{K_{\bC}^{G,\cX}(K)\ar[rr]_-{\simeq}^-{\oplus_{S}q^{\cX}_{S}}\ar[d]^{p_{ K }}&&\ar[d]_{\simeq}^{\oplus_{S}p_{ S }} \bigoplus_{S\in  G\backslash K}   K_{\bC}^{G,\cX}(S) \\
K_{\bC}^{G,\An}(\iota^{\topp}(K)) \ar[rr]_-{\simeq}^-{\oplus_{S}q^{\An}_{S}}&&\bigoplus_{S\in   G\backslash K} K_{\bC}^{G,\An}(\iota^{\topp}(S))  }$$
Since we assume that  $G\backslash K$ is finite  the  horizontal morphisms are equivalences by excision. Furthermore,  the right vertical morphism is an equivalence by Assumption \ref{eoirgjwergerregwgreg}.\ref{ergoiwejrgrgrgwrgwergw}. 
Consequently, the left vertical morphism is an equivalence.

Let $n$ be in $\nat$ and    assume that we have shown  that $p_{K  } $ is an equivalence provided   $K$ is  $G$-finite  $G$-simplicial complex  {of dimension $n$ with stabilizers in $\cF$}   {and} with  structures induced by its spherical path metrics.
Let then $ X $ be in $G\UBC $ and 
assume that there exists a homotopy equivalence
 $X \to  K $. By the naturality of the Paschke transformation
  we can consider the  commutative square 
$$\xymatrix{K_{\bC}^{G,\cX}(X)\ar[r]^{\simeq}\ar[d]^{p_{ X }}&K_{\bC}^{G,\cX}(K)\ar[d]_{\simeq}^{p_{ K }}\\ K_{\bC}^{G,\An}(\iota^{\topp}(X))\ar[r]^{\simeq}&K_{\bC}^{G,\An}(\iota^{\topp}(K))}$$
{Since} the functors and $K_{\bC}^{G,\An}$ and $K_{\bC}^{G,\cX}$   are homotopy invariant by \cite[Thm.~1.15]{KKG}  and Proposition \ref{oirgjoerwgwegregw9}, respectively, the horizontal morphisms are equivalences. By assumption  the right vertical morphism in an equivalence, too. Consequently, the left vertical morphism is also an equivalence.

We now show {the 
induction} step.
Assume that $K $ in $G\UBC $ is such that $K$ is a $G$-finite $G$-simplicial complex  {of dimension $n$ with stabilizers in $\cF$}  with  structures induced by its spherical path metrics.
 Let $Y$ be the closed $1/2$-neighbourhood of the $(n-1)$-skeleton $K_{n-1}$ of $K$ and set $Z \coloneqq K\setminus \inter(Y)$.  Then $(Y,Z)$ is a closed decomposition of~$K$.
 
We can consider $Y$, $Z$ and $Y\cap Z$ as objects in $G\UBC $ with the induced structures.   We then have the following commutative diagram
\begin{equation}
\xymatrix{K_{\bC}^{G,\cX}(Y\cap Z)\ar[ddd]\ar[dr]_-{\simeq}^-{\quad p_{Y\cap Z}}\ar[rrr]& &&K_{\bC}^{G,\cX}(Z)\ar[ddd]\ar[dl]^-{\simeq}_-{p_{Z}}\\
&K_{\bC}^{G,\An}(\iota^{\topp}(Y\cap Z))\ar[r]\ar[d]&K_{\bC}^{G,\An}( \iota^{\topp}(Z))\ar[d]&\\
&K_{\bC}^{G,\An}(\iota^{\topp}(Y) )\ar[r]&K_{\bC}^{G,\An}(\iota^{\topp}(K))&\\
K_{\bC}^{G,\cX}(Y)\ar[rrr]\ar[ur]^-{\simeq}_-{\ p_{Y}}&&&K_{\bC}^{G,\cX}(K)\ar[ul]^-{p_{(K,\tau_{K})}}}\ .
\end{equation} 
  Since $Y, Z$ and $Y\cap Z$ are homotopy equivalent in $G\UBC $ to $G$-finite $G$-simplicial complexes of dimension $<n$ with stabilizers in $\cF$  their Paschke morphisms are equivalences by {the} induction hypothesis. Since the functors $K_{\bC}^{G,\An}\circ \iota^{\topp}$ and $K_{\bC}^{G,\cX}$ 
 are excisive for  this closed decomposition {(for $K^{G,\An}_{\bC}$ we use \cite[Prop. 5.1.2]{KKG})} the inner and the outer square are push-out squares. 
 Alltogether we can then conclude that the Paschke morphism $p_{K }$ is an equivalence, too.
 \end{proof}

In order to prepare the proof of Theorem \ref{qreoigjoergegqrgqerqfewf}.\ref{wrtigjiogowrefwerfewrf}  we replace the Paschke morphism $p$ in Proposition \ref{eoirgjwergerregwgreg}  by the locally finite version $p^{\mathrm{lf}}$ with target $K^{G,\An,\mathrm{lf}}_{\bC}$. In Assumption  \ref{eoirgjwergerregwgreg}.\ref{ergoiwejrgrgrgwrgwergw} we further replace $G_{\cF}\Orb$ by $G_{\cF}\Set$. {Note that this is a stronger assumption.}
 Let $X $ be in $G\UBC $.  The argument for Proposition \ref{eoirgjwergerregwgreg}
then also shows the following statement. 
\begin{prop}\label{eoirgjwergerregwgreg-allg} 
Assume: 
\begin{enumerate}
\item \label{ergoiwejrgrgrgwrgwergw-allg}
The Paschke morphism ${p^{\mathrm{lf}}_{ S }\colon K^{G,\cX}_{\bC}(S)\to K^{G,\An,\mathrm{lf}}_{\bC}(\iota^{\topp}(S))}$ is an equivalence for every countable
 $S$ in $ G_{\cF}\Set$.
\item \label{weliorgjwgwreegwre9-allg} $X$ is homotopy equivalent to a countable,  finite-dimensional  $G$-simplicial complex with stabilizers in $\cF$  {and} with  structures induced by its spherical path metrics.
\end{enumerate}
Then the Paschke morphism ${p^{\mathrm{lf}}_{ X }\colon K^{G,\cX}_{\bC}(X)\to K^{G,\An,\mathrm{lf}}_{\bC }(\iota^{\topp}(X))}$ is an equivalence.
\end{prop}
\begin{proof}
Using the stronger Assumption  \ref{eoirgjwergerregwgreg-allg}.\ref{ergoiwejrgrgrgwrgwergw} instead of Assumption \ref{eoirgjwergerregwgreg}.\ref{ergoiwejrgrgrgwrgwergw} one can redo the proof of Proposition \ref{eoirgjwergerregwgreg} for {$ p^{\mathrm{lf}}$} avoiding the step where we decompose the zero-dimensional complex $K$ into a finite union of $G$-orbits.
\end{proof}

In the following lemma we show that   Assumption   {\ref{eoirgjwergerregwgreg}.}\ref{ergoiwejrgrgrgwrgwergw} implies  Assumption  {\ref{eoirgjwergerregwgreg-allg}.}\ref{ergoiwejrgrgrgwrgwergw-allg}
  provided $G$ is finite {and $\bC$ admits all very small AV-sums.} 
  
  \begin{lem} \label{wergoepgrwegefwf}We assume that $G$ is finite {and that  $\bC$  admits all very  small  {orthogonal} AV-sums.}
If 
the Paschke morphism $p_{T }$ is an equivalence for every
 $T$ in $ G\Orb$, then  the Paschke morphism ${p^{\mathrm{lf}}_{ S }}$ is an equivalence for every countable
 $S$ in $ G\Set$.
\end{lem}
\begin{proof}
 {The functor 
 $K^{G,\An,\mathrm{lf}}_{\bC}\circ \iota^{\topp}$}  sends countable disjoint unions  into
products.
Hence we have an equivalence
\begin{equation}\label{qewfqojkdjsd}
K^{G,\An{,\mathrm{lf}}}_{\bC}(S_{disc})\simeq \prod_{T\in G\backslash S} K^{G,\An{,\mathrm{lf}}}_{\bC}(\iota^{\topp}(T_{disc}))\, .
\end{equation}
   If $G$ is finite, then we have an equality $G_{can,max}=G_{max,max}$. Recall the notion of the free union from \cite[Ex.~2.16]{equicoarse}.
As in the proof 
   of \cite[Lem.\ 3.13]{equicoarse}, by exploiting the equality $G_{can,max}=G_{max,max}$,  we have an isomorphism
\begin{equation}\label{ccdscqcasdcascasdca}
S_{min,min}\otimes G_{can, {max}}\cong (\bigsqcup^{\mathrm{free}}_{T\in G\backslash S} T_{min,min}) \otimes G_{can,max}\cong 
\bigsqcup_{T\in G\backslash S}^{\mathrm{free}} (T_{min,min} \otimes G_{can,max})\, .
\end{equation}
in $G\BC$.
 {The additional assumption on $\bC$ implies that $ K\bC\cX^{G}$} is strongly additive by
\cite[Thm.~11.1]{coarsek}, see also Theorem \ref{qeroigqwefqewfefewfq}. {It therefore} sends free unions to products. Applying now $K\bC\cX^{G}$ to \eqref{ccdscqcasdcascasdca} and
  using Definition \ref{qriofjqofewfefqeqff}  we consequently   have an equivalence 
\begin{equation}\label{qewfqojkdjsd1}
K^{G,\cX}_{\bC}(S_{min,min,disc})\simeq \prod_{T\in G\backslash S} K^{G,\An{,\mathrm{lf}}}_{\bC}(\iota^{\topp}(T_{min,min,disc}))
\end{equation}
 {arising in the following way:
\begin{align*}
K^{G,\cX}_{\bC}(S_{min,min,disc}) \ & = \ K\bC\cX^{G}_{G_{can,max}}( \cO^\infty ( S_{min,min,disc} ) )\\
&\stackrel{{!}}{ \simeq }\ \Sigma K\bC\cX^{G}_{G_{can,max}}( S_{min,min,disc} )\\
& \stackrel{\mathclap{\eqref{ccdscqcasdcascasdca}}}\simeq \ \Sigma K\bC\cX^{G} \Big( \bigsqcup_{T\in G\backslash S}^{\mathrm{free}} (T_{min,min} \otimes G_{can,max}) \Big)\\
& \simeq \ \Sigma \prod_{T\in G\backslash S} K\bC\cX^{G} ( T_{min,min} \otimes G_{can,max} )\\
& \stackrel{{!}}{ \simeq } \ \prod_{T\in G\backslash S} K^{G,\cX}_{\bC} ( T_{min,min,disc} )\\
& \stackrel{\mathclap{\prod_{T}p_{T }}}\simeq \:\:\:\:\: \prod_{T\in G\backslash S} K^{G,\An}_{\bC}(\iota^{\topp}(T_{min,min,disc}))\\   &\simeq \ \prod_{T\in G\backslash S} K^{G,\An{,\mathrm{lf}}}_{\bC}(\iota^{\topp}(T_{min,min,disc}))\, .
\end{align*}
{Here we use \cite[Prop.\ 9.35]{equicoarse} for the  equivalences  marked by $!$.}
By} naturality of the Paschke transformation,  under the equivalences \eqref{qewfqojkdjsd} and   \eqref{qewfqojkdjsd1} the Paschke morphism $p^{{\mathrm{lf}}}_{S }$ corresponds to the product of the Paschke morphisms $p_{ T}$ for the $G$-orbits $T$  in $S$. 
If the latter are equivalences, then $p_{ S }$ is an equivalence.
\end{proof}

At the moment we do not know whether this lemma generalizes to infinite groups, possibly with restrictions on allowed stabilizers.

Combining Proposition \ref{eoirgjwergerregwgreg-allg} with Lemma \ref{wergoepgrwegefwf} we get the following result.
\begin{kor}\label{wtoigjwgowergwregwregw} 
Assume: 
\begin{enumerate}
\item $G$ is finite.
\item {$\bC$ admits all very small AV-sums.}
\item \label{ergoiwejrgrgrgwrgwergw-allg1} The Paschke morphism $p_{T }$ is an equivalence for every
 $T$ in $ G\Orb$.
 \item \label{weliorgjwgwreegwre9-allg1} $X$ is homotopy equivalent to a countable,  finite-dimensional  $G$-simplicial complex  with  structures induced by its spherical path metric.
\end{enumerate}
Then the Paschke morphism $p^{\mathrm{lf}}_{ X } \colon K^{G,\cX}_{\bC}(X)\to K^{G,\An,\mathrm{lf}}_{\bC}(\iota^{\topp}(X))$ is an equivalence.
\end{kor}

\begin{rem}\label{aergijoarvvavadsvsdva}
We can not expect that the Paschke morphism is an equivalence for spaces which are not proper $G$-spaces. More precisely, we do not expect that  Assumption \ref{eoirgjwergerregwgreg}.\ref{ergoiwejrgrgrgwrgwergw}    is  satisfied if $\cF$ contains infinite subgroups.

 {Indeed, a}ssume that $S =G/H$ with $H$ infinite. 
Then we have  \begin{eqnarray}
K_{\bC}^{G,\cX}((G/H)_{min,min,disc})&\stackrel{\text{def.}}{\simeq} &K\bC\cX^{G}_{G_{can,max}}(\cO^{\infty}((G/H)_{min,min,disc} )))\nonumber\\&\stackrel{(1)}{\simeq}&\Sigma K\bC\cX^{G}_{G_{can,max}}((G/H)_{min,min})
\label{rgpfojewropgegwgwrgrgwegwerg}\\
&\stackrel{\text{def.}}{\simeq}&
 \Sigma K\bC\cX^{G}( (G/H)_{min,min}\otimes G_{can,max})\nonumber\\&\stackrel{(2)}{\simeq} &
 0\, , \nonumber
\end{eqnarray}
 {where t}he equivalence $(1)$ is an instance of \cite[Prop.\ 9.35]{equicoarse} since $(G/H)_{min,min,disc}$ is discrete. 
In order to see the equivalence $(2)$  we use that the functor   $K\bC\cX^{G}$ is continuous: We refer  
 to \cite[Def.\ 5.15]{equicoarse} for the definition of this notion and to \cite[Thm.\ 6.3]{coarsek} for the fact.
Continuity implies that the value of $K\bC\cX^{G}(X)$ for any $X$ in $G\BC$  is given as a colimit  of the values
$K\bC\cX^{G}(L)$ over the locally finite invariant subsets $L$ of $X$. 
We now observe that  if $H$ is infinite, then
$(G/H)_{min,min}\otimes G_{can,max}$ does not admit any non-empty invariant locally finite subset. 
Indeed, if $L$ would be such a subset, then on the one hand 
$(eH\times G)\cap L$ is finite,  but the infinite group $H$ acts freely on this set on the other hand.

 In contrast,  the spectrum
$$K_{\bC}^{G,\An} (  (G/H)_{disc} )\simeq \KKG(C_{0}((G/H)_{disc}),  \bQ^{(G)}_{\std})$$ does not vanish in general.
As an example we consider the case $G=H$, and we  further specialize to 
 {$\bC=\Hilb^{G}_{c}(A)$} for a {unital} $G$-$C^{*}$-algebra $A$.
 By Proposition \ref{wegojeogrregregwergwegre}
 we have an equivalence 
$$K_{\bC}^{G,\An} ( (G/H)_{disc})  \simeq \Sigma  \KKG(\C,A)\, .$$
We claim that this spectrum is non-trivial if we take $A=\C$ with the trivial $G$-action.
Indeed, in this case we have the class $ \id_{\kkG(\C)}$ in $\KKG_{0}(\C,\C)$ and
$ \id_{\kkG(\C)}\simeq 0$ if and only if $\KKG(\C,\C)\simeq 0$.
 Since $\kkG(\C)$ is the  tensor  unit of $\KKG$
 we have  $  \KKG(\C,\C)\simeq 0$ if and only if $\KKG\simeq 0$. 
 But since
 $$\KKG(C_{0}(G),\C)\simeq \Kast(\C)\simeq KU$$ by \cite[Thm.\ 1.23]{KKG}
this never happens.
\hB
\end{rem}

Consider $Y$  in $\ppGTop$.
At various places we will use the following properties of this functor.  
\begin{lem}\label{eroigjweogregwgefwe}
If $Y$ is homotopy equivalent to a $G$-finite $G$-CW-complex with finite stabilizers, then:
\begin{enumerate}
\item\label{qroigwoeigwjereferwfwerf} $\KKG(C_{0}(Y),-)$ sends exact sequences  in $\Fun(BG,\nCcat)$ to fibre sequences.
\item\label{qroigwoeigwjereferwfwerf1}  $\KKG(C_{0}(Y),-)$ annihilates flasque objects  in $\Fun(BG,\nCcat)$.
\item\label{qroigwoeigwjereferwfwerf3}  {$\KKG(C_{0}(Y),-)$  sends   relative    Morita equivalences to equivalences.}
\end{enumerate}
\end{lem}
\begin{proof}
By  \cite[Prop.~1.26]{KKG} the object $\kkG(C_{0}(Y))$ is  {$G$-proper and therefore $\mathrm{ind}$-$G$-proper} in the sense of   \cite[Def.~1.25]{KKG}.
{The} assertions now follow from  \cite[Thm.~1.32]{KKG}.
\end{proof}

Let $X$ be in $G\UBC $.  Then we have the multiplication map \eqref{oreihoijfvoisfdvervfdsvdfvsvv}
$$\mu^{\bQ}_{X} \colon C_{0}(X)\otimes \bQ(X)\to  { \bQ}^{(G)}_{\std}\, .$$
We add a superscript $\bQ$ since we are going to consider other versions of this map which will be distinguished by other choices for this superscript.  The main ingredient in the verification that $\mu^{\bQ}_{X}$ is well-defined was
 Lemma  \ref{wtioghjwergfgrefwref}
 saying that for a morphism 
  $A \colon (C,\rho,\mu)\to (C',\rho',\mu')$ in $\bCgtsmc(\cO(X)\otimes G_{can,max})$
we have $\phi'(f)A-A\phi(f)\in {\bC}$  for all  $f$ in $C_{0}(X)$.
If $X$ is discrete, 
then we actually have
$\phi'(f)A-A\phi(f)=0$ for all  such $f$.
This has the effect that  in the construction of $\mu_{X}$  {in \eqref{oreihoijfvoisfdvervfdsvdfvsvv}} on morphisms  {(see Item  \ref{vefvvdfvweewfvfdsv} in the list below   \eqref{oreihoijfvoisfdvervfdsvdfvsvv})} we do not have to go to the quotients in order to ensure compatibility with the composition.

From now one we assume that $X $  is discrete.   
Using the observation just made  we can lift $\mu^{\bQ}_{X}$ to a multiplication map
$$\mu^{\bD}_{X} \colon C_{0}(X)\otimes  \bD(X)\to   \bM\bC^{(G)}_{\std}\, , \quad    f\otimes A\mapsto fA\, ,$$
where    $\bD(X)$ is defined in \eqref{rgfqfewfqewf1}.
Using in addition  Lemma  \ref{ewiogjoerwgfdsfdgsfg}  and the definition \eqref{rgfqfewfqewf}  of $\bC(X)$
the map $\mu^{\bD}_{X} $ restricts to a map $\mu^{\bC}_{X}$ so that 
 we get a morphism of exact sequences in $\Fun(BG,\nCcat)$ 
\begin{equation}\label{qfqwfqewfeqrrredw}
\xymatrix{0\ar[r]& C_{0}(X)\otimes \bC (X)\ar[r]\ar[d]^{\mu_{X}^{\bC}}& C_{0}(X)\otimes \bD (X)\ar[d]^{\mu_{X}^{\bD}} \ar[r]& C_{0}(X)\otimes \bQ (X)\ar[d]^{\mu_{X}^{\bQ}}\ar[r]&0\\0\ar[r]&  {\bC}^{(G)}_{\std}\ar[r]& {\bM}\bC^{(G)}_{\std}\ar[r]& \bQ^{(G)}_{\std}\ar[r]&0}
\end{equation}
 Here in the upper line  we used \eqref{qwefoujujfqew09ufewewfewfqfqef} and  that  $C_{0}(X)\otimes -$  {(involving the maximal tensor product) preserves exact sequences  of $C^{*}$-categories  by  
  \cite[Prop. 7.23.1]{KKG}.} 

In the definition \eqref{rgergergegwegerg} of  the diagonal morphism $ \delta_{X }$
we could replace $\bQ (X)$ by $\bC (X)$ or $\bD (X)$. 
Using the obvious   naturality of the construction of $\delta_{X }$ in this  variable
we get a commutative diagram \begin{align}\label{qfwefeeqfeqfqwfqerrrwwfeqwf}
\\
\mathclap{
\xymatrix{ 
 \KK(\C,\bC (X))\ar[r]\ar[d]^{\delta_{X}^{\bC}} & \KK(\C,\bD (X))
 \ar[d]^{\delta_{X}^{\bD}}\ar[r]& \KK(\C,\bQ (X))
 \ar[d]^{\delta_{X}^{\bQ}} 
 \\ \KKG(C_{0}(X),C_{0}(X)\otimes \bC (X))  \ar[r]& \KKG(C_{0}(X),C_{0}(X)\otimes \bD  (X)) \ar[r]& \KKG(C_{0}(X),C_{0}(X)\otimes \bQ (X))  
 }\ .\notag }
\end{align}
Recall that we assume that $X$ is discrete. We now in addition assume that $X$ is $G$-finite and has finite stabilizers.
Using the exactness of the upper horizontal sequence in \eqref{qfqwfqewfeqrrredw} and  \eqref{qwefoujujfqew09ufewewfewfqfqef} 
we can conclude with Lemma  \ref{eroigjweogregwgefwe}.\ref{qroigwoeigwjereferwfwerf}
 that the horizontal sequences are segments of fibre sequences.
 Applying $\KKG(C_{0}(X),-)$ to   \eqref{qfqwfqewfeqrrredw} and composing  the resulting   morphism of fibre sequences with the morphism   \eqref{qfwefeeqfeqfqwfqerrrwwfeqwf} 
we get the morphism of fibre sequences
\begin{equation}\label{qfwefeeqfeqfqwfqffferrrwwfeqwf}  
\xymatrix{ 
 \KK(\C,\bC (X))\ar[r]\ar[d]^{p_{X}^{\bC}} & \KK(\C,\bD (X))
 \ar[d]^{p_{X}^{\bD} }\ar[r]& \KK(\C,\bQ (X))
 \ar[d]^{p_{X}^{\bQ}}
 \\ \KKG(C_{0}(X), {\bC}^{(G)}_{\std})\ar[r]&\KKG(C_{0}(X), {\bM}\bC^{(G)}_{\std})\ar[r]& \KKG(C_{0}(X),  \bQ^{(G)}_{\std}) }
\end{equation}
where $p_{X}^{\bQ}$ is the Paschke morphism \eqref{nlkkmlmvfdvsdva}.

For a family of subgroups $\cF$ we denote by $G_{\cF}\Set$   the full subcategory of $G\Set$ of $G$-sets with   stabilizers in $\cF$.
Let $Y $ be  a discrete object  of $G\UBC $.
\begin{prop}\label{lem_pseudoloc_vanishes_discrete}\mbox{}
\begin{enumerate}
\item
We have $ \KK(\C,\bD(Y))\simeq 0$.
\item \label{erthojrehtrhtrgegrge}
If $Y$ is in $G_{\Fin}\Set$ 
and $G\backslash Y$ is finite,   then   $\KKG(C_{0}(Y),{\bM}\bC^{(G)}_{\std})\simeq 0$.
\end{enumerate}
\end{prop}
\begin{proof}
We have the chain of equivalences: 
\begin{eqnarray*}
 \KK(\C,\bD(Y))&\stackrel{\eqref{rgfqfewfqewf1} \ \& \  \text{Def.~}\ref{qrogijeqoifefewfefewqffe}}{\simeq}&
 K\bC\cX^{G}_{G_{can,max}}(\cO (Y))\\ &\simeq&
 0
\end{eqnarray*}
since  the cone $\cO(Y)$  of a discrete object in $G\UBC$ is a flasque object in $G\BC$ by \cite[Ex. 9.25]{equicoarse}  and
the coarse homology theory  $K\bC\cX^{G}_{G_{can,max}}$ vanishes on flasques.
 %
%
%
%
%
Since $ \bM\bC^{(G)}_{\std}$ is flasque by Lemma \ref{wrtijhoerhgrtgertg} we conclude Assertion \ref{erthojrehtrhtrgegrge} with  
  Lemma   \ref{eroigjweogregwgefwe}.\ref{qroigwoeigwjereferwfwerf1}.
\end{proof}

Using {Proposition} \ref{lem_pseudoloc_vanishes_discrete} and  the morphism of fibre sequences \eqref{qfwefeeqfeqfqwfqffferrrwwfeqwf}
we get the following corollary.
\begin{kor}\label{qerpgoijqpwfeqfewfq}
If $X$ is in $G_{\Fin}\Orb$, then we have a commutative square $$\xymatrix{\Omega \KK(\C,\bQ (X))\ar[r]^{\simeq}\ar[d]^{\Omega p^{\bQ}_{X}}& \KK(\C,\bC (X))\ar[d]^{p_{X}^{\bC}}\\ \Omega  \KKG(C_{0}(X),\bQ^{(G)}_{\std})\ar[r]^{\simeq}& \KKG(C_{0}(X), {\bC}^{(G)}_{\std})}$$
In particular, the Paschke morphism for $X$ in $G_{\Fin} \Orb$ is an equivalence if and only if
the morphism $p_{X}^{\bC} \coloneqq \mu_{X}^{\bC}\circ \delta_{X}^{\bC}$ is an equivalence.
\end{kor}

In view of Corollary \ref{qerpgoijqpwfeqfewfq} and Proposition \ref{eoirgjwergerregwgreg} and  {Corollary  \ref{wtoigjwgowergwregwregw}\footnote{This corollary is needed only for Theorem~\ref{qreoigjoergegqrgqerqfewf}.\ref{wrtigjiogowrefwerfewrf}.}} the following proposition finishes the proof of the
Theorems \ref{qreoigjoergegqrgqerqfewf}.\ref{qregiojqwfewfqwfqewf} and  \ref{qreoigjoergegqrgqerqfewf}.\ref{wrtigjiogowrefwerfewrf}. 
{We assume that $\bC$ in $\Fun(BG,\nCcat)$ is effectively additive and  admits countable AV-sums.}
 \begin{prop}\label{qeoigjqergqwedewdqed}
  {If  $X$ is in $G_{\Fin}\Orb$, 
then} 
\begin{equation}\label{verwejvweopvwerverwv}
p_{X}^{\bC} \colon \KK(\C,\bC (X))\to  \KKG(C_{0}(X),{\bC}^{(G)}_{\std})
\end{equation} 
is an equivalence.
\end{prop}

The whole of Section \ref{oifjqweofwefewffqfef} is devoted to the proof of this proposition.

\section{Verification of the Paschke equivalence on \texorpdfstring{$\boldsymbol{G}$}{G}-orbits}\label{oifjqweofwefewffqfef}

{We assume that $\bC$ in $\Fun(BG,\nCcat)$ is effectively additive and  admits countable AV-sums.}
We fix a finite subgroup  $H$   of $G$ and consider the $G$-set  $G/H$ in $G_{\Fin}\Orb$. 
As a first step we construct an explicit functor $\Theta$ in $\nCcat$ and show in Proposition \ref{wqeroighowergergregrgwergwreg} that
$p^{\bC}_{G/H}$ is an equivalence if and only if $\Kcat(\Theta)$ is an equivalence.
In the second  step we then verify in Proposition \ref{4oijotrherhrthrhe} that $\Kcat(\Theta)$  is an equivalence.

We form
 the $G$-bornological coarse space
$(G/H)_{min,min}\otimes G_{can,max}$.
It contains  
  the locally finite subset   
\begin{equation}\label{qrewfpokfpowefqweffqwfwefq}
X \coloneqq G(H, e)\, ,
\end{equation}
the $G$-orbit of the point $(H,e)$ in $G/H\times G$.  
Note that in contrast to the example in Remark \ref{aergijoarvvavadsvsdva}  the group $H$ is finite.
We equip $X$
 with the  bornological coarse structures  induced from $(G/H)_{min,min}\otimes G_{can,max}$.    {The map $g\mapsto g(H,e)$ is a $G$-equivariant  bijection of sets  between $G$ and $X$  which will be used below to name points and subsets of $X$.} The   induced bornology on $X$ is the minimal one. The induced $G$-coarse structure
 reflects the information about the finite subgroup $H$ and is in general smaller than the canonical coarse structure on $G$.  {For instance,    the subset $H $ is a coarse component of $X$.}

  The following lemma states 
  that the inclusion  $X\to  (G/H)_{min,min}\otimes G_{can,max}$
  is a continuous equivalence in the sense of   
\cite[Sec.~7]{desc}. 
 \begin{lem}\label{wegoijwegowiergwegerw}
 The inclusion $X\to  (G/H)_{min,min}\otimes G_{can,max}$ induces an equivalence
 $E(X)\to E((G/H)_{min,min}\otimes G_{can,max})$  for any continuous equivariant {coarse} homology theory $E$.
 \end{lem}
 \begin{proof} 
For $Y$ in $G\BC$ we let $\LF(Y)$ denote the poset of $G$-invariant locally finite subsets.
Let $L$ be  in $\LF((G/H)_{min,min}\otimes G_{can,max})$.
Then $L_{0} \coloneqq L\cap ( \{H\}\times G)  $ is a finite set which we will sometimes consider as a subset of $G$. Since every $G$-orbit in $L$ meets $L_{0}$ we have $L = GL_{0}$.
  
  We claim that 
  for every $L$ in $\LF((G/H)_{min,min}\otimes G_{can,max})$  the inclusion  $i \colon X\to L\cup X$ is a coarse equivalence.
 Indeed, we can construct
 an  inverse equivalence $p \colon L\cup X\to X$. The 
 map $p$ is the identity on $X$, and  it sends a point $g(H,h)$ (with $h$ in $L_{0}\setminus \{e\}$) in $L\setminus X$ to $g(H,e)$ in $X$. Then $p\circ i=\id_{X}$
 and  $i\circ p $ is close to the identity. In order to see the second assertion note that $L_{0}$ is finite and therefore $  \diag(G/H)\times \{(gh,g)\:|\: h\in L_{0},g\in G\}$ is a coarse entourage   
 of $(G/H)_{min,min}\otimes G_{can,max}$. We then use that    
  $$(\id_{X},i\circ p)(\diag(L\cup X)) \subseteq \diag(G/H)\times \{(gh,g)\:|\: h\in L_{0}, g\in G\}\, .$$  
  
 If $E$ is any equivariant coarse homology theory, 
 then  the canonical morphism $$E(X)\to \colim_{L\in \LF((G/H)_{min,min}\otimes G_{can,max}) } E(L)$$
  is an equivalence since the elements of $ \LF((G/H)_{min,min}\otimes G_{can,max}) $ containing $X$ are cofinal and for {those} elements the 
 inclusions $X\to   L$ are coarse equivalences. Since we assume in addition that 
  $E$  is continuous, the canonical morphism      $$  \colim_{L\in \LF((G/H)_{min,min}\otimes G_{can,max}) } E(L)\to E((G/H)_{min,min}\otimes G_{can,max})$$ is an equivalence. Hence  the composition of these equivalences is an equivalence
\[E(X)\to E((G/H)_{min,min}\otimes G_{can,max})\, . \qedhere\]
\end{proof}

Using the inclusion
\begin{equation}\label{evrwvjoivjiowewrervewrv}
i \colon {X}  \to  G/H\times G\to Z_{0}
\end{equation}
(see \eqref{fqewqewfkqwpefowqefewfef} for the notation $Z_{0}$ {as a subspace of $\cO((G/H)_{min,min})\otimes G_{can,max}$})   
\begin{equation}\label{evrwvjoivjiowewrervewrv234}
i_{*} \colon \bCgtsmc(X)\to \bC(G/H )
\end{equation}
(where we use \eqref{rgfqfewfqewf} for $\bC(G/H):=\bC((G/H)_{min,min,disc})$) we get an  inclusion
which identifies $\bCgtsmc(X)$
  with the full subcategory of objects of $ \bC(G/H ) $ supported on  $i(X)$.

    In the following $
    \Idem(\Res^{G}_{H}(   {\bC}^{(G) }_{\std})\rtimes H ) $
 is
 the relative idempotent completion using the embedding of $\Res^{G}_{H}(  {\bC}^{(G) }_{\std})\rtimes H$ as an ideal into $\Res^{G}_{H}( {\bM} \bC^{(G) }_{\std})\rtimes H $, \cite[{Def.~17.5}]{cank}. {In order to keep the notation readable\footnote{{i.e., to avoid symbols like $\Idem^{\Res^{G}_{H}( {\bM} \bC^{(G) }_{\std})\rtimes H)}(\Res^{G}_{H}(  {\bC}^{(G) }_{\std})\rtimes H)$}}, in contrast to the reference we will not indicate the bigger unital category by a superscript.}
Recall the notation for morphisms in crossed {products} from   \cite[Def.~5.1]{crosscat}. {In the formulas below, e.g., in order to interpret  the term $\mu(H)$ in \eqref{erergwgerw4r43r3},   we use the bijection between $G$ and $X$ mentioned above.}
    \begin{ddd}\label{weoigjowergrgweerffrfwff}
We define the functor 
   \begin{equation}\label{qwefewfefefweqfewfqwef}
\Theta \colon \bCgtsmc(X)\to  \Idem(\Res^{G}_{H}(  {\bC}^{(G) }_{\std})\rtimes H ) 
\end{equation}
   as follows:
 \begin{enumerate}
 \item objects: $\Theta$ sends the object $(C,\rho,\mu)$ in $ \bCgtsmc(X)$ to the object
 $( C,\rho ,{\pi} )$ in the category $\Idem(\Res^{G}_{H}(  {\bC}^{(G) }_{\std})\rtimes H)$,
 where the orthogonal projection  $\pi$ on $(C,\rho)$ is given by  \begin{equation}\label{erergwgerw4r43r3}
\pi: =\frac{1}{|H|}\sum_{h\in H} (\mu(H) ,h)\, .
\end{equation}
  \item \label{qewrglijqeovewfwerfwerfwerferfw} morphisms:    $\Theta$ sends $A \colon (C,\rho,\mu)\to (C',\rho',\mu')$ in $  \bCgtsmc(X)$ to the morphism $$\pi' (A,e)\pi \colon ( C,\rho,  \pi)
\to ( C',{\rho'} ,\pi')$$
in $\Idem(\Res^{G}_{H}(   {\bC}^{(G) }_{\std})\rtimes H)$.   
\end{enumerate}
\end{ddd}
{Note that $A \colon C\to C'$ belongs to $\bM\bC$, but since $H $ is a finite and hence bounded subset of $X$,
the projection  $\mu(H )$ belongs to $\bC$ by the local finiteness of $(C,\rho,\mu)$. Therefore $\pi'(A,e)\pi$ belongs to the ideal
$\Idem(\Res^{G}_{H}(   {\bC}^{(G) }_{\std})\rtimes H)$ as stated.} 
{In order to see that $\Theta$ is compatible with the composition}
     note that the relations    {$A\mu(H )=\mu'(H )A$ (since $H$ is a coarse component of $X$)} and   $h\cdot A=A$ for all $h$ in $H$  imply that
$(A,e)\pi=\pi'(A,e)$. 

 \begin{prop}\label{wqeroighowergergregrgwergwreg}
 The morphism $p^{\bC}_{G/H}$ in \eqref{verwejvweopvwerverwv} is an equivalence if and only if the morphism
  $\Kcat(\Theta) $ is  an equivalence, where $\Theta$ is as  in Definition \ref{weoigjowergrgweerffrfwff}.
  \end{prop}
  \begin{proof}
  Recall that we consider $G/H$ as the object $G/H_{min,min,disc}$ of $G\UBC$ so that $C_{0}(G/H)$ is given by Definition \ref{kohperthrtgertg}.\ref{kohethertgtergelabel}.  
   In analogy to the diagonal  {morphism} \eqref{rgergergegwegerg} we define 
  \begin{eqnarray}\label{rgergergegffewefwefwefewfwfwegerg}
\delta' \colon \KK(\C,\bCgtsmc(X))
&\stackrel{C_{0}(G/H)\otimes -}{\to}&
 \KKG(C_{0}(G/H) ,C_{0}(G/H) \otimes \bCgtsmc(X)) \, .\nonumber
\end{eqnarray}
We then have a commutative diagram
\begin{align}
\label{erjioqergqrfqwefwqefqwfqef}\\
\mathclap{
\xymatrix{
\KK(\C,\bCgtsmc(X))\ar@/^-4cm/[dd]_{p'}\ar[rr]^{\KK(\C,i_{*})}\ar[d]^{\delta'} && \ar@/^4cm/[dd]^{p^{\bC}_{G/H}}  KK(\C,\bC(G/H))\ar[d]^{\delta_{G/H}^{\bC}}\\
\KKG(C_{0}(G/H) ,C_{0}(G/H) \otimes \bCgtsmc(X))\ar[rr]^{C_{0}(G/H)\otimes i_{*}}
\ar[d]^{\mu'} &&           \KKG(C_{0}(G/H) ,C_{0}(G/H) \otimes \bC(G/H) )\ar[d]^{\mu^{\bC}_{G/H}}\\
\KKG(C_{0}(G/H), {\bC}^{(G)}_{\std})\ar@{=}[rr] && \KKG(C_{0}(G/K),   {\bC}^{(G)}_{\std})
}
}
\notag 
\end{align}
where $\mu' \coloneqq \mu^{\bC}_{G/H}\circ ( C_{0}(G/H)\otimes i_{*})$ {and $i_*$ is as in \eqref{evrwvjoivjiowewrervewrv234}.}
The filler of the upper square is induced from the fact that \eqref{saCWDQKOIOJ1} is a bifunctor. Implicitly we also   used the Lemma \ref{glkbijowerfvvfevsdfvsfdv} in order to relate $\hatotimes$ and $\otimes$.

\begin{lem}\label{wegiojergwergwrefrfref}
The morphism $\KK(\C,i_{*}) \colon \KK(\C,\bCgtsmc(X))\to  \KK(\C,\bC(G/H))$ is an equivalence.
\end{lem}
\begin{proof}
 Using the definitions
$K\bC\cX^{G}(-):=\Kcat(\bCgtsmc(-))$ and $\Kcat(-) \coloneqq \KK(\C, -)$  and \eqref{evrwvjoivjiowewrervewrv} we can rewrite the morphism in question as 
\begin{equation}\label{wegergeferferfwef}K\bC\cX^{G}(X)\to K\bC\cX^{G}((G/H)_{min,min}\otimes G_{can,max})\to \Kcat (\bC(G/H)){\, ,}
\end{equation}
where the morphisms are induced by the canonical inclusions of $C^{*}$-categories.  We have seen in the proof of Proposition \ref{wrtoihgjoergergegwgrgwerg}  that the second morphism in \eqref{wegergeferferfwef} (it  is an instance of the left vertical morphism in \eqref{qfwefeeqfeqfqwfqewwfeqwf} applied  to
  $(G/H)_{min,min,disc}$ in place of $X$)  
  is an equivalence. The first morphism  in \eqref{wegergeferferfwef}  is induced by the inclusion $X\to (G/H)_{min,min}\otimes G_{can,max}$. Since $ K\bC\cX^{G}$ is a continuous equivariant coarse homology theory {it is an equivalence by} Lemma \ref{wegoijwegowiergwegerw}.
  \end{proof}
{We {continue with} the proof of Proposition \ref{wqeroighowergergregrgwergwreg}.} We define
$p':=\mu'\circ \delta'$.
In view of \eqref{erjioqergqrfqwefwqefqwfqef} and Lemma \ref{wegiojergwergwrefrfref} 
we conclude that
\begin{equation}\label{eq14toirj13o}
p'\simeq p^{\bC}_{G/H}\, .
\end{equation}

 We consider the morphism $\epsilon \colon \C\to \C\rtimes H$ which sends $1$ to the projection $\frac{1}{|H|}\sum_{h\in H} (1,h)$. 
Let furthermore  $\iota \colon \C\to \Res^{G}_{H}(C_{0}(G/H))$ be the homomorphism sending $z$ in $\C$ to $z\chi_{H}$, where $\chi_{H}$ is the characteristic function of the orbit $H$ in $G/H$. 
We then  have the following {commutative} diagram:
\begin{align}
\label{oigjeroigergergergergwe}\\
\mathclap{
\xymatrix{\KK^{G}(C_{0}(G/H),C_{0}(G/H)\otimes \bCgtsmc(X))\ar[r]^-{\mu'}\ar[d]\ar@/_5.3cm/[dd]_-{r^{G}_{H}}^{\simeq}& \KKG(C_{0}(G/H),  {\bC}^{(G)}_{\std}) \ar[d]\ar@/^4cm/[dd]_{r^{G}_{H}}^{\simeq}\\ 
\KKH(\Res^{G}_{H}(C_{0}(G/H)), \Res^{G}_{H}(C_{0}(G/H)\otimes \bCgtsmc(X)))\ar[r]^-{\Res^{G}_{H}(\mu')}\ar[d]^{\iota^{*}}&\KKH(\Res^{G}_{H}(C_{0}(G/H)),\Res^{G}_{H}(   {\bC}^{(G)}_{\std}))\ar[d]^{\iota^{*}}\\
\KKH(\C, \Res^{G}_{H}(C_{0}(G/H)\otimes \bCgtsmc(X)))\ar@/_5.3cm/[dd]_{j^{H}}^{\simeq}\ar[r]^-{\Res^{G}_{H}(\mu')}\ar[d]^{-\rtimes H}&\ar@/^4cm/[dd]_{j^{H}}^{\simeq}\KKH(\C,\Res^{G}_{H}( {\bC}^{(G)}_{\std}))\ar[d]^{-\rtimes H }\\\KK(\C\rtimes H, (\Res^{G}_{H}(C_{0}(G/H)\otimes \bCgtsmc(X)))\rtimes H)\ar[r]^-{\Res^{G}_{H}(\mu')\rtimes H}\ar[d]^{\epsilon^{*}}&\KK(\C\rtimes H,\Res^{G}_{H} (    {\bC}^{(G) }_{\std}) \rtimes H)\ar[d]^{\epsilon^{*}}\\\KK(\C, (\Res^{G}_{H}(C_{0}(G/H)\otimes \bCgtsmc(X)))\rtimes H)\ar[r]^-{\Res^{G}_{H}(\mu')\rtimes H} &\KK(\C,\Res^{G}_{H}(     {\bC}^{(G) }_{\std})\rtimes H)
}
}
\notag 
\end{align}
The second and the last middle {square} commute by the associativity of the composition in $\KKH$ and $\KK$, respectively. The first and the third {square} commute since
$\Res^{G}_{H}
$ and $-\rtimes H$ are functors. In order to see that 
  $r^{G}_{H}$ and $j^{H}$ are equivalences we observe that
  $\iota$ and $\epsilon$ are instances of the units of the adjunctions in  \cite[Thm.~1.23.1 \& 2]{KKG} (induction and restriction $(\Ind^{G}_{H} \dashv \Res^{G}_{H})$ and the Green-Julg adjunction $(\Res^{H} \dashv -\rtimes H)$) and that  
  $r^{G}_{H}$ and $j^{H}$ are precisely the corresponding equivalences
  of mapping spectra.

We furthermore have the diagram
\begin{align}
\label{efqefeqfqefefqewf}\\
{\scriptsize\mathclap{
\xymatrix{ \Hom_{\Fun(BG,\nCalg )}(C_{0}(G/H),C_{0}(G/H))\times \KK(\C,\bCgtsmc(X))\ar[r]^-{{\hat  \otimes}}\ar[d]&\KK^{G}(C_{0}(G/H),C_{0}(G/H)\otimes \bCgtsmc(X)) \ar[d] \ar@/^4.3cm/[dd]_{r^{G}_{H}}^{\simeq}\\ 
\Hom_{\Fun(BH,\nCalg )}(\Res^{G}_{H}C_{0}(G/H),\Res^{G}_{H}C_{0}(G/H))\times \KK(\C,\bCgtsmc(X))\ar[r]^-{ {\hat  \otimes}}\ar[d]^{\iota^{*}\times \id} &\KKH(\Res^{G}_{H}(C_{0}(G/H)), \Res^{G}_{H}(C_{0}(G/H))\otimes \bCgtsmc(X))\ar[d]^{\iota^{*}} \\
\Hom_{\Fun(BH,\nCalg )}(\C,\Res^{G}_{H}C_{0}(G/H))\times \KK(\C,\bCgtsmc(X))\ar[r]^-{{\hat  \otimes}}\ar[d]^{(-\rtimes H)\times \id} &\KKH(\C, \Res^{G}_{H}(C_{0}(G/H))\otimes \bCgtsmc(X)) \ar[d]^{-\rtimes H}\ar@/^4.3cm/[dd]_{j^{H}}^{\simeq} \\
\Hom_{\nCalg }(\C\rtimes H,\Res^{G}_{H}C_{0}(G/H)\rtimes H) \times \KK(\C,\bCgtsmc(X))\ar[r]^-{ {\hat  \otimes}}\ar[d]^{\epsilon^{*}\times \id} &\KK(\C\rtimes H, (\Res^{G}_{H}(C_{0}(G/H))\otimes \bCgtsmc(X))\rtimes H)  \ar[d]^{\epsilon^{*}} \\
\Hom_{\nCalg }(\C,\Res^{G}_{H}C_{0}(G/H)\rtimes H) \times \KK(\C,\bCgtsmc(X))\ar[r]^-{{\hat  \otimes} }&\KK(\C, (\Res^{G}_{H}(C_{0}(G/H))\otimes \bCgtsmc(X))\rtimes H)
}}}\notag \end{align}
In the targets of the two lower maps  we implicitly used the identification \begin{equation}\label{eqwffiouhq9wefewfqfewf}
(A\rtimes H)\otimes \bB\cong  (A\otimes \bB) \rtimes H 
\end{equation}  for
$A$ in $\Fun(BH,\nCalg)$ and $\bB$ in $\nCcat$.
The second and the last {square} commute since  ${\hat  \otimes}$  in \eqref{saCWDQKOIOJ1} is a bifunctor. 
We now provide the fillers for the first and the third {square}. 
We consider the diagram 
$$\xymatrix{ \Fun(BG,\nCalg )\times  \nCalg\ar[r]^-{\otimes}\ar[d]^{\Res^{G}_{H}\times \id}&\ar[d] \Fun(BG,\nCalg) \ar[r]^-{\kkG}\ar[d]^{\Res^{G}_{H}}&\KKG \ar[d]^{\Res^{G}_{H}} \\  \Fun(BH,\nCalg )\times  \nCalg \ar[r]^-{\otimes}& \Fun(BG,\nCalg)\ar[r]^-{\kkH}& \KKH  }$$ 
The left cell obviously commutes, and the right cell commutes 
by \cite[Thm.\ 1.22]{KKG}.
 We now extend using the universal property of $\kk \colon 
  \nCalg\to \KK$ \cite[Thm.\ 1.19]{KKG}  in order to get a {commutative} diagram
  $$\xymatrix{ \Fun(BG,\nCalg )\times  \KK \ar[r]^-{  {\hat  \otimes}}\ar[d]^{\Res^{G}_{H}\times \id}&\ar[d] \KKG \ar[d]^{\Res^{G}_{H}}  \\  \Fun(BH,\nCalg )\times  \KK   \ar[r]^-{ {\hat  \otimes}}& \KKH   }$$
  This applied to morphism spaces yields the filler of the first  middle square in 
  \eqref{efqefeqfqefefqewf}. 
  In order to justify the third middle square we argue similarly.
We consider the diagram 
$$\xymatrix{ \Fun(BH,\nCalg )\times  \nCalg\ar[r]^-{\otimes }\ar[d]^{ -\rtimes H\times \id}&  \Fun(BH,\nCalg) \ar[r]^-{\kkG}\ar[d]^{ - \rtimes H}&\KKH \ar[d]^{-  \rtimes H} 
\\  \Fun(BH,\nCalg )\times  \nCalg  \ar[r]^-{\otimes}& \Fun(BG,\nCalg )\ar[r]^-{\kk}& \KK  }$$ 
The left square commutes  because of \eqref{eqwffiouhq9wefewfqfewf}, and the right cell commutes by \cite[Thm. 1.22]{KKG}. 
We now extend using the universal property of $\kk \colon 
  \nCalg\to \KK$  in  \cite[Thm.\ 1.19]{KKG} in order to get a {commutative} diagram 
$$\xymatrix{ \Fun(BH,\nCalg )\times\KK \ar[r]^-{ {\hat  \otimes} }\ar[d]^{ -  \rtimes H\times \id}& \KK \ar[d]^{ -  \rtimes H} 
\\  \Fun(BH,\nCalg )\times \KK  \ar[r]^-{ {\hat  \otimes}}&\KK  }$$
This square yields the of the  third middle square in
  \eqref{efqefeqfqefefqewf}. 
  
  We specialize the diagram \eqref{efqefeqfqefefqewf} at $\id_{C_{0}(G/H)}$ in
  $ \Hom_{\Fun(BG,\nCalg )}(C_{0}(G/H),C_{0}(G/H))$. Then we get 
\begin{align}
\label{efqefeqfqfvdfvdfvefefqfwefwefewf}
\xymatrix{  \KK(\C,\bCgtsmc(X))\ar[rr]^-{ \delta'}\ar@{=}[d]&&\KK^{G}(C_{0}(G/H),C_{0}(G/H)\otimes \bCgtsmc(X)) \ar[d] \ar@/^5.35cm/[dd]_{r^{G}_{H}}^{\simeq}\\ 
  \KK(\C,\bCgtsmc(X))\ar[rr]^-{\id_{\Res^{G}_{H}C_{0}(G/H)}{\hat  \otimes}}\ar@{=}[d]  &&\KKH(\Res^{G}_{H}(C_{0}(G/H)), \Res^{G}_{H}(C_{0}(G/H))\otimes \bCgtsmc(X))\ar[d]^{\iota^{*}} \\
  \KK(\C,\bCgtsmc(X))\ar[rr]^-{\iota {\hat  \otimes}}\ar@{=}[d] &&\KKH(\C, \Res^{G}_{H}(C_{0}(G/H))  \otimes  \bCgtsmc(X)) \ar[d]^{-\rtimes H}\ar@/^5.35cm/[dd]_{j^{H}}^{\simeq} \\
 \KK(\C,\bCgtsmc(X))\ar[rr]^-{( \iota\rtimes H) {\hat  \otimes}}\ar@{=}[d]  &&\KK(C^{*}(H), (\Res^{G}_{H}(C_{0}(G/H))  \otimes  \bCgtsmc(X))\rtimes H)  \ar[d]^{\epsilon^{*}} \\
  \KK(\C,\bCgtsmc(X))\ar[rr]^-{  \delta''}&&\KK(\C, (\Res^{G}_{H}(C_{0}(G/H)  \otimes  \bCgtsmc(X))\rtimes H)  }
\end{align}
where  
\begin{align*}
\delta'' \coloneqq {\epsilon^{*}}(\iota\rtimes H){\hat  \otimes}- \colon \KK(\C,\bCgtsmc(X)) & \to \KK(\C, (\Res^{G}_{H}(C_{0}(G/H)) \rtimes H)\otimes \bCgtsmc(X)))\\
&\simeq  \KK(\C, (\Res^{G}_{H}(C_{0}(G/H))\otimes \bCgtsmc(X))\rtimes H)    \, .
\end{align*}
Composing \eqref{efqefeqfqfvdfvdfvefefqfwefwefewf} with \eqref{oigjeroigergergergergwe} we get a {commutative} square
$$\xymatrix{  \KK(\C,\bCgtsmc(X))\ar[rrr]^-{p'=  \mu' \circ \delta'}\ar@{=}[d]&&&  \KKG(C_{0}(G/H),    {\bC}^{(G)}_{\std})\ar[d]_{\simeq}^{j^{H}\circ r^{G}_{H}}\\  \KK(\C,\bCgtsmc(X))\ar[rrr]^-{p'':=\Res^{G}_{H}(\mu')\rtimes H\circ \delta''}
&&&\KK(\C,\Res^{G}_{H}(  {\bC}^{(G) }_{\std})\rtimes H)
}$$
 We therefore have an equivalence 
 \begin{equation}\label{qwfewpeofkqwepfqew1f}
p''\simeq p'\stackrel{\eqref{eq14toirj13o}}{\simeq} p^{\bC}_{G/H}\, .
\end{equation}
 By construction the morphism $p''$ is induced 
 by an explicit  functor 
 \begin{equation}\label{vervevervwekvpewvervevwev}
\Theta' \colon \bCgtsmc(X)\to \Res^{G}_{H}(  {\bC}^{(G)}_{\std})\rtimes H\, .
\end{equation}

 Inserting all definitions we see that $\Theta'$  is  given by follows:
 \begin{enumerate}
 \item objects:  $\Theta'$ sends the object $(C,\rho,\mu)$ in $ \bCgtsmc(X)$ to
 $(C, \rho  )$ in $ \Res^{G}_{H}( { \bC}^{(G) }_{\std})\rtimes H$.
\item morphisms:  The functor $\Theta'$ sends a morphism $A \colon (C,\rho,\mu)\to (C',\rho',\mu')$ in $  \bCgtsmc(X)$ to the morphism $$ {\pi'A\pi}\colon (C,\rho)\to  (C',\rho' )$$
in $\Res^{G}_{H}( {\bC}^{(G) }_{\std})\rtimes H$, {where $\pi$ is as in \eqref{erergwgerw4r43r3}}.  
 \end{enumerate}
 {The observations made after the Definition \ref{weoigjowergrgweerffrfwff} of $\Theta$ also show that $\Theta'$  is well-defined.} Note,  {however,} that $\Theta'$ is not full.

 Let $$ c\colon \Res^{G}_{H}(  {\bC}^{(G) }_{\std})\rtimes H\to \Idem(\Res^{G}_{H}(  {\bC}^{(G) }_{\std})\rtimes H )$$ be the inclusion {into the relative idempotent completion}.
We consider the two  functors  $$ \Theta , c\circ \Theta' \colon \bCgtsmc(X)\to   \Idem(\Res^{G}_{H}( {\bC}^{(G) }_{\std})\rtimes H)$$
 in $\nCcat$.
 
{Recall the notion of a   Murray von Neumann   (MvN)  equivalence  \cite[Def.\ {17.12}]{cank}.
\begin{lem}\label{egeroigjregwergege} There is a MvN equivalence $\Theta \to c\circ \Theta'$. 
  In particular  \begin{equation}\label{werfpokfpwerferffwrf}
\Kcat ( \Theta )\simeq \Kcat ( c\circ \Theta' ) \colon \Kcat ( \bCgtsmc(X ))\to \Kcat(\Idem(\Res^{G}_{H}(  {\bC}^{(G)}_{\std})\rtimes H)) \, .
\end{equation} 
\end{lem}}
\begin{proof}

{Applying   \cite[Rem.\ {17.13}]{cank} to the inclusion of $ \Idem(\Res^{G}_{H}( \bC^{(G)}_{\std})\rtimes H)$ as an ideal  into $ \Idem(\Res^{G}_{H}( \bM\bC^{(G)}_{\std})\rtimes H)$
 it suffices to construct a natural transformation 
 $v \colon \Theta\to c\circ \Theta'$ implemented 
  by a family $(v_{(C,\rho,\mu)})_{(C,\rho,\mu)\in  \bCgtsmc(X)}$ of partial isometries  in
  $\Idem(\Res^{G}_{H}( \bM\bC^{(G)}_{\std})\rtimes H)$.}

  {We define 
   $v_{(C,\rho,\mu)} \colon (C,\rho,p)\to (C,\rho)$ to be the canonical inclusion.}
Since the formulas for the actions of $\Theta$ and $\Theta'$ {on morphisms} are equal, this family is indeed a natural transformation.  
\end{proof}

{We  {continue with} the proof of Proposition \ref{wqeroighowergergregrgwergwreg}.} 
{Since  the homological functor $\Kcat$ is  Morita invariant by  \cite[Thm.\  {16.18}]{cank}
the morphism 
 $$\Kcat(c)\colon  \Kcat (\Res^{G}_{H}(  {\bC}^{(G)}_{\std})\rtimes H)\to \Kcat(\Idem(\Res^{G}_{H}(  {\bC}^{(G)}_{\std})\rtimes H))$$
is an equivalence {by  \cite[Prop.\ 17.{8}]{cank}}.} Therefore
$\Kcat(\Theta) $ is an  equivalence if and only if $ \Kcat(\Theta' )$ is an  equivalence. 
The Proposition \ref{wqeroighowergergregrgwergwreg} now follows from 
the
   combination of  \eqref{qwfewpeofkqwepfqew1f} and the fact that $p''$ is induced by the functor $\Theta'$.
  \end{proof}
  
 Recall the Definition \ref{weoigjowergrgweerffrfwff} of the functor $\Theta$ and that $H$ denotes a finite subgroup of $G$.
   The next proposition finishes the proof of 
 Proposition  \ref{qeoigjqergqwedewdqed} and hence of Theorem \ref{qreoigjoergegqrgqerqfewf}.
  \begin{prop}\label{4oijotrherhrthrhe}
   {
  $\Kcat(\Theta)$ is an equivalence.}
  \end{prop}
\begin{proof}
The proof of  Proposition \ref{4oijotrherhrthrhe} is based on the factorization of $\Theta$ as described by the commutative diagram \eqref{oijiojiovawvasfvvsdvdsva}. The functors in this diagram will all induce equivalences in $K$-theory, but for different reasons. {The rest of this section is devoted to the proof of Proposition \ref{4oijotrherhrthrhe} {which} is split in several lemmas.}

   \begin{lem}\label{qwroiqrgegwregrgwregwg}
 The functor $\Theta$ is fully faithful.
 \end{lem}
 \begin{proof}
 Recall that   $X=G(H,e)$ is a  subspace  of $(G/H)_{min,min}\otimes G_{can,max}$, see  \eqref{qrewfpokfpowefqweffqwfwefq}.  Let $(C,\rho,\mu)$ and $ (C',\rho',\mu')$ be objects of $ \bCgtsmc(X)$.
 Then $\Theta(C,\rho,\mu)=(C,\rho,\pi)$  with $\pi$ given by \eqref{erergwgerw4r43r3}, and similarly $\Theta(C',\rho',\mu')= (C',\rho',\pi')$.
 Let $$B \colon (C,\rho,\pi)\to (C',\rho',\pi')$$
 by any morphism. We can write $B=\sum_{h\in H} (B_{h},h)$, where $B_{h} \colon C\to C'$. The condition
 $\pi'B\pi=B$ implies that
 $B_{h}= \mu'(H)B_{e}\mu(H)$ and $h \cdot B_{e}=B_{e}$ for every $h$ in $H$.  {Using \cite[Lem.\ 7.8]{cank}
 we} can define the  morphism 
\begin{equation}\label{qewfqowifheiofewfqewfewf}
A \coloneqq \frac{1}{|H|}\sum_{g \in G}  g\cdot B_{e} \colon (C,\rho,\mu)\to (C',\rho',\mu')
\end{equation}  
in      $ \bCgtsmc(X)$.
 Then $\Theta(A)=B$.
 The formula \eqref{qewfqowifheiofewfqewfewf} defines an inverse of $\Theta$ on the level of morphisms.
  \end{proof}

 In $ \Idem(\Res^{G}_{H}(   {\bC}^{(G)}_{\std})\rtimes H)$
 we consider the full subcategory $\bD$  of objects of the form 
 $(C,\rho,(\mu(H ),e))$, where $(C,\rho,\mu)$ is in $\tCglf(X)$. 
 We let furthermore
 $\bD'$ be the {full} subcategory {of  $ \Idem(\Res^{G}_{H}( {\bC}^{(G)}_{\std})\rtimes H)$} on objects of the form $(C,\rho,(\mu(Z),e))$,
 where $(C,\rho,\mu)$ is in $ \tCglf(Y)$ for some free $G$-set $Y$ and $Z$ is a $H$-invariant subset  of $Y$.  By $\Lambda$ we denote the
 canonical inclusion of $\bD$ into $\bD'$.  {Below, the idempotent completions  of   $\bD$ and $\bD'$ are  formed relative to the full subcategories of $ \Idem(\Res^{G}_{H}(  \bM\bC^{(G)}_{\std})\rtimes H)$ on objects from $\bD$ or $\bD'$, respectively.}
 Then we have the following diagram
 \begin{equation}\label{oijiojiovawvasfvvsdvdsva}
\xymatrix{\ar@/^1cm/@{..>}[rrrr]^{\Theta} \bCgtsmc(X) \ar[r]^{\Phi}&\Idem(\bD)\ar@{..>}[rr]^{\Idem(\Lambda)} && \Idem(\bD')\ar[r]^-{\Delta} & \Idem(\Res^{G}_{H}(  {\bC}^{(G)}_{\std})\rtimes H)\, ,\\
&\bD\ar[u]^{\Xi}\ar[rr]^{\Lambda} && \bD'\ar[u]^{\Psi}&}
\end{equation}
where $\Delta$ is again  the canonical inclusion.
The upper line is then a factorization of $\Theta$ as indicated.

 In the following we will show that  all {solid}  morphisms in \eqref{oijiojiovawvasfvvsdvdsva} induce equivalences after applying $\Kcat$.
  It is clear that this implies  that   $\Kcat(\Theta) $ is  an equivalence.  
 To this end
we   use that $\Kcat$ sends   {unitary equivalences, Morita equivalences, relative  idempotent completions, and weak Morita equivalences (see \cite[Sec.~16--18]{cank})} to equivalences. 
In the following lemmas we argue case by case that all  {solid} arrows  {in the above diagram} have one of these properties. 

{Recall the notion of a relative  idempotent completion \cite[Def. 17.5]{cank}.}
\begin{lem}\label{wotihgjogregwgregw}  {$\Xi$ and $\Psi$   are relative idempotent completions.} 
\end{lem}
\begin{proof} {This is true by construction.}
\end{proof}{
Therefore $\Kcat(\Xi)$ and $\Kcat(\Psi)$ are  equivalences by  \cite[Prop.\ 17.4]{cank}.}

\begin{lem}\label{qergpoqjekrgpgqgre}
 $\Delta$  
is a  {unitary} equivalence in the sense of \cite[Def. 3.19]{cank}.
 \end{lem}
\begin{proof}{
It suffices to show the claim that  every object of $\Idem(\bD')$  admits a unitary   isomorphism to an object of
$\Idem(\Res^{G}_{H}(\bC^{(G)}_{\std})\rtimes H)$ in $\Idem(\Res^{G}_{H}(\bM\bC^{(G)}_{\std})\rtimes H)$. 
Since 
 $\bD'$ in particular contains  all objects of the form 
$(C,\rho,(\mu(Y),e))$ for all   free $G$-sets $Y$ and all $(C,\rho,\mu)$  in $ \tCglf(Y)$,
every object of $  \Res^{G}_{H}(\bC^{(G)}_{\std})\rtimes H$ is unitarily isomorphic   in $\Res^{G}_{H}(\bM\bC^{(G)}_{\std})\rtimes H$ to an object of
$\bD'$. 
This implies the claim by going over to the relative idempotent completions.}
%
%
%
%
\end{proof}
{Since $\Kcat$ is a homological functor by \cite[Thm.\ 14.4]{cank} the morphism $\Kcat(\Delta)$ is an equivalence.}
%

\begin{lem}\label{wetiogwtgwgreregwr}
$\Phi$  is a Morita equivalence.
\end{lem}
\begin{proof}  
The functor 
{$
\Idem(\Lambda)$ is fully faithful by construction.} Since 
 $ \Theta$ is fully faithful by Lemma \ref{qwroiqrgegwregrgwregwg}  {and $\Delta$ is also fully faithful},  we can conclude  that $\Phi$ is fully faithful, too.

Let $(C,\rho,\mu)$ be an object of $\bCgtsmc(X)$. Then we define 
  $$U:=\frac{1}{\sqrt{|H|}} \sum_{h\in H}(\mu({\{h\}})  ,h)$$ in $\End_{\bD}( 
  {(C,\rho,(\mu(H) ,e)))}$.
We calculate that 
$$UU^{*}= (\mu( {\{e \}}),e)\, , \quad U^{*}U=\pi\, ,$$
{where $\pi$ is as in \eqref{erergwgerw4r43r3}.}
This calculation shows that the projection 
$\pi$ is {MvN}-equivalent to $(\mu({\{e\}}),e)$.
For $h$ in $H$ we consider the unitary $V_{h}:=(\mu(H ),h^{-1})$ in $\End_{\bD}((C,\rho,{({\mu}(H),e))})$.
 Then  $$V_{h}(\mu(\{e\}),e) V_{h}^{*}=(\mu(\{h\}),e)\, .$$
So the projection $(\mu({\{h\}}),e)$ is also MvN-equivalent to $\pi$ for every $h$ in $H$.
Since the projections
$((\mu(\{h\}),e))_{h\in H}$  are mutually orthogonal and
$\sum_{h\in H} (\mu(\{h\}),e)=(\mu(H),e)$
we see that any object of $\bD$ is an orthogonal summand  of a finite orthogonal  sum of objects in the essential image of $\Phi$.  This implies that also every object of $\Idem(\bD)$ is an orthogonal summand of a   finite orthogonal  sum of objects in the essential image of $\Phi$. 
\end{proof}

Since $\Kcat$ is Morita invariant by \cite[Thm.~16.18]{cank} the morphism $\Kcat(\Phi)$ is an equivalence.


\begin{lem}\label{ergoijreogwergegwgwrgwr} 
  {$\Lambda$ is a weak Morita equivalence.}
\end{lem}
\begin{proof}
{The functor $\Lambda$ is fully faithful by definition. 
{Furthermore, $\bD$ is unital since the identity on an object  $(C,\rho,(\mu(H ),e))$ of $\bD$ is   given by $(\mu(H),e)$ and $\mu(H)$ is in $\bC$.} 
It remains to  show} that the set of objects of $\bD$ is weakly generating in $\bD'$, see \cite[Def.\  {18.1}]{cank}.

Let $(C,\rho,(\mu(Z),e))$ be any object of $\bD'$, where $(C,\rho,\mu)$ is in $   \tCglf(Y)$ for some free $G$-set $Y$ and $Z$ is a $H$-invariant subset of $Y$.
 Let $y$ be a point in $Y$.
Then we can form the object $(C,\rho,(\mu(Hy),e))$ in $\bD'$.
 We claim that this object is isomorphic to an object in $\bD$. 
 We consider the  $G$-equivariant   injection $i \colon X\to Y$ which sends $(H,e)$ to $y$. 
 We choose an image $u \colon C'\to C$ {in $\bM\bC$} of the projection $\mu(Gy)$. 
 Then we define 
 $(C',\rho',\mu')$ in $\bCgtsmc(X)$ {by setting}
 {$\rho'_{g}=gu^{*}\rho_{g}u$ for every $g$ in $G$} and $\mu'(W)=u^{*}\mu(i(W)) u$ for every subset $W$ of $X$.
  Then we have an isomorphism
 $$(u,e) \colon (C',\rho',(\mu'(H),e))\to (C,\rho,(\mu(Hy),e))$$
 in  $\bD'$. 
 More generally, if $Z$ is any finite $H$-invariant subset of $Y$ (note that $H$ is finite), then $(C,\rho,(\mu(Z),e))$ is isomorphic to a finite sum of objects in $\bD$.

Let now $(A_{j})_{j\in J}$ with $A_{j} \colon (C_{j},\rho_{j},p_{j})\to (C,\rho,p) 
 $ be a finite family of morphisms in $\bD'$.
 
Let  $\epsilon$ in $(0,\infty)$ be given.
 Since 
 $C$ is  isomorphic to the  AV-sum {in $\bC$} of the family of 
 projections $(\mu(S))_{S\in H\backslash Y}$ {the sum
 $\sum_{S\in H\backslash Y} \mu(S)$ converges strictly  in $\bM\bC$ to $\id_{C}$.}
 Since the morphisms $A_{j}$ belong to $\bC$ 
  there exists a finite $H$-invariant subset $Z$ of
  $Y$ such that $$\|A_{j}- (\mu(Z),e)A_{j}\|<\epsilon$$ for all $j$ in $J$.   \end{proof}

{By \cite[Thm.\  {18.6}]{cank} the morphism $\Kcat(\Lambda)$ is an equivalence.}

{Applying $\Kcat$ to the diagram in \eqref{oijiojiovawvasfvvsdvdsva} and combining the results above we conclude 
the proof of Proposition \ref{4oijotrherhrthrhe}.} 
\end{proof}
{Therefore  the proofs of the Theorems \ref{qreoigjoergegqrgqerqfewf}.\ref{qregiojqwfewfqwfqewf} and  \ref{qreoigjoergegqrgqerqfewf}.\ref{wrtigjiogowrefwerfewrf} are   also complete.}\phantomsection\label{wefgiojwoegergregegw}

\section{Calculation of the domain and target of the Paschke transformation}\label{weogiwjiergreggwrgwerg}

\newcommand{\pcc}{\mathrm{pcc}}

The domain of the Paschke transformation is  the functor 
 $$
 K^{G,\cX}_{\bC}\colon G\UBC\to \Sp\, .$$
 {The   first goal  of} this section is to describe its values  on sufficiently  nice spaces in terms
 of the equivariant  homology theory $$K\bC^{G} \colon G\Orb\to \Sp$$ introduced in    \eqref{asdvijqoirgfvfsdvsva}, see Definition \ref{qroeigjoqergerqfewewfqewfeqwf} below for the technical description. {Our final result is stated in Proposition \ref{weroigjwerogregrgwffe}.}
 
{In order to understand why the construction of the comparison map in Proposition \ref{weroigjwerogregrgwffe} is  difficult, note that
 on the one hand for $X$ in $G\UBC$ the spectrum $ K^{G,\cX}_{\bC}(X)$ is defined as the $K$-theory of an explicitly constructed $C^{*}$-category associated to $X$ and the coefficient category $\bC$. On the other  hand the spectrum $K\bC^{G}(X)$ is the value on the underlying $G$-topological space of $X$ of the equivariant homology theory given by a spectrum-valued functor $K\bC^{G}$ on the orbit category $G\Orb$ of $G$ determined by $\bC$. The construction of a natural  map between $K^{G,\cX}_{\bC}(X)$ and
$K\bC^{G}(X)$ will involve a classification of functors with certain homological properties on subcategories of $G\Top$. This classification    is related to 
  {Elmendorf's} theorem 
and the techniques behind it.}

{The second theme of the present section is the calculation of the domain and target of the Paschke transformation. Our main example of a coefficient category is $\bC =\Hilb_{c}(A)$ for a $C^{*}$-algebra $A$ with an action of $G$.  If $A$ is unital, then  one can express the values of the functors 
$K^{G,\cX}_{\bC}$ on $G$-orbits  and of $K^{G,\An}_{\bC}$  on sufficiently nice spaces directly in terms of constructions with the algebra $A$. The results are stated as Corollary \ref{wtiejgowegwergwrefff} and Propositions \ref{wegojeogrregregwergwegre} and \ref{werigowergrefwerfwe}.}

  
  We start with the statement of Elmendorf's theorem.
Let $\bM$ be a cocomplete stable $\infty$-category.  
In the present paper we adopt the following simple definition which in some sense reverses the history of this notion. \begin{ddd}\label{weigjorgdg}
An equivariant $\bM$-valued homology theory  is  a functor 
 $$E \colon G\Orb\to \bM\, .$$ 
 \end{ddd}

 Recall that a weak equivalence between topological spaces is a  continuous map which induces a bijection between the sets of connected components and  isomorphisms between the higher homotopy groups at all base points. We have  a functor
 \begin{equation}\label{adsfasdfasdfadsf}
\ell \colon \Top\to \Spc
\end{equation}
   which presents $\Spc$ as the localization of $\Top$ at the weak equivalences. 
We now consider the functor 
\begin{equation}\label{wefqwefweqdxasdc}
Y^{G} \colon G\Top\to \PSh(G\Orb)
\end{equation}
 which sends  
$X$ in $ G\Top$ to the presheaf $$S\mapsto \ell(\Map_{ G\Top}(S_{disc},X))\, ,$$
where $\Map_{ G\Top}(S_{disc},X)$  in $\Top$ is the topological mapping space of equivariant maps. 
By definition, a map $f:X\to Y$  between $G$-topological spaces is an equivariant weak equivalence if it induces weak equivalences $\Map_{G\Top}(S_{disc},X)\to \Map_{G\Top}(S_{disc},Y)$  for  all $S$ in $G\Orb$.
 \begin{theorem}[Elmendorf's theorem] \label{weiogjodfbsefsffgsfgdfg}The functor $Y^{G}$ presents $ \PSh(G\Orb)$ as the Dwyer-Kan  localization  of $ G\Top $ at the  equivariant weak equivalences. \end{theorem}

    By the universal property of  presheaves, the pull-back along the  Yoneda embedding   $\yo \colon G\Orb\to \PSh(G\Orb)$ induces  an equivalence 
  $$\yo^{*} \colon \Fun^{\colim}(\PSh(G\Orb),\bM)\stackrel{\simeq}{\to} \Fun(G\Orb,\bM)\, .$$
  Let $E \colon G\Orb\to \bM$ be an equivariant homology theory. Its
 colimit preserving extension  
  to presheaves     is the  left Kan-extension  $\yo_{!}E \colon \PSh(G\Orb)\to \bM$  of $E$ along $\yo$.
    \begin{ddd} \label{weiogwegerffs}
    The evaluation of $E$ on $G$-topological spaces is  defined as composition
 (which we will again denote by $E$)
 \begin{equation}\label{adfiuhuihavfvasdcacsadcasca}
E \colon G\Top  \stackrel{Y^{G}}{\to} \PSh(G\Orb)\stackrel{\yo_{!}E}{\to} \bM\, .
\end{equation}
\end{ddd}
 If  $S$ is in $G\Orb$, then the value
 of the original functor  $E$ on $S$ and the evaluation of $E$ on the discrete  $G$-space $S_{disc}$ coincide so that there is no conflict of notation.
The value of the   equivariant homology theory  on a general  space $X$ in $   G\Top$ is given by the coend
\begin{equation}\label{qwefpojfopwefqwefqwefqew}
E(X) \coloneqq \int^{G\Orb}  E \wedge  \Sigma^{\infty}_{+} Y^{G}(X)\, ,
\end{equation}
where  
 $\wedge \colon \bM\times \Sp\to \bM$ is the tensor structure of $\bM$ (the same as \eqref{sdvjqr0evj0fvsfvvfv})  which exists by the cocompleteness and stability assumptions on $\bM$.

 We let $G\UBC^{\pcc}$ be the full subcategory of 
$G\UBC$ of $G$-uniform bornological coarse spaces which have the following properties:
\begin{enumerate} 
\item the underlying topological space is Hausdorff,
\item  the bornology is generated by relatively compact subsets,
\item the coarse structure is generated by all entourages of the form
$G(K\times K)$, where $K$ is a relatively compact connected subset,
\item $G$ acts properly and cocompactly.
\end{enumerate}
The category $G\UBC^{\pcc}$ contains all $G$-finite $G$-simplicial complexes with finite stabilizers with the structures induced by the spherical path metric. 
   We consider the functor  
$ \iota \colon G\UBC\to G\Top$ which takes the underlying $G$-topological space. 
\begin{lem}\label{hfegwiugheifofdsfd} The restriction
$\iota_{|G\UBC^{\pcc}} \colon   G\UBC^{\pcc} \to G\Top$ is fully faithful.  
\end{lem}
  \begin{proof}
  It is clear that $\iota_{|G\UBC^{\pcc}}$ is faithful. We must show that it is full.
  Let $X,Y$ be in $ G\UBC^{\pcc}$ and $f \colon X\to Y$ be an equivariant continuous map.
  We must show that it is controlled, uniformly continuous and proper.
  
  We first show 
  that $f$ is proper.
{Let $K$ be a  relatively compact subset of $Y$ and let $(x_{\alpha})_{\alpha}$ be a net in $f^{-1}(K)$.}
 Since $K$ is relatively compact, and $G\backslash X$ is compact,  we can assume by taking a subnet that
 $(f(x_{\alpha}))_{\alpha}$ and  $([x_{\alpha}])_{\alpha}$ converge in $Y$ and $G\backslash X$, respectively.
 By the {latter} there exists a family $(g_{\alpha})_{\alpha}$ in $G$ such that $(g_{\alpha}x_{\alpha})_{\alpha}$ converges. 
 Since then $(g_{\alpha}f(x_{\alpha}))_{\alpha}$ also converges  and $G$ acts properly on $Y$ we can  assume after taking a subnet that
 $(g_{\alpha})_{\alpha}$ is constant. 
 {But this means that  $(x_{\alpha})_{\alpha}$ has a subnet converging in $X$, which shows that $f^{-1}(K)$ is relatively compact.}
  
 We claim that any invariant open entourage of the diagonal of $X$ is uniform. 
 The claim 
 implies that $f$ is uniformy continuous: Indeed, if $V$ is any uniform entourage of $Y$, then by the axioms for a $G$-uniform structure there exists an invariant  uniform entourage $V'$ of $Y$ such that $V'\subseteq V$. But then $(f\times f)^{-1}(V')$ is invariant and open, hence a uniform entourage of $X$ by the claim. The relation $(f\times f)^{-1}(V')\subseteq (f\times f)^{-1}(V)$ implies that $ (f\times f)^{-1}(V)$ is uniform.

 We now show the claim. Assume by contradiction that $U$ is not uniform.
 Then for every invariant uniform entourage $V$ of $X$ there exists 
 $(x_{V},y_{V})$ in $V\setminus U$. By compactness of the quotient we can assume, after taking a cofinal subnet $(V_{\alpha})_{\alpha}$ of uniform entourages,  that 
 $[x_{V_{\alpha}}]\to [x]$ and $[y_{V_{\alpha}}]\to [y]$.  We can find a net $(g_{ \alpha})_{\alpha}$ in $G$ such that $g_{\alpha}x_{V_{\alpha}}\to x$ in $X$.
 But then  also $g_{\alpha}y_{V_{\alpha}}\to x$ since $X$ is Hausdorff and the  net $(V_{\alpha})_{\alpha}$ of uniform entourages is cofinal. Since $U$ is $G$-invariant we have
 $(g_{\alpha}x_{V_{\alpha}},g_{\alpha}y_{V_{\alpha}})\not\in U$ for all $\alpha$, and since $U$ is open we conclude that also $(x,x)\not\in U$.
 But this is impossible since $U$ was an open neighbourhood of the diagonal.  
   
 We check on generators that $f$ is controlled. Let $K$ be a  relatively compact    connected subset of $X$ and consider the generator 
 $G(K\times K)$ of the coarse structure of $X$. Then $f(K)$ is relatively compact and connected, too.
  Therefore
 $(f\times f)G(K\times K) =G(f(K)\times f(K))$ is a coarse entourage of $Y$. 
 \end{proof}

Recall that  {$K^{G,\cX}_{\bC}$} 
is 
defined on $G\UBC$. By the Lemma \ref{hfegwiugheifofdsfd} 
 we  can restrict {$K^{G,\cX}_{\bC} $} 
  to a functor defined on the full subcategory $G\UBC^{\pcc}$ of $G\Top$. In contrast, the equivariant homology theory   $K\bC^{G}$ gives rise to a 
  functor defined on all of $G\Top$ by Definition \ref{weiogwegerffs}.
 Therefore, as a preparation we present a  general result which helps to compare a functor with homological  properties defined on some full subcategory of  $G\Top$ with  an associated equivariant homology theory.

Let $\bV$ be a simplicial model category  with weak equivalences  $W$, {homotopy equivalences $W_h$}, and
with functorial  factorizations. 
 The associated $\infty$-category of  $\bV$ is defined by $\bV_{\infty}:=\bV[W^{-1}]$. We let $\ell \colon \bV\to \bV_{\infty}$ denote the canonical functor.
We furthermore  let $\bV^{\cf}$ denote the full subcategory of cofibrant/fibrant objects in $\bV$. {The following lemma is of course well-known, but for lack of reference, we include a proof here.}

\begin{lem}
The inclusion $\bV^\cf \to \bV$ induces an equivalence of Dwyer--Kan localizations $\bV^\cf[W_h^{-1}] \simeq \bV[W^{-1}]$.
\end{lem}
\begin{proof}
We consider  the following square 
$$\xymatrix{\bV^{\cf}\ar[r]^-{\ell_{h}}\ar[d]&\bV^{\cf}[W^{-1}_{h}]\ar@{..>}[d]\\ \bV\ar[r]^-{\ell}&\bV_{\infty}}$$
where the dotted arrow is obtained from the universal property of the localization $\ell_{h}$. {We claim that it is} an equivalence as desired. In order to produce an inverse 
 we consider the square 
$$\xymatrix{\bV\ar[r]^-{\ell}\ar[d]^{RL}&\bV_{\infty}  \ar@{..>}[d]\\ \bV^{\cf}\ar[r]^-{\ell_{h}}& \bV^{\cf}[W^{-1}_{h}]}$$
where $RL$ is the fibrant-cofibrant replacement functor. 
The dotted arrow is obtained from the universal property
of $\ell$ since $RL$ sends weak equivalences to homotopy equivalences.   
We have a diagram $$\xymatrix{ &L\ar[dr]\ar[dl]& \\ \id &&RL}$$ of endofunctors of $\bV$, where $L$ and $R$ are the fibrant and cofibrant replacement functors. It is 
  sent by $\ell$ to a diagram of equivalences. 
Similarly, $\ell_{h}$ sends  the restriction of this  diagram to $\bV^{\cf}$ to a diagram of equivalences.
From this we can conclude that the two dotted arrows are inverse to each other.
\end{proof}

Let $E \colon \bV\to \bM$ be a homotopy invariant functor.

\begin{lem}\label{ewiorgjowergwrewfre}
There exists a functor
$E^{\infty}\colon\bV_{\infty}\to \bM$ such that the following square commutes:
$$\xymatrix{\bV^{\cf}\ar[r]^{E_{|\bV^{\cf}}}\ar[d]&\bM\\ \bV \ar[r]&\bV_{\infty}\ar[u]_{E^{\infty}}}$$
\end{lem}
\begin{proof}
{We obtain} the desired square from 
$$\xymatrix{\bV^{\cf}\ar[r]^-{\ell_{h}}\ar[d]\ar@/^1cm/[rr]^{E_{|\bV^{\cf}}}&\bV^{\cf}[W^{-1}_{h}]\ar@{..>}[r]\ar[d]^{\simeq}&\bM\ar@{=}[d]\\ \bV\ar[r]^{\ell}&\bV_{\infty}\ar[r]^{E^{\infty}}&\bM}$$
where the dotted arrow exists since $E$ sends homotopy equivalences to equivalences.
\end{proof}

We consider  some full subcategory  $\bW$      of  $G\Top$  and let 
$E \colon \bW\to \bM$ be some functor. 
We assume that $\cF$ is a family of subgroups of $G$
and that $G_{\cF}\Orb\subseteq \bW$.  
We then define the equivariant homology theory  $E^{\%,\cF} \colon G\Orb\to \bM$ as the left Kan extension of the functor  $E_{|G_{\cF}\Orb}$ along $i_{\cF}$:
 $$\xymatrix{G_{\cF}\Orb\ar[rr]^{E_{|G_{\cF}\Orb}}\ar[d]_{i_{\cF}} &&\bM\,. \\G\Orb \ar[urr]_{E^{\%,\cF} }&&  }$$
Following Definition \ref{weiogwegerffs} we will consider
$E^{\%,\cF}$ also as a functor $E^{\%,\cF}\colon G\Top\to \bM$.

We now use that $G\Top$ admits a simplicial model category structure
with the  weak  equivalences  as described {after \eqref{wefqwefweqdxasdc}} and such that the notion of homotopy is the usual one. By Theorem \ref{weiogjodfbsefsffgsfgdfg} the functor
$Y^{G} \colon G\Top\to \PSh(G\Orb)$ is equivalent to the functor $G\Top\to G\Top_{\infty}$
 in the notation introduced before Lemma \ref{ewiorgjowergwrewfre}.  {Let $j \colon \bW\to G\Top$ denote the inclusion.}
 \begin{lem} \label{sogipegrggwege}Assume:
\begin{enumerate}
\item $G_{\cF}\Orb\subseteq \bW\subseteq G\Top^{\cf}$
\item  $\bW$ is closed under taking the product with $[0,1]$.
\item  
$E$ is homotopy invariant. 
\end{enumerate}
Then we have a canonical natural transformation of functors
$$j^{*}E^{\%,\cF} \to E \colon \bW\to \bM\, .$$
\end{lem}
\begin{proof}
{Since $j$ is fully faithful,
 we have an equivalence $E\stackrel{\simeq}{\to} j^{*}j_{!}E$.}  We claim that $j_{!}E$ is homotopy invariant.  Let 
$X$ be in $G\Top$. Then we must show that $(j_{!}E)([0,1]\times X)\to (j_{!}E)(X)$ is an equivalence.
We use the point-wise formula for the left Kan extension in order to rewrite this map as
\begin{equation}\label{vfdvdfvfdvasfvsfv}
\colim_{(Y\to [0,1]\times X)\in \bW_{/[0,1]\times X}} E(Y)\to \colim_{(Z\to   X)\in \bW_{/  X}}E(Z)\, . 
\end{equation} 
We now observe that the maps of the form
$[0,1]\times Z\to [0,1]\times X$ for maps $Z\to X$ are cofinal in the index category of the left colimit. At this point we use that $\bW$ is closed under taking products with an interval.
Indeed, let $(a,b) \colon Y\to [0,1]\times X$ be a map. Then we consider the factorization
$$Y\stackrel{(a,\id_{Y})}{\to}[0,1]\times Y\stackrel{(\id_{[0,1]},b)}{\to} [0,1]\times X\, .$$
Consequently, the morphism  in \eqref{vfdvdfvfdvasfvsfv} is equivalent to 
$$\colim_{(Z\to X)\in \bW_{/  X}} E([0,1]\times Z)\to \colim_{(Z\to   X)\in \bW_{/  X}}E(Z)\, .$$
This map is an equivalence since $E$ is homotopy invariant. This finishes the proof of the claim.

 By Lemma \ref{ewiorgjowergwrewfre} we get a functor
$(j_{!}E)^{\infty}\colon \PSh(G\Orb)\to \bM$ fitting into the commutative square in \begin{equation}\label{gfdndfngfbdbdgb} \xymatrix{ &&& \bW\ar[ddl]^(0.65){j} \ar[rrd]^{E}\ar[dl]&& \\G_{\cF}\Orb\ar[d]^{i_{\cF}}\ar[urrr] &&G\Top^{\cf}\ar[d]\ar[rrr]^{(j_{!}E)_{|G\Top^{\cf}}}&&&\bM\\ G\Orb\ar[rr]\ar@/_1cm/[rrrrr]_{\yo} &&G\Top \ar[rrr]^{Y^{G}} &&& \PSh(G\Orb) \ar[u]^{(j_{!}E)^{\infty}} }
\end{equation}
Here the   triangle involving $(j_{!}E)_{|G\Top^{\cf}}$  commutes since $j^{*}j_{!}E\simeq E$ as observed already above.
The commutative diagram provides an equivalence $  E_{|G_{\cF}\Orb}   \simeq    i_{\cF}^{*}\yo^{*}(j_{!}E)^{\infty}$.
Applying the left Kan extension $\yo_{!} i_{\cF,!} $ we get an equivalence 
 $$\yo_{!} E^{\%,\cF} \simeq  \yo_{!}  i_{\cF,!} E_{|G_{\cF}\Orb}    \simeq  \yo_{!}   i_{\cF,!} i_{\cF}^{*}\yo^{*}(j_{!}E)^{\infty}   \, .$$
The counit $ \yo_{!}   i_{\cF,!} i_{\cF}^{*}\yo^{*}\to \id   $ then yields the transformation 
 $\yo_{!}E^{\%,\cF}\to (j_{!}E)^{\infty}$.
 We finally apply $j^{*} (Y^{G})^{*}$ and get the desired transformation 
 $$j^{*}E^{\%,\cF}\to j^{*} (Y^{G})^{*} (j_{!}E)^{\infty}\simeq E \colon \bW\to \bM\, ,$$
where the second equivalence follows form the commutativity of  {a}  part {of} the diagram \eqref{gfdndfngfbdbdgb} above.
\end{proof}

 Recall that $\bW$  is a full subcategory of $G\Top$ and that
 $E \colon \bW\to \bM$ is some functor. We call $E$ reduced if $E(\emptyset)\simeq 0$.
We let $ \bW_{\cF}^{\hfin}$ denote the  full subcategory of $\bW$ of spaces which are homotopy equivalent to 
a  $G$-finite $G$-CW complex with stabilizers in $\cF$.  
\begin{prop}\label{wefgbkjowgefwevffvsvfdvsfd}
Assume:
\begin{enumerate}
\item    $\bW\subseteq G\Top^{\cf}$  and $\bW$ contains all $G$-finite $CW$-complexes with stabilizers in $\cF$.
\item
$\bW$ is closed under taking the product with $[0,1]$.
\item $E$ is reduced, homotopy invariant, and excisive for  cell attachments.
\end{enumerate}
Then the natural transformation from Lemma \ref{sogipegrggwege}   induces an equivalence
$$(j^{*}E^{\%,\cF})_{| \bW_{\cF}^{\hfin}}  \to E_{| \bW_{\cF}^{\hfin}}\, .$$
\end{prop}
\begin{proof}   We note that $j^{*}E^{\%,\cF} \colon \bW\to \bM$ is reduced, homotopy invariant, and excisive for cell attachments.

We must show that $E^{\%,\cF}(X)\to E(X)$ is an equivalence for all $X$ in  $\bW_{\cF}^{\hfin} $.   Since $j^{*}E^{\%,\cF}$ and $E$ are homotopy invariant we can assume that $X$ is  a $G$-finite $CW$-complex with stabilizers in $\cF$.

 We  then argue by induction by the number of cells. 
The assertion is clear for the empty $G$-CW-complex since both functors are reduced.
 Assume now that the assertion is true for  the $G$-CW-complex $Y$, and that $X$ is obtained from $Y$ by a cell-attachement. Then we have a push-out diagram
 $$\xymatrix{G/K\times S^{n}\ar[r]\ar[d]&Y\ar[d]\\G/K\times D^{n+1}\ar[r]&X}$$
 where $n$ is in $\nat$ and $K$ is a subgroup of $G$ belonging to $\cF$.
 The natural transformation induces a map of push-out diagrams
 $$\xymatrix{E^{\%,\cF}(G/K\times S^{n})\ar[r]\ar[d]&E^{\%,\cF}(Y)\ar[d]\\E^{\%,\cF}(G/K\times D^{n+1})\ar[r]&E^{\%,\cF}(X)}\qquad \to\qquad  \xymatrix{E(G/K\times S^{n})\ar[r]\ar[d]&E(Y)\ar[d]\\E(G/K\times D^{n+1})\ar[r]&E(X)}$$
 which is implemented by equivalences  at the two upper and the lower left corners by  the induction hypothesis.
 We conclude that
 $E^{\%,\cF}(X)\stackrel{\simeq}{\to} E(X)$.
 \end{proof}

We now consider two functors $E,F \colon \bW\to \bM$ and assume that we are given  an equivalence
$$\phi \colon E_{|G_{\cF}\Orb}\to F_{|G_{\cF}\Orb}\, .$$
 \begin{kor}\label{wegtopwgwgregergw}
Assume:
\begin{enumerate}
\item    $\bW\subseteq G\Top^{\cf}$  and $\bW$ contains all $G$-finite $CW$-complexes with stabilizers in $\cF$.
 \item
$\bW$ is closed under taking the product with $[0,1]$.
\item $E$ and $F$ are reduced, homotopy invariant, and excisive for  cell attachments.
\end{enumerate}
Then 
$\phi$ extends to an equivalence
$$\tilde \phi \colon E_{|\bW_{\cF}^{\hfin}}\to F_{ |\bW_{\cF}^{\hfin}}\, .$$
\end{kor}
\begin{proof}
The equivalence $\phi$ induces an equivalence
$\tilde \phi \colon E^{\%,\cF}\stackrel{\simeq}{\to} F^{\%,\cF}$.
The desired equivalence is now given by the composition 
$$ E_{|\bW_{\cF}^{\hfin}}\stackrel{\simeq}{\leftarrow}(j^{*}E^{\%,\cF})_{|\bW_{\cF}^{\hfin}}\stackrel{\tilde \phi,\simeq}{\to}
(j^{*}F^{\%,\cF})_{|\bW_{\cF}^{\hfin}} \stackrel{\simeq}{\to}F_{|\bW_{\cF}^{\hfin}}$$
where the outer equivalences are supplied by Proposition \ref{wefgbkjowgefwevffvsvfdvsfd}
\end{proof}

  We let $G\UBC^{\pcc,\hfin}$ be the full subcategory of $G\UBC^{\pcc}\cap G\Top^{\cf}$ of $G$-spaces which are homotopy equivalent to $G$-finite $G$-CW complexes with stabilizers in $\Fin$.
{We consider  $\bC$  in $  \Fun(BG,\nCcat)$  which is  effectively additive and admits countable  AV-sums.}
  
 \begin{prop} \label{weroigjwerogregrgwffe}We have an equivalence 
$$K^{G,\cX}_{\bC}(-)_{|G\UBC^{\pcc,\hfin}}\simeq \Sigma K\bC^{G}(-)_{|G\UBC^{\pcc,\hfin}}\,  .$$
\end{prop}
\begin{proof}
We start with the equivalence
$$K^{G,\cX}_{\bC}(S_{min,min,disc}) \stackrel{\textrm{def}}{=}  K\bC\cX^{G}_{G_{can,max}}(\cO^{\infty}(S_{min,min,disc})) \simeq \Sigma  K\bC\cX^{G}_{G_{can,max}}( S_{min,min})\, ,$$ where the second equivalence is given by the cone boundary  \cite[Prop.\ 9.35]{equicoarse}.
  For every $S$ in $G_{\Fin}\Orb$ the sets of invariant  locally finite subsets 
$\LF(S_{min,max}\otimes G_{can,min})$
and
$\LF(S_{min,min}\otimes G_{can,max})$ are equal.  
 Using that  $  K\bC\cX^{G}$ is a continuous equivariant coarse homology theory we get the  
middle equivalence in 
\begin{align*}
\mathclap{
K\bC\cX^{G}_{G_{can,max}}( (-)_{min,min})\stackrel{\textrm{def}}{=} K\bC\cX^{G}((-)_{min,min}\otimes G_{can,max})
\simeq  K\bC\cX^{G}((-)_{min,max}\otimes G_{can,min}))\stackrel{\textrm{def}}{=}  K\bC^{G}(-)
}
\end{align*}
of functors
on $G_{\Fin}\Orb$.
We now apply Corollary \ref{wegtopwgwgregergw}  with $\bW= G\UBC^{\pcc}\cap G\Top^{\cf} $, $\cF=\Fin$,
$E= K^{G,\cX}_{\bC}(-)$ and $F=\Sigma K\bC^{G}(-)$ in order to get the desired equivalence.
 \end{proof}

Using Proposition \ref{weroigjwerogregrgwffe} we can express  the domain of the Paschke transformation in terms of the equivariant homology theory $K\bC^{G}$.
In the following we describe the values of this functor on $G$-orbits in some detail.
 We use remark environments in order to be able to refer to this discussion later.

\begin{rem}\label{eroigjeogiegregegegweg}
{We assume that $\bC$ in $\Fun(BG,\nCcat)$ is effectively additive.}
By \cite[Prop.\ 8.2.{3}]{coarsek} we  have an explicit description of the values of the functor $K\bC^{G}$ on $G$-orbits $S$: \begin{equation}\label{qrforfofwqewfqefqefe}
K\bC^{G}(S)\simeq \Kcat(\btCtsmc(S_{min,max})\rtimes_{r} G)\, .
\end{equation}
Here  $\btCtsmc(S_{min,max})$ 
in $\Fun(BG,\nCcat)$ is the $C^{*}$-category $\bCtsmc(S_{min,max})$
with the $G$-action induced by functoriality by the actions of $G$ on $S$ and $ \bC$,
and $-\rtimes_{r} G$ is the reduced crossed product for $G$-$C^{*}$-categories introduced in {\cite[Thm.\ 12.1]{cank}.}
Note that the objects of $\bCtsmc(S_{min,max})$ are objects of $\bC$ which are decomposed   as    AV-sums of $S$-indexed families of objects of ${\bC}^{u}$ with finitely many non-zero terms, and morphisms are morphisms in ${\bM\bC}$ which are diagonal with respect to this decomposition. 
We note that \eqref{qrforfofwqewfqefqefe} implies that the functor $K\bC^{G}$ is the functor
defined in \cite[Def.\ {19.12}]{cank} {for $\Homol=\Kcat$ and denoted there  by $(\Kcat)^{G}_{{\bC^{u}},r}$.}

The right-hand  side of the equivalence in \eqref{qrforfofwqewfqefqefe} reflects the functorial dependence on $S$ in an obvious manner.  If one is not interested in functoriality, then  one can give even simpler formulas.  For a   subgroup
 $H$ of $G$ we have the equivalence
$$ K\bC^{G}(G/H)\stackrel{\eqref{qrforfofwqewfqefqefe}}{\simeq}  \Kcat(\bCtsmc((G/H)_{min,max})\rtimes_{r} G)\simeq  \Kcat( {\bC}^{u}\rtimes_{r} H)$$
by using \cite[Cor.\ {19.13}]{cank} and the Morita invariance of $\Kcat$.
  \hB
\end{rem}

\begin{rem}\label{oiejgoigsvfdvfvs}
We continue the calculations from Remark \ref{eroigjeogiegregegegweg} but now specialize further
to the case {$\bC= \Hilb_{c}(A)$}  for an $A$ in $\Fun(BG,\Calg)$.
Since $A$ is unital, {the inclusion $A\to \Hilb_{c}(A)^{u}$ is a Morita equivalence {(combine \cite[Ex.\ 16.9 \&  18.15]{cank})}  and therefore induces by \cite[Prop.\ 16.11]{cank} (stating that $-\rtimes_{r}H$ preserves Morita equivalences) and \cite[Thm.\   {16.18}]{cank} (stating that $\Kcat$ is Morita invariant)
 an equivalence 
$$  \Kast( A\rtimes_{r} H) \stackrel{\simeq}{\to} \Kcat(  \Hilb_{c}(A)^{u}\rtimes_{r} H) \, .  $$}
So in this case
$$K\bC^{G}(G/H)\simeq   \Kast( A\rtimes_{r} H)\, .$$
We see that  the functor $K\bC^{G}$ has the same values as the functor introduced  in  \cite{davis_lueck}  (with additions by \cite{joachimcat}
  or alternatively  by \cite{Land:2017aa})\footnote{To be precise, in   \cite{davis_lueck} only the case $A=\C$ is considered, but the generalization to unital $C^{*}$-algebras with trivial $G$-action is straightforward. 
  The additions concern a correction in the construction of a $K$-theory functor  for $C^{*}$-categories.}.
 If $A$ is unital and {is equipped with} the trivial $G$-action, 
 then  by \cite[Prop.\  {19.18}]{cank}  the functor  $K\bC^{G}$ and the Davis--L\"uck functor are  actually  equivalent as functors.  
\hB
\end{rem}

Using  \eqref{rgpfojewropgegwgwrgrgwegwerg} and Proposition \ref{weroigjwerogregrgwffe} combined with Remark 
\ref{oiejgoigsvfdvfvs} we can describe the values  on the orbit category for the functor   $K^{G,\cX}_{{\Hilb_{c}(A)}}$ appearing in the domain of the Paschke morphism.
{Let $A$ be in $\Fun(BG,\nCalg)$.}
 \begin{kor}\label{wtiejgowegwergwrefff}{ If $A$ is unital, then}
 for every subgroup $H$ of $G$ we have  an equivalence $$K^{G,\cX}_{{\Hilb_{c}(A)}}((G/H)_{min,min,disc})\simeq \left\{\begin{array}{cc}
 0&|H|=\infty\,,
 \\\Sigma  \Kast(A\rtimes_{r}H)&|H|<\infty\,.
 \end{array}\right.$$
 \end{kor}

We now turn {our} attention to the target of   the Paschke morphism.   We {show  {that in the case of  {$ \bC = \Hilb_{c}(A)$} for unital $A$,
  we} can express the functor 
  $$K^{G,\An}_{{\bC}}(-)\stackrel{\eqref{vfdsvsvvsvsdvdsadsvdsva}}{=}\KKG(C_{0}(-),\bQ^{(G)}_{\std})$$
  in terms of the more familiar functor
 $$ K_{A}^{G,\an}(-) \coloneqq \KKG(C_{0}(-),A)
$$
from $\ppGTop$ to $\Sp$, see \cite[Def.\ 1.14]{KKG}. 
 {In order to state the results properly, we introduce the following notation.}  {
\begin{ddd}
\label{werigjweogwergwregwerg}\mbox{}\begin{enumerate} \item  \label{wrijgowegewrgwregwreg}We let  $\ppGTopf$ denote  the full subcategory of $\ppGTop $ on spaces which are  homotopy equivalent in $\ppGTop $  to   $G$-finite $G$-CW complexes with finite stabilizers. \item \label{qewfwefdwdwqedweq1} We let $\ppGTopo$ denote the full subcategory of $\ppGTop$ of second countable spaces with proper $G$-action which are homotopy equivalent in $\ppGTop $ to countable  $G$-CW complexes with  proper $G$-action. \end{enumerate} \end{ddd}}

{Let $A$ be in $\Fun(BG,\nCalg)$.}
 \begin{prop}\label{wegojeogrregregwergwegre}\mbox{} If $A$ is unital, then
 we have  an equivalence of functors  $$(\Sigma K_{A}^{G,\an})_{|\ppGTopf}\simeq  (K_{{\Hilb_{c}(A)}}^{G,\An})_{|\ppGTopf}\, .$$
  \end{prop}
\begin{proof} {We abbreviate $\bC \coloneqq \Hilb_{c}(A)$.}
{Using the notation of \cite[{Def.\ 1.14}]{KKG} we have the equality }  $$K_{\bQ^{(G)}_{\std}}^{G,\an}(-) = K^{G,\An}_{\bC}(-)\, .$$ 

 If $X$ is in $\ppGTopf$, then by 
Lemma \ref{eroigjweogregwgefwe}
 the functor 
$\bB\mapsto K_{\bB}^{G,\an}({X})$ sends exact sequences  in $\Fun(BG,\nCcat)$ to fibre sequences of functors on $\ppGTopf$,
 annihilates flasques, and sends relative  Morita   equivalences to equivalences. By \cite[Thm.\ 1.32.3]{KKG} it also sends weak Morita equivalences to equivalences.

We apply  the exactness property to the exact sequence \begin{equation}\label{QEWFQWEFFQWDWD}
0\to {\bC}^{(G)}_{\std}\to {\bM\bC}^{(G)}_{\std}\stackrel{{\pi}}{\to} \bQ^{(G)}_{\std}\to 0\, .
 \end{equation}
{Since  $  \bC^{(G)}_{\std}$  admits countable AV-sums,  we know by Lemma \ref{wrtijhoerhgrtgertg}  that $\bM\bC^{(G)}_{\std}$  is flasque. Therefore}   $ K_{ {\bM\bC}^{(G)}_{\std}}^{G,\an} (-)\simeq 0$
{and the boundary map of the fibre sequence obtained by applying {$ K^{G,\an}_{-}$}
to 
  \eqref{QEWFQWEFFQWDWD} is  an}
equivalence 
\begin{equation}\label{weogjkopfegwreffsv}
K_{\bC}^{G,\An}(-) =    K_{ \bQ^{(G)}_{\std}}^{G,\an} (-) \xrightarrow{\simeq} \Sigma K_{ {\bC}^{(G)}_{\std}}^{G,\an} (-)
\end{equation} of functors on $\ppGTopf$.
We  consider   the zig-zag 
\begin{equation}\label{aefefdfafadfdffsa}
A\to  (\bC^{u})^{(G)} \to \bC^{(G)}_{\std,+}\leftarrow   \bC^{(G)}_{\std}
\end{equation} 
 in $\Fun(BG,\nCcat)$,
where by Lemma \ref{eriughwierewfwerfwf}.\ref{wergjiergjosfdgdfg2} the first map is a Morita equivalence, the second is a weak Morita equivalence, and the third one is a split relative Morita equivalence  by Lemma \ref{eriughwierewfwerfwf}.\ref{wergjiergjosfdgdfg1}.   We therefore get an associated zig-zag of equivalences
  \begin{equation}\label{weogjkopfegwreffsv1}K^{G,\an}_{A}(-)\stackrel{\simeq}{\to} K^{G,\an}_{(\bC^{u})^{(G)}}(-) \stackrel{\simeq}{\to} K^{G,\an}_{ \bC^{(G)}_{\std,+}}(-)\stackrel{\simeq}{\leftarrow}  K^{G,\an}_{  \bC^{(G)}_{\std}}(-)\end{equation}
 of functors on $\ppGTopf$.

  Composing the equivalences in \eqref{weogjkopfegwreffsv} and \eqref{weogjkopfegwreffsv1}
  we get the asserted equivalence.
 \end{proof}

{In the next proposition we calculate the values}
  of the functor $K_{\bC}^{G,\An,\mathrm{lf}}$ from \eqref{regrwefefefewf}.
 {We use the notation introduced in Definition \ref{werigjweogwergwregwerg}.\ref{qewfwefdwdwqedweq1}}.
{Let $A$ be in $\Fun(BG,\nCalg)$.}

\begin{prop} \label{werigowergrefwerfwe}
If $A$ is unital and separable, then we have an equivalence
$$(\Sigma K_{A}^{G,\an})_{|\ppGTopo}\simeq (K_{{\Hilb_{c}(A)}}^{G,\An,\mathrm{lf}})_{|\ppGTopo}\, .$$
\end{prop}
\begin{proof}
The argument is similar as for Proposition \ref{wegojeogrregregwergwegre}.   {However,} if $X$ is in $\ppGTopo$, then 
$\kkG(C_{0}(X))$ is not ind-$G$-proper in general so that $\bB\to K^{G,\an}_{\bB}(X)$ does not send all exact sequences to fibre sequences, i.e., Lemma \ref{eroigjweogregwgefwe} is not directly applicable.  

  In analogy with \eqref{wergeefeferfewrferf} we can define 
 the  locally finite evaluation  $ F^{\mathrm{lf}}$ of 
 any functor $F$ on $\ppGTop$  (with complete target)   
  by
$$ F^{\mathrm{lf}}(X)\coloneqq \lim_{U\subseteq X}F(U)\, ,$$
where $U$ runs over the open subsets of $X$ with $G$-compact closure.  
We have a natural transformation $c_{F}\colon F\to F^{\mathrm{lf}}$, and the transformation
$c_{F^{\mathrm{lf}}}\colon F^{\mathrm{lf}}\to (F^{\mathrm{lf}})^{\mathrm{lf}}$ is an equivalence by a cofinality argument.

We again  abbreviate $\bC \coloneqq \Hilb_{c}(A)$. 
We will construct an equivalence 
\begin{equation}\label{qwefrfrfrfrfrfqweoifjqwioefeqwfq1}
(\Sigma K_{A}^{G,\an,\mathrm{lf}})_{|\ppGTopo}\simeq (K_{\bC}^{G,\An,\mathrm{lf}})_{|\ppGTopo} 
\end{equation}
and furthermore show that the canonical morphism $c_{K_{A}^{G,\an}}$ induces an equivalence
 \begin{equation}\label{werfwefewfqwefqwedqdqwedwer}
  (K_{A}^{G,\an })_{|\ppGTopo}\stackrel{\simeq}{\to}  (K_{A}^{G,\an,\mathrm{lf}})_{|\ppGTopo}\, .
  \end{equation}
 The asserted equivalence is  then defined  as the composition of the equivalences
 in \eqref{qwefrfrfrfrfrfqweoifjqwioefeqwfq1} and \eqref{werfwefewfqwefqwedqdqwedwer}.

We start with the construction of  \eqref{qwefrfrfrfrfrfqweoifjqwioefeqwfq1}.
We consider the following diagram in $\KKG$ 
\begin{equation}\label{fqwfwefwefewfewffqewf} 
 \xymatrix{&{\kkGA}({\bC}^{(G)}_{\std})\ar[r]\ar@{..>}[d]^{i}&{\kkGA}({\bM\bC}^{(G)}_{\std})\ar[rr]^{{\kkGA}(\pi)}\ar@{=}[d]&&\kkG(\bQ^{(G)}_{\std})\ar@{=}[d]\\
 \Sigma^{-1} {\kkGA}(\bQ^{(G)}_{\std})\ar[r]^-{j}&
 F(\pi)\ar[r]&{\kkGA}({\bM\bC}^{(G)}_{\std})\ar[rr]^{{\kkGA}(\pi)}&&{\kkGA}(\bQ^{(G)}_{\std})\, .
}
 \end{equation}
 The lower part is a segment of  a fibre sequence with $F(\pi)$ defined as the fibre of $\kkG(\pi)${, where $\pi$ is the quotient morphism ${\bM\bC}^{(G)}_{\std} \to \bQ^{(G)}_{\std}$}.   The upper composition vanishes  since \eqref{QEWFQWEFFQWDWD} is exact, but it is not necessarily part of a fibre sequence since  $\kkGA$  is only conditionally exact.   
 The dotted arrow and the corresponding square is then given by the universal property of the fibre.

 We  consider an ind-$G$-proper object $P$ and apply the exact  functor 
 { $\KKG(P,-)\colon \KKG \to \Sp$  to \eqref{fqwfwefwefewfewffqewf}}. We then  get the following diagram in $\Sp$  (as usual we drop the symbol ${\kkGA}$ if we insert objects in $\KKG(-,-)$)
 \begin{equation}\label{fqwfwefwefewfewffqewf1} 
 \xymatrix{&\KKG(P, {\bC}^{(G)}_{\std})\ar[r]\ar@{..>}[d]^{i_{*}}_{\simeq}&\KKG(P,{\bM\bC}^{(G)}_{\std})\ar[r]^{\pi_{*}}\ar@{=}[d]&\KKG(P,\bQ^{(G)}_{\std})\ar@{=}[d]\\
 \Sigma^{-1} \KKG(P,\bQ^{(G)}_{\std})\ar[r]_-{\simeq}^-{j_{*}}&
 \KKG(P,F(\pi))\ar[r]&\KKG(P,{\bM\bC}^{(G)}_{\std})\ar[r]^{\pi_{*}}&\KKG(P,\bQ^{(G)}_{\std})\, .
}
 \end{equation}
By  \cite[Thm.\ 1.32.5]{KKG} the upper sequence becomes a fibre sequence, too. Therefore the dotted arrow becomes an equivalence. Furthermore, $ {\bM\bC}^{(G)}_{\std}$ is flasque by   Lemma \ref{wrtijhoerhgrtgertg}  so that 
$\KKG(P,{\bM\bC}^{(G)}_{\std})\simeq 0$ by \cite[Thm.\ 1.32.7]{KKG},  and $j_{*}$ becomes  an equivalence.


    We consider the following two natural transformations
 \begin{equation}\label{wqefqwefqewfdqdewdqwd}
\Sigma^{-1} K_{\bC}^{G,\An} \stackrel{\textrm{def}}{=}\Sigma^{-1} K_{\bQ^{(G)}_{\std}}^{G,\an}        \xrightarrow{j_{*}}  K^{G,\an}_{F(\pi)}
\end{equation}  
  and 
   \begin{equation}\label{wqefqwefqewfdqdewdqwd1}
K^{G,\an}_{A}\stackrel{\eqref{weogjkopfegwreffsv1}}{\simeq} K_{ {\bC}^{(G)}_{\std}}^{G,\an} \xrightarrow{i_{*}}  K^{G,\an}_{F(\pi)}\end{equation}  
  of $\Sp$-valued functors on $\ppGTop$, where $i_{*}$ and $j_{*}$ are induced by the morphisms $i$ and $j$ in \eqref{fqwfwefwefewfewffqewf}.
Since by \cite[Prop.\ 1.26]{KKG} the restriction of $\kkG\circ C_{0}(-)$ to $\ppGTopf$  takes values in ind-$G$-proper objects, 
the restrictions of $j_{*}$ in \eqref{wqefqwefqewfdqdewdqwd} and     $i_{*}$  in \eqref{wqefqwefqewfdqdewdqwd1} 
to $\ppGTopf$  are equivalences.

 We apply the $(-)^{\lf}$-construction to the transformations in \eqref{wqefqwefqewfdqdewdqwd} and
\eqref{wqefqwefqewfdqdewdqwd1} and get transformations
\begin{equation}\label{qwefrfrfrfrfrfqweoifjqwioewerwerwerrfeqwfq1}
 \Sigma^{-1} K^{G,\An,\mathrm{lf}}_{\bC} \to    K_{F(\pi)}^{G,\an,\mathrm{lf}}
 \end{equation}
 and
  \begin{equation}\label{qwefrfrfrfrfrfqweoifjqwioefeqwfq}
  K^{G,\an,\mathrm{lf}}_{A} \to    K_{F(\pi)}^{G,\an,\mathrm{lf}}\, .
 \end{equation}


We now show that the evaluations of  \eqref{qwefrfrfrfrfrfqweoifjqwioewerwerwerrfeqwfq1}  and \eqref{qwefrfrfrfrfrfqweoifjqwioefeqwfq}    at   $X$  in $\ppGTopo$  are  equivalences. By homotopy invariance of the domains and targets   we can assume that $X$ is a countable  $G$-CW-complex with proper $G$-action.  By local compactness, it admits a cofinal family of open subsets $U$ with $G$-compact closure
belonging to $\ppGTopf$. This implies that $j_{*}$ in \eqref{wqefqwefqewfdqdewdqwd}  and $i_{*}$  in \eqref{wqefqwefqewfdqdewdqwd1}  
become equivalences after evaluation at  such $U$. We get the equivalences  \eqref{qwefrfrfrfrfrfqweoifjqwioewerwerwerrfeqwfq1}  and \eqref{qwefrfrfrfrfrfqweoifjqwioefeqwfq} as   limits of equivalences. The desired equivalence \eqref{qwefrfrfrfrfrfqweoifjqwioefeqwfq1} is now defined as the suspension of the composition  
\begin{equation}\label{fqwefwefweqdwedqe}
 (K^{G,\an,\mathrm{lf}}_{A})_{|\ppGTopo}    \xrightarrow{\eqref{qwefrfrfrfrfrfqweoifjqwioefeqwfq},\simeq}
  (K_{F(\pi)}^{G,\an,\mathrm{lf}})_{|\ppGTopo}\xleftarrow{\simeq, \eqref{qwefrfrfrfrfrfqweoifjqwioewerwerwerrfeqwfq1}} (\Sigma^{-1} K_{\bC}^{G,\An,\mathrm{lf}})_{|\ppGTopo}  \, .
\end{equation}
%

  It now remains to show that the canonical  transformation
  \eqref{werfwefewfqwefqwedqdqwedwer}
      is an equivalence. 
We can again assume that  $X$ is a countable  $G$-CW-complex with proper $G$-action. 
 We let $(U_{n})_{n\in \nat}$ be an exhaustion of $X$ by an increasing family of invariant open  subsets with $G$-compact closure. Then setting $Y_{n}:=X\setminus  U_{n}$ the family $(Y_{n})_{n\in \nat}$ is a decreasing family of closed invariant  subsets  of $X$ with $\bigcap_{n\in \nat} Y_{n}=\emptyset$.    We get a diagram of  maps 
   $$\xymatrix{\vdots\ar[d]&\vdots\ar[d]&\vdots\ar[d]\\K^{G,\an}_{A}(Y_{n+1}) \ar[r]\ar[d]&K^{G,\an}_{A}(X)\ar[d]\ar[r]&K^{G,\an}_{A}(U_{n+1})\ar[d]\\
 K^{G,\an}_{A}(Y_{n})\ar[r]\ar[d]&K^{G,\an}_{A}(X)\ar[r]\ar[d]&K^{G,\an}_{A}(U_{n})\ar[d]\\\vdots &\vdots &\vdots }$$
 whose horizontal pieces are  fibre sequences by \cite[Thm. 1.15.3]{KKG}. 
Here we use that the inclusions $Y_{n}\to X$ are split-closed by \cite[Prop. 5.1.1]{KKG} and our topological assumptions on $X$.
 We now consider the fibre sequence  obtained as the limit of this diagram in the vertical direction. Using that $A$ is separable and 
  \cite[Thm. 1.15.6]{KKG}  the limit of the left column vanishes. Hence
  we get an equivalence
  $$K^{G,\an}_{A}(X)\stackrel{\simeq}{\to} \lim_{n\in \nat} K^{G,\an}_{A}(U_{n+1})\simeq K_{A}^{G,\an,\mathrm{lf}}(X)$$
  as desired.
  %
  \end{proof}

 \section{Comparison with classical constructions}\label{wrtiohgrhdgfhgh}

As explained already in the introduction 
the classical  definition of the domain of the Paschke morphism
does not involve a $C^{*}$-category of controlled Hilbert spaces but it involves the choice of
a single sufficiently large  continuously controlled Hilbert space.  So in order to compare the approach of the present paper with the classical one 
we 
 specialize to the case of trivial coefficients characterized by 
 {$\bC = \Hilb_{c}(\C)$ and $\bM\bC=\Hilb(\C)$.} 
 According to Definition \ref{qegijqgiofjfqewfqwef} the objects of {$\Hilb(\C)^{(G)}$} are pairs $(H,\rho)$ of a Hilbert space 
 $H$ and a unitary representation $\rho \colon G^{\op}\to U(H)$, $g\mapsto \rho_{g}$.
 The morphisms are given by  $\Hom_{{\Hilb(\C)^{(G)}}}((H,\rho),(H',\rho'))=B(H,H')$, {the bounded linear operators from $H$ to $H'$.}
 The group $G$ fixes the objects of {$\Hilb(\C)^{(G)}$} and acts on the morphisms
 by $g\cdot A \coloneqq \rho_{g}^{\prime, -1}A \rho_{g}$.

 We consider a second countable proper metric space $X$ with an isometric action of the 
 group $G$.
In the following we construct an exact sequence of  $C^{*}$-categories   
\begin{equation}\label{frefreferfererwwerfrwef}0\to \bC^{G}(X) \to \bD^{G}(X)\to \bQ^{G}(X)\to 0\, .
\end{equation}
 We start with the definition of a $C^{*}$-category   $\bB(X)$ with $G$-action. Its objects   are triples $(H,\rho,\phi)$, where $(H,\rho)$ is in $\Hilb(\C)^{(G)}$ such that  $H$ is separable   and
 $\phi \colon C_{0}(X)\to B(H)$ is  homomorphism of $C^{*}$-algebras 
 {satisfying the following properties:
 \begin{enumerate}
 \item The representation $\phi$ is equivariant, i.e., we have  $g^{-1}\cdot \phi(f)=\phi(g^{*}f)$ for all $f$ in $C_{0}(X)$ and $g$ in $G$,  see \eqref{gwwgegdfsfvfdvsdvsfv}.
  \item The representation $\phi$ is non-degenerate in the sense that $\overline{\phi(C_{0}(X))H}=H$.
 \item  There exists an equivariant unitary  isomorphism $(H,\rho)\cong (L^{2}(G)\otimes H',\lambda\otimes \id_{H})$, where $\lambda$ is the left-regular representation of $G$ on $ L^{2}(G)$ and $H'$ is some auxiliary separable Hilbert space.
   \end{enumerate}}
 
 The morphisms of $\bB(X)$  are inherited from ${\Hilb (\C)^{(G)}}$. The group $G$ fixes the objects of $\bB(X)$ and acts on morphisms as in $ {\Hilb (\C)^{(G)}}$.

 Let $(H,\rho,\phi)$ and $ (H',\rho',\phi')$ be objects of $\bB(X)$.
  An operator $A$ in $B(H,H')$ is called locally compact  if $\phi'(f)A$ and $A\phi(f)$  belong to $K(H,H')$
 for all $f$ in $C_{0}(X)$, {where $K(H,H')$ denotes the set of compact linear operators from $H$ to $H'$.} Further, $A$ is called pseudolocal if $\phi'(f)A-A\phi(f)\in K(H,H')$ for all $f$ in $C_{0}(X)$.
 Finally, it is called controlled if there exists $R$ in $(0,\infty)$ such that
 $d(\supp(f'),\supp(f))\ge R$ implies  that $\phi'(f)A\phi(f)=0$. 
 The $C^{*}$-category $\bC^{G}(X)$ is the wide $C^{*}$-subcategory of $\bB(X)$ 
   generated by the invariant,  locally compact and controlled operators. Similarly the $C^{*}$-category
   $\bD^{G}(X)$ is generated by the invariant, pseudolocal and controlled operators.
   Finally $\bQ^{G}(X)$ is defined as the quotient, see \eqref{frefreferfererwwerfrwef}.
   If $(H,\rho,\phi)$ is an object of $\bB(X)$, then the corresponding endomorphism algebras form  an exact sequence 
   $$0\to  C^{G}(H,\rho,\phi)\to D^{G}(H,\rho,\phi)\to Q^{G}(H,\rho,\phi)\to 0$$
   which is the equivariant generalization of \eqref{wergpowkgpowergregregwegrg} from the introduction.
  \begin{ddd}\label{werogihetgergfdsf}
{An object $(H,\rho,\phi)$ of $\bD^{G}(X)$} is called absorbing if for every other $(H',\rho',\phi')$ {in $\bD^{G}(X)$} there exists  an isometry $u \colon (H',\rho',\phi')\to (H,\rho,\phi)$
in $\bD^{G}(X)$.
\end{ddd}
 The existence of absorbing objects in the case of trivial $G$ follows from \cite[Lem.\ 7.7]{hr}.\footnote{We {neither} know a reference {nor have a proof} for the existence of absorbing objects in the equivariant case  {in full generality}, {see Remark \ref{eiruhguiewgewrgerfwefewrf}}.}
For the following {discussion}, we assume that we can choose an absorbing object $(H,\rho,\phi)$.  We set $Q^{G}(X):=Q^{G}(H,\rho,\phi)$ and let $Q(H)$ be the Calkin algebra of $H$ with the induced $G$-action.
 With these choices we can define the Paschke morphism 
 $$p^{(H,\rho,\phi)}_{X} \coloneqq \mu_{X}\circ \delta_{X} \colon \KK(\C,Q^{G}(X))\to \KK^{G}(C_{0}(X),Q(H))$$ as in 
\eqref{rfqwpofkjpqowefewfeqwfqef}.  
We can consider $X$ as an object of $G\UBC$ with the structures induced by the metric.
We furthermore assume that $X$ is homotopy equivalent to a  $G$-compact $G$-CW-complex with finite stabilizers.
The following proposition asserts that  the Paschke morphism $p_{X}$ from \eqref{Paschke-morphisms} 
is compatible with   $p^{(H,\rho,\phi)}_{X}$.

\begin{prop}\label{iuhugiergwergergwg}
There exists a commutative   square \begin{equation}\label{feqwfewfefqewfefqewfqe}
\xymatrix{K_{\bC}^{G,\cX}(X)\ar[r]^-{\gamma}\ar[d]^{p_{X}}&\KK(\C,Q^{G}(X))\ar[d]^{p_{X}^{(H,\rho,\phi)}}\\ K_{\bC}^{G,\An} (\iota^{\topp}(X))&\ar[l]_-{\simeq}  \KK^{G}(C_{0}(X) ,Q(H))  }
\end{equation}
 \end{prop}
\begin{proof}
 We use the identifications
$$K_{\bC}^{G,\cX}(X)\stackrel{\text{Lem.~}\ref{wrtoihgjoergergegwgrgwerg}}{\simeq}  \KK(\C,\bQ(X))$$
and
$$K_{\bC}^{G,\An} (\iota^{\topp}(X)) \stackrel{\eqref{vfdsvsvvsvsdvdsadsvdsva}}{=} \KK^{G}(C_{0}(X),\bQ^{(G)}_{\std})\, .$$ 
 The objects of  $\bQ(X) $ (and also of $\bD(X)$ and $\bC(X)$, see \eqref{rgfqfewfqewf1} and \eqref{rgfqfewfqewf})  are the objects of $\bCgtsmc(\cO(X)\otimes G_{can,max})$.   If $(H',\rho',\mu')$ is such an object,
 we get the object $(H',\rho',\phi')$ of {$\bD^{G}(X)$}  with $\phi'$ as in \eqref{ewqfpojlqkrmeflkwefqef}.  
  Note that  since $X$  is second countable and has the bornology of relatively compact subsets, the Hilbert space $H'$ is separable by the local finiteness conditions (see Definition \ref{qeroigergeggweerg}) on  $(H',\rho',\mu')$.
  {Furthermore, using that $X\times G$ is a free $G$-set we see that $(H',\rho') $ is a multiple of the regular representation of $G$ on $L^{2}(G)$.}
  Since we assume that  $(H,\rho,\phi)$ is absorbing 
 there exists  an isometry 
 ${u'} \colon (H',\rho',\phi')\to (H,\rho,\phi)$ in $\bD^{G}(X)$.
 
 We consider the category $\bD^{u}(X) $ consisting of pairs
 $((H',\rho',\mu'),{u'})$ of an object  $(H',\rho',\mu')$ in $\bD(X)$
 and an  isometry $u$ as above. A morphism 
 $A \colon ((H',\rho',\mu'),{u'})\to ((H'',\rho'',\mu''),{u''})$ is a morphism
 $A \colon (H',\rho',\mu')\to  (H'',\rho'',\mu'')$ in $\bD(X)$.
 We define $\bC^{u}(X) $ and $\bQ^{u}(X)$ similarly.
 Then we have a diagram of maps of exact sequences of $C^{*}$-categories 
 $$\xymatrix{0\ar[r]&\bC (X)\ar[r] &\bD (X)\ar[r] &\bQ (X) \ar[r]&0
 \\0\ar[r]&\bC^{u}  (X)\ar[r]\ar[d]\ar[u]&\bD^{u} (X)\ar[r]\ar[d]\ar[u]&\bQ^{u} (X)\ar[u]\ar[d]\ar[r]&0\\ 
0\ar[r]&C^{G}(X)\ar[r]&D^{G}(X)\ar[r]&Q^{G}(X)\ar[r]&0
 }$$
 where in the lower sequence we consider the $C^{*}$-algebras as $C^{*}$-categories with a single object. 
 The upper vertical functors  just forget the embedding ${u'}$
 {and  are} 
 unitary equivalences.
  The definition of the lower vertical functors  on the objects is clear. 
  The functor $\bD^{u}(X)\to D^{G}(X)$ sends a morphism $A \colon ((H',\rho',\mu'),{u'})\to ((H'',\rho'',\mu''),{u''})$ to ${u''Au^{\prime,*}}$. The other functors are defined similarly. 
  Since $\Kcat$ sends   unitary  equivalences to  equivalences,
  we get the following {morphism} 
$$\xymatrix{   \Kcat(\bC (X))\ar[r]\ar[d]^{\alpha}_{\simeq} &\Kcat(\bD (X))\ar[r]\ar[d] &\Kcat(\bQ (X) )\ar[d]^{\gamma}  \\ 
 \Kast(C^{G}(X))\ar[r]&\Kast(D^{G}(X))\ar[r]&\Kast(Q^{G}(X)) 
 }$$
of fibre sequences. {The} {right} vertical map is the map $\gamma$ in   the square \eqref{feqwfewfefqewfefqewfqe}.
The map $\alpha$ is an equivalence   by \cite[Thm.\ 6.1]{indexclass}, but this will not be used here.

If $X$ {is}  homotopy equivalent to a $G$-finite $G$-CW complex with finite stabilizers, {then} the functor  $\KKG(C_{0}(X),-)$ sends exact sequences in $\Fun(BG,\nCcat)$  to fibre sequences by a combination of \cite[Prop.\ 1.26]{KKG} and \cite[Thm.\ 1.32.5]{KKG}. 
 The lower horizontal map in  \eqref{feqwfewfefqewfefqewfqe} is induced by the functor
 $Q(H)\to \bQ^{(G)}_{\std}$ which just views $(H,\rho)$ as an object of $ \bQ^{(G)}_{\std}$. 
In order to show that it is an equivalence
we consider the map of fibre sequences obtained by applying $\KK^{G}(C_{0}(X),-)$  to the  map of exact sequences
\begin{equation}\label{sdaoijfiojafassdfasdf}  
\xymatrix{ 0\ar[r]&K(H) \ar[r]\ar[d] & B(H)\ar[r]\ar[d] & Q(H)\ar[d]\ar[r]&0\\ 
0\ar[r]&{\Hilb_{c}(\C)}^{(G)}_{\std}\ar[r]&{\Hilb(\C)}^{(G)}_{\std}\ar[r]&\bQ^{(G)}_{\std}\ar[r]&0
 }
 \end{equation}
The vertical maps send the unique object of the domain to the object $(H,\rho)$.
 We    have   $\KKG(C_{0}(X),B(H))\simeq 0 $ by \cite[Cor.\ 6.{22}]{KKG}, and we also have   $\KKG(\C,{\Hilb(\C)^{(G)}_{\std}})\simeq 0 $ by  \cite[Thm.\ 1.32.7]{KKG} 
since ${\Hilb(\C)^{(G)}_{\std}}$   is flasque {by Lemma \ref{wrtijhoerhgrtgertg}.}

We will show that the left vertical map {in \eqref{sdaoijfiojafassdfasdf}} induces an equivalence after
applying $\KK^{G}(C_{0}(X), -)$. 
 We let $ {\Hilb_{c}(\C)}^{(G),\sepa}_{\std}$ and $ {\Hilb(\C)}^{(G),\sepa}_{\std}$ denote the full subcategories of $ {\Hilb_{c}(\C)}^{(G)}_{\std}$  and $ {\Hilb(\C)}^{(G)}_{\std}$, respectively, 
 {of separable Hilbert spaces.}
{Then we have a factorization of the left vertical morphism in \eqref{sdaoijfiojafassdfasdf} as}
\begin{equation}\label{qkweflkqwfnlkwefweqfqweqdqwd} 
K(H) \to   {\Hilb_{c}(\C)}^{(G),\sepa}_{\std} \to  {\Hilb_{c}(\C)}^{(G)}_{\std}\, .
\end{equation}
{We claim that first morphism is an idempotent completion   relative to the ideal inclusion $K(H)\to B(H)$, and therefore
 a relative Morita equivalence by \cite[Prop.\ 17.8]{cank}.
  In order to see  the claim  note  we have an equivariant  unitary isomorphism  $(H,\rho)\cong (L^{2}(G)\otimes H',\lambda\otimes \id_{H'})$. Since $(H,\rho,\phi)$ is absorbing  we can in addition assume that    $\dim(H')=\infty
$. Since every separable Hilbert space is isomorphic to a subspace of $H'$ we see that
every object of $  \Hilb(\C)^{(G),\sepa}_{\std}$ admits an isometry to $(H,\rho)$.
 We now consider the square
$$\xymatrix{K(H)\ar[r]\ar[d]&B(H)\ar[d]\\ {\Hilb_{c}(\C)}^{(G),\sepa}_{\std}\ar[r]&  {\Hilb(\C)}^{(G),\sepa}_{\std}}$$
where the horizontal maps are ideal inclusions. By the observation above  the 
  right vertical map presents  $ \Hilb(\C)^{(G),\sepa}_{\std}$ as the idempotent completion of $B(H)$.} 
     
     {The  second morphism in \eqref{qkweflkqwfnlkwefweqfqweqdqwd} is easily seen to be a weak Morita equivalence.}
    Since $\KKG(C_{0}(X), -)$ sends both relative Morita equivalences and weak Morita equivalences to equivalences by   \cite[Thm.\ 1.32.8]{KKG} and  \cite[Thm.\ 1.32.3]{KKG}, respectively,  
   the left vertical morphism in \eqref{sdaoijfiojafassdfasdf} induces an equivalence after applying  $ \KKG(C_{0}(X),  -)$.


  This 
  {together   with 
  the fact that this functor annihilates $B(H)$ and $\Hilb(\C)_{\std}^{(G)}$}
  implies  that    $$\KKG(C_{0}(X),Q(H))\to \KKG(C_{0}(X), \bQ^{(G)}_{\std})$$ is an equivalence.
 This explains the lower horizontal equivalence in \eqref{feqwfewfefqewfefqewfqe}.

It is obvious from the definitions of the Paschke morphisms in \eqref{rfqwpofkjpqowefewfeqwfqef}  and Definition \ref{qeorigjoqfeqwfqfewf}   that the diagram commutes.
\end{proof}

 In the following we assume that $X$ satisfies the assumptions  {of  Theorem} \ref{qreoigjoergegqrgqerqfewf}.\ref{qregiojqwfewfqwfqewf} such that   $p_{X}$  is an equivalence.
\begin{kor}\label{tgjiogweregwergre}
The morphism $\gamma$ is an equivalence if and only if $p^{(H,\rho,\phi)}_{X}$ is an equivalence. 
\end{kor}

This says that in all cases where the classical Paschke morphism $p^{(H,\rho,{\phi})}_{X}$ is an equivalence it is equivalent
to our morphism $p_{X}$ as a spectrum map.
An independent proof\footnote{We do not know a reference for such a proof.}  that $\gamma$ is an equivalence would then allow us to conclude  from  Theorem \ref{qreoigjoergegqrgqerqfewf}.\ref{qregiojqwfewfqwfqewf}  that
$p^{(H,\rho,{\phi})}_{X}$ is an equivalence.

\begin{rem}\label{eiruhguiewgewrgerfwefewrf}
This is a remark about the existence of absorbing objects an in Definition \ref{werogihetgergfdsf}. First of all the discussion above depends on the existence of an absorbing object in $\bD^{G}(X)$ for which we {neither} have a reference {nor a proof}.  Related results are \cite[Lem.\ 4.5.5 \& Prop.\ 4.5.14]{willett_yu_book}. They are adapted for the approach based on localization algebras but 
do not imply directly what we need.  
A similar remark applies to   \cite[Thm.\ 1.3]{Benameur:2020aa}.

In the non-equivariant {case} the existence of  absorbing objects  is settled {in   \cite[Lem.\ 7.7]{hr}
by an application of Voiculescu's Theorem.}

We furthermore  do not know a reference for the fact that
$p_{X}^{(H,\rho,{\phi})}$ is an equivalence. In fact,  
 \cite[Thm.\ 1.5]{Benameur:2020aa}  states a Paschke duality
 isomorphism in the equivariant case. But it is not obvious how to identify the targets and the maps {in} \cite[Thm.\ 1.5]{Benameur:2020aa} with $p_{X}^{(H,\rho,\phi)}$.
\hB
\end{rem}

 \section{Homotopy theoretic  and analytic assembly maps}\label{weroigjwoergergwgre}
 
 In this section we describe the homotopy theoretic and the analytic assembly maps which we will eventually compare in Theorem \ref{wtoiguwegwergergregwe}.  The homotopy theoretic assembly introduced in Definition \ref{qroeigjoqergerqfewewfqewfeqwf}  is a standard construction from equivariant homotopy theory \cite{davis_lueck}.  For the historic development of the analytic assembly map we refer to \cite{Aparicio:2019aa}.
 Our Definition \ref{wergoijowergerreggwgw} is a spectrum valued refinement of the assembly map of \cite{kasparovinvent,baum_connes_higson} which is new in this form.
 
We begin with the homotopy theoretic assembly map. Let $G\Orb$ denote the orbit category of $G$ and $\bM$ be some cocomplete stable $\infty$-category. Recall that by Definition \ref{weigjorgdg}
  an equivariant $\bM$-valued homology theory  is simply  a functor 
 $$E \colon G\Orb\to \bM\, .$$ 
   
 Let  $\cF $ be a  family of subgroups of $G$. By $G_{\cF}\Orb$ we denote the full subcategory of the orbit category $G\Orb$ of transitive $G$-sets with stabilizers in the family $\cF$.  
Since 
  $*$ is a final object of $G\Orb$ we have a natural transformation    $E\to\underline{E(*)}$ in $\Fun( G\Orb , \bM)$. This transformation induces the  homotopy theoretic assembly map:
\begin{ddd}\label{woiejtgoiwegegrgergwer}The homotopy theoretic assembly map for $E$ and $\cF $  is the canonical morphism  $$\Ass^{h}_{E,\cF} \colon \colim_{G_{\cF}\Orb}E\to E(*)$$
  in $\bM$.  
\end{ddd}

Recall that we can evaluate the equivariant homology theory $E$
on $G$-topological spaces using \eqref{adfiuhuihavfvasdcacsadcasca}.
For every  $X$ in $ G\Top$ we get  a morphism \begin{equation}\label{qwefeqwdqwed}
\Ass_{E,X}^{h} \colon E(X)\to E(*)
\end{equation}
which is induced by the projection $X\to *$.  
We let $E_{\cF}G^\cw$ be a $G$-CW complex representing the homotopy type of the classifying space for the family $\cF$. It is characterized 
essentially uniquely by the condition that \begin{equation}\label{sggsfdge}Y^{G}(E_{\cF}G^\cw)(S)\simeq \left\{\begin{array}{cc} \emptyset&S\not\in G_{\cF}\Orb\,,\\{*}&S\in G_{\cF}\Orb\,. \end{array}\right.
\end{equation}
$$ $$
As a consequence  of \eqref{qwefpojfopwefqwefqwefqew} we then get the equivalence
$E(E_{\cF}G^\cw)\simeq \colim_{G_{\cF}\Orb}E$, and under this identification we have the equivalence \begin{equation}\label{ewfqpijoijoioij}
 \Ass_{E,\cF}^{h}\simeq   \Ass_{E,E_{\cF}G^\cw}^{h} 
\end{equation}
 of assembly maps.
 {Further below, in the special case of the  functor $E=\hat K_{A}^{G}$ introduced in Definition \ref{egiueheiugheiugwegerggwfewr} for  $A$ in $ \KKG$  we will use the notation 
 \begin{equation}\label{rtherhthetfff}\mu^{DL}_{A,X} \coloneqq \Ass_{\hat K^{G}_{A},X}^{h}\, , \quad \mu^{DL}_{A,\cF} \coloneqq \Ass_{\hat K^{G}_{A},\cF}^{h}
\end{equation}  indicating that 
 $\mu^{DL}_{A,\cF}$ is the assembly map  introduced by Davis--L\"uck  \cite{davis_lueck}.

We have a functor \begin{equation}\label{adfvoijqiorjvoifvafsvsdv}
\iota \colon G\Orb\to G\BC\, , \quad S\mapsto S_{min,max}\, ,
\end{equation} 
where $S_{min,max}$ is the $G$-set $S$ equipped with the minimal coarse structure and  the maximal bornology. For a coefficient category $\bC$ in $\Fun(BG,\nCcat)$ which is effectively additive and admits countable AV-sums we have an
equivariant coarse  $K$-homology functor 
$$K\bC\cX_{G_{can,min}}^{G} \colon G\BC\to \Sp^{\la}$$
(see Definition \ref{qrogijeqoifefewfefewqffe} for $K\bC\cX^{G}$ and  Definition \ref{regoiergowregrwergwreg} for the   {twist of an}{ equivariant  coarse homology theory by an object of $G\BC$,   in the present case by $G_{can,min}$}). The following is the technical definition of the functor described in \eqref{asdvijqoirgfvfsdvsva}.
\begin{ddd}\label{qroeigjoqergerqfewewfqewfeqwf}
We define the functor 
$$K\bC^{G} \colon  G\Orb \stackrel{\iota}{\to} G\BC \stackrel{K\bC\cX^{G}_{G_{can,min}}}{\lto}  \Sp^{\la} \, .$$
\end{ddd}
 
We now apply the definitions of assembly maps explained above  to the functor $K\bC^{G}$ in place of $E$ and introduce a shorter notation.
\begin{ddd}\label{wrtiogrgergfgs}
The homotopy theoretic assembly   map  associated to $G$, $\cF$ and {$ \bC$} is defined to be the map
$$\Ass_{{\bC},\cF}^{h} \coloneqq \Ass_{K\bC^{G},\cF}^{h} \colon\colim_{G_{\cF}\Orb} K\bC^{G} \to K\bC^{G}(*)\, .$$
\end{ddd}

More generally, for every $X$ in  $G\Top$, specializing \eqref{qwefeqwdqwed}, we have the  morphism
\begin{equation}\label{vevopjrepovervfvdfv}
\Ass_{{\bC},X}^{h} \coloneqq \Ass^{h}_{K\bC^{G},X}:K\bC^{G}(X)\to K\bC^{G}(*)
\end{equation}
  induced by the projection $X\to *$.
Since $K\bC\cX^{G}$ depends naturally  on the coefficient category $\bC$ in $\Fun(BG,\nCcat_{\ndeg,\eadd,\omega\add})${, see \eqref{eq_defn_functoriality_category} for the definition of this category}, so do the assembly maps  
$\Ass_{{\bC},X}^{h}$  and $ \Ass_{{\bC},\cF}^{h}$.

We now turn to the analytic assembly map  whose final definition will be stated in Definition \ref{wergoijowergerreggwgw}.
  We start with introducing the notation for its domain. Recall that
    $\ppGTop$ is the  category of  locally compact {Hausdorff} $G$-spaces with partially defined   proper maps.
     \begin{ddd}\label{wrtoigjiowrgerferfw}
 We {denote by} $\pGTopc$   the full category  of $\ppGTop$
  of spaces on which $G$ acts properly and cocompactly.
  \end{ddd}
We will  describe the analytic assembly map
$\Ass^{\an}_{{\bC},\cF}$ associated to {$  \bC $    in $\Fun(BG,\nCcat)$} and a family $\cF$ contained in $\Fin$. In analogy to \eqref{qwefeqwdqwed} we will further  describe  a natural  transformation 
$$\Ass_{{\bC}}^{\an} \colon K_{\bC}^{G,\An}(-)\to \underline{\Sigma \KK(\C, {\bC}_{\std}^{(G)}\rtimes_{r}G)}
$$ of functors from $ \pGTopc$ to $\Sp^{\la}$.
  Note that  for infinite $G$ the morphism  
 \begin{equation}\label{fwfqwefedewdedwqedqewd}
\Ass^{\an}_{{\bC},X} \colon K_{\bC}^{G,\An}(X)\to \Sigma \KK(\C, {\bC}_{\std}^{(G)}\rtimes_{r}G)
\end{equation} 
 can  not  simply be  induced by a map $X\to *$ since $*$ and therefore this map   are  not in the category $\pGTopc$.
  If $E_{\cF}G^\cw$ happens to be in $\pGTopc$, then we will have an equivalence
$\Ass^{\an}_{{\bC},\cF}\simeq \Ass^{\an}_{{\bC},E_{\cF}G^\cw}$ in analogy to \eqref{ewfqpijoijoioij}.

The classical definition of the analytic assembly map is based on  a construction of a    family $(\Ass^{\an}_{{\bC},X,*})_{X\in\pGTopc}$ of homomorphisms in $\Ab^{\Z}$
  \begin{equation}\label{efovjefoivfdvsdfvsfdv}
\Ass_{{\bC},X,*}^{\an} \colon K_{\bC,*}^{G,\An}(X)= \KK_{*}^{G}(C_{0}(X),  \bQ^{(G)}_{\std}  )\to  \KK_{*-1}(\C, {\bC}^{(G)}_{\std}\rtimes_{r} G)\, ,
\end{equation}
which implement a natural transformation
\begin{equation}\label{qewfoihioqjfiofewfqwefqwefqewfefqwef}
K_{\bC,*}^{G,\An}(-)\to \underline{\KK_{*-1}(\C, {\bC}^{(G)}_{\std}\rtimes_{r} G)}
\end{equation}
of functors from $\pGTopc$ to $\Ab^{\Z}$.


 In the following we describe the details of the construction of $\Ass_{{\bC},X,*}^{\an}$ {in \eqref{efovjefoivfdvsdfvsfdv}}
thereby lifting it to the spectrum level.
The construction has three steps. The first is an application of   functor $-\rtimes G$ from \cite[Thm. 1.22.3]{KKG}, where $\rtimes$ without subscript refers to the maximal crossed product. The second is a pull-back along the Kasparov projection given by \eqref{wefqwfefqfqewfqfe1} below.  
The last step consists of changing target categories \eqref{wefqwfefqfqewfqfe}.

\phantomsection\label{afvoijvodvasdvsdvavdv}The following discussion will be used to get rid of the choice of cut-off functions involved in the Kasparov projection. Here we can take full advantage of the $\infty$-categorical set-up.
 We let $$\cR \colon  {\pGTopc}^{\op}\to \Set$$ be the following functor:
 \begin{enumerate}\item objects: The  functor $\cR$ sends $X$ to the set $\cR(X)$ of all   functions 
 $\chi$ in $C_{c}(X)$ such that 
 \begin{equation}
 \label{goqiejgoigqfqwfwefqwef}\sum_{g\in G} g^{*}\chi^{2}=1\, .
 \end{equation}   \item morphisms: 
 The functor $\cR$ sends a morphism
 $f \colon X\to X'$ in $ \pGTopc $ to the map
$\cR(f) \colon \cR(X')\to \cR(X)$ which sends $\chi'$ in $\cR(X')$ to $f^{*}\chi'$ in $\cR(X)$. 
\end{enumerate}

For $\chi$ in $\cR(X)$
we define the Kasparov  projection \begin{equation}\label{398z9823zf983rferwfwefwrffwr}
p_{\chi} \coloneqq \sum_{g\in G} (\chi \cdot g^{*}\chi ,g)
\end{equation}
 in $C_{0}(X)\rtimes G$. Note that this sum has finitely many non-zero terms.

 If $f \colon X\to X'$ is a morphism in  
 $  {\pGTopc}$ and $\chi'$ is in $\cR(X')$, then we have the relation
 $$(f^{*}\rtimes G)(p_{\chi'}) = p_{f^{*}\chi'}\, .$$
 Hence  we get a natural transformation of contravariant $\Set$-valued functors
 $$\cR(-)\to \Hom_{\nCalg}(\C,C_{0}(-)\rtimes G)$$ on $  {\pGTopc}$
 which sends $\chi$ in $\cR(X)$ to the homomorphism
 $$\C\ni \lambda\mapsto  \lambda p_{\chi} \in C_{0}(X)\rtimes G\, .$$
 Composing with $\kk$  we get a natural transformation of $\Spc$-valued functors 
 $$\ell'\cR(-)\to \Omega^{\infty}\KK(\C,C_{0}(X)\rtimes G)\, ,$$
 where $\ell' \colon \Set\to \Spc$ is the canonical inclusion.
 Using the      $(\Sigma^{\infty}_{+},\Omega^{\infty})$-adjunction we can interpret the result as   a transformation 
 \begin{equation}\label{adsckuhqiuhfiucdcdscads}
\Sigma^{\infty}_{+}\ell'\cR(-)\to \KK(\C,C_{0}(-)\rtimes G)
\end{equation}
 of $\Sp^{\la}$-valued functors.
 
 Let $E \colon   {\pGTopc}^{\op}\to \bM$ be  any functor to a cocomplete target.
We have a  functor  $$q \colon  {\pGTopc}\times \Delta\to    {\pGTopc}$$ which sends $(X,[n])$ to $X\times \Delta^{n}$ with the $G$-action only  on the first factor.      
We define the homotopification of $E$ by 
 $$ \cH(E) \coloneqq q_{*} q^{*} E \colon   {(}{\pGTopc}{)}^{\op}\to \Sp^{\la}\, ,$$   where $q^{*}$ is the pull-back along $q$ and $q_{*}$ is the right-adjoint of $q^{*}$, the right Kan-extension functor. 
  The unit of the adjunction $(q^{*},q_{*})$ provides a natural transformation
$E\to \cH(E)$.
We say that $E$ is homotopy invariant if the projection $  X\times \Delta^{1}  \to X$ induces an equivalence
$E(X)\stackrel{\simeq}{\to} E(X\times \Delta^{1})$. A proof of the following lemma  {is for instance implicitly given in the proof of \cite[Lem.\ 7.5]{Bunke:2013uq1}} 
\begin{lem}[{cf.\ \cite[Lem.\ 7.5]{Bunke:2013uq1}}]\label{fjewoifjqowfqewqwefqwefq}\mbox{}\begin{enumerate} \item $\cH(E)$ is homotopy invariant.
\item \label{fqwoeihfeowdqewdwqedewd}
 $E$ is homotopy invariant if and only if the canonical morphism $E\to \cH(E)$ is an equivalence.
\end{enumerate}
\end{lem}

Let $S$ denote the sphere spectrum and $\underline{S} \colon {\pGTopc}^{\op}\to \Sp$ be the constant functor with value $S$.
\begin{lem}\label{riojoijiowervervfsdvfdvsfd} The projection $\cR\to \underline{*}$ induces
 an equivalence $\cH(\Sigma^{\infty}_{+}\ell' \cR)\simeq \underline{S}$.
\end{lem}
 \begin{proof}
 By the pointwise formula for the left Kan extension $q_{!}$ we must show that the projection $\cR\to \underline{*}$ induces for every $X$ in $ {\pGTopc}$ an equivalence 
 $$\colim_{[n]\in \Delta^{\op}} \Sigma^{\infty}_{+}\ell'\cR(X\times \Delta^{n})\stackrel{\simeq}{\to} S\, .$$
 Since $\Sigma^{\infty}_{+}\colon\Spc\to \Sp$ preserves colimits it actually suffices to show that
 $ \colim_{[n]\in \Delta^{\op}} \ell'\cR(X\times \Delta^{n})\stackrel{\simeq}{\to}  *$ in $\Spc$. 
For a simplicial set $W$ the colimit $\colim_{\Delta^{\op}}\ell' W $ is given by 
$\ell(|W|)$, where  $\ell$ is as in \eqref{adsfasdfasdfadsf} and
$|W|$ in $\Top$ is the geometric realization of $W$.  Since the  geometric realization of the total space of
a trivial Kan fibration over a point is contractible it therefore 
 suffices to show that
 the map of simplicial sets
 $\cR(X\times \Delta^{-}) \to *$ is a trivial Kan fibration.      
 So we must show that 
 for every $n$ in $\nat$  a function 
  $\chi$ in  $\cR( {X\times \partial \Delta^{n}}  )$  can be extended to a function $\tilde \chi $ in $\cR( {X\times \Delta^{n}} )$. 
  
  For the case $n=0$ we observe that for any $X$ in $  {\pGTopc}$ we have $\cR(X)\not=\emptyset$. For $n\ge 1$,
 using barycentric coordinates  we can write a point in $\Delta^{n}$ in the form $\sigma t$ where $\sigma$ is in $[0,1]$ and
 $t$ is   in $ \partial\Delta^{n}$. Then  an  extension of $\chi$ is given by 
 $$\tilde \chi({x,\sigma t}) \coloneqq \sqrt{\sigma \chi({x,t})^{2}+(1-\sigma) \chi({x,t_{0}})^{2}}\, ,$$ where $t_{0}$ is the zero'th vertex of the simplex.
 \end{proof}

 We now use that $ \KK(\C,C_{0}(-)\rtimes G)$ is a homotopy invariant $\Sp^{\la}$-valued functor. Applying $\cH$ to \eqref{adsckuhqiuhfiucdcdscads} we get a transformation \begin{equation}\label{qewfoihiew9ufj9qweofqewfewfq}
\epsilon \colon
\underline{S}\stackrel{\textrm{Lem.~}\ref{riojoijiowervervfsdvfdvsfd}}{\simeq} \cH({\Sigma^\infty_+}\ell'\cR)\stackrel{\cH\eqref{adsckuhqiuhfiucdcdscads}}{\to} \cH( \KK(\C,C_{0}(-)\rtimes G))\stackrel{\simeq, \textrm{Lem.~}\ref{fjewoifjqowfqewqwefqwefq}.\ref{fqwoeihfeowdqewdwqedewd}}{\leftarrow}  \KK(\C,C_{0}(-)\rtimes G)\ .
\end{equation}

Let $A$ be an object of $\KKG$ and consider the functor from \cite[Def. 1.14]{KKG}:
\begin{equation}\label{qwefqwefewdewdwqdwedqewd}
K^{G,\an}_{A} \coloneqq \KK(C_{0}(-),A) \colon \ppGTop\to \Sp\, .
\end{equation}
We have the maximal\footnote{In the present paper  we use the convention to denote the maximal crossed product by $\rtimes$ and the reduced by $\rtimes_{r}$.} crossed product functor {$-\rtimes G$   \cite[Thm. 1.22.3]{KKG} whose action on mapping spectra   induces the following natural transformation  \begin{equation}\label{wefqwfefqfqewfqfe2}
 -\rtimes G:K_{A}^{G,\an}(-)=\KK^{G}(C_{0}(-),  {A})\to \KK(C_{0}(-)\rtimes G,  {A}\rtimes G)\, .
\end{equation} of functors from $\pGTopc$ to $\Sp$}. 
The composition of morphisms  in $\KK$ provides a  natural transformation
$$ \KK(\C,C_{0}(-)\rtimes G)\to \map_{\Sp^{\la}}( \KK(C_{0}(-)\rtimes G, {A}\rtimes G), \KK(\C , {A}\rtimes G))\, .$$ We interpret its pre-composition with 
 \eqref{qewfoihiew9ufj9qweofqewfewfq}  as  a natural transformation 
\begin{equation}\label{wefqwfefqfqewfqfe1}
\epsilon^{*} \colon \KK(C_{0}(-)\rtimes G, {A}\rtimes G) \to
\underline{ \KK(\C , {A}\rtimes G)}
\end{equation}
of functors on $ \pGTopc$ with values in $\Sp^{\la}$. The composition of \eqref{wefqwfefqfqewfqfe2} and
\eqref{wefqwfefqfqewfqfe1} is a natural transformation \begin{equation}\label{rgfewrfwweferf}
\mu^{\Kasp}_{A,-,\max} \colon K^{G,\an}_{A}(-)\to \underline{ \KK(\C , {A}\rtimes G)}
\end{equation} of functors from $\pGTopc$ to $\Sp$.
 We  now assume $\cF\subseteq \Fin$. 
 In general 
$E_{\cF}G^\cw$ does not belong to $ {\pGTopc}$ so that we can  not apply $K^{G,\an}_{A}$  or $ \mu^{\Kasp}_{-,A,\max}$  to $E_{\cF}G^\cw$ directly. Therefore we adopt the following definition.
\begin{ddd}\label{rigrgbjoirjvorgvfbdfgbfgbdfgbdfbg}
We let $$RK_{{A}}^{G,{\an}}\colon G\Top\to \Sp^{\la}$$ be the left Kan extension of
$(K_{{A}}^{G,{\an}})_{| {\pGTopc}}$ along the inclusion
$$ {\pGTopc}\to G\Top\, .$$
 \end{ddd} In particular 
 we have the diagram
\begin{equation*}
\xymatrix{ {\pGTopc}\ar[rr]^{K_{{A}}^{G,{\an}}}\ar[dr]&& \Sp^{\la}\,.\\& G\Top\ar@{-->}[ur]_{RK_{{A}}^{G,\an}}\ar@{}[u]^{\Rightarrow}&}
\end{equation*}

The following {definition} introduces the spectrum-valued refinement of the classical  Kasparov assembly map as introduced in \cite{kasparovinvent,baum_connes_higson}.
 \begin{ddd}\label{wergoijowergerrrfrfrfrfeggwgw}
The  Kasparov   assembly map associated to $G$, $\cF$ and $A$
is defined as  the map $$ \mu^{\Kasp}_{{A,\cF},\max} \colon RK^{G,{\an}}_{{A}} (E_{\cF}G^\cw) \to   \KK(\C, A\rtimes G)$$ induced by
the  natural transformation in \eqref{rgfewrfwweferf}.
We further define
$$ \mu^{\Kasp}_{{A,\cF}}\colon RK^{G,{\an}}_{{A}} (E_{\cF}G^\cw) \to   \KK(\C, A\rtimes_{r} G)$$
as the composition of $ \mu^{\Kasp}_{{A,\cF},\max}$ with the canonical morphism
$A\rtimes G\to A\rtimes_{r}G$.
  \end{ddd}
  Note that both versions of the Kasparov assembly map are, by construction, natural in the coefficient object $A$ in $\KKG$.

Using the functor $\kk_{\Ccat}\colon \Fun(BG,\nCcat)\to \KKG$ we consider the Kasparov assembly map as depending on a coefficient  $C^{*}$-category with $G$-action in place of $A$. Recall that we drop $\kk_{\Ccat}$ from the notation. 

Consider a morphism $\bC\to \bD$ in $\Fun(BG,\Ccat) $.
 \begin{lem}\label{wtogopwerfgerwfwref}
 If $\bC\to \bD$ is a Morita equivalence, then the induced morphism $ \mu^{\Kasp}_{{\bC,\cF}}\to  \mu^{\Kasp}_{{\bD,\cF}}$ is an equivalence.
 \end{lem}
 \begin{proof}
 By the functoriality of the  Kasparov assembly map we have a commutative square
$$ \xymatrix{ RK^{G,{\an}}_{{\bC}} (E_{\cF}G^\cw) \ar[r]^-{ \mu^{\Kasp}_{{\bC,\cF}}}\ar[d] &  \KK(\C, \bC\rtimes_{r} G) \ar[d] \\  RK^{G,{\an}}_{{\bD}} (E_{\cF}G^\cw) \ar[r]^-{ \mu^{\Kasp}_{{\bD,\cF}}} &  \KK(\C, \bD\rtimes_{r} G) }$$
  It suffices to show that the vertical morphisms are equivalences. We start with the left vertical morphism.
  Note that
  $$RK^{G,{\an}}_{{\bC}} (E_{\cF}G^\cw) \simeq \colim_{W\subseteq E_{\cF}G^\cw} \KKG(C_{0}(W),\bC)\, ,$$
  where $W$ runs over the $G$-finite subcomplexes of $ E_{\cF}G^\cw$.
  By \cite[Prop.\ 1.26]{KKG} the objects $\kkG(C_{0}(W))$ of $\KKG$ are $G$-proper and hence ind-$G$-proper ({recall that we assume that the family $\cF$ is contained in $\Fin$).}
  By \cite[Thm.\ 1.32.8]{KKG} the functor  $\KKG(C_{0}(W),-)$ sends relative Morita equivalences to equivalences.
  Hence the left vertical arrow in the square above is equivalent to 
  the colimit of equivalences
  $$ \colim_{W\subseteq E_{\cF}G^\cw}   \KKG(C_{0}(W),\bC)\to  \colim_{W\subseteq E_{\cF}G^\cw}  \KKG(C_{0}(W),\bD)$$
  and hence itself an equivalence.
  
  The right vertical arrow in the square is an equivalence since $-\rtimes_{r}G$ preserves Morita equivalences by \cite[Prop.\ 16.11]{cank},  and    
   $\KK(\C,-)$ sends Morita equivalences to equivalences by \cite[Thm.\ 16.18]{cank}.
    \end{proof}

  \begin{ex}
  Assume that  $A$ is an object in $\Fun(BG,\Calg)$ and set $\bC \coloneqq \Hilb_{c}(A)$ in $\Fun(BG,\nCcat)$.
  Then by Lemma \ref{eriughwierewfwerfwf}.\ref{wergjiergjosfdgdfg2} we have a Morita equivalence $A\to  (\bC^{u})^{(G)}$ induced by the canonical inclusion. We then have an equivalence
 \begin{equation}\label{fdvsfdvvsfdvav} \mu^{\Kasp}_{{A,\cF}}\stackrel{\simeq}{\to}  \mu^{\Kasp}_{{(\bC^{u})^{(G)},\cF}}
\end{equation}  
by Lemma \ref{wtogopwerfgerwfwref}.
\hB  
\end{ex}

We now derive the analytic assembly map \eqref{fwfqwefedewdedwqedqewd} associated to  {$\bC $ in $\Fun(BG,\nCcat)$.}
The composition of the two transformations $-\rtimes G\to -\rtimes_{r}G$ and 
$\id\to \Idem$     
 {yields} a morphism of exact sequences 
 \begin{equation}\label{rfreffewqefewfwefewffwfwfwefewfewfewfwfwefewfewfqfdfdfwefqwef} 
\xymatrix{0\ar[r]&
  {\bC}^{(G)}_{\std}\rtimes G \ar[d] \ar[r] &  {\bM\bC}^{(G)}_{\std}\rtimes G \ar[d] \ar[r]&   \bQ^{(G)}_{\std}\rtimes G \ar[d]^{\ctc} \ar[r]&0\\0\ar[r]& \Idem( {\bC}^{(G)}_{\std}\rtimes_{r} G)\ar[r]& \Idem ( {\bU} )\ar[r]& \frac{\Idem(  {\bU}) }{\Idem( {\bC}^{(G)}_{\std}\rtimes_{r} G)}   \ar[r]&0}
\end{equation}
{where the middle vertical arrow   is the composition of \eqref{asvasvadvadsvadvdvoihjio} with the inclusion $\bU\to \Idem(\bU)$.}
In the upper line we also used that the functor $-\rtimes G$ is exact.
The functor $\ctc$  will be called the change of target categories functor.

The change of target categories functor $\ctc$ in \eqref{rfreffewqefewfwefewffwfwfwefewfewfewfwfwefewfewfqfdfdfwefqwef} yields the first morphism in the following composition. The second is  the boundary map associated to the second exact sequence in \eqref{rfreffewqefewfwefewffwfwfwefewfewfewfwfwefewfewfqfdfdfwefqwef}.
Finally, {the left-pointing morphism is an equivalence by}   the Moria invariance of $ \KK(\C,-)= \Kcat$ \cite[Thm.\ {16.18}]{cank}: \begin{align}
\KK(\C,\bQ^{(G)}_{\std}\rtimes G) & \xrightarrow{\ctc}  \KK(\C,  \frac{\Idem(  {\bU} )}{\Idem( {\bC}^{(G)}_{\std}\rtimes_{r} G)} )\label{wefqwfefqfqewfqfe}\\
& \xrightarrow{\phantom{\ctc}}  \Sigma \KK(\C, \Idem( {\bC}^{(G)}_{\std}\rtimes_{r} G))\xleftarrow{\simeq}\Sigma \KK(\C,  {\bC}^{(G)}_{\std}\rtimes_{r} G)\notag\, .
\end{align}
 
{We  now specialize {the assembly maps introduced in Definition \ref{wergoijowergerrrfrfrfrfeggwgw}}  to $A=\kkG(\bQ_{\std}^{(G)})$, but we will
drop the symbol $\kkG$   in order to shorten the formulas. We use that $K^{G,\an}_{\bQ^{(G)}_{\std}}= K^{G,\An}_{\bC}$, compare \eqref{qwefqwefewdewdwqdwedqewd} and \eqref{vfdsvsvvsvsdvdsadsvdsva}.}
\begin{ddd} \label{qeriughioergewgergwer9}
We define the natural transformation
{$$\Ass^{\an}_{{\bC},-} \colon K_{\bC}^{G,\An}(-) \stackrel{\mu^{\Kasp}_{{\bQ_{\std}^{(G)},-,}\max}}{\to}
 \underline{  \KK(\C,  \bQ_{\std}^{(G)}\rtimes  G)}\stackrel{\eqref{wefqwfefqfqewfqfe}}{\to}
 \underline{\Sigma \KK(\C,  {\bC}^{(G)}_{\std}\rtimes_{r} G)}$$ 
of functors from $ \pGTopc$ to $\Sp^{\la}$.}  
We then define $\Ass^{\an}_{{\bC},X,*} \coloneqq \pi_{*}(\Ass^{an}_{{\bC},X})$.
\end{ddd}

 We now use Definition \ref{rigrgbjoirjvorgvfbdfgbfgbdfgbdfbg} for $RK^{G,\An}_{\bC}=
 RK^{G,\an}_{\bQ^{(G)}_{\std}}$. 
     \begin{ddd}\label{wergoijowergerreggwgw}
The analytic assembly map associated to $G$, $\cF$ and {$ \bC$}
is defined as  the map $$\Ass^{\an}_{{\bC},\cF} \colon RK^{G,\An}_{\bC} (E_{\cF}G^\cw) \to  \Sigma \KK(\C, {\bC}^{(G)}_{\std}\rtimes_{r} G)$$ induced by  {the natural  transformation} $\Ass^{\an}_{{\bC}}$ in Definition \ref{qeriughioergewgergwer9}.
 \end{ddd}

The assembly maps $\Ass^{\an}_{\bC}$ and  $\Ass^{\an}_{\bC,\cF}$ depend naturally on the coefficient category $\bC$ in $\Fun(BG,\nCcat_{\ndeg,\eadd,\omega\add})$.
  
\section{\texorpdfstring{$\boldsymbol{C^{*}}$}{Cstar}-categorical model for the homotopy theoretic assembly map}\label{ggerigjogsdfgdfgds}

The homotopy theoretic assembly map $\Ass_{{\bC},\cF}^{h}$ in  Definition \ref{wrtiogrgergfgs} is defined  
  in terms of the equivariant homology theory  $K\bC^{G}$.
On the other hand, the analytic assembly map $\Ass^{\an}_{{{\bC},\cF}}$ is constructed in Definition \ref{wergoijowergerreggwgw} in terms of $\KK$-theory.  Our   goal is to compare these two assembly maps.  As a first step, in this section we  will construct  an assembly map $\Ass_{X}^{\Theta}$ 
induced by an explicit  functor $\Theta_{X}$ between $C^{*}$-categories and show that it is equivalent to the homotopy theoretic assembly map $\Ass^{h}_{{\bC},X}$ on $G$-finite $G$-simplicial complexes. 
$\Ass_{X}^{\Theta}$ also depends on  $\bC$, but we drop this subscript  from the notation   in order to simplify the notation.

{Let $\bC$ be in $\Fun(BG,\nCcat)$ and assume that it is  effectively additive and admits  countable   AV-sums. In the following we will use  the $C^{*}$-category $\bU$ defined 
in Definition \ref{witgjierojgoewrfre} which contains $\bC^{(G)}_{\std}\rtimes_{r}G$ as an ideal,   and the morphism
 $\sigma \colon {\bM\bC^{(G)}_{\std} \rtimes} G\to \bU$ from \eqref{asvasvadvadsvadvdvoihjio}.}
 Recall the Definition \ref{wtogjoergrwegreggwege} of the functor $\bCgtsmc \colon G\BC\to \Ccat$. 
   Let $X$ be in $G\BC$. For an object $(C,\rho,\mu)$ in $   \bCgtsmc(X\otimes G_{can,min})$ we use the abbreviation  \begin{equation}
\mu_{g} \coloneqq \mu(X\otimes \{g\})
\end{equation}
denoting a  projection in $\bM\bC$ on $C$. 
We refer to Proposition \ref{weoigjweogergrrewgwregwrg} below for  the verifications 
 related with the following definition.
 \begin{ddd}\label{eoigjoeirgrwegreregrgewrg}
 We define a functor   
$$\Theta_X \colon   \bCgtsmc(X\otimes G_{can,min})\to {\Idem( \bU)}\, .$$
 as follows:
\begin{enumerate}
\item objects: The functor $\Theta_X$ sends the object $(C,\rho,\mu)$ in $  \bCgtsmc(X\otimes G_{can,min})$ to the object $(C,\rho,{ p})$ in  ${\Idem( \bU)}$, where \begin{equation}\label{qewfeqwfpokpoqwefqwefqewf}
  p:= {\sigma}( \mu_{e},e)\, .
\end{equation} 

\item morphisms:
The functor $\Theta_{X}$ sends the morphism $A \colon (C,\rho,\mu)\to (C',\rho',\mu')$  in $  \bCgtsmc(X\otimes G_{can,min})$ to the morphism \begin{equation}\label{ewqifhwqeifewfewfwefwefqwef}
\Theta_{X}(A):=\sum_{g\in G} {\sigma} (    \mu'_{g^{-1}}  A \mu_{e}   ,g) \colon (C,\rho, p)\to (C',\rho',   p')
\end{equation}  in ${\Idem( \bU)}$. 
 \end{enumerate}
 \end{ddd}
{For the interpretation of the infinite sum in \eqref{rgoijerwoggwergregwgreg} we refer to the proof of Lemma \ref{rgoijerwoggwergregwgreg} below.}
 Let $G\BC_{\bd}$ denote the full category of $G\BC$ of bounded $G$-bornological coarse spaces. 

 \begin{prop}\label{weoigjweogergrrewgwregwrg}\mbox{}
 \begin{enumerate}
 \item\label{ewtgojkropgwergregergwerg} {For every $X$ in $G\BC$, t}he functor $\Theta_{X}$ is well-defined.
 \item  \label{ewtgojkropgwergregergwerg1} The family $(\Theta_{X})_{X\in G\BC}$ is a natural transformation \begin{equation}\label{grpogkprewgref}
\Theta \colon \bCgtsmc(-\otimes G_{can,min})\to \underline{ {\Idem(\bU)}}
\end{equation}
of functors from $G\BC$ to $\nCcat$.
\item \label{ewtgojkropgwergregergwerg2} The transformation $\Theta$ restricts to a transformation  
\begin{equation}\label{grpogkpfrfrfrfrfrewgref}
\Theta \colon \bCgtsmc(-\otimes G_{can,min})\to \underline{\Idem({ {\bC}^{(G)}_{\std}\rtimes_{r}G)}} 
\end{equation}
of functors from $G\BC_{\bd}$ to $\nCcat$.

 \end{enumerate}
 \end{prop}

  \begin{proof}
 We start with Assertion  \ref{weoigjweogergrrewgwregwrg}.\ref{ewtgojkropgwergregergwerg}. 
  Since $X\times G$ is a free $G$-set and $(C,\rho,\mu)$ is locally finite it follows that
$(C,\rho)$ belongs to ${\bC}^{(G)}_{\std}$. 
Furthermore,   $  p$  belongs to $\bU$ since $\mu_{e}$ belongs to $\bM\bC$.
 Consequently, $(C,\rho,  p)$ is a well-defined object in  ${\Idem(  \bU)}$.

{The following lemma finishes the verification   that $\Theta_{X}$ is a well-defined functor between $C^{*}$-categories and therefore  proves   Assertion  \ref{weoigjweogergrrewgwregwrg}.\ref{ewtgojkropgwergregergwerg}.}

\begin{lem} \label{rgoijerwoggwergregwgreg}
The  {formula} \eqref{ewqifhwqeifewfewfwefwefqwef} {determines an isometric}   map {$\Theta_{X}(-)$} on morphism spaces {which is compatible with the composition and the involution.}
\end{lem}
\begin{proof}  We first observe that if {$A\colon (C,\rho,\mu)\to (C',\rho',\mu')$} has controlled propagation
then the sum  in  \eqref{ewqifhwqeifewfewfwefwefqwef}  has finitely many non-zero terms {which all belong to $\bU$ since $A$ belongs to $\bM\bC$.}

%

It follows from Definition \ref{qeroigergeggweerg}.\ref{qeroighjoergfqewfqf} and \cite[Lem.\ {7.10}]{cank} that 
 $C$ is isomorphic to the orthogonal {AV}-sum of the images of the family of projections  $(\mu_{g})_{g\in G}$.  {Using \cite[Lem.\ 7.8]{cank} we therefore get  a multiplier isometry}
 \begin{equation}\label{gwerfwerfwerfwerf}u\colon C \to  \bigoplus_{g\in G} C\, , \quad u \coloneqq \sum_{g\in G} e_{g}\mu_{g}  \, ,
\end{equation}
where the sum converges strictly.
We have an analogous {multiplier isometry} $u'\colon C' \to  \bigoplus_{g\in G} C'$.
{Still assuming that $A$ is controlled, we} calculate by using \eqref{wfoizeqwfu98u9e8wfqewf} (saying that $g\cdot \mu_{h}=\mu_{gh}$ for all $g,h$ in $G$) and the  $G$-invariance of $A$ {(saying that   $g\cdot A=A$  for all $g$ in $G$)} that 
\begin{equation} \label{wthoigwjopgwrregw}
 \Theta_{X}(A):=\sum_{g\in G} {\sigma(   \mu'_{g^{-1}}  A \mu_{e}   ,g)} =%
  u' A  u^{*} \, .
\end{equation}
{Since $\bU$ is closed in $\bL^{2}(G,\bC^{(G)}_{\std})$,}  this formula shows that 
{$\Theta_{X}$ extends by continuity} {to an isometric map  defined on all morphisms in  $\bCgtsmc(X\otimes G_{can,min})$ with values in $\bU$.}
{Using the equality \begin{equation}\label{gwergerfwerfrewf}uu^{*}= p
\end{equation}   and  \eqref{ewqifhwqeifewfewfwefwefqwef}
 we see that
 $${  p'}\Theta_{X}(A)
 = \Theta_{X}(A)  p= \Theta_{X}(A)\, .$$}
 {Altogether we obtain an
 isometric} map 
$$\Theta_{X}(-) \colon \Hom_{ \bCgtsmc(X\otimes G_{can,min})}((C,\rho,\mu),(C',\rho',\mu'))\to \Hom_{\Idem(\bU)}( (C,\rho,  p), (C',\rho', p'))\, .$$

    We finally show that $\Theta_{X}(-)$ is compatible with the composition and the involution. Let $A \colon (C,\rho,\mu)\to (C',\rho',\mu')$  and $A' \colon (C',\rho',\mu')\to (C'',\rho'',\mu'')$
 be   morphisms in $\bCgtsmc(X\otimes G_{can,min})$.
 Since $\Theta_{X}$ is  continuous, {as shown above,}   we can assume for simplicity that the morphisms are controlled. 
 We then calculate   using that $u$
  and $u'$ are isometries and \eqref{wthoigwjopgwrregw}  that 
 \[\Theta_{X}(A')\Theta_{X}(A)=\Theta_{X}(A'A)\, , \quad \Theta_{X}(A)^{*}=\Theta_{X}(A^{*})\, .\qedhere\]
%
 \end{proof}

 In order to see the Assertion  \ref{weoigjweogergrrewgwregwrg}.\ref{ewtgojkropgwergregergwerg1} we consider a map
 $f \colon X\to X'$ in $G\BC$. Then $\bCgtsmc(f)(C,\rho,\mu)=(C,\rho,f_{*} \mu)$.
 We observe by inspection of the definitions  that
 $$\Theta _{X'}(\bCgtsmc(f)(C,\rho,\mu))=\Theta_{X}(C,\rho,\mu)\, , \quad \Theta_{X}(\bCgtsmc(f)(A))=\Theta_{X}(A)\, .$$
 
 We finally show  Assertion  \ref{weoigjweogergrrewgwregwrg}.\ref{ewtgojkropgwergregergwerg2}. 
 If $X$ is bounded, then
 $X\times \{g\}$ is a bounded subset of $X\otimes G_{can,min}$ for every $g$ in $G$.
 Consequently,
 $\mu_{g}$ belongs to ${\bC}$, see the explanations in  Remark \ref{rqfqiofioewjeowfewfqewfqwefe}. 
 Every summand of 
 \eqref{ewqifhwqeifewfewfwefwefqwef}
is a morphism in $ {\bC}_{\std}^{(G)}\rtimes_{r}G$. {Since
$\bC_{\std}^{(G)}\rtimes_{r}G$ is closed in $\bU$  we conclude that}
 $\Theta_{X}$ takes values in the {wide subcategory
 $\Idem({\bC}_{\std}^{(G)}\rtimes_{r}G)$ of $\Idem(\bU)$}, {provided $X$ is bounded}.
 \end{proof}

For $X$ in $G\UBC_{\bd}$ we will also write 
\begin{equation}\label{fqoifoifjoeiwfjqowfqwefqwefqwefq}
\Theta_{X}\colon \bCgtsmc(\cZ\subseteq \cO(X)\otimes G_{can,min})\to {\underline{\Idem(\bC^{(G)}_{\std}\rtimes_{r}G)}}
\end{equation}
for the  restriction of the functor $\Theta_{\cO(X)}$ to the ideal $\bCgtsmc(\cZ\subseteq \cO(X)\otimes G_{can,min})${, see} \eqref{f23r8z89fz2893zf892334234f2} for the notation.  

We now apply the functor $\Kcat(-)=\KK(\C,-)$ to the transformations \eqref{grpogkprewgref} and \eqref{grpogkpfrfrfrfrfrewgref}. 
Using the    
  Definition \ref{qrogijeqoifefewfefewqffe}  of  $K\bC\cX^{G}$ in order to rewrite the domain and the Morita invariance of $\Kcat(-)$ {together with \cite[Prop.\ 17.4 \& 17.8]{cank}}   in order to remove $\Idem(-)$  in the target we get the assertions of the following corollary.
\begin{kor}\label{wopijgowergerfwerfwerfewrf}\mbox{} 
\begin{enumerate} \item 
We have a natural transformation
\begin{equation}
\theta \colon K\bC\cX^{G}_{G_{can,min}} \to \underline{\KK(\C, {\bU}) }
\end{equation} of functors from $G\BC$ to $\Sp^{\la}$.
\item \label{wegiojreogewrfwerffw}
We have a natural transformation
\begin{equation}
\theta \colon K\bC\cX^{G}_{G_{can,min}} \to \underline{\KK(\C, {\bC}^{(G)}_{\std}\rtimes_{r}G)}
\end{equation}
of functors from $G\BC_{\bd}$ to $\Sp^{\la}$.
\end{enumerate}
\end{kor}

 \begin{prop}\label{wrtoihjwogregergwergwreg}
 The  morphism  
\begin{equation}\label{qwefweofjwpoefeqfeqwfqwe}
\theta_{*} \colon K\bC\cX^{G}_{G_{can,min}}(*)\to \KK(\C,  {\bC}^{(G)}_{\std}\rtimes_{r}G)
\end{equation}
  is an equivalence.
\end{prop}
\begin{proof}
The proof is very similar to the proof of Proposition \ref{4oijotrherhrthrhe}. But the difference is that here $G$ is infinite while {in Proposition \ref{4oijotrherhrthrhe}} $H$ was finite. 
  By definition the morphism  in question is 
$$ K \bC\cX^{G}_{G_{can,min}}(*)\simeq \Kcat(\bCgtsmc(G_{can,min}))\stackrel{\Kcat(\Theta_{*})}{\to} \Kcat(\Idem({\bC}^{(G)}_{\std}\rtimes_{r}G)) 
\,.$$

We will   construct a factorization of $\Theta_{*} $  as
$$\bCgtsmc(G_{can,min}) \to \bD\to \Idem(\bD)\to \Idem( {\bC}^{(G)}_{\std}\rtimes_{r}G)\, ,$$ where  $\bCgtsmc(G_{can,min}) \to \bD$ is a weak Morita equivalence, $\bD\to \Idem(\bD)$ is a {relative idempotent completion}, and $\Idem(\bD)\to  
 \Idem(  {\bC}^{(G)}_{\std}\rtimes_{r}G)$   is a  {unitary}  equivalence. Since $\Kcat$ sends   functors with  any of these properties to equivalences \cite[Sec. 14 -16]{cank} the assertion then follows.

\begin{lem}\label{reoithgjowrtegwrgwergreg}
 $\Theta_{*}$ is fully faithful.
 \end{lem}
 \begin{proof}
{By Lemma \ref{rgoijerwoggwergregwgreg} 
the functor  $\Theta_{*}$  is an isometric inclusion on morphisms. It remains to show that it is surjective.}
 
  Let $(C,\rho,\mu)$ and $(C',\rho',\mu')$ be two objects of $\bCgtsmc(G_{can,min})$.
  Note that $\Theta_{*}(C,\rho,\mu)=(C,\rho,  p)$ and $\Theta_{*}(C',\rho',\mu')=(C'\rho',  p')$ in $\Idem(  {\bC}^{(G)}_{\std}\rtimes_{r}G)$.
 Let $A \colon (C,\rho,  p) \to (C'\rho', p')$ be a morphism  in $\Idem( {\bC}^{(G)}_{\std}\rtimes_{r}G)$. We will construct a morphism $\hat A \colon (C,\rho,\mu)\to (C',\rho',\mu')$ in 
 $\bCgtsmc(G_{can,min})$ such that $\Theta_{*}(\hat A)=A$.
 
 Note that $A$ is a morphism $(C,\rho)\to (C',\rho')$ in $  {\bC}^{(G)}_{\std}\rtimes_{r}G  $ which in addition satisfies   $ \tilde p' A\tilde p=A$.  
{There is a unique family $(A_{g})_{g\in G}$ of morphisms    $A_{g}\colon C\to C$ in $\bC$ such that 
$$A=\sum_{g\in G} \sigma(A_{g},g)\, ,$$
where the sum converges in norm in $\bU$.   From the equality 
$$\sum_{g\in G} \sigma(A_{g},g)=A=
   p' A  p  = \sum_{g\in G}  \sigma(\mu'_{g^{-1}} A_{g}\mu_{e},g)  
$$
 we conclude that
\begin{equation}\label{ewqfqwedqwedqwed}
\mu'_{g^{-1}} A_{g}\mu_{e}=A_{g}
\end{equation} for all $g$ in $G$.
Using the notation from   \eqref{wthoigwjopgwrregw} 
 we define $$\hat A:=u^{\prime,*} \sum_{g\in G} \sigma(A_{g},g)u$$ in $\Hom_{\bM\bC}(C,C')$.  Inserting all definitions we get $\hat A=\sum_{k\in G}\sum_{g\in G}  k\cdot     A_{g}$
 where the $g$-sum converges in norm and the $k$-sum converges strictly. This formula shows that
 $\hat A$ is $G$-invariant. 
  Furthermore, by \eqref{ewqfqwedqwedqwed} for every  $g$ in $G$ the support of $ \sum_{k\in G}k\cdot     A_{g}$ is the coarse entourage $G(\{(g^{-1},e)\})$ of $G_{can,min}$. It follows that $\hat A$ 
  can be approximated in norm by controlled and invariant operators, i.e., we have $\hat A\in \bCgtsmc(G_{can,min})$.}
 By construction we have $\Theta_{*}(\hat A)=A$.
 This finishes the verification that $\Theta_{*}$ is full faithful. \end{proof}

  For every free $G$-set $Y$,   
  every  subset $F$ of $Y$, and  every object $(C,\rho,\mu)$ in $\tCglf(Y_{min})$
   we can consider the projection
$  p_{F}:={\sigma}(\mu(F),e)$ on $(C,\rho)$ considered as an object of  {$\bU$}. 
We let $\bD$ be the full subcategory of
$\Idem( {\bC}^{(G)}_{\std}\rtimes_{{r}}  G)$  
of objects of the form
$(C,\rho, p_{F})$ for some choice of $Y$, $F$ and $(C,\rho,\mu)$ as above. 
We can consider $Y=G$ and $F=\{e\}$. Then
$  p= p_{\{e\}}$ so that $\bD$ contains the image of $\Theta_{*}$. 

Recall the notion of a weak Morita equivalence from \cite[Def.\ {18.3}]{cank}.
\begin{lem} The functor
$\bCgtsmc(G_{can,min})\to \bD$ is a weak Morita equivalence.
\end{lem}
\begin{proof}
It follows from Lemma \ref{reoithgjowrtegwrgwergreg} that  the morphism in question is fully faithful. It remains to show that 
set of objects $\Theta_{*}(\Ob(\bCgtsmc(G_{can,min})))$ is weakly generating \cite[Def. 16.1]{cank}.
In order to simplify the notation we write $  p_{g}:= p_{\{g\}}$ and note that $  p= p_{e}$.
We have
$${\sigma}(\id_{C},g)  p_{e} {\sigma}(\id_{C},g)^{*}=  p_{g}\, .$$
This shows that for every $g$ in $G$ and $(C,\rho,\mu)$ in $\bCgtsmc(G_{can,min})
$ the  object 
$(C,\rho,  p_{g})$ in $\Idem( {\bC}^{(G)}_{\std}\rtimes_{r}G)$ is isomorphic to the object  $(C,\rho,  p)$  
  which belongs to the image of $\Theta_{*}$.
Furthermore, 
for every finite subset $F$  of some free $G$-set $Y$
and $(C,\rho,\mu)$ in  $\tCglf(Y_{min})$   the object
$(C,\rho, p_{F})$   is isomorphic to a finite sum of objects in the image of $\Theta_{*}$.

Let now $(C,\rho,  p_{F})$ be any object of $\bD$ and 
$(A_{i})_{i\in I}$ be a finite family of morphisms in $\Idem( {\bC}^{(G)}_{\std}\rtimes G) $ with target $(C,\rho, p_{F})$.  Let $\epsilon$ be in $(0,\infty)$. 
We write $A_{i}=\sum_{h\in G} {\sigma}(A_{i,h},h)$ where  {the} $A_{i,h}$ belong to ${\bC}$. Since these sums converge in norm and $I$ is finite 
  there exists a finite subset $F'$ of $G$ such that
$\|A_{i}- \sum_{h\in F'} {\sigma}(A_{i,h},h)\|\le \epsilon/2$ for all $i$ in $I$.
{Since $\sum_{y\in Y}\mu(\{y\})$ converges strictly to $\id_{C}$}
  we can find a finite subset $F''$ of $Y$ such that $$\|A_{i,g}-\mu(g^{-1}F'')A_{i,g}\| \le \frac{\epsilon}{2|F'| }$$ for all $i$ in $I$ and $g$ in $F'$. 
Then
$\|A_{i}- p_{F''}A_{i}\|\le \epsilon$ for all $i$ in $I$.
\end{proof}

 Recall the definition \cite[Def.\ {17.5}]{cank} of a relative idempotent completion. 
  In the following we let
   $\bE$ be 
     the full subcategory   of ${\Idem( \bU)}$ with the same objects as $\bD$.  Then $\bD$ is an ideal in $\bE$ and
     the  idempotent completion
  $\bD\to \Idem(\bD)$   is understood relative to $\bE$. {We summarize this in the following corollary:}

\begin{kor} The functor 
$\bD\to \Idem(\bD)$ is a relative idempotent completion.
\end{kor}

  
\begin{lem}
The inclusion
$\Idem(\bD)\to \Idem( {\bC}^{(G)}_{\std}\rtimes_{r}G)$ is a {unitary} equivalence.
\end{lem}
\begin{proof} {We  apply the characterization of unitary equivalence given in \cite[Rem.\ 3.20.3]{cank}. We consider the square $$\xymatrix{\Idem(\bD)\ar[r]\ar[d]&\Idem( \bE) \ar[d]\\ \Idem( {\bC}^{(G)}_{\std}\rtimes_{r}G)\ar[r]&\Idem({\bU})}$$
Its horizontal morphisms are ideal inclusions by construction.
It remains to show that the right vertical morphism is a unitary equivalence in $\Ccat$.
In fact, it is fully faithful by definition. Since $\bE$ contains the all objects  of the form
$(C,\rho,\tilde p_{Y})=(C,\rho)$  for  free $G$-sets $Y$  and $(C,\rho,\mu)$ in $\tCglf(Y_{min})$ conclude that it is also essentially surjective.}  
\end{proof}

This finishes the proof of Proposition \ref{wrtoihjwogregergwergwreg}.
 \end{proof}

We now apply the cone sequence \eqref{wefqwefqewewfwdqewdewdqede} to the functor 
 $K\bC\cX^{G}_{G_{can,min}}$ and obtain a boundary map
 $$\partial^{\Cone} \colon K\bC\cX^{G}_{G_{can,min}}(\cO^{\infty}(-))\to \Sigma K\bC\cX^{G}_{G_{can,min}}(-)$$
between functors
 from $G\UBC$ to $\Sp^{\la}$.

\begin{ddd}\label{werigowerferfrwef}We denote by $G\UBC_{\bd}$ the full subcategory of $G\UBC$ of bounded $G$-uniform bornological coarse spaces.
\end{ddd}
We have a forgetful functor $G\UBC_{\bd}\to G\BC_{\bd}$ which we always drop from the notation. We can also restrict the cone boundary
transformation along the inclusion $G\UBC_{\bd}\to G\UBC$.
  
 Let $X$ be in $G\UBC_{\bd}$. 
 We use the Corollary \ref{wopijgowergerfwerfwerfewrf}.\ref{wegiojreogewrfwerffw} in order to see the that the natural transformation defined  below  takes values in the correct target.
 \begin{ddd}\label{kopeherthrggertg}
 We define the natural transformation 
 $$\Ass^{\Theta} \coloneqq \theta \circ \partial^{\Cone} \colon K\bC\cX^{G}_{G_{can,min}}(\cO^{\infty}(-))\to  \underline{\Sigma \KK(\C,   {\bC}^{(G)}_{\std}\rtimes_{r}G)}$$
 of functors from $G\UBC_{\bd}$ to $\Sp^{\la}$.
   \end{ddd}

We consider the functor \begin{equation}
\tilde \iota\colon G\Set\to G\UBC_{\bd}\, , \quad S\mapsto S_{min,max,disc}\, ,
\end{equation}  
where $disc$ stands for the discrete uniform structure.  
A $G$-simplicial complex is a simplicial complex with a simplicial $G$-action.
We assume that if $g$ in $G$ fixes a point in the interior of a simplex, then it fixes the whole simplex pointwise.  This can always be ensured by going over to a barycentric subdivision.
We let $G\Simpl$
denote the category of $G$-simplicial complexes and simplicial equivariant maps.  

Let $G\Simpl^{\textrm{fin-dim}}$ 
 be the full subcategory of $G\Simpl$ of finite-dimensional $G$-simplicial complexes. 
We have a natural functor
$$\tilde s \colon G\Simpl\to G\UBC_{\bd}$$
which sends a $G$-simplicial complex $X$ to the $G$-uniform bornological coarse space
$\tilde s(X)$ 
given by $X$ with the coarse and the uniform structures induced by the spherical path metric, and with the maximal bornology. We have a commutative diagram 
of canonical functors   \begin{equation}\label{retgertg365ztheh}\xymatrix{G\Set\ar[rr]^{  {\tilde \iota}}\ar[dr]^{{(1)}}&&G\UBC_{\bd}\ar[dr]^{r}&&\\&G\Simpl^{\textrm{fin-dim}}\ar[dr]^{f}\ar[ur]^{ s}\ar[rr]^(0.3){ t}&&G\Top\\&&G\Simpl \ar[uu]^(0.3){\tilde s}\ar[ru]^{\tilde  t}&&}
\end{equation}
where  arrow  $(1)$ interprets a $G$-set as a zero-dimensional $G$-simplicial set, and arrow $r$ sends a uniform bornological coarse space to the underlying $G$-topological space.

\begin{prop}\label{weqfiuhqiuwhefiqwefewewfqfwefqwef}
\mbox{}
\begin{enumerate}
\item \label{qwhioqwehfiouqewfqewqewfqf} The transformation ${\tilde  \iota}^{*}\partial^{\Cone} \colon  {\tilde \iota}^{*}K\bC\cX^{G}_{G_{can,min}}(\cO^{\infty}(-))\to\tilde \iota^{*}r^{*} \Sigma K\bC^{G}(-)$ of functors from $G\Set$ to $\Sp^{\la}$ is an equivalence.
\item \label{qfeiohqoifweqewfqf}  We have an  equivalence \begin{equation}\label{adsvkanjovasdvassdvavads}
s^{*}K\bC\cX^{G}_{G_{can,min}}(\cO^{\infty}(-))\simeq  t^{*}\Sigma K\bC^{G}(-)
\end{equation}  of functors from $G\Simpl^{\textrm{fin-dim}}$ to $\Sp^{\la}$.
\item \label{fopepofkopqwkfpoqwefkqwefqwefqe}  We have a commutative square of natural transformations
  \begin{equation}\label{qerffqwefwfwefewqfwf} \xymatrix{ {t^{*}}\Sigma K\bC^{G}(-)\ar[rr]^{{t^{*}}\Sigma \Ass_{{\bC}}^{h}}\ar@{-}[d]_{\simeq}^{\eqref{adsvkanjovasdvassdvavads}}&&\underline{\Sigma K\bC^{G}(*)}\ar@{-}[d]_{\eqref{wrtoihjwogregergwergwreg}}^{\simeq}\\   {s^{*}}K\bC\cX^{G}_{G_{can,min}}(\cO^{\infty}(-))\ar[rr]^{{s^{*}}\Ass^{\Theta}}&&  \underline{\Sigma \KK(\C,  {\bC}^{(G)}_{\std}\rtimes_{r}G)}} 
\end{equation}
between functors from $G\Simpl^{\textrm{fin-dim}}$ to $\Sp$
which  depends naturally on the coefficient category $\bC$ in $\Fun(BG,\nCcat_{\ndeg,\eadd,\omega\add})$.

\end{enumerate}
\end{prop}
 \begin{proof}
 By  \cite[Prop. 9.35]{equicoarse} for every $S$ in $G\Set$ we have an equivalence
 $$\partial^{\Cone}:K\bC\cX^{G}_{G_{can,min}}(\cO^{\infty}(\tilde \iota (S))\stackrel{\simeq}{\to} \Sigma K\bC\cX^{G}_{G_{can,min}}( \tilde \iota (S))\ .$$
 If $L$ is a locally finite subset of $S_{min,max}\otimes G_{can,min}$, then
 $L\cap (S\times \{e\})$ is finite. It follows that $L$ is a finite union of $G$-orbits.
By continuity of $K\bC\cX^{G}$ we have
$$
  \Sigma K\bC\cX^{G}_{G_{can,min}}( \tilde \iota (S))\simeq  \bigoplus_{T\in G\backslash S} 
   \Sigma K\bC\cX^{G}_{G_{can,min}}(  \iota (T) )\ ,$$
   where $\iota$ is as in \eqref{adfvoijqiorjvoifvafsvsdv}.
By Definition \ref{qroeigjoqergerqfewewfqewfeqwf}
$$   
   \Sigma K\bC\cX^{G}_{G_{can,min}}(  \iota (T) )\simeq \Sigma K\bC^{G}(r(\tilde \iota(T)))\ .$$
  Since the homology theory $K\bC^{G}$ sends disjoint unions of orbits to sums we conclude that
  $$ \bigoplus_{T\in G\backslash S} 
 \Sigma  K\bC^{G}(r(\tilde \iota(T)))\simeq \Sigma K\bC^{G}(r(\tilde \iota(S)) )\ .$$
 Combining these equivalences we get Assertion  \ref{weqfiuhqiuwhefiqwefewewfqfwefqwef}.\ref{qwhioqwehfiouqewfqewqewfqf}.
 
%
%

We now show Assertion  \ref{weqfiuhqiuwhefiqwefewewfqfwefqwef}.\ref{qfeiohqoifweqewfqf}.
Note that $ K\bC^{G}(-)$ in the statement is   the evaluation of an equivariant homology   theory defined on all of $G\Top$ by \eqref{adfiuhuihavfvasdcacsadcasca}.     The 
other functor $K\bC\cX^{G}_{G_{can,min}}(\cO^{\infty}(-))$ is defined on $G\UBC$. By  restricting $K\bC^{G}(-)$ along the forgetful functor $ G\UBC\to G\Top$  we can consider them on the same domain $G\UBC$. Assertion  \ref{weqfiuhqiuwhefiqwefewewfqfwefqwef}.\ref{qwhioqwehfiouqewfqewqewfqf} then provides an equivalence between the further restrictions of both functors  to  zero-dimensional simplicial complexes. We then argue that  this natural equivalence   canonically extends to an equivalence between these functors at least on $G\Simpl^{\textrm{fin-dim}}$ since they are both homotopy invariant 
 and excisive for cell-attachements.


%
%
 
We will construct the desired equivalence by   induction with respect to the dimension. 
We let $G\Simpl_{\le n}$ be the full subcategory of $G$-simplicial complexes of dimension $\le n$.
We let $s_{n}$ and $t_{n}$ denote the restrictions of $s$ and $t$ to $G\Simpl_{\le n}$.

The case of zero-dimensional simplicial complexes is  done by Assertion \ref{qwhioqwehfiouqewfqewqewfqf}.

 We assume  {now} that we have constructed an equivalence
  \begin{equation}\label{weergwrgkpdsdsdsdogergwe}
s_{n-1}^{*}K\bC\cX^{G}_{G_{can,min}}(\cO^{\infty}(-))\simeq  t_{n-1}^{*}\Sigma K\bC^{G}(-)
\end{equation}
for $n\ge 1$.
The induction step exploits the fact that
$s_{n}(X)$ in $G\Simpl_{\le n}$  has a canonical decomposition 
$(Y,Z)$ in $G\UBC_{\bd}$, where $Z$ is the disjoint union of $2/3$-scaled $n$-simplices,
and $Y$ is  the complement of the disjoint union of the interiors of the $1/3$-scaled $n$-simplices 
 (see the pictures in \cite[P.~80]{buen}).
 {We equip the subspaces $Y$ and $Z$ with the uniform bornological coarse structures induced from $s_{n}(X)$.}


Since both functors $ K\bC\cX^{G}_{G_{can,min}}(\cO^{\infty}(-))$ and
$  \Sigma K\bC^{G}(-)$ are excisive for such decompositions
we get push-out squares
\begin{equation}
\xymatrix{K\bC\cX^{G}_{G_{can,min}}(\cO^{\infty}(Y\cap Z)) \ar[r]\ar[d]&K\bC\cX^{G}_{G_{can,min}}(\cO^{\infty}(Z)) \ar@{..>}[d]\\K\bC\cX^{G}_{G_{can,min}}(\cO^{\infty}(Y)) \ar@{..>}[r]&K\bC\cX^{G}_{G_{can,min}}(\cO^{\infty}(X))}
\end{equation} and
\begin{equation}
\xymatrix{\Sigma K\bC^{G}  (Y\cap Z) \ar[r]\ar[d]&\Sigma K\bC^{G}(  Z)\ar@{..>}[d]\\
\Sigma K\bC^{G} ( Y)\ar@{..>}[r]&\Sigma K\bC^{G} ( X)}
\end{equation}
We now use that both functors are homotopy invariant.
The projection of $Z$ to the $G$-set $Z_{0}$ of barycenters is a homotopy equivalence in $G\Top$ and $G\UBC_{\bd}$. Similarly, there is a projection of $Y$ to the $(n-1)$-skeleton  $X_{n-1}$ of $X$ and a projection of $Y\cap Z$ to a disjoint union $(Y\cap Z)_{n-1}$ of boundaries of the $n$-simplices. These two maps are  
homotopy equivalences in  $G\Top$ and $G\UBC_{\bd}$.  These projections identify the bold parts of the push-out squares above canonically with the respective bold parts {of the push-out} squares below:
\begin{equation}\label{ewffdfadsfadfasdfasdf}
\xymatrix{K\bC\cX^{G}_{G_{can,min}}(\cO^{\infty}( (Y\cap Z)_{n-1})) \ar[r]\ar[d]&K\bC\cX^{G}_{G_{can,min}}(\cO^{\infty}( Z_{0}))\ar@{..>}[d]\\K\bC\cX^{G}_{G_{can,min}}(\cO^{\infty}( X_{n-1}))\ar@{..>}[r]&K\bC\cX^{G}_{G_{can,min}}(\cO^{\infty}(X))}
\end{equation} and   
\begin{equation}\label{ewffdfadsfadfasdfasdf1}
\xymatrix{\Sigma K\bC^{G}  ( (Y\cap Z)_{n-1}) \ar[r]\ar[d]&\Sigma K\bC^{G}(  Z_{0} )\ar@{..>}[d]\\\Sigma  K\bC^{G} (  Y_{n-1} )\ar@{..>}[r]&\Sigma K\bC^{G} (  X )}
\end{equation}
The induction hypothesis now provides an equivalence between the bold parts of 
\eqref{ewffdfadsfadfasdfasdf} and \eqref{ewffdfadsfadfasdfasdf1}. This equivalence then provides the desired equivalence of push-outs
$$K\bC\cX^{G}_{G_{can,min}}(\cO^{\infty}( X ))\simeq \Sigma  K\bC^{G} (  X)\, .$$
The whole construction is functorial in $X$. To see this interpret  the symbols $X,Y,Z$ as placeholders for entries of diagram valued functors.

\begin{rem}\label{kohprhertgrtgetr}
In order to give a more formal argument for naturality we could proceed   as in the proof of  Corollary  \ref{wegtopwgwgregergw}.
 Let $q:G\Set \to G\Simpl$ be the canonical inclusion.
Then we have a counit  morphism
$$q_{!}q^{*}\tilde s^{*}  K\bC\cX^{G}_{G_{can,min}}(\cO^{\infty}(-))\to  \tilde  s^{*}  K\bC\cX^{G}_{G_{can,min}}(\cO^{\infty}(-))\ .
$$ Using excision and homotopy invariance one checks that
$$f^{*}q_{!}q^{*}\tilde s^{*}  K\bC\cX^{G}_{G_{can,min}}(\cO^{\infty}(-))\to    s^{*}  K\bC\cX^{G}_{G_{can,min}}(\cO^{\infty}(-))\ .
$$
is an equivalence.
 Since $K\bC^{G}$ is an equivariant homology theory the counit  
 $$q_{!}q^{*}\tilde t^{*}K\bC^{G}\stackrel{\simeq}{\to} \tilde t^{*}K\bC^{G}$$
 is an equivalence.
 Finally, applying $q_{!}$ to  the equivalence from Assertion \ref{qwhioqwehfiouqewfqewqewfqf}. we get the equivalence
 $$q_{!}q^{*}\tilde s^{*}  K\bC\cX^{G}_{G_{can,min}}(\cO^{\infty}(-))\stackrel{\simeq}{\to}
 q_{!} q^{*}\tilde t^{*}K\bC^{G}\ .$$
 The desired equivalence is now given by
 $$s^{*}  K\bC\cX^{G}_{G_{can,min}}(\cO^{\infty}(-))\stackrel{\simeq}{\leftarrow}
f^{*}q_{!}q^{*}\tilde s^{*}  K\bC\cX^{G}_{G_{can,min}}(\cO^{\infty}(-))\stackrel{\simeq}{\to}
f^{*}q_{!} q^{*}\tilde t^{*}K\bC^{G}
\stackrel{\simeq}{\to}    t^{*}K\bC^{G}\ .$$
\hB

\end{rem}

 Assertion 
 \ref{weqfiuhqiuwhefiqwefewewfqfwefqwef}.\ref{fopepofkopqwkfpoqwefkqwefqwefqe}.
  becomes obvious if we 
  expand the square  \eqref{qerffqwefwfwefewqfwf} as follows 
  \begin{equation}
  \xymatrix{ {t^{*}}\Sigma K\bC^{G}(-)\ar[r]^{{t^{*}}\Ass^{h}_{{\bC}}} \ar[r]&\Sigma K\bC^{G}(*) \ar[dr]_{\simeq }^{ \theta_{*},\eqref{qwefweofjwpoefeqfeqwfqwe}} \\ {s^{*}}K\bC\cX^{G}_{c,G_{can,min}}(\cO^{\infty}(-))\ar@{-}[u]^{\simeq}_{\eqref{adsvkanjovasdvassdvavads}} \ar@/_2cm/[rr]^{ {s^{*}}\Ass^{\Theta} }\ar[r] &K\bC\cX^{G}_{c,G_{can,min}}(\cO^{\infty}(*))\ar[r]_-{  \Ass^{\Theta}_{*}}^{  }  \ar[u]^{\simeq}_{\partial^{\Cone},\eqref{adsvkanjovasdvassdvavads}}& \Sigma \KK(\C,  {\bC}^{(G)}_{\std}\rtimes_{r}G)} 
\end{equation}
 The left horizontal maps in the square are induced by the natural transformation $(-)\to \underline{\const_{*}}$ (see
 \eqref{vevopjrepovervfvdfv} for $\Ass^{h}_{\bC}$), and the upper-left square commutes by the naturality    statement  in Assertion   \ref{weqfiuhqiuwhefiqwefewewfqfwefqwef}.\ref{qfeiohqoifweqewfqf}.
 The upper right triangle commutes by the definition of $\Ass_{*}^{\Theta}$, and finally the lower triangle commutes by the naturality of $\Ass^{\Theta}$.
 \end{proof}

\section{\texorpdfstring{$\boldsymbol{C^{*}}$}{Cstar}-categorical model for  the analytic assembly map}\label{wtogpwgregwegreg}

At the end of this section we finish the proof of Theorem \ref{wtoiguwegwergergregwe}.

The analytic assembly map $\Ass^{\an}_{{\bC}}$ in Definition \ref{qeriughioergewgergwer9} was obtained using a construction on the level of spectrum-valued $\KK$-theory. If we precompose this assembly map with the Paschke transformation 
from Theorem \ref{troigjwrtoijgergrweg}, then we get a functor  whose domain is also expressed through
the coarse $K$-homology functor $K\bC\cX^{G}$ and therefore in terms of $C^{*}$-categories of controlled objects.
In the present section we construct an assembly map
$\Ass^{\Lambda}$ in terms of a natural functor $\Lambda$ between $C^{*}$-categories which models this composition.
We then relate $\Ass^{\Lambda}$ with both $\Ass^{\Theta}$ and $\Ass^{\an}_{{\bC}}$.
The intermediate objects also depend on  $\bC$, but we again drop this subscript   in their notation  in order to simplify the notation.

\begin{ddd}\label{wtrkogerfrfwerf}We let $G\UBC_{\pc}$ denote the full subcategory of $G\UBC$ of $G$-uniform bornological coarse spaces  which have the bornology of relative compact subsets and whose underlying $G$-topological space belongs to $\pGTopc$ introduced in Definition \ref{wrtoigjiowrgerferfw}.
\end{ddd} 
We consider $\bC$ in $\Fun(BG,\nCcat)$ and assume that it   is  {effectively} additive and admits countable AV-sums.
Let $X$ be in $G\UBC_{\pc}$ and choose
$\chi$ in $\cR(X)$, {where the functor $\cR$ is as in}  \eqref{afvoijvodvasdvsdvavdv}. If $ (C,\rho,\mu)$ is an object in $\bCgtsmc(\cO(X)\otimes G_{can,max})$, then we can consider the homomorphism $\phi \colon C_{0}(X)\to {\End_{\bM\bC}(C)}$ defined in \eqref{ewqfpojlqkrmeflkwefqef}. The sum  
\begin{equation}\label{ewqfoihioqwhejfioewqfewfeqfew}
p_{\chi} \coloneqq \sum_{m\in G}  {\sigma}(\phi(\chi) \phi(m^{*}\chi),m)
\end{equation} has finitely many non-zero terms and defines
 a projection
   on $(C,\rho)$  considered as an object   
in the $C^{*}$-category {$\bU$}  described in the Definition \ref{witgjierojgoewrfre}, where $\sigma$ is as in \eqref{asvasvadvadsvadvdvoihjio}. 
We refer to Proposition \ref{weoigjweogergrrewgwregfewfwwrg} for the necessary verifications related with the following definition.
\begin{ddd}\label{erwoighfgojrefqwfewfqwefeqwf}
We define a   functor
$$\Lambda_{(X,\chi)}\colon \bCgtsmc(\cO(X)\otimes G_{can,max})\to  \Idem({\bU}
)$$ in $\nCcat$ as follows: 
\begin{enumerate}
\item objects: The  functor  $\Lambda_{(X,\chi)}$ sends the object
$(C,  \rho,\mu)$ in $\bCgtsmc( \cO(X)\otimes G_{can,max})$ to the object
$(C,\rho,p_{\chi})$ in  $\Idem({\bU})$, where $p_{\chi}$ is as in 
 \eqref{ewqfoihioqwhejfioewqfewfeqfew}.
 
\item  morphisms: The  functor $\Lambda_{(X,\chi)}$ sends the 
 morphism 
  $A\colon (C,\rho,\mu)\to (C',\rho',\mu')$ in $\bCgtsmc( \cO(X)\otimes G_{can,max})$  
 to the morphism \begin{equation}\label{fqwefqoijofqwefqwefqef}
\Lambda_{(X,\chi)}(A):=\sum_{m\in G}  {\sigma} (  \phi'(m^{*}\chi)
  A \phi(\chi),m) 
\end{equation} in $\Idem({\bU})$.
\end{enumerate}
 \end{ddd}
{We refer to the proof of Lemma \ref{ezrhjieorhgtre} below for the 
interpretation of the infinite sum in \eqref{fqwefqoijofqwefqwefqef}.}

In order to state the naturality of $\Lambda_{(X,\chi)}$ we introduce the category
$G\UBC^{\cR}_{\pc}$ given by the Grothendieck construction of the functor $\cR$. Its objects are pairs $(X,\chi)$ of an object $X$ in $G\UBC_{\pc}$ and $\chi$ in $\cR(X)$, and a morphism $f\colon (X,\chi)\to (X',\chi')$ in $G\UBC_{\pc}^{\cR}$ is a morphism $f \colon X\to X'$ in $G\UBC_{\pc}$ such that $f^{*}\chi'=\chi$.  We have a forgetful functor $G\UBC_{\pc}^{\cR}\to G\UBC_{\pc}$ 
which we will not write explicitly in formulas.

 \begin{prop}\label{weoigjweogergrrewgwregfewfwwrg}\mbox{}
 \begin{enumerate}
 \item\label{ewtgojkropgwerefwefgregergwerg} {For every $(X,\chi)$ in  
 $G\UBC_{\pc}^{\cR}$},
  the functor  $\Lambda_{(X,\chi)}$ is well-defined.
 \item \label{qroijgoijerogwgergergwergwergw}  The family $(\Lambda_{(X,\chi)})_{(X,\chi)\in G\UBC_{\pc}^{\cR}}$ is a natural transformation
 $$\Lambda \colon \bCgtsmc(\cO(-)\otimes G_{can,max})\to  \underline{\Idem({\bU})}$$
 of functors from $G\UBC_{\pc}^{\cR}$ to $\Sp^{\la}$.
  \label{ewtgojkropgwergwefwfwefregergwerg1}
  \item\label{wetoiguwoefmw09i09m}
 The  transformation     restricts to a  {natural}  transformation      \begin{equation}
 \label{grpogkprwefwefwsdfsdfsdffewgref}
\Lambda \colon \bCgtsmc(\cZ\subseteq \cO(-)\otimes G_{can,max})\to \underline{\Idem( {\bC}^{(G)}_{\std}\rtimes_{r}G)}
\end{equation}
of functors from $G\UBC_{\pc}^{\cR}$ to $\Sp^{\la}$.
 \end{enumerate}
 \end{prop}
 \begin{proof}
 The structure of this proof is the same as for Proposition \ref{weoigjweogergrrewgwregwrg}.

 {We first observe that}   $(C,\rho,p_{\chi})$ is an object of $\Idem(\bU)$.
 \begin{lem} \label{ezrhjieorhgtre}
The {formula \eqref{fqwefqoijofqwefqwefqef}  determines  a continuous map of morphism spaces which is compatible with the composition and the involution.} 
\end{lem}
\begin{proof}
In analogy to \eqref{gwerfwerfwerfwerf} 
for every 
 $(C,\rho,\mu)$ in $\bCgtsmc( \cO(X)\otimes G_{can,max})$   we 
 consider
the isometry \begin{equation}\label{gwgwreferferfewf3}v\colon C\to \bigoplus_{g\in G}C\ , \quad v:=\sum_{g\in G} e_{g} \phi(g^{-1,*}\chi)\ .
\end{equation}
 Then similarly as \eqref{wthoigwjopgwrregw} we have \begin{equation}\label{erwfwerf25}\Lambda_{(X,\chi)}(A)=v^{\prime}Av^{*}
\end{equation}
and \begin{equation}\label{werfwerferfwrefwrwr}p_{\chi}=vv^{*}
\end{equation}  in analogy to \eqref{gwergerfwerfrewf}.  
\end{proof}

{This finishes the verification of Assertion \ref{weoigjweogergrrewgwregfewfwwrg}.\ref{ewtgojkropgwerefwefgregergwerg}.}
We continue {with} Assertion  \ref{weoigjweogergrrewgwregfewfwwrg}.\ref{qroijgoijerogwgergergwergwergw}.
Let $f \colon (X,\chi)\to (X',\chi')$ be a morphism in $G\UBC_{\pc}^{\cR}$ and note $\bCgtsmc(f)(C,\rho,\mu)=(C,\rho,f_{*}\mu)$.  We let $f_{*}(\phi) \colon C_{0}(X')\to \End_{\bC}(C)$  be the homomorphism defined with $f_{*}\mu$. Then we have the relation
$$f_{*}\phi(\theta')=\phi(f^{*}\theta')$$
for all $\theta'$ in $C_{0}(X')$. In particular
$(f_{*}\phi)(\chi')=\phi(\chi)$.
 This relation implies that
$p_{\chi}=p_{\chi'}$ and $\Lambda_{(X',\chi')}(\bCgtsmc(f)(A))=\Lambda_{(X,\chi)}(A)$ (note 
Definition \ref{wtogjoergrwegreggwege}.\ref{oiwejgiowegregwr}).
 These equalities imply the assertion.

 We finally verify Assertion  \ref{weoigjweogergrrewgwregfewfwwrg}.\ref{wetoiguwoefmw09i09m}.
 If $A \colon (C,  \rho,\mu)\to (C',\rho',\mu')$ is a morphism in $\bCgtsmc( \cZ\subseteq \cO(X)\otimes G_{can,max})$, then    $A\phi(\chi) $ is in ${\bC}$ by Lemma \ref{ewiogjoerwgfdsfdgsfg}.    This implies that $\Lambda_{(X,\chi)}(A)$ is a morphism in the ideal
$ {\Idem({\bC}^{(G)}_{\std}\rtimes_{r} G)}$.
\end{proof}

We now consider the cone sequence \eqref{wefqwefqewewfwdqewdewdqede} for $E=K\bC\cX^{G}_{G_{can,max}}$  whose boundary is the natural transformation
\begin{equation}\label{eqwfojpoqwefqefewqf}
\partial^{\Cone} \colon K\bC\cX^{G}_{G_{can,max}}(\cO^{\infty}(-))\to \Sigma K\bC\cX^{G}_{G_{can,max}}(-)
\end{equation} 
of functors from $G\UBC$ to $\Sp^{\la}$. The canonical inclusions
$\bCgtsmc(X\otimes G_{can,min})\to  \bCgtsmc(\cZ\subseteq \cO(X)\otimes G_{can,min})$
give a further transformation
\begin{equation} \label{eqwfojpoqwefqefewqf1}
\Sigma K\bC\cX^{G}_{G_{can,max}}(-) 
  \xrightarrow{\simeq}  \Sigma \Kcat( \bCgtsmc(\cZ\subseteq \cO(-)\otimes G_{can,min})) \end{equation}
which is actually an equivalence (see the argument for the left vertical equivalence in \eqref{friuhiufhqiwfewfqwefqfqewf} applied to the case $Y=G_{can,min}$). 
The composition of  the  transformations \eqref{eqwfojpoqwefqefewqf} with the equivalence \eqref{eqwfojpoqwefqefewqf1}  will also  be  called the cone boundary transformation 
$$\hat \partial^{\Cone} \colon K\bC\cX^{G}_{G_{can,max}}(\cO^{\infty}(-))\to   \Sigma \Kcat(\bCgtsmc(\cZ\subseteq \cO(-)\otimes G_{can,max}))$$ of functors from $G\UBC$ to $\Sp^{\la}$, but we add the $\hat{-}$ in order to distinguish it from \eqref{eqwfojpoqwefqefewqf}.
\begin{ddd}\label{qeoirgjoqwrfewffewfq9}
 We define the natural transformation
\begin{equation}
\Ass^{\Lambda} \coloneqq \Kcat(\Lambda)\circ \hat \partial^{\Cone}:K\bC\cX^{G}_{G_{can,max}}(\cO^{\infty}(-))\to \underline{ \Sigma \KK(\C,   {\bC}^{(G)}_{\std}\rtimes_{r}G)}
\end{equation}  of functors from $G\UBC_{\pc}^{\cR}$ to $\Sp^{\la}$.
 \end{ddd}

{If $X$ is in $G\UBC_{\pc}$ (see Definition \ref{wtrkogerfrfwerf}), then it is $G$-bounded, but not necessarily bounded.}
We let $X_{\cB_{max}}$ denote the object of $G\UBC_{\bd}$ ({see Definition \ref{werigowerferfrwef})}  obtained from $X$ by replacing the bornology of $X$ by the maximal bornology.

\begin{prop}\label{tghquifhiweewfqe}
There is a canonical equivalence of functors  \begin{equation}\label{gu98reug98wugergwergwergwwregwerg}
K\bC\cX_{G_{can,min}}(\cO^{\infty}((-)_{\cB_{max}}))\simeq K\bC\cX_{G_{can,max}}(\cO^{\infty}(-)))
\end{equation}
from $G\UBC_{{\pc}}$ to $\Sp^{\la}$.
\end{prop}
\begin{proof}

We  employ the notion of continuous equivalence introduced in \cite [Def. 3.21]{desc}.
  Recall the  Definition \ref{oiguoeigqregregwergregwgrg} of a locally finite subset of a $G$-bornological space.
  In the present situation we have a $G$-coarse space $Z$ with two $G$-bornologies.
  We denote the two objects in $G\BC$ by $Z_{0}$ and $Z_{1}$.
  The identity map of $Z$ is a continuous equivalence between $Z_{0}$
 and $Z_{1}$ if the following conditions on every $G$-invariant subset $L$ of $Z$ are equivalent:
 \begin{enumerate} 
 \item $L$ is locally finite in 
 $Z_{0}$.
 \item  $L$ is locally finite in    $Z_{1}$.
 \end{enumerate}
   In this case we have an obvious equality in $\nCcat$  \begin{equation}\label{vapojvoidjvadvavadsv}
\bCgtsmc(Z_{0})=  \bCgtsmc(Z_{1})\, .
\end{equation}

  \begin{lem}\label{wegoijerogregwegrgwgewergw}
  If $X$ in $G\UBC$ is $G$-bounded and such that $G$ acts properly, then the bornological coarse spaces
\[
 X\otimes G_{can,max} \quad  \textrm{and} \quad  X_{\cB_{max}}\otimes G_{can,min}
\]
 {are continuously equivalent, and}
\[
 \cO (X)\otimes G_{can,max} \quad  \textrm{and} \quad  \cO (X_{\cB_{max}})\otimes G_{can,min} 
\]
are continuously equivalent ({in both cases} by the identity map of the underlying sets).
  \end{lem}
\begin{proof}

We consider the second case. The first is similar and simpler.  Let $L$ be a $G$-invariant subset of
$[0,\infty)\times X{\times} G$.  Since $X$ is $G$-bounded we can choose a bounded subset  $B$ of $X$ such that $GB=X$.
For $n$ in $\nat$ and   subset $A$ of $X$ we
consider the intersections $L_{n,e}:=L\cap ([0,n]\times X\times \{e\})$ and $L_{n,A}:=L\cap ([0,n]\times A\times G)$.
\begin{enumerate}
\item   $L$ is locally finite in $\cO(X)\otimes G_{can,max}$ if and only if  $L_{n,A}$ is finite for every $n$ in $\nat$  and bounded subset $A$ of $X$. In particular $L_{n,B}$ is finite.
Hence $L$ is locally finite in $\cO(X)\otimes G_{can,max}$ if and only    if $L_{n,X}$ consists of finitely many $G$-orbits for every $n$ in $\nat$. 
Here we use that every $G$-orbit is locally finite in  $\cO(X)\otimes G_{can,max}$ since $G$ acts propertly on $X$.
\item  If $L$ is locally finite in $\cO(X_{\cB_{\max}})\otimes G_{can,min}$, if and only if 
$L_{n,e}$ is finite for every $n$ in $\nat$. This is the case exactly if $L_{n,X}$ consists of finitely many $G$-orbits. \qedhere
\end{enumerate}
 \end{proof}

 Let $Y$ be  any object in $G\BC$ and $X$ be in $G\UBC$.
  Then we have a  diagram in $\nCcat$
  \begin{equation}\label{}
\xymatrix{&\bCgtsmc(X\otimes Y)\ar[d] \ar[r] &\ar@{=}[d]\bCgtsmc(\cO(X)\otimes Y)\ar[r]&\bCgtsmc(\cO^{\infty}(X)\otimes Y)\\
0\ar[r]&\bCgtsmc(\cZ\subseteq \cO(X)\otimes Y)\ar[r]&\bCgtsmc(\cO(X)\otimes Y)\ar[r]&\frac{\bCgtsmc(\cO(X)\otimes Y)}{\bCgtsmc(\cZ\subseteq \cO(X)\otimes Y)}\ar[r]&0}
\end{equation} which is natural in $X$,  
where the lower sequence is exact,  and where the square commutes. 
If we apply $\Kcat$ and use  Definition \ref{qrogijeqoifefewfefewqffe}, then we get the (natural in $X$) commutative diagram
{\tiny  \begin{equation}\label{friuhiufhqiwfewfqwefqfqewf}
\xymatrix{\ar[r]&K\bC\cX^{G}_{Y} (X)\ar[d]^{\simeq}  \ar[r] &\ar@{=}[d]K\bC\cX^{G}_{Y}(\cO(X)) \ar[r]&K\bC\cX^{G}_{Y}(\cO^{\infty}(X)\otimes Y)\ar[r]\ar@{..>}[d]^{\simeq }&\\
\ar[r]&\Kcat(\bCgtsmc(\cZ\subseteq \cO(X)\otimes Y))\ar[r]&\Kcat(\bCgtsmc(\cO(X)\otimes Y))\ar[r]&\Kcat\left(\frac{\bCgtsmc(\cO(X)\otimes Y)}{\bCgtsmc(\cZ\subseteq \cO(X)\otimes Y)}\right)\ar[r]&}
\end{equation}}The lower sequence is a fibre sequence  by the exactness of $\Kcat$ (\cite[Thm.~1.32.5]{KKG} or \cite[Prop.{14.7}]{cank}), and the upper sequence is an instance of the cone sequence \eqref{wefqwefqewewfwdqewdewdqede}.
We now argue that the left vertical morphism is an equivalence (essentially the same argument as for the left vertical arrow in \eqref{qfwefeeqfeqfqwfqewwfeqwf}). First of all
for every $n$ in $\nat $ the inclusion 
$$ \bCgtsmc(Z_{n})\simeq  \bCgtsmc(Z_{n}\subseteq \cO(X)\otimes Y)
$$ 
is a  unitary equivalence by \cite[Lem.\ 6.10(2)]{coarsek}, where $Z_{n}:=[0,n]\times X\times Y$ has the structures induced from $\cO(X)\otimes Y$. The inclusion  $X\otimes Y\to Z_{n}$ given by
$(x,y)\mapsto (0,x,y)$ is a coarse equivalence.
Hence the induced map 
$$K\bC\cX^{G}( X\otimes Y)\to  K\bC\cX^{G}(\bCgtsmc(Z_{n}))\stackrel{\simeq}{\to}  \Kcat(\bCgtsmc(Z_{n}\subseteq \cO(X)\otimes Y))$$ is an equivalence for every $n$ in $\nat$. 
We now use that by definition
$$\bCgtsmc(\cZ\subseteq \cO(X)\otimes Y)\cong \colim_{n\in \nat}
\bCgtsmc(Z_{n}\subseteq \cO(X)\otimes Y)
$$ and that $\Kcat$ commutes with filtered colimits by    \cite[Thm.\ 14.4]{cank}. Hence we get an equivalence
$$K\bC\cX^{G}( X\otimes Y)\stackrel{\simeq}{\to} \Kcat(\bCgtsmc(\cZ\subseteq \cO(X)\otimes Y))$$ induced by the canonical inclusion. This is exactly the left vertical arrow in \eqref{friuhiufhqiwfewfqwefqfqewf}.

  We now assume that  $X$ is in $G\UBC_{{\pc}}$ (see Definition \ref{wtrkogerfrfwerf}) and note that $X$ is then $G$-bounded.
Using two instances of the the diagram  \eqref{friuhiufhqiwfewfqwefqfqewf}, one for $ X$ and $Y=G_{can,max}$ , and one  for $X_{\cB_{\max}}$ and $Y=G_{can,min}$,  and the equalities of $C^{*}$-categories resulting from Lemma \ref{wegoijerogregwegrgwgewergw} and  \eqref{vapojvoidjvadvavadsv} saying that the corresponding lower  fibre sequences of the two diagrams are equivalent
we get the desired equivalence \eqref{gu98reug98wugergwergwergwwregwerg}.
\end{proof}

  Let $(X,\chi)$ be in $G\UBC_{{\pc}}^{\cR}$ {(see the text before Proposition \ref{weoigjweogergrrewgwregfewfwwrg})}. Recall Definition \ref{kopeherthrggertg} of $\Ass^{\Theta}$ and
  Definition \ref{qeoirgjoqwrfewffewfq9} of $ \Ass^{\Lambda} $.
    \begin{prop}\label{erqoiheroigjorgiqfwefqwefqewf}
We have a commutative square
\begin{equation}\label{qoirfgjqoirqwfeqewfqfwefq}
\xymatrix{K\bC\cX^{G}_{G_{can,min}}(\cO^{\infty}(X_{\cB_{\max}}))\ar[rr]^-{ \Ass^{\Theta}_{X_{\cB_{max}}}}\ar@{-}[d]_{\eqref{gu98reug98wugergwergwergwwregwerg}}^{\simeq}&&  \Sigma \KK(\C,     {\bC}^{(G)}_{\std}\rtimes_{r}G) 
 \ar@{=}[d]   \\  K\bC\cX^{G}_{G_{can,max}}(\cO^{\infty}(X ))   \ar[rr]^{\Ass^{\Lambda}_{(X,\chi)}}&&  \Sigma \KK(\C,   {\bC}^{(G)}_{\std}\rtimes_{r}G) }
\end{equation}
which depends naturally on the coefficient category $\bC$ in $\Fun(BG,\nCcat_{\ndeg,\eadd,\omega\add})$.
\end{prop}

\begin{proof}
 Recall the construction of the functor $\Theta$ in  Definition \ref{eoigjoeirgrwegreregrgewrg}
 (see also   \eqref{fqoifoifjoeiwfjqowfqwefqwefqwefq}) and of $\Lambda$ in Definition \ref{erwoighfgojrefqwfewfqwefeqwf}. We  get the following morphism of exact sequences of $C^{*}$-categories.
\begin{align}
\label{rfreffewqewfewfewfqwefqwef}\\
{\small 
\mathclap{
\xymatrix{
0\ar[r]&\bCgtsmc( \cZ\subseteq\cO (X)\otimes G_{can,max}) \ar@{=}@/_2cm/[dd]\ar[d]^{\Lambda_{(X,\chi)}}\ar[r] & \ar@{=}@/_2cm/[dd]\bCgtsmc( \cO (X)\otimes G_{can,max}) \ar[d]^{\Lambda_{(X,\chi)}}\ar[r]& \frac{\bCgtsmc(  \cO (X)\otimes G_{can,max})}{\bCgtsmc( \cZ\subseteq\cO(X)\otimes G_{can,max})} \ar@{=}@/_2cm/[dd] \ar[d] \ar[r]&0\\0\ar[r]&\Idem({\bC}^{(G)}_{\std}\rtimes_{r} G)  \ar[r]&\Idem( {\bU} )\ar[r]& \frac{\Idem( {\bU})}{\Idem({\bC}^{(G)}_{\std}\rtimes_{r} G)} \ar[r]&0\\ 
0\ar[r]&\bCgtsmc( \cZ\subseteq\cO(X_{\cB_{max}})\otimes G_{can,min}) \ar[u]_{\Theta_{X_{\cB_{max}}}}\ar[r] &\bCgtsmc( \cO(X_{\cB_{max}})\otimes G_{can,min}) \ar[u]_{\Theta_{\cO(X_{\cB_{\max}})}}\ar[r]& \frac{\bCgtsmc(  \cO (X_{\cB_{max}})\otimes G_{can,min})}{\bCgtsmc( \cZ\subseteq\cO (X_{\cB_{max}})\otimes G_{can,min})} \ar[u] \ar[r]&0
}
}
}\notag
\end{align}
The right vertical maps are induced from the universal property of quotients. The round equalities are consequences of Lemma \ref{wegoijerogregwegrgwgewergw}
and \eqref{vapojvoidjvadvavadsv}. The right equality is responsible for the left vertical equivalence in 
\eqref{qoirfgjqoirqwfeqewfqfwefq} up to identifications, see the proof of Proposition \ref{tghquifhiweewfqe}. We apply $\Kcat$ and consider the segment of the long exact sequences which involve the boundary map. We use the identification given by the 
right vertical equivalences in the two instances of \eqref{friuhiufhqiwfewfqwefqfqewf}
with $X$ and $G_{can,max}$ and $X_{\cB_{max}}$ and $Y=G_{can,min}$ in order to
express the $K$-theory of the quotient categories in terms of coarse $K$-homology.
\begin{align}\label{weiojgoijregoiergwergwergwegrwergwf}
\mathclap{
\xymatrix{K\bC\cX^{G}_{c,G_{can,max}}(\cO^{\infty}(X))\ar@{-}[dd]_{\simeq}\ar[r]^-{\hat \partial^{\Cone}}& \ar@{-}[dd]_{\simeq}\Sigma \Kast( \bCgtsmc( \cZ\subseteq\cO (X)\otimes G_{can,max}))\ar[rd]^-{\quad\Kast(\Lambda_{(X,\chi)})}&\\&& \KK(\C,   {\bC}^{(G)}_{\std}\rtimes_{r} G)  \\
K\bC\cX^{G}_{c,G_{can,min}}(\cO^{\infty}(X_{\cB_{max}}))\ar[r]^-{\hat \partial^{\Cone}} &\Sigma \Kast( \bCgtsmc( \cZ\subseteq\cO (X_{\cB_{max}})\otimes G_{can,min}))\ar[ru]_-{\quad\Kast(\Theta_{X_{\cB_{max}}})}&}
}\ .
\end{align}
The left square commutes since it is induced by an equality of exact sequences of $C^{*}$-categories. We must  provide the filler of the right triangle.

This filler will be given by a  unitary equivalence {(see \cite[Def.\ 17.9]{cank} for the definition of this notion in the non-unital case)} of functors on the level of $C^{*}$-categories which will be induced from the equivalence provided by the following lemma.
\begin{lem}\label{epgjpfewdfewfqwfewfqef}
The following  triangle is filled by a natural  unitary equivalence:
$$\xymatrix{\bCgtsmc( \cO (X)\otimes G_{can,max}) \ar[dr]^-{\quad \Lambda_{(X,\chi)}}\ar@{=}[dd]&\\
&\Idem(  {\bU} 
)\\\bCgtsmc( \cO(X_{\cB_{max}})\otimes G_{can,min}) \ar[ur]_-{\quad \Theta_{X_{\cB_{max}}}}&
}\ .$$
\end{lem}
\begin{proof}
We consider an object $(C,\rho,\mu)$ on the common domain of the functors. 
We define $$U:=uv^{*}$$  in ${\bU}$ with $u$ as in \eqref{gwerfwerfwerfwerf}
and $v$ as in \eqref{gwgwreferferfewf3}.
By \eqref{gwergerfwerfrewf} and \eqref{werfwerferfwrefwrwr} we have 
%
%
 $$U  U^{*}
   =    p\, , \quad U^{*}U=p_{\chi}\, ,$$
  where $\tilde p$ and $p_{\chi}$ are as in \eqref{qewfeqwfpokpoqwefqwefqewf} and
\eqref{ewqfoihioqwhejfioewqfewfeqfew}, respectively.
We conclude that
$U  p_{\chi} =   p U$ {and that
we therefore} have a unitary  isomorphism $ U\colon ( C, \rho,  p_{\chi }) \to  ( C,  \rho,    p)$ in $\Idem({\bU})$
  as desired.

In order to verify that $U$ implements a natural transformation we must check the
compatibility with morphisms. 
 Let $A\colon (C,\rho,\mu)\to (C',\rho',\mu')$ be a morphism in the domain of the functors. 
 We let  $U'$ be defined as above for $(C',\rho',\mu')$. 
 Then by \eqref{wthoigwjopgwrregw} and \eqref{erwfwerf25} we have
$$U' \Lambda_{(X,\chi)}(A)= \Theta_{X}(A) U\ .$$
%
%
\end{proof}

{In view of \cite[Rem.\ 17.10]{cank}, the unitary equivalence from Lemma \ref{epgjpfewdfewfqwfewfqef} implements} a unitary equivalence filling $$\xymatrix{\bCgtsmc( \cZ\subseteq \cO (X)\otimes G_{can,max}) \ar[dr]^-{\quad \Lambda_{(X,\chi)}}\ar@{=}[dd]&\\
&\Idem(  {\bC}^{(G)}_{\std}\rtimes_{r} G)\\\bCgtsmc(\cZ\subseteq  \cO(X_{\cB_{max}})\otimes G_{can,min}) \ar[ur]_-{\quad \Theta_{X_{\cB_{max}}}}&
}$$

We now use \cite[Lem.\ 17.11]{cank} which provides the desired filler of the right  triangle in \eqref{weiojgoijregoiergwergwergwegrwergwf}.
\end{proof}

\begin{rem}
In Proposition \ref{erqoiheroigjorgiqfwefqwefqewf} we could state a stronger assertion 
saying that there is an equivalence of natural transformations from
$G\UBC_{{\pc}}^{\cR}$. The constructions    on the $C^{*}$-category level done in the proof  are sufficiently natural.   But writing out the details would amount to write out large higher coherence diagrams.
Since we do not really need this naturality, we refrain from doing so.   \hB
\end{rem}

We consider  $(X,\chi)$  in $G\UBC_{\pc}^{\cR}$.  
 Recall the Paschke  morphism $p_{X}$ from  \eqref{qwoeifjoifjewfqwefwfqewfq}.
 We use  Definition \ref{qriofjqofewfefqeqff} in order to rewrite the domain of $\Ass_{(X,\chi)}^{\Lambda}$ introduced in Definition \ref{qeoirgjoqwrfewffewfq9}. 
 Recall  the Definition \ref{qeriughioergewgergwer9} of $\Ass^{\an}_{\bC}$.
 
 \begin{prop} \label{erwogijiojfoiqewfewfqwefqwef}
We have a  commutative square \begin{equation}\label{weiojgoijogergwrgwger}
\xymatrix{  K_{\bC}^{G,\cX}( X )  \ar[rr]^-{\Ass^{\Lambda}_{(X,\chi)}}\ar[d]_{p_{X}}&&  \Sigma \KK(\C,   {\bC}^{(G)}_{\std}\rtimes_{r}G)\ar@{=}[d]\\   K_{\bC}^{G,\An}(\iota^{\topp}(X))\ar[rr]^-{\Ass^{\an}_{{\bC},\iota^{\topp}(X)}}&&  \Sigma \KK(\C,   {\bC}^{(G)}_{\std}\rtimes_{r}G) } \end{equation} which   depends naturally on the coefficient category $\bC$ in $\Fun(BG,\nCcat_{\ndeg,\eadd,\omega\add})$.
\end{prop}
\begin{proof}
  We consider the following commutative diagram  of exact sequences in $\nCcat$
\begin{align}
\label{rrwer3423ggergwegegwergwegergewge}\\
{\small 
\mathclap{
\xymatrix{0\ar[r]&\bC  (X)  \ar@{=}[d]^{ \eqref{rgfqfewfqewf}} \ar[r] & \bD   (X)\ \ar@{=}[d]^{ \eqref{rgfqfewfqewf1}} \ar[r]&  \bQ (X)  \ar@{=}[d]^{\eqref{qwefoujujfqew09ufewewfewfqfqef}} \ar[r]&0\\
0\ar[r]&\bCgtsmc( \cZ\subseteq\cO (X)\otimes G_{can,max}) \ar[d]^{\Lambda_{(X,\chi)}}\ar[r] &  \bCgtsmc( \cO (X)\otimes G_{can,max}) \ar[d]^{\Lambda_{(X,\chi)}}\ar[r]& \frac{ \bCgtsmc( \cO (X)\otimes G_{can,max})}{\bCgtsmc( \cZ\subseteq\cO  (X)\otimes G_{can,max})}    \ar[d]^{\bar \Lambda_{(X,\chi)}} \ar[r]&0\\0\ar[r]&\Idem({\bC}^{(G)}_{\std}\rtimes_{r} G)  \ar[r]&\Idem({\bU})\ar[r]& \frac{\Idem({\bU})}{\Idem({\bC}^{(G)}_{\std}\rtimes_{r} G)} \ar[r]&0 }
}
}\notag \ .
\end{align}
 We use the right vertical  equivalence  of \eqref{friuhiufhqiwfewfqwefqfqewf} for $X$ and $Y=G_{can,max}$ and Definition \ref{qriofjqofewfefqeqff} in order to get the equivalence
$$K_{\bC}^{G,\cX}(X)\simeq \Kcat\left(\frac{ \bCgtsmc( \cO (X)\otimes G_{can,max})}{\bCgtsmc( \cZ\subseteq\cO  (X)\otimes G_{can,max})}\right)\simeq \Kcat(\bQ(X))\, .$$
We now expand the square \eqref{weiojgoijogergwrgwger}
as follows:
\begin{align}\label{wiuhiuwehiuehiugwegergreg}\\
\mathclap{
\xymatrix{  K_{\bC}^{G,\cX}( X )  \ar@/^1.5cm/[rrrr]^-{\Ass^{\Lambda}_{(X,\chi)}} \ar[rr]^-{\hat \partial^{\Cone}} && \Sigma \Kcat(\bCgtsmc(\cZ\subseteq \cO(X)\otimes G_{can,max}))\ar[rr]^-{\Kcat(\Lambda_{(X,\chi)}) } &&  \Sigma \KK(\C,  {\bC}^{(G)}_{\std}\rtimes_{r}G)\ar@{=}[d]\\
\Kcat(\bQ(X))\ar[u]_{\simeq} \ar[d]_{p_{X}}\ar[rr]^-{\Kcat(\bar \Lambda_{(X,\chi)})} &&\Kcat(\frac{\Idem({\bU})}{\Idem({\bC}^{(G)}_{\std}\rtimes_{r} G)})\ar@{=}[d]\ar[rr]^{\partial}&&\ar@{=}[d] \Sigma \KK(\C,  {\bC}^{(G)}_{\std}\rtimes_{r}G)\\
K_{\bC}^{G,\An}(\iota^{\topp}(X))\ar@/_1.5cm/[rrrr]^-{\Ass^{\an}_{\iota^{\topp}(X)}}\ar[rr]^-{\ctc\circ \epsilon^{*}\circ (-\rtimes G)} &&\Kcat(\frac{\Idem( {\bU})}{\Idem({\bC}^{(G)}_{\std}\rtimes_{r} G)})\ar[rr]^{\partial} &&  \Sigma \KK(\C,  {\bC}^{(G)}_{\std}\rtimes_{r}G) }
}\notag
\end{align}
where $\ctc$ is the change-of-target functor \eqref{rfreffewqefewfwefewffwfwfwefewfewfewfwfwefewfewfqfdfdfwefqwef}
and $\epsilon^{*}$ is as in \eqref{wefqwfefqfqewfqfe1}.
The  commutativity of the upper triangle reflects the definition of $\Ass_{(X,\chi)}^{\Lambda}$ in Definition \ref{qeoirgjoqwrfewffewfq9}. The filler of the middle hexagon is obtained from the naturality of boundary operators  for the  morphism  of fibre sequences obtained by applying $\Kcat$ to  \eqref{rrwer3423ggergwegegwergwegergewge}.
The lower triangle reflects the Definition \ref{qeriughioergewgergwer9}  of $\Ass^{\an}_{\iota^{\topp(X)}}$ where also the notation appearing on the lower left horizontal arrow is explained.

So in order to produce a filler of the square
\eqref{weiojgoijogergwrgwger} we must provide a filler of the lower left square in \eqref{wiuhiuwehiuehiugwegergreg}. This is the assertion of the following lemma.

\begin{lem}\label{qoirfjqofqwefefqewfq}
We have a commutative square 
$$\xymatrix{ \Kcat(\bQ(X)) \ar[d]_{p_{X}}\ar[rrr]^-{\Kcat(\bar \Lambda_{(X,\chi)} )} &&&\Kcat(\frac{\Idem( {\bU})}{\Idem({\bC}^{(G)}_{\std}\rtimes_{r} G)})\ar@{=}[d] \\
K_{\bC}^{G,\An}(\iota^{\topp}(X)) \ar[rrr]^-{\ctc\circ \epsilon^{*}\circ (-\rtimes G)
}&&&\Kcat(\frac{\Idem({\bU})}{\Idem({\bC}^{(G)}_{\std}\rtimes_{r} G)})  }$$
\end{lem}

\begin{proof}

We start with  the following  diagram:

\begin{equation}\label{erfgeoifjoiewfqwfewfq}
\xymatrix{ \KKG(C_0(X),C_{0}(X)\otimes \bQ (X))\ar[r]^-{ \mu_{X},\eqref{oreihoijfvoisfdvervfdsvdfvsvv} }  \ar[d]^{-\rtimes G}&\KKG(C_0(X), \bQ^{(G) }_{\std})  \ar[d]^{-\rtimes G}\\ \KK(C_{0}(X)\rtimes G , (C_{0}(X)\otimes  \bQ (X)) \rtimes G)\ar[r]^-{ \mu_{X} \rtimes G}\ar[d]^{\epsilon^{*} }
&\KK(C_{0}(X)\rtimes G , \bQ^{(G) }_{\std} \rtimes G)\ar[d]^{\epsilon^{*}}\\\KK(\C ,(C_{0}(X)\otimes  \bQ (X)) \rtimes G) \ar[r]^{\mu_{X }\rtimes G}
&\KK(\C , \bQ^{(G) }_{\std}  \rtimes G) 
} 
\end{equation}
where $\epsilon^{*}$ is given by pre-composition in $\KK$  with the morphism described in \eqref{qewfoihiew9ufj9qweofqewfewfq}. 
The first square commutes since $-\rtimes G$ is a functor.
The second square commutes since $\KK$ is a bifunctor.

The  next diagram extends \eqref{erfgeoifjoiewfqwfewfq} to the left:
\begin{align}\label{qrqoijfqoifewefqwefwef}\\
\mathclap{
\xymatrix{ 
\Hom_{\Fun(BG,\nCalg )}( C_0(X),C_0(X))\times   \KK( \C,  \bQ (X))   \ar[r]^-{\hatotimes }\ar[d]^{(-\rtimes G)\times \id}&\KKG(C_0(X),C_{0}(X)\otimes  \bQ (X))   \ar[d]^{-\rtimes G} \\ 
\Hom_{ \nCalg}( C_0(X)\rtimes G,C_0(X)\rtimes G) \otimes   \KK( \C,  \bQ (X))   \ar[r]^-{\hatotimes }\ar[d]^{\epsilon^{*}\times \id}&\KK(C_{0}(X)\rtimes G , (C_{0}(X)\otimes  \bQ (X)) \rtimes G) \ar[d]^{\epsilon^{*}}\\
\Hom_{ \nCalg }( \C,C_0(X)\rtimes G) \otimes   \KK( \C,  \bQ (X))      \ar[r]^-{\hatotimes }&\KK(\C ,(C_{0}(X)\otimes  \bQ (X)) \rtimes G) 
 } 
 }\notag
\end{align}
The second square commutes since $\hatotimes $ in \eqref{saCWDQKOIOJ1} is a bifunctor.
The argument for the commutativity of the first square is the same as for the third square in 
\eqref{efqefeqfqefefqewf}.
We finally specialize \eqref{qrqoijfqoifewefqwefwef} at $\id_{C_{0}(X)}$  in $\Hom_{\Fun(BG,\nCalg )}( C_0(X),C_0(X))$ and get 
\begin{equation}\label{qrqoijfqoifewefqwefwerf}
\xymatrix{ 
    \KK( \C,  \bQ (X))   \ar[rr]^-{\id_{C_{0}(X)\hatotimes }}\ar@{=}[d] &&\KKG(C_0(X),C_{0}(X)\otimes  \bQ (X))   \ar[d]^{-\rtimes G} \\ 
   \KK( \C,  \bQ (X))   \ar[rr]^-{\id_{C_{0}(X)}\rtimes G\hatotimes }\ar@{=}[d]  &&\KK(C_{0}(X)\rtimes G , (C_{0}(X)\otimes  \bQ (X)) \rtimes G) \ar[d]^{\epsilon^{*}}\\    
   \KK( \C,  \bQ (X))      \ar[rr]^-{ \epsilon\hatotimes \id_{ \bQ (X)}\rtimes G}&&\KK(\C ,(C_{0}(X)\otimes  \bQ (X)) \rtimes G)
 }
\end{equation}
Forming the horizontal composition of \eqref{qrqoijfqoifewefqwefwerf} and \eqref{erfgeoifjoiewfqwfewfq} and using Definition \ref{nlkkmlmvfdvsdva} of $p_{X}$
  yields the bold part of the commutative  diagram
  \begin{equation}
\xymatrix{\KK(\C, \bQ(X) \ar[r]^-{p_{X} }\ar@{=}[d]&\KKG(C_{0}(X), \bQ^{(G) }_{\std})\ar[d]^{\epsilon^{*}\circ (-\rtimes G)}\\  \KK(\C, \bQ(X) \ar[r]\ar@{.>}[dr]_-{\Kcat(\Gamma_{(X,\chi)})\quad}\ar[r]&\KK(\C, \bQ^{(G) }_{\std}\rtimes G)  \ar[d]^{\ctc}\\& \Kcat(\frac{\Idem( {\bU} )}{\Idem({\bC}^{(G)}_{\std}\rtimes_{r} G)}) }
\end{equation}

 Unfolding the definitions we see that  
the dotted   morphism is induced by a  functor
\begin{equation}\label{qwojoijqwvfwevqvwwc}
\Gamma_{(X,\chi)}\colon \bQ(X)\to \frac{\Idem({\bU})}{\Idem({\bC}^{(G)}_{\std}\rtimes_{r} G)}
\end{equation}
which has the following description:
\begin{enumerate}
\item objects:  The functor $\Gamma_{(X,\chi)}$ sends the object $(C,\rho,\mu)$ in $\bQ(X)$  to the object $( C,\rho ,{\id_{C}})$ in $ \frac{\Idem( {\bU})}{\Idem({\bC}^{(G)}_{\std}\rtimes_{r} G)}
$. 
\item morphisms: The functor $\Gamma_{(X,\chi)}$ sends a  morphism  $[A] \colon (C,\rho,\mu)\to (C',\rho',\mu')$ in   $\bQ(X)$  to the morphism  
\begin{equation}\label{svasdcdscascsdca}{[\sum_{g\in G} \sigma(\phi'(\chi) \phi'(g^{*}\chi) A, g)]}\colon ( C,\rho,{\id_{C}} )\to ( C',\rho,{\id_{C'}} ) \end{equation} 
in $ \frac{\Idem( {\bU}  )}{\Idem({\bC}^{(G)}_{\std}\rtimes_{r} G)}$. Here we use the formula \eqref{398z9823zf983rferwfwefwrffwr} for $p_{\chi}$ which enters the definition of $\epsilon^{*}$, and $\sigma$ is as in \eqref{asvasvadvadsvadvdvoihjio}.

\end{enumerate}
Note that the sum in \eqref{svasdcdscascsdca} has finitely many non-zero terms.
In order to show Lemma \ref{qoirfjqofqwefefqewfq} we must provide an equivalence \begin{equation}\label{adsoijoiajfoiafqasdfsfs}
\Kcat(\Gamma_{(X,\chi)})\simeq \Kcat(\bar \Lambda_{(X,\chi)})\, ,
\end{equation}
where  
\begin{equation}\label{weqfiuhiuhiu}
\bar \Lambda_{(X,\chi)} \colon  \bQ (X) \to  \frac{\Idem( {\bU})}{\Idem({\bC}^{(G)}_{\std}\rtimes_{r} G)} 
\end{equation} is  as in \eqref{rrwer3423ggergwegegwergwegergewge}.
It has the following explicit description derived from Definition \ref{erwoighfgojrefqwfewfqwefeqwf}:
\begin{enumerate}
\item objects:  The functor $ \bar \Lambda_{(X,\chi)} $  sends the object $(C,\rho,\mu)$ in $\bQ (X)$  to the object $   (C,\rho ,p_{\chi})$ in $  \frac{\Idem( {\bU})}{\Idem({\bC}^{(G) }_{\std}\rtimes_{r} G)}$. 
\item morphisms: The functor $ \bar \Lambda_{(X,\chi)} $ sends a  morphism  $[A]\colon (C,\rho,\mu)\to (C',\rho',\mu')$ in   $\bQ (X)$  to the morphism  
 \begin{equation}\label{kpoherrgrtgertg}
 [\sum_{g\in G} \sigma(\phi'(g^{*}\chi)  A\phi( \chi) , g)] \colon ( C,\rho, p_{\chi})\to ( C',\rho ',p'_{\chi})
 \end{equation}   
in   $\frac{\Idem(  {\bU})}{\Idem({\bC}^{(G) }_{\std}\rtimes_{r} G)}$, see \eqref{fqwefqoijofqwefqwefqef}

\end{enumerate}

Recall the notion of a   MvN equivalence of functors from \cite[{Def.\ 17.12}]{cank}.
We claim that 
the functors $ \bar \Lambda_{(X,\chi)} $ and $ \Gamma_{(X,\chi)}$ are    MvN  equivalent. The claim implies the equivalence \eqref{adsoijoiajfoiafqasdfsfs}  by \cite[Prop.\ {16.18} \& {17.14}]{cank}. 
 
 The MvN equivalence $v\colon \bar \Lambda_{(X,\chi)} \to  \Gamma_{(X,\chi)}$ is given by the family of partial    isometries
 $v=({[}v_{(C,\rho,\mu)}{]})_{(C,\rho,\mu)\in \bQ(X)}$, where
 $v_{(C,\rho,\mu)}\colon(C,\rho,p_{\chi})\to (C,\rho,{\id_{C}})$
 is the canonical inclusion. This inclusion is given by the morphism $p_{\chi}\colon C\to C$ which indeed belongs to $\bU$.
Note that in the summands in \eqref{svasdcdscascsdca}, we can replace $A\phi(\chi) $ by $\phi'(\chi)A$ since
$A$ is pseudo-local by Lemma \ref{wtioghjwergfgrefwref} and we take the quotient by
$\Idem(\bC^{(G)}_{\std}\rtimes_{r}G)$. Naturality of $v$ is now obvious since the formulas \eqref{svasdcdscascsdca}
and \eqref{kpoherrgrtgertg} for the action of the functors on morphisms coincide after this replacement.
  This finishes the proof of Lemma \ref{qoirfjqofqwefefqewfq}.
  \end{proof} 
  
To complete the  proof of Proposition \ref{erwogijiojfoiqewfewfqwefqwef} we observe by an inspection of the constructions that they  depend naturally on  the coefficient category $\bC$ in
   $\Fun(BG,\nCcat_{\ndeg,\eadd,\omega\add})$.
\end{proof}

By equipping a $G$-simplicial complex $X$ with the structures induced by the metric  we obtain an object $m(X)$  of $G\UBC$. We further
 use the notation introduced in the diagram \eqref{retgertg365ztheh} in order to interpret $X$ in $G\UBC_{\bd}$ or $G\Top$. In the following statement and its proof we must be very precise about this interpretation.
\begin{prop} 
{If 
 $X$ is a $G$-finite $G$-simplicial complex with finite stabilizers, then}
 we have a commutative square
\begin{equation}\label{jfjoiwefjwqefwfwefqwefwefq}
\xymatrix{\Sigma K\bC^{G}(t(X))\ar[rrr]^-{\Sigma \Ass_{{\bC},t(X)}^{h}}\ar@{-}[d]^{\simeq}&&&\Sigma K\bC^{G}(*)\ar[d]^{\simeq }\\ K^{G,\An}_{\bC}(\iota^{\topp}(m(X)))\ar[rrr]^-{\Ass_{{\bC},\iota^{\topp}(m(X))}^{\an}}&&&\Sigma \KK(\C,  {\bC}^{(G)}_{\std}\rtimes_{r}G)}
\end{equation}
which depends naturally on $\bC$ in $\Fun(BG,\nCcat_{\ndeg,\eadd,\omega\add})$.
\end{prop}
\begin{proof}
     Note that $s(X) =m(X)_{\cB_{\max}}$ in the notation introduced before Proposition \ref{tghquifhiweewfqe}.  Note further that $m(X)$ actually belongs to the subcategory 
   $G\UBC_{{\pc}}$  described in Definition \ref{wtrkogerfrfwerf}.
 We can therefore choose $\chi$ in $\cR(m(X))$. 
 {We  consider the diagram}
\[\begin{tikzcd}[column sep=large]
	\Sigma K\bC^G(t(X)) \ar[r,"{\Sigma \Ass_{{\bC},t(X)}^h}"] \ar[d,"{\ref{weqfiuhqiuwhefiqwefewewfqfwefqwef}}","\simeq"'] & \Sigma K\bC^G(*) \ar[d,"{\ref{wrtoihjwogregergwergwreg}}"',"\simeq"] \\
	K\bC\cX^G_{G_{can,min}}(\cO^\infty(s(X)) \ar[r,"\Ass^\Theta_{{s(X) }}"] \ar[d,"{\ref{tghquifhiweewfqe}}","\simeq"'] & \Sigma \KK(\C,  {\bC}^{(G)}_{\std}\rtimes_{r}G) \ar[d,equal] \\
	K\bC\cX^G_{G_{can,max}}(\cO^\infty(m(X))) \ar[r,"{\Ass^\Lambda_{(m(X),\chi)}}"] \ar[d,equal,"{\mathrm{def}}"] & \Sigma \KK(\C,  {\bC}^{(G)}_{\std}\rtimes_{r}G) \ar[d,equal] \\ 
	K^{G,\cX}_{\bC}(m(X)) \ar[r,"{\Ass^\Lambda_{(m(X),\chi)}}"] \ar[d,"{p_{m(X)}}","\simeq"'] & \Sigma \KK(\C,   {\bC}^{(G)}_{\std}\rtimes_{r}G) \ar[d,equal] \\ 
	K^{G,\An}_{\bC}(\iota^{\topp}(m(X)) \ar[r,"\Ass^\an_{{\bC},\iota^{\topp}(m(X))}"] & \Sigma \KK(\C,   {\bC}^{(G)}_{\std}\rtimes_{r}G) \\
\end{tikzcd}\]
{The} lowest left vertical map is an equivalence by an application of our main Theorem~\ref{qreoigjoergegqrgqerqfewf}.\ref{qregiojqwfewfqwfqewf}.
The statement that each of the above squares commute is proven, from top to bottom, in Proposition~\ref{weqfiuhqiuwhefiqwefewewfqfwefqwef}.\ref{fopepofkopqwkfpoqwefkqwefqwefqe}, Proposition~\ref{erqoiheroigjorgiqfwefqwefqewf}, the definitions, and Proposition~\ref{erwogijiojfoiqewfewfqwefqwef}.
 All squares  depend naturally on  the coefficient category $\bC$ in
   $\Fun(BG,\nCcat_{\ndeg,\eadd,\omega\add})$.
 This shows the proposition.
\end{proof}

  \begin{proof}[Proof of Theorem \ref{wtoiguwegwergergregwe}]
  We choose a model for $E_{\cF}G^\cw$ which is a $G$-simplicial complex. 
  Then we apply $\pi_{*}$ to the square \eqref{jfjoiwefjwqefwfwefqwefwefq} and form the colimit   of the resulting squares of {homotopy} groups for $X$ running over the $G$-finite subcomplexes of $E_{\cF}G$.
  This yields \eqref{avvoijfoivjoafvasddsvdvasdv}. 
     \end{proof}

\begin{rem} In the proof of Theorem \ref{wtoiguwegwergergregwe} 
 we must apply $\pi_{*}$ before taking the colimit over the subcomplexes. The reason is that we have only constructed the boundary of the square  \eqref{jfjoiwefjwqefwfwefqwefwefq}
 naturally in $X$. For the fillers we just have shown existence for every $X$ separately.
 \hB
\end{rem}

\section{Davis--L\"uck functors and the argument of Kranz} \label{thelast}
  
 In this section we review the argument of Kranz \cite{kranz} for the comparison of the Davis--L\"uck assembly map with the Kasparov assembly map which involves the Meyer--Nest assembly map as an intermediate step. 
{In more detail, Kranz compares the Davis--L\"uck assembly map with the Meyer--Nest assembly map, which is known to coincide with the analytical assembly map. We will review these comparisons below. In fact, Kranz' paper has two separate parts. On the one hand, he shows that the Davis--L\"uck assembly map associated to a functor}
 $$K^{G}\colon \KKGs\to \Fun(G\Orb,\Sp)$$ satisfying certain axioms (stated in Assumption \ref{tgioewrgergwerg}) {is equivalent to the Meyer--Nest assembly map. On the other hand, he provides a concrete construction of such a functor $  K^{G}$.}
We recall this construction in detail with the goal of showing that it only involves
 formal manipulations using the calculus of equivariant $\KK$-theory as developed in \cite{KKG}.
 
 

We first recall the Meyer--Nest approach to the Baum--Connes assembly map  \cite{MR2193334}. {Given the results of  \cite{KKG} and the present paper, we will  give an almost self-contained treatment, the only exception is  the usage of  \cite[Prop.\ 4.6]{MR2193334} in the proof of  Proposition \ref{weg8w9egwergweg1} below.}  We    interpret the terminology introduced in \cite{MR2193334} in the stable $\infty$-category $\KKGs$ introduced in \cite[Def.\ 1.8]{KKG} instead of the triangulated homotopy category of $\KKGs$ as considered   by Meyer--Nest.
We call a subcategory of $\KKGs$ localizing\footnote{Usually, localizing subcategories are stable, cocomplete subcategories of stable, cocomplete $\infty$-categories. Since $\KKGs$ is only known to admit countable colimits, we must use this ad-hoc definition.} if it is thick and closed under countable direct sums. In the following we use the restriction,  induction and crossed-product functors   on the level of stable $\infty$-categories as introduced in  \cite[Sec. 1.5]{KKG}. 
\begin{ddd}\label{def:CI-and-CC}\mbox{}
 \begin{enumerate} \item
 We define $\ci$ as the localizing subcategory of $\KKGs$ generated by the objects of the form $\Ind_{H,s}^G(A)$ for  all finite subgroups $H$    of $G$ and objects $A $ in $\KKHs$. {The objects of $\ci$ will be called compactly induced.}
 \item We 
 define   $\cc$ as the localizing subcategory of $\KKGs$  given by  all objects $A$ with $\Res_{H,s}^G(A) = 0$ for all finite subgroups $H$ of $G$. \end{enumerate}
\end{ddd}
We note here that $\cc$ is localising because the restriction functors commute with countable  sums  \cite[Lem.\ 4.3]{KKG}.
The proof of the following proposition is based on  a general adjoint functor theorem applicable in this situation.
\begin{prop} \label{weg8w9egwergweg1}
There exists an adjunction  \begin{equation}\label{wefqwdxcassdcsdca}
\incl:\ci\leftrightarrows  \KKGs:C
\end{equation}
    \end{prop}  
\begin{proof}
For any object $A$ in $\KKGs $  by \cite[Prop.\ 4.6]{MR2193334} there is an object $\tilde A$ in $\ci$ with a morphism $\tilde A\to A$ (called the Dirac morphism) inducing an equivalence of functors $\KKGs(-,\tilde A)\to \KKGs(-,A)$ from $\ci^{\op}$ to $\Sp$. 
Hence for any $A$ in $\KKGs$ the functor
$\KKGs(-,A)_{|\ci^{{\op}}}\colon\ci^{\op}\to \Sp$ is representable by an object of $\ci$. {This implies the existence of the {right adjoint $C$ to $\incl$}}  {as follows for instance from \cite[Prop.\ 5.1.10]{Land-infinity}.}
\end{proof}

{Let  $\cC$ be a stable $\infty$-category.} Recall that a semi-orthogonal decomposition of $\cC$ is a pair $(\cA,\cB)$ of full stable subcategories such that 
$\map_{\cC}(A,B)\simeq 0$ for all $A$  in $\cA$ and $B$ in $\cB$, and such that 
 for every object $C$ of $\cC$ there exists a fibre sequence
$A\to C\to B$ with $A$ in $\cA$ and $B$ in $\cB$. For the sake of  completeness of the presentation, we give the following list of equivalent conditions on a pair $(\cA,\cB)$ of stable subcategories,  and refer for more details to \cite[Sec.\ 7.2.1]{SAG}:
\begin{enumerate}
\item\label{ertherhrt1}  The pair $(\cA,\cB)$ is a semi-orthogonal decomposition of $\cC$.
\item The pair $(\cA,\cB)$ is a $t$-structure on $\cC$. 
\item\label{ertherhrt3} The inclusion $\cA \to \cC$ has a right adjoint and $\cB$ is the right orthogonal complement of $\cA$.
\item The inclusion $\cB \to \cC$ has a left adjoint and $\cA$ is the left orthogonal complement of $\cB$.
\end{enumerate}

\begin{prop}
\label{weg8w9egwergweg}
The pair $(\ci,\cc)$ is a semi-orthogonal decomposition of $\KKGs$.
\end{prop}
\begin{proof}
For every subgroup $H$ of $G$ we have an adjunction 
$$\Ind_{H,s}^{G}:\KKHs \leftrightarrows  \KKGs: \Res^{G}_{H,s}$$
 which  can be obtained from 
\cite[Thm.\ 1.23.1]{KKG} by restriction to the separable subcategories.
It is an 
  immediate consequence of the existence of  these adjunctions
  that $\KKGs(A,B)\simeq 0$ for all $A$ in $ {\ci}$ and $B$ in $\cc$. We get in fact the following stronger assertion that $\cc$ consists precisely of the objects 
  $B$ of $\KKGs$ with $\KKGs(A,B)  {\simeq 0}$ for all $A$ in $\ci$, {i.e.\ that $\cc$ is the right orthogonal complement to $\ci$.} 
  
{In view of Proposition \ref{weg8w9egwergweg1}, the following is precisely a specialization of the  argument that Condition \ref{ertherhrt3} above 
 implies Condition \ref{ertherhrt1}.}
We must show  that  for any object $A$ of $\KKGs$, there is a   fibre sequence \begin{equation}\label{erwgwegergewrwferwerc}
 C(A) \lto A \lto N(A)\ 
\end{equation}
with $C(A)$ in $\ci$ and $N(A)$ in $\cc$.
By Proposition \ref{weg8w9egwergweg1} we have a fibre sequence of functors
$  C\to \id_{\KKGs}\to N$, where $N \colon \KKGs\to \KKGs$ is defined as the 
cofibre of the counit of the adjunction {in} \eqref{wefqwdxcassdcsdca}. It suffices to show that $N$ takes values in $\cc$. Let $A$ be in $\KKGs$. Then
for every $B$ in $\ci$ we have
$\KKGs(B,N(A))\simeq \cofib(\KKGs(B,C(A))\to \KKGs(B,A))$.
But $\KKGs(B,C(A))\to \KKGs(B,A)$
 is an equivalence by the construction of $C$
 so that $\KKGs(B,N(A))\simeq 0$. Since, as seen above, $\cc$ is precisely the right-orthogonal complement of $\ci$ 
   this implies  that $N(A)$  belongs to $\cc$.
 \end{proof}

Let $A$ be in $\KKGs$.
\begin{ddd}\label{igowergwergwegwerg}
The Meyer--Nest assembly map for $G$ 
is the map $$\mu_{\ast}^{\MN} \colon \KKs(\C,C(A)\rtimes_r G)\to  \KKs(\C, A\rtimes_r G)$$
induced by $C(A)\to A$ in $\KKGs$.
\end{ddd}

The following theorem is an immediate consequence of \cite[Prop.\ 5.2]{MR2193334}
which yields the comparison of the Meyer--Nest assembly map and Kasparov's assembly map. 
\begin{theorem}\label{wrigoghgegrtgeg}
 There is a commutative   square
 $$\xymatrix{  {RK^{G,\an}_{C(A)} (E_{\Fin}G^{\cw})}\ar[r]^{{\simeq}}\ar[d]_{{\simeq} }^{\mu^{\Kasp}_{C(A)}}&{RK^{G,\an}_{A} (E_{\Fin}G^{\cw})}\ar[d]^{\mu^{\Kasp}_{A}}\\ 
 \KK_{\sepa}(\C, C(A)\rtimes_{r}G) \ar[r]^-{\mu^{\MN}_{\ast}}&\KK_{\sepa}(\C, A\rtimes_{r}G)}$$
where the vertical maps are  instances of Kasparov's assembly map of {Definition \ref{wergoijowergerrrfrfrfrfeggwgw}} for the family of finite subgroups, and the horizontal maps are induced by the morphism $C(A)\to A$. 
 \end{theorem}
 \begin{proof} 
 {First we note that the square commutes by the naturality of the Kasparov assembly map with respect to {morphisms between} coefficients.} 
{Using Definition \ref{rigrgbjoirjvorgvfbdfgbfgbdfgbdfbg} the upper horizontal map is equivalent to the map 
$$\colim_{W\subseteq E_{\Fin}G^{\cw}}\KK^{G}_{\sepa}(C_{0}(W), C(A)) \to  \colim_{W\subseteq E_{\Fin}G^{\cw}}\KK^{G}_{\sepa}(C_{0}(W), A)\, ,$$
where the colimits run  over the $G$-finite sub-complexes of $E_{\Fin}G^{\cw}$.
 It is} an  equivalence by the definition of $C(A)\to A$, since $C_{0}(W)$ belongs to  $ \ci$ for every $W$ appearing in the colimit. 
 
 The verification of the fact that $\mu^{\Kasp}_{C(A)}$ is an equivalence is more complicated. {The reference \cite{MR2193334}
  employs the work of \cite{OO}
(isomorphism of the induction map) and
\cite[Prop. 2.3]{MR1836047} (compatibility of induction with the Kasparov assembly map).} 
{Using the results of the present paper, Theorem  \ref{wtkogwegerfwerf} gives an independent proof of this fact  {in the case of discrete groups. Note that \cite{MR2193334} considers the more general case of locally compact groups.
}}
\end{proof}

We now consider a family $(K^{H})_{H\subseteq G}$ of functors 
 $$K^{H} \colon \KKGs\to  \Fun(H\Orb,\Sp)\, , \quad A\mapsto K_{A}^{H}$$
 indexed by the subgroups $H$ of $G$.
 In order to formulate the  properties of this family required for Kranz' argument we consider the functor $$i_{H}^{G} \colon H\Orb\to G\Orb\, , \quad S\mapsto G\times_{H}S$$ and  let $i_{H,!}^{G}$ denote the left Kan extension functor along $i_{H}^{G}$. 
We assume $(K^{H})_{H\subseteq G}$ has the following properties:
\begin{ass}\label{tgioewrgergwerg}\mbox{}
 \begin{enumerate}
 \item\label{wrtkhiowgwergefrefw} $K^{G}$ preserves countable colimits.
\item  \label{rgegfgsdfg} For every  $A$ in $\KKGs$ and subgroup $H$ of $G$ we have
 an equivalence\footnote{The subscript $s$ at various functors indicates their restriction to the subcategory of separable algebras.} \begin{equation}\label{fvdsfvsavsdcasdcasdc}
K^{G}_{A}(G/H)\simeq \KKs(\C,(\Res^{G}_{H,s}(A)\rtimes_{r}H)_{s})\, .
\end{equation}  
\item \label{giowjwgoegwergwergr} For any  subgroup $H$
  of $G$ we have a commutative square
\begin{equation}\label{cdscasdc}
\xymatrix{\KKHs\ar[r]^-{K^{H}}\ar[d]^{\Ind^{G}_{H,s}}&\Fun(H\Orb,\Sp)\ar[d]^{i^{G}_{H,!}}\\\KKGs\ar[r]^-{K^{G}}&\Fun(G\Orb,\Sp) }
\end{equation}
\end{enumerate}
  \end{ass}

Note that we are mainly interested in the  member $K^{G}$ of the family $(K^{H})_{H\subseteq G}$.
The other members are only used to formulate Assumption \ref{tgioewrgergwerg}.\ref{giowjwgoegwergwergr}.
In the example   of the family $(K^{H})_{H\subseteq G}$  used below the functors
$K^{H}$ are constructed by applying 
  Definition \ref{egiueheiugheiugwegerggwfewr} to $H$ in place of $G$. In this case
the members $K^{H}$ have analoguous properties as $K^{G}$.

In view of Definition \ref{weigjorgdg} we consider $K^{G}$ as a functor from $\KKGs$ to the stable $\infty$-category of $\Sp$-valued equivariant homology theories. In particular, for $A$ in $\KKGs$ and $X$  in $G\Top$ we have a well-defined evaluation $K^{G}_{A}(X)$ in $\Sp$.

The argument of Kranz is then based on the following commutative diagram \begin{equation}\label{wergwerrefwercsdcsdc}
 \xymatrix{
	K^G_{C(A)}(E_{\Fin} G^{\cw}) \ar[rr]^-{\mu^{\MN}_{A,E_{\Fin}G^{\cw}}} \ar[d]^{\mu^{\DL}_{C(A),E_{\Fin} G^{\cw}}} && K^G_{A}(E_\Fin G^{\cw}) \ar[d]^{\mu^{\DL}_{A},E_{\Fin} G^{\cw}} \\
	K^G_{C(A)}(\ast) \ar[rr]^-{\mu^{\MN}_{A,\ast}} && K^G_{A}(\ast)
 }
\end{equation}
 Here the vertical Davis--Lück assembly maps \eqref{rtherhthetfff} are induced by the map $E_{\Fin} G^{\cw}\to \ast$. 
 Moreover, the horizontal Mayer--Nest   assembly maps   are  induced by the map $C(A)\to A$. 
 {By  Assumption \ref{tgioewrgergwerg}.\ref{rgegfgsdfg}} the map   $\mu^{\MN}_{A,\ast}$ is indeed the map from Definition \ref{igowergwergwegwerg}.

\begin{theorem}[Kranz]\label{qerighojwergerfewferf}
We have an equivalence $\mu^{\DL}_{A,E_{\Fin} G^{\cw}}\simeq \mu_{A,\ast}^{\MN}$. 
\end{theorem}
\begin{proof}

The square in \eqref{wergwerrefwercsdcsdc}  yields an equivalence of $\mu^{\DL}_{A,E_{\Fin} G^{\cw}}$ with $\mu^{\MN}_{A,*}$  provided one can show that  $\mu^{\DL}_{C(A),E_{\Fin} G^{\cw}}$ and $\mu^{\MN}_{A,E_{\Fin}G^{\cw}}
 $ are equivalences. This is the content of the following two  lemmas.
 
 Let $A$ be in $\KKGs$.
 \begin{lem}  The Meyer--Nest assembly map 
 $\mu^{MN}_{A,E_{\Fin}G^{\cw}}$ is an equivalence.
 \end{lem}
 \begin{proof}
 Since  $K^{G}$ is exact by Assumption \ref{tgioewrgergwerg}.\ref{wrtkhiowgwergefrefw}, 
 using    \eqref{erwgwegergewrwferwerc} we see that  it suffices to show that \begin{equation}\label{qwefqwefwefewfqwf}
K^{G}_{N(A)}(E_{\Fin}G^{\cw})\simeq 0\, .
\end{equation}
  Since $N(A)$ belongs to $\cc$  we have $\Res^{G}_{H,s}(N(A))\simeq 0$   for all $H$ in  $\Fin$. As a consequence of
 \eqref{fvdsfvsavsdcasdcasdc}   we conclude
 $K^{G}_{N(A)}(G/H)\simeq 0$ for every $H$ in $\Fin$.
 On the other hand, by the characterization  \eqref{sggsfdge} of the homotopy type of $E_{\Fin}G^{\cw}$ we have    $Y^{G}(E_{\Fin}G^{\cw})(G/H)\simeq 0$  (see \eqref{wefqwefweqdxasdc} for $Y^{G}$) provided $H\not\in \Fin$.
 As an immediate consequence of   the formula  \eqref{qwefpojfopwefqwefqwefqew} for the evaluation of a homology theory on a $G$-topological space we get the desired equivalence \eqref{qwefqwefwefewfqwf}.
\end{proof}

Let $A$ be in $\KKGs$.
\begin{lem}\label{egiowgwgeregegwref}
If $A$ is in $\ci$, then the Davis--L\"uck assembly map $\mu^{\DL}_{A,E_{\Fin} G^{\cw}}$
is an equivalence.
\end{lem}
\begin{proof} 
Since $ K^{G}$ preserves countable colimits and $\ci$ is generated by $\Ind_{H,s}^{G}(B)$ for all $B$ in $\KKGs$ and {all} finite subgroup{s $H$} of $G$ it suffices to show that 
 $\mu^{\DL}_{\Ind_{H,s}^{G}(B),E_{\Fin} G^{\cw}}$ is an equivalence for such data. 
 By Diagram \eqref{cdscasdc} we have an equivalence
 $K^{G}_{\Ind_{H,s}^{G}(B)}\simeq i^{G}_{H,!} K^{G}_{B}$.
 It is now a general fact (see e.g.\ \cite[Lem.\ 19.{25}]{cank} for an argument) that for  a functor $E \colon H\Orb\to \bM$ with cocomplete stable target $\bM$  we have a natural equivalence of functors
 $$i_{H,!}^{G}E\simeq E\circ \Res^{G}_{H} \colon G\Top\to \bM\, .$$
 We therefore get the commutative square
  $$\xymatrix{K^{G}_{\Ind_{H,s}^{G}(B)}(E_{\Fin}G^{\cw}) \ar[rr]^-{\mu^{\DL}_{\Ind_{H,s}^{G}(B),E_{\Fin} G^{\cw}}}\ar[d]^{\simeq}&& K^{G}_{\Ind_{H,s}^{G}(B)}(*) \ar[d]^{\simeq}\\  
  K^{H}_{ B}(\Res^{G}_{H}(E_{\Fin}G^{\cw}))
  \ar[rr]^-{!}&&K^{H}_{ B}(\Res^{G}_{H}(*))}$$
  Since  
 $\Res^{G}_{H}(E_{\Fin}G^{\cw})\to \Res^{G}_{H}(*)$ is a homotopy equivalence in $H\Top$ we conclude that the  map marked by $!$ is an equivalence.
 This implies that the map $\mu^{\DL}_{\Ind_{H,s}^{G}(B),E_{\Fin} G^{\cw}}$ is an equivalence. 
   \end{proof}
   This finishes the proof of Theorem \ref{qerighojwergerfewferf}.
 \end{proof}

 We now discuss the construction of the functor $K^{G}$. {It is based on the ideas of Kranz \cite{kranz}, but we  reformulate the construction such that it only uses the formal aspects of the calculus of equivariant $\mathrm{KK}$-theory as developed in \cite{KKG}.
 We give full details since we use them   crucially  in the argument for  Proposition \ref{weroigjoewrgergergfewrf}, which in turn is used in Theorem \ref{wtkogwegerfwerf}.
 }

We start with the adjunction
\begin{equation}\label{ewrfewrfwefefewfefwefe}
\C[-]:G\Set\leftrightarrows \Fun(BG,\Ccat):\Ob
\end{equation} 
whose left adjoint sends a $G$-set  $S$ to the $G$-{$C^{*}$}-category $\C[S]$ with the $G$-set $S$ of objects   and morphisms generated by the identities \cite[Lem.\ 3.8]{crosscat}.  
By \cite[Lem.\ 3.7]{crosscat} the inclusion 
$ \Fun(BG,\Ccat)\to  \Fun(BG,\nCcat)$ is again a left-adjoint. 
 By post-composition   with this inclusion we therefore get a left-adjoint functor $$\C[-] \colon G\Set\to \Fun(BG,\nCcat)$$ which we denote by the same  symbol for simplicity. 
 
 Recall the functor $y^{G}:\KKGs\to \KKG$ from \cite[Def. 1.8]{KKG}.
\begin{ddd}\label{egiueheiugheiugwegerggwfewr}
We define the functors
$$\hat K^{G} \colon \KKG\to \Fun(G\Orb,\Sp)\, , \quad A\mapsto  K^{\Ccat}((A\otimes_{\max} \kkGA(\C[-]))\rtimes_{r} G)$$
and
$$K^{G} \coloneqq \KKGs\xrightarrow{y^{G}} \KKG \xrightarrow{\hat K^{G}} \Sp\, .$$
\end{ddd}
In order to verify that $K^{G}$ satisfies the Assumption \ref{tgioewrgergwerg} we  analyse the construction of these functors  through  various intermediate constructs.
The most difficult part is thereby Assumption \ref{cdscasdc}.\ref{giowjwgoegwergwergr}.
 If one is not interested in the details of the argument one could skip the material  until Theorem \ref{wlgpwegerfrwefer} and just accept its statement.

We start with the functor
\begin{eqnarray}
\Fun(BG,\nCcat)\times G\Set& \stackrel{\id_{ \Fun(BG,\nCcat)}\times \C[-]}{\to}& \Fun(BG,\nCcat)\times  \Fun(BG,\nCcat) \nonumber\\&\stackrel{-\otimes_{\max}-}{\to}&  \Fun(BG,\nCcat) \nonumber\\&\stackrel{\kkGA}{\to}&\KKG\, .\label{trhoiwrhrhehrthh} 
\end{eqnarray}
{Using the exponential law, the above defines a functor}
$$R^{G} \colon \Fun(BG,\nCcat)\to \Fun(G\Set,\KKG)\, , \quad \bC\mapsto R^{G}_{\bC}\, .$$
Let $i_{\omega} \colon G\Set_{\omega}\to G\Set$ denote the inclusion of the full subcategory of countable $G$-sets, and let $\Fun^{\coprod_{\omega}}$ denote the full subcategory of a functor category of countable coproduct preserving functors. 
\begin{lem}\label{wegoiwegergrweg}
\mbox{}
\begin{enumerate}
\item \label{wethwthgegergr}$R^{G}$ is $s$-finitary.
\item \label{wethwthgegergr1}The restriction of $R^{G}$ to $\Fun(BG,\Ccat)$ sends unitary equivalences to equivalences.
\item \label{wethwthgegergr2}The functor $R^{G}$  sends weak Morita equivalences to equivalences.
\item \label{wethwthgegergr3}We have a canonical factorization
$$\xymatrix{\Fun(BG,\nCalg)\ar[rr]^-{\kkG}\ar[d]^{\incl} && \KKG\ar@{..>}[d]^{F^{G}}\\
\Fun(BG,\nCcat)\ar[rr]^-{R^{G}}\ar[urr]^{\kkG_{\Ccat}\quad} && \Fun(G\Set,\KKG)}$$
\item  \label{wethwthgegergr4}The functor $F^{G}$ preserves    colimits.
\item \label{wethwthgegergr5} We have a factorization \begin{equation}\label{afdvoihvioasvsdsdvsavasdvds}
\xymatrix{\KKGs\ar[r]^{y^{G}}\ar@{..>}[dd]^{F^{G}_{s}}&\KKG\ar[d]^{F^{G}}\\&\Fun(G\Set,\KKG)\ar[d]^{i_{\omega}^{*}}\\ \Fun^{\coprod_{\omega}}(G\Set_{\omega},\KKGs) \ar[r]^-{y^{G}}&\Fun(G\Set_{\omega},\KKG) }
\end{equation}such that $F_{s}^{{G}}$ 
preserves countable colimits.
\end{enumerate}
\end{lem}
\begin{proof}
Using the fact that $\kkGA$ is symmetric monoidal  \cite[Thm.\ 1.35]{KKG} we can rewrite the  functor  in  \eqref{trhoiwrhrhehrthh} as 
 \begin{equation}\label{fqwefqefdqewd}
 \Fun(BG,\nCcat)\times G\Set   \stackrel{ \kkGA\times \kkGA(\C[-])}{\to} \KKG\times \KKG \stackrel{-\otimes_{\max}-}{\to}   \KKG\, .
\end{equation}
 
The Assertions  \ref{wethwthgegergr}, \ref{wethwthgegergr1} and \ref{wethwthgegergr2} now follow from the corresponding properties of the functor
$\kkGA$ stated in \cite[Thm.\ 1.32]{KKG}, where  for  \ref{wethwthgegergr} we also use that
 the tensor structure on $\KKG$ preserves colimits in each variable. 
 In order to show Assertion \ref{wethwthgegergr3} we again use the Formula \eqref{fqwefqefdqewd}. It is then clear that we must define $F^{G}$ by the composition
 \begin{align}\label{wergfrfwefweerf}\\
 \mathclap{
F^{G} \colon \KKG\stackrel{\id_{\KKG}\times \kkGA(\C[-])}{\to} \KKG\times \Fun(G\Set,\KKG)\stackrel{-\otimes_{\max}-}{\to} \Fun(G\Set,\KKG)\, , \quad A\mapsto F^{G}_{A}
}\notag
\end{align} 
 Since $-\otimes_{\max}-$   preserves colimits in each argument  we conclude Assertion~\ref{wethwthgegergr4}.

 We finally show Assertion \ref{wethwthgegergr5}.  
We let $\C_{s}[-]$ denote the restriction of $\C[-]$ to countable sets. We consider $\C_{s}[-]$ as a functor with values in the full subcategory $\nCcat_{\sepa}$ of $\nCcat$ of $C^{*}$-categories with countably many objects and separable morphism spaces. The functor  $\C_{s}[-]$ is still  a left-adjoint. 
      The restriction of the adjunction     $$A^{f}:\nCcat\leftrightarrows \nCalg : \incl$$   (see e.g. \cite[Lem.\ 3.9]{crosscat})  to separable objects
 yields an adjunction
 $$A^{f}_{s} : \nCcat_{\sepa}\leftrightarrows \nCalg_{\sepa} :\incl\, .$$
 We define
 $F^{G}_{s}$ by the formula \begin{equation}\label{ewfqwefwdqwedwqwed}
F^{G}_{s,(-)}(-) \coloneqq (-)\otimes_{\max} \kkGs(A_{s}^{f}(\C_{s}[-]))\, .
\end{equation}
   The following chain of equivalences yields the commutative square \eqref{afdvoihvioasvsdsdvsavasdvds}, where for the moment we ignore the superscript
   $\coprod_{\omega}$ at the lower left corner
  \begin{align*}
  \mathclap{
  y^{G}\circ F^{G}_{s,(-)}( -)\stackrel{\textrm{def}}{\simeq} 
   y^{G}( (-)\otimes_{\max} \kkGs(A^{f}_{s}(\C_{s}[-])))\stackrel{!}{\simeq} 
   y^{G}(-)\otimes_{\max}  \kkGA (\C[i_{\omega}(-)])))\stackrel{\textrm{def}}{\simeq}  F^{G}_{y^{G}(-)}(i_{\omega}(-))\,.
   }
   \end{align*}
   For  the marked equivalence we use that $y^{G}$ is symmetric monoidal and the obvious equivalence $y^{G}( \kkGs(A_{s}^{f}(\C_{s}[-])))\simeq \kkGA(\C[i_{\omega}(-)])$
   of functors from $G\Set_{\omega}$ to $\KKG$.
 
   It remains to show that for any $A$ in $\KKGs$ the functor
   $F_{s,A}^{G}$ preserves countable coproducts.
By definition,  we have  an equivalence
   $$F^{G}_{s,A}(-)\stackrel{\textrm{def}}{\simeq} 
 A\otimes_{\max} \kkGs(A_{s}^{f}(\C_{s}[  - ])) $$
      of functors from $G\Set_{\omega}$ to $\KKGs$.
 Because $\C_{s}[-]$  is a left-adjoint it  preserves countable coproducts.
  The functor 
 $  \kkGs\circ A^{f}_{s}$  sends the relevant countable coproducts to sums by  
 \cite[Lem.\ 6.6]{KKG}.  Finally, by \cite[Prop.\ 1.7]{KKG}
 the tensor product $-\otimes_{\max}-$ on $\KKGs$ preserves countable sums in each argument. This finishes the construction of the factorization $F^{G}_{s}$ asserted in \ref{wethwthgegergr5}.

 It immediately follows from the definition in \eqref{ewfqwefwdqwedwqwed} that 
 the functor $F_{s}^{G}$  preserves countable colimits.
 Here we use again that $-\otimes_{\max}-$  on $\KKGs$ 
 preserves countable colimits in each argument \cite[Prop.\ 1.7]{KKG}.
This finishes the verification of Assertion  \ref{wethwthgegergr5}.  
  \end{proof}

  Let $H$ be a subgroup of $G$ and consider the object $G/H$ in $G\Set$. We let
$r^{G}_{H} \colon G\Set\to H\Set$ denote  the functor which restricts the $G$-action on a set to an $H$-action.
We consider  the object $G/H$ in $G\Set$.
\begin{lem}\label{ewgoweprgergwergwerg}
\mbox{}
\begin{enumerate}
\item \label{ergwioeogewgergergewrg}We have a commutative square  \begin{equation}\label{qwfeqweqfqwdqdqwed}
\xymatrix{\KKG\ar[d]^{\Res^{G}_{H}}\ar[r]^-{F^{G}}&\Fun(G\Set,\KKG) \ar[d]^{\ev_{G/H}}\\ \KKH\ar[r]^{\Ind_{H}^{G}}&\KKG}\ .
\end{equation}
\item \label{wetgioweggwregwreg}We have a commutative square
 \begin{equation}\label{fewfqewfwfwefqewfqewf}
\xymatrix{\KKH\ar[r]^-{F^{H}}\ar[dd]^{\Ind^{G}_{H}}&\Fun(H\Set,\KKH)\ar[d]^{r^{G,*}_{H}}\\ &\Fun(G\Set,\KKH)\ar[d]^{\Ind_{H}^{G}}\\\KKG\ar[r]^-{F^{G}}&\Fun(G\Set,\KKG)} \ .
\end{equation}
 \end{enumerate}
 \end{lem}
\begin{proof}
We use the functor $A \colon \nCcat_{{\mathrm{inj}}}\to \nCalg$ (see e.g. \cite[Def.\ 6.5]{crosscat}) and note that  for $S$ in $G\Set$ we get $A(\C[S])\cong C_{0}(S)$ in $\Fun(BG,\nCalg)$.  Applying this to $G/H$ in place of $S$ und using the definition of the induction functor
$\Ind_{H}^{G}$ from $\Fun(BH,\nCalg)$ to $\Fun(BG,\nCalg)$
applied to $\C$ with the trivial $H$-action      we obtain the isomorphisms $$A(\C[G/H])\cong C_{0}(G/H)\cong \Ind_{H}^{G}(\C)\, .$$
By  \cite[Prop.\ 6.9]{KKG} for every $\bC$ in $ \Fun(BG,\nCcat)$ {we} have an equivalence  
$$\kkGA(\bC)\stackrel{\textrm{def}}{=}\kkG ( A^{f}(\bC)) \stackrel{\simeq}{\to}  \kkG(A(\bC))\, .$$
Hence 
$$\kkGA(\C[G/H]) \simeq \kkG(A(\C[G/H]))\simeq
\kkG(\Ind_{H}^{G}(\C))\simeq \Ind_{H}^{G}(\kkG(\C))\, ,$$
where the symbol $\Ind_{H}^{G}$ on the right-hand side is the induction functor from $\KKH$ to $\KKG$
\cite[Thm.\ 1.22]{KKG}.
Using \eqref{wergfrfwefweerf}, the following  projection formula  \cite[Cor.\ 4.13]{KKG} 
\begin{equation}\label{fwerfwfwfwefqwef}
\Ind_{H}^{G}(-)\otimes_{\max} (-)\simeq \Ind_{H}^{G}((-)\otimes_{\max} \Res^{G}_{H}(-))
\end{equation} 
for functors $\KKH\times \KKG\to \KKG$, and that 
$\kkG(\C)$ is the tensor unit of $\KKG$ we get the following chain of equivalences    of endofunctors 
 \begin{eqnarray*}
\ev_{G/H}\circ F^{G}_{(-)}&\stackrel{}{\simeq} &( -) \otimes_{\max}  \Ind_{H}^{G}(\kkG(\C))\\&\stackrel{}{\simeq}&
\Ind_{H}^{G}(\Res^{G}_{H} (-)\otimes_{\max}   \kkG(\C))\\&\simeq&
\Ind_{H}^{G}(\Res^{G}_{H} (-))
\end{eqnarray*}
of $\KKG$
which provides the filler of the square in \eqref{qwfeqweqfqwdqdqwed}.

In order to construct the filler of the pentagon in \eqref{fewfqewfwfwefqewfqewf}
we note that we have obvious equivalences
\begin{align}\label{qewfqwfewfwfqwdewdqwed}\\
\mathclap{
r^{G,*}_{H}(\kkHA(\C[-]))\simeq \kkHA(r^{G,*}_{H}(\C[-]))\simeq 
 \kkHA(\Res^{G}_{H}(\C[-]))\simeq \Res^{G}_{H}(\kkGA(\C[-]))\,.
 }\notag
\end{align} 
The chain of equivalences
 \begin{eqnarray*}
 \Ind_{H}^{G}\circ r^{G,*}_{H}\circ F^{H}_{(-)}&\stackrel{\text{def}}{\simeq}&
 \Ind_{H}^{G}((-)\otimes_{\max}  r^{G,*}_{H}(\kkHA(\C[-])))\\&\stackrel{\eqref{qewfqwfewfwfqwdewdqwed}}{\simeq}&
  \Ind_{H}^{G}(-\otimes_{\max}  \Res^{G}_{H} (\kkGA(\C[-])))\\
  &\stackrel{\eqref{fwerfwfwfwefqwef}}{\simeq}&
  \Ind_{H}^{G}(-)\otimes_{\max}  \kkGA(\C[-])\\&\stackrel{\text{def}}{\simeq}&F^{G}_{(-)}\circ \Ind_{H}^{G}
\end{eqnarray*}
 provides the filler of   the pentagon.
\end{proof}

We now {consider} the functor 
$$L^{G}\colon \KKG\stackrel{F^{G}}{\to} \Fun(G\Set,\KKG)\stackrel{-\rtimes_{r}G}{\to}
\Fun(G\Set,\KK)\ .$$
By \cite[Lem.\ 4.16]{KKG} {the} restriction of $-\rtimes_{r}G$ to the subcategories of compact objects is a countable  sum preserving functor 
$$(-\rtimes_{r}G)_{s} \colon \KKGs\to \KKs\, .$$
We can therefore also {consider}
$$L^{G}_{s} \colon \KKGs\stackrel{F_{s}^{G}}{\to} \Fun^{\coprod_{\omega}}(G\Set_{\omega},\KKGs)\stackrel{-\rtimes_{r}G}{\to}
\Fun^{\coprod_{\omega}}(G\Set_{\omega},\KK_{\sepa})\, .$$

\begin{lem}\label{qerigowegwergrewf}\mbox{}
\begin{enumerate}
\item \label{ewrgoiwjogegergwegr}$L^{G}$ 
preserves colimits.
\item  \label{ewrgoiwjogegergwegr1} For every subgroup $H$ of $G$ we have a commutative  square  \begin{equation}\label{ergwerrewgergregwerg}
\xymatrix{\KKH\ar[r]^-{L^{H}}\ar[d]^{\Ind_{H}^{G}}& \Fun(H\Set,\KK)\ar[d]^{r^{G,*}_{H}}\\\KKG\ar[r]^-{L^{G}}& \Fun(G\Set,\KK)}
\end{equation}
\item \label{rthkoperhrtehgrteg}For every subgroup $H$ of $G$ we have a commutative square
\begin{equation}\label{foijeovfvjkourf98adsfik}
\xymatrix{\KKG\ar[d]^{\Res^{G}_{H}}\ar[r]^-{L^{G}}&\Fun(G\Set,\KK)\ar[d]^{\ev_{G/H}}\\\KKH\ar[r]^{-\rtimes_{r}H}&\KK }
\end{equation}
  \item \label{ewrgoiwjogegergwegr2} We have a commutative diagram
 \begin{equation}\label{afdvoihvioasvsdsdvsavasdvd2e2e2e2es}
\xymatrix{\KKGs\ar[r]^{y^{G}}\ar[dd]^{L^{G}_{s}}&\KKG\ar[d]^{L^{G}}\\&\Fun(G\Set,\KK)\ar[d]^{i_{\omega}^{*}}\\ \Fun^{\coprod_{\omega}}(G\Set_{\omega},\KKs) \ar[r]^{y}&\Fun(G\Set_{\omega},\KK) }
\end{equation}
\item \label{ewrgoiwjogegergwegr3} For every subgroup $H$ of $G$ we have a commutative  square  \begin{equation}\label{ergwerrewgergregwdsfasdfsferg}
\xymatrix{\KKHs\ar[r]^-{L_{s}^{H}}\ar[d]^{\Ind_{H,s}^{G}}& \Fun^{\coprod_{\omega}}(H\Set_{\omega},\KKs)\ar[d]^{r^{G,*}_{H}}\\\KKGs\ar[r]^-{L_{s}^{G}}& \Fun^{\coprod_{\omega}}(G\Set_{\omega},\KKs)}
\end{equation}
\item  \label{tgokwpegwergergeg}The functor $L_{s}$ 
preserves countable colimits.
\end{enumerate}

\end{lem}
\begin{proof}
Assertion \ref{ewrgoiwjogegergwegr} follows from   \ref{wegoiwegergrweg}.\ref{wethwthgegergr4} and the fact that $-\rtimes_{r}G \colon \KKG\to \KK$  preserves colimits \cite[Thm.\ 1.22]{KKG}.
 
 For Assertion \ref{ewrgoiwjogegergwegr1} we expand the square in \eqref{ergwerrewgergregwerg} as follows:
  \begin{equation}\label{ergwerrewgergregwerg1}
\xymatrix{\KKH\ar[r]^-{F^{H}}\ar[dd]^{\Ind_{H}^{G}}\ar[r]& \Fun(H\Set,\KKH)\ar[d]^{r^{G,*}_{H}}\ar[r]^{\rtimes_{r}H}& \Fun(H\Set,\KK)\ar[dd]^{r^{G,*}_{H}}\\& \Fun(G\Set,\KKH)\ar[d]^{\Ind_{H}^{G}}\ar[dr]^{-\rtimes_{r}H}&\\\KKG\ar[r]^-{F^{G}}& \Fun(G\Set,\KKG)\ar[r]^{-\rtimes_{r}G}& \Fun(G\Set,\KK)}
\end{equation}
The left pentagon is precisely \eqref{fewfqewfwfwefqewfqewf} and commutes by   \ref{ewgoweprgergwergwerg}.\ref{wetgioweggwregwreg}. The upper right square in \eqref{ergwerrewgergregwerg1} commutes by the associativity of composition of functors.
Finally, the lower triangle commutes by the equivalence 
\begin{equation}\label{erferfwerfwerfwerfe}
(-)\rtimes_{r}H\simeq \Ind_{H}^{G}(-)\rtimes_{r}G 
\end{equation}  of functors
from $\KKH$ to $\KK$ \cite[Thm.\ 1.23]{KKG}.

In order to show Assertion \ref{rthkoperhrtehgrteg} we expand the square \eqref{foijeovfvjkourf98adsfik} as follows:
\begin{equation}\label{foijeovfvjkourf98adsfike}
\xymatrix{\KKG\ar[d]^{\Res^{G}_{H}}\ar@/^1cm/[rr]^{L^{G}}\ar[r]^-{F^{G}}&\Fun(G\Set,\KKG)\ar[d]^{\ev_{G/H}} \ar[r]^{-\rtimes_{r}G}&\Fun(G\Set,\KK)\ar[d]^{\ev_{G/H}}\\\KKH\ar@/^-1cm/[rr]_{-\rtimes_{r}K}\ar[r]^{\Ind_{H}^{G}}&\KKG\ar[r]^{-\rtimes_{r}G}&\KK }
\end{equation}
The right square commutes obviously, and the commutativity of the left square is considered in \ref{ewgoweprgergwergwerg}.\ref{ergwioeogewgergergewrg}.
The upper triangle reflects the definition of $L^{G}$, and the lower triangle commutes 
by \eqref{erferfwerfwerfwerfe}.

{By composing} \ref{wegoiwegergrweg}.\ref{wethwthgegergr5} with $-\rtimes_{r}G$ and the equivalence
$$(-\rtimes_{r}G)\circ y^{G}\simeq y\circ (-\rtimes_{r}G)_{s}$$ 
{we conclude Assertion \ref{ewrgoiwjogegergwegr2}.}

 
 In order to show Assertion \ref{ewrgoiwjogegergwegr3} we precompose  the square in  \eqref{ergwerrewgergregwerg} with $y^{H}$ and $y^{G}$, respectively,  and restrict the results to countable sets. We use that $\Ind_{H}^{G}\circ y^{H}\simeq {y^{G}\circ \Ind_{H,s}^{G} }$.
This gives the outer square in 
 \begin{equation}
 \label{ergwerrewgergregwdsfascdcdcdfsferg}
\xymatrix{
\KKHs \ar@/^1cm/[rr]^{(L^{H}_{ y^{H}})_{|H\Set_{\omega}}}\ar[r]^-{L_{s}^{H}}\ar[d]^{\Ind_{H,s}^{G}}& \Fun^{\coprod_{\omega}}(H\Set_{\omega},\KKs)\ar[d]^{r^{G,*}_{H}}\ar[r]^-{y}&\Fun(H\Set_{\omega},\KK)\ar[d]^{r^{G,*}_{H}}\\
\KKGs\ar[r]^-{L_{s}^{G}}\ar@/_1cm/[rr]_{(L^{G}_{y^{G}})_{|G\Set_{\omega}}}& \Fun^{\coprod_{\omega}}(G\Set_{\omega},\KKs)\ar[r]^-{y}&\Fun(G\Set_{\omega},\KK)}
\end{equation}

 We then use that
$r^{G}_{H}$ preserves countability and coproducts and therefore that  $r^{G,*}_{H}$  preserves countable  coproduct preserving functors. If we now employ the fact that    $y$ is fully faithful, then we get the filler of the left square.

Assertion \ref{tgokwpegwergergeg} follows from  \ref{qerigowegwergrewf}.\ref{wethwthgegergr5} and the fact that $(-\rtimes _{r}G)_{s}$ 
preserves countable colimits \cite[Lem.\ 4.16]{KKG}.
 \end{proof}

Let $H$ be a subgroup of $G$.
We have an adjunction
$$i^{G}_{H}:H\Set\leftrightarrows G\Set: r^{G}_{H}\, ,$$
where $i^{G}_{H}$ sends the $H$-set $S$ to the $G$-set $G\times_{H}S$.
 Consequently, we have an equivalence 
 $$r^{G,*}_{H}\simeq i^{G}_{H,!} \colon \Fun(H\Set,\cC)\to \Fun(G\Set,\cC)$$ for any target category $\cC$, where $ i^{G}_{H,!}$ is the left Kan-extension functor.
 It restricts to an equivalence 
 $$r^{G,*}_{H}\simeq i^{G}_{H,!} \colon \Fun^{\coprod_{\omega}}(H\Set_{\omega},\cC)\to \Fun^{\coprod_{\omega}}(G\Set_{\omega},\cC)$$ 
 provided $\cC$ has countable coproducts.
 
  The functor $i^{G}_{H}$ restricts to a functor
$i^{G}_{H} \colon H\Orb\to G\Orb$.
We note that the slice categories 
$H\Orb_{/S}$ for any $S$ in $G\Orb$ are   countable  discrete.  
Therefore the left Kan extension functor 
$$i^{G}_{H,!} \colon \Fun(H\Orb,\cC)\to \Fun(G\Orb,\cC)$$ exists provided $\cC$ admits all countable  coproducts.
We let $i^{G} \colon G\Orb\to G\Set$ denote the inclusion.
From now on we assume that $\cC$ admits  countable coproducts.
We consider the square 
$$\xymatrix{\Fun(H\Set_{\omega},\cC)\ar[r]^{i^{H,*}}\ar[d]^{r^{G,*}_{H}}&\Fun(H\Orb,\cC)\ar[d]^{i^{G}_{H,!}}\\ \Fun(G\Set_{\omega},\cC)\ar[r]^{i^{G,*}}&\Fun(G\Orb,\cC)}$$
In general we do not expect that the square commutes.  
 \begin{lem}\label{wetigowegerferfwerf}
The restriction of the square to countable coproduct preserving functors
is a commutative square 
\begin{equation}\label{wegwergfergerwgrgf}
\xymatrix{\Fun^{\coprod_{\omega}}(H\Set_{\omega},\cC)\ar[r]_-{\simeq}^-{i^{H,*}}\ar[d]^{r^{G,*}_{H}}&\Fun(H\Orb,\cC)\ar[d]^{i^{G}_{H,!}}\\ \Fun^{\coprod_{\omega}}(G\Set_{\omega},\cC)\ar[r]_-{\simeq}^-{i^{G,*}}&\Fun (G\Orb,\cC)}
\end{equation}
\end{lem}
\begin{proof}
The inverse of the horizontal arrows are the left Kan-extension functors 
along $i^{H}$ and $i^{G}$, respectively. Since we have a canonical isomorphism 
 Since $ i^{G}_{H}\circ i^{H}\cong  i^{G}\circ i^{G}_{H}$ of functors from $H\Orb$ to $G\Set$ the  square
 \begin{equation}\label{ergwergergerfwewrfefref}
\xymatrix{\Fun^{\coprod_{\omega}}(H\Set_{\omega},\cC) \ar[d]^{i^{G}_{H,!}\simeq r^{G,*}_{H}}&\ar[l]^-{\simeq}_-{i^{H}_{!}}\Fun(H\Orb,\cC)\ar[d]^{i^{G}_{H,!}}\\ \Fun^{\coprod_{\omega}}(G\Set_{\omega},\cC) &\ar[l]^-{\simeq}_-{i^{G}_{!}}\Fun (G\Orb,\cC)}
\end{equation}
commutes. We obtain \eqref{wegwergfergerwgrgf} from \eqref{ergwergergerfwewrfefref} by inverting the horizontal arrows.
\end{proof}

Note that $\KKs$ admits countable colimits \cite[Thm. 1.4]{KKG}.
\begin{prop}We have a commutative square \begin{equation}\label{wqoiwehfiwqeofwefwefweq}
\xymatrix{\KKHs\ar[rr]^-{i^{H,*}L^{H}_{s}}\ar[d]^{\Ind_{H,s}^{G}}&& \Fun(H\Orb,\KKs)\ar[d]^{i^{G}_{H,!}}\\\KKGs\ar[rr]^-{i^{G,*}L^{G}_{s}}&&\Fun(G\Orb,\KKs)}
\end{equation}
 \end{prop}
\begin{proof}
We expand the square as 
$$\xymatrix{\KKHs\ar[r]^-{{L^{H}_{s}}  }\ar[d]^{\Ind_{H,s}^{G}} \ar[d]&\Fun^{\coprod_{\omega}}(H\Set_{\omega},\KKs)\ar[d]^{r^{G,*}_{H}}\ar[r]^-{{i^{H,*}}}& \Fun(H\Orb,\KKs)\ar[d]^{i^{G}_{H,!}}\\{\KKGs}\ar[r]^-{ {L^{G}_{s}} }&\ar[r]^-{{i^{G,*}}  }\Fun^{\coprod_{\omega}}(G\Set_{\omega},\KKs)&\Fun(G\Orb,\KKs)}$$
The left square commutes by  \ref{qerigowegwergrewf}.\ref{ewrgoiwjogegergwegr3}.
The right square commutes by Lemma \ref{wetigowegerferfwerf}.
\end{proof}

We now observe by an inspection of the constructions:
\begin{kor}\label{egiueheiugheiugwegerggwfewr222}
We have a canonical equivalence    of functors 
$$\hat K^{G}\simeq  \KKG(\C,-)\circ i^{G,*}L^{G} \colon \KK^{G} \to \Fun(G\Orb,\Sp)\,.$$
\end{kor}

\begin{kor}\label{wegiwrogerfrweferfwerf}\mbox{}
\begin{enumerate}
\item\label{wthgjiwoegerferferfwef} The functor $\hat K^{G}$ 
preserves colimits.
\item \label{wthgjiwoegerferferfwef1} For every subgroup $H$ of $G$ we have an equivalence $$ \hat K^{G}_{(-)}(G/H)\simeq \KK(\C,\Res^{G}_{H}(-)\rtimes_{r}H)$$ of functors $\KK^{G} \to \Sp$.
\item \label{wthgjiwoegerferferfwef2}
The composition $$\Fun(BG,\nCcat)\xrightarrow{\kkGA} \KKG\xrightarrow{\hat K^{G}}\Fun(G\Orb,\Sp)$$
 sends Morita equivalences to equivalences.
\end{enumerate}
\end{kor}
\begin{proof}
Assertion \ref{wthgjiwoegerferferfwef} follows from Lemma \ref{qerigowegwergrewf}.\ref{ewrgoiwjogegergwegr}, the fact that $i^{G,*}$ obviously preserves colimits, and that $\KK^{G}(\C,-)$ preserves  colimits since $\KKG$ is stable and $\kkG(\C)$ in $\KKG$ is compact.

Assertion  \ref{wthgjiwoegerferferfwef1} is a consequence of  the commutativity of 
\eqref{foijeovfvjkourf98adsfik} and the definitions.

In order to show Assertion \ref{wthgjiwoegerferferfwef2} note that the collection of evaluations at the orbits $G/H$ for all subgroups $H$ of $G$ detects equivalences.
In view of Assertion \ref{wthgjiwoegerferferfwef1} it thus suffices to show that
$\KK(\C,\Res^{G}_{H}(-)\rtimes_{r}H)$ sends Morita equivalences to equivalences. But this is true since
$\Res^{G}_{H}(-)$ obviously preserves Morita equivalences,
$-\rtimes_{r} G$ preserves Morita equivalences by \cite[Prop.\ 16.11]{cank}, and $ \KK(\C,-)=K^{\Ccat}(-)$ sends Morita equivalences to equivalences by \cite[Prop.\ 16.18]{cank}.
 \end{proof}

Using the equivalence
$\KKG(\C,y^{G}(-))\simeq \KKGs(\C,-)$ of functors from $\KKs$ to $\Sp$ we get the formula
\begin{equation}\label{qwefweqdqwedqw}
K^{G}\simeq \KKs(\C,i^{G,*}L^{G}_{s}(-))\, .
\end{equation}
\begin{theorem}\label{wlgpwegerfrwefer}
The functor $K^{G}$ satisfies the Assumption \ref{tgioewrgergwerg}.
\end{theorem}
\begin{proof}
The functor $K^{G}$ is exact since  {$\hat K^{G}$ is exact by Corollary \ref{wegiwrogerfrweferfwerf}.\ref{wthgjiwoegerferferfwef} and  $y^{G}$  is exact.}

 
In order to show that the functor $K^{G}$ preserves countable colimits we use \eqref{qwefweqdqwedqw}, that  $L_{s}^{G}$
preserves countable colimits by \ref{qerigowegwergrewf}.\ref{tgokwpegwergergeg},
and that  $ \KKGs(\C,-)$ preserves  countable colimits:
Indeed, $\KKGs(\C,-)$ is exact by definition. To see that it preserves countable sums, we use   the identification $ \KKGs(\C,\kks(-))\simeq \Kast(-)$ of functors from $\nCalg_{\sepa}\to \Sp$, the fact that countable sums in $\KKs$ are presented by countable sums in $\nCalg_{\sepa}$, and that 
$\Kast$ sends countable sums to coproducts. 

For $A$ in $\KKGs$ we have a natural equivalence
\begin{eqnarray*}
K^{G}_{A}(G/H)&\simeq&\KK(\C, L^{G}_{y^{G}(A)}(G/H))\\
&\stackrel{\ref{qerigowegwergrewf}.\ref{rthkoperhrtehgrteg}}{\simeq}&
\KK(\C,  \Res^{G}_{H}(A)\rtimes_{r}H)\, .
\end{eqnarray*}
Finally the commutativity of the square in \eqref{cdscasdc}  is obtained by applying $\KKs(\C,-)$ to the right part of the square in \eqref{wqoiwehfiwqeofwefwefweq} and  
using that $\KKs(\C,-) \colon \KKs\to \Sp$ preserves countable colimits in order to commute
$i^{G}_{H,!}$ with this functor.
 \end{proof}

 \section{The generalized Green--Julg Theorem}

In this section  we show
 a version of the generalized Green--Julg theorem, see \cite[Thm.\ 13.1]{Guentner_2000} 
 stating that the Kasparov assembly map for the family $\Fin$  and proper $G$-$C^{*}$-algebras
  is an equivalence.  In our statement we replace the condition that the separable $G$-$C^{*}$-algebra $A$ is proper by the weaker  (see \cite[Cor.\ 7.3]{MR2193334})  homotopy theoretic condition that $\kkGs(A)$ belongs to the set $\ci$ generated by the compactly induced objects, see Definition \ref{def:CI-and-CC}.  
  
  In   \cite{MR1836047} it was shown more generally for locally compact groups $G$ that
 the Kasparov assembly map is an  equivalence for compactly induced coefficients.
   Our proof  for discrete groups   is logically independent of the results of   \cite{MR1836047}  
     and also  different  from the one in \cite{Guentner_2000}. 
 In particular, it  makes the proof of  Theorem \ref{wrigoghgegrtgeg} independent of   \cite{MR1836047}.  
Our approach is based on  the equivalence between the analytic and Davis--L\"uck assembly maps and that the analoguous assertion
 for the latter is known.
 


  We consider  $A$ be in $\KKGs$.
 \begin{theorem}\label{wtkogwegerfwerf}
 If $A$   belongs to  $\ci$, then the Kasparov assembly map
 $$\mu^{\Kasp}_{A,\Fin} \colon RK^{G,\an}_{A}(E_{\Fin}G^{\cw})\to \KK(\C,A\rtimes_{r}G)$$
 is an equivalence.
 \end{theorem}
\begin{proof}
%

%
%
%

The proof of this theorem is based on a chain of comparison results of independent interest which eventually will be combined to provide an  equivalence between $\mu^{\Kasp}_{A,\Fin}$ and $\mu^{\DL}_{A,\Fin}$. The latter is known to be an equivalence by Lemma \ref{egiowgwgeregegwref}.

Let $\bC$ be in $\Fun(BG,\nCcat_{\ndeg,\eadd,\omega\add})$   so that $K\bC^{G} \colon G\Orb\to \Sp$ is given by Definition \ref{qroeigjoqergerqfewewfqewfeqwf}. 
We then form $\bC^{u}$ in $\Fun(BG,\Ccat)$ by Definition \ref{wefbsfopkfvdfsv} and $\hat K^{G}_{\bC^{u}} \colon G\Orb\to \Sp$ by Definition \ref{egiueheiugheiugwegerggwfewr}. Note that
the latter only depends on the object $\kkG_{\Ccat}(\bC^{u})$ in $\KK^{G}$, but according to our general conventions we   dopped the symbol $\kkG_{\Ccat}$ from the notation.

Recall the Definition   \ref{wrtiogrgergfgs} of  $\Ass^{h}_{\bC,\cF}$  and  $\mu^{\DL}_{\bC^{u},\cF}$ from  \eqref{rtherhthetfff}.  
\begin{prop}\label{weroigjoewrgergergfewrf}
 We have a canonical equivalence 
$K\bC^{G}\simeq \hat K^{G}_{\bC^{u}}$ and therefore for any family  $\cF$  of subgroups of $G$ a commutative diagram
\begin{equation}
\xymatrix{\hat K^{G}_{\bC^{u}}(E_{\cF}G^{\cw})\ar[r]^-{\simeq}\ar[d]^{\simeq}& \colim_{G_{\cF}\Orb} \hat K^{G}_{\bC^{u}}   \ar[rrr]^-{\mu^{\DL}_{\bC^{u},\cF} }\ar[d]^{\simeq} &&& \hat K_{\bC^{u}}^{G}(*)\ar[d]^{\simeq} \\  K\bC^{G} (E_{\cF}G^{\cw})\ar[r]^-{\simeq}& \colim_{G_{\cF}\Orb} K\bC^{G}  \ar[rrr]^-{\Ass^{h}_{\bC,\cF}} &&& K\bC^{G}(*) } 
\end{equation} 
which is natural for $\bC$ in $\Fun(BG,\nCcat_{\ndeg,\eadd,\omega\add})$.
\end{prop}
 
\begin{proof}
For any   effectively additive $C^{*}$-category $\bD$ we define a functor
$$\bD^{u}[-] \colon \Set\to \Ccat\, .$$
It sends a set
$X$  to the $C^{*}$-category $\bD^{u}[X]$ whose objects are pairs
$(D,(p_{x})_{x\in X})$ of an object $D$ of $\bD^{u}$ and a family of  mutually orthogonal 
effective projections on $D$ such that $\{x\in X\mid p_{x}\not=0\}$ is finite and $\sum_{x\in X} p_{x}=\id_{D}$. 
The morphisms $ (D,(p_{x})_{x\in X})\to (D',(p'_{x})_{x\in X})$ in $\bD^{u}[X]$ are morphisms $a \colon D\to D'$ in $\bD$ such that for all $x,x'$ in $X$ with $x\not= x'$ we have
$p'_{x'}ap_{x}=0$.  A morphism $f \colon X\to X'$  of sets induces a unital functor
$\bD^{u}[X]\to \bD^{u}[X']$ which sends
$(D,(p_{x})_{x\in X})$ to $(D,(\sum_{x\in f^{-1}(x')}p_{x})_{x'\in X'})$ (here we use the assumption that $\bD$ is effectively additive)  and acts as identity on morphisms.

The construction of $\bD^{u}[-]$ from $\bD$ is functorial  in $\nCcat_{\ndeg,\eadd}$.  If $G$ acts on $X$ and $\bD$, then we get an induced action on $\bD^{u}[X]$ by functoriality.
We have thus  defined a functor from $\Fun(BG,\nCcat_{\ndeg,\eadd})$   
 to $\Fun(G\Set,\Fun(BG,\Ccat))$.

For $X$ in $G\Set $ and $\bC$ in $\Fun(BG,\nCcat_{\ndeg,\eadd})$ we let $\mathbf{\tilde{\bar{C}}}^{\mathrm{ctr}}_{\lf}(X_{min,max})$ in $\Fun(BG,\Ccat)$ denote the object $\bar{C}^{\mathrm{ctr}}_{\lf}(X_{min,max})$ introduced in Definition \ref{rfquhwfiuqwhfiufewqefqwefqwefwefqwef} for the trivial group
with the $G$-action induced by functoriality.
In  \cite[Prop.\ 9.12 (1)]{coarsek} we have constructed an  isomorphism
$$\mathbf{\tilde{\bar{C}}}^{\mathrm{ctr}}_{\lf}((-)_{min,max})\cong \bC^{u}[-]$$
of functors from $G\Set$ to $\Fun(BG,\Ccat)$. For $X$ in $G\Set$ it sends the object $(C,\mu)$ in $\mathbf{\tilde{\bar{C}}}^{\mathrm{ctr}}_{\lf}(X_{min,max})$ to the object
$(C,(\mu(\{x\}))_{x\in X}) $ in $\bC^{u}[X]$ and acts as identity on morphisms. This isomorphism is clearly natural for $\bC$ in $\Fun(BG,\nCcat_{\ndeg,\eadd})$. Restricting along $G\Orb\subseteq G\Set$  and applying $-\rtimes_{r} G$ we therefore get an equivalence
\begin{equation}\label{svsdfvsdfvsfvfdv}  K^{\Ccat}(\mathbf{\tilde{\bar{C}}}^{\mathrm{ctr}}_{\lf}((-)_{min,max})\rtimes_{r}G)\simeq
K^{\Ccat}(\bC^{u}[-]\rtimes_{r}G)
\end{equation}  of functors from $G\Orb$ to $\Sp$
which is natural for $\bC$ in $\Fun(BG,\nCcat_{\ndeg,\eadd})$.

 We now use that $\bC$ admits countable $AV$-sums.
By \eqref{svsdfvsdfvsfvfdv} and \cite[Prop. 9.12 (3)]{coarsek} we have a unitary equivalence
$$\phi \colon \mathbf{\tilde{\bar{C}}}^{\mathrm{ctr}}_{\lf}((-)_{min,max}) \rtimes_{r}G\stackrel{\simeq}{\to}\mathbf{\bar{C}}^{\mathrm{ctr}}_{\lf} ((-)_{min,max}\otimes G_{can,min})$$
of functors from $G\Set$ to $\Ccat$.
This construction is not natural in $\bC$ since  the first step in the proof of 
\cite[Prop.\ p.1]{coarsek} going into  \cite[Prop. 9.12 (3)]{coarsek}   involves the choice of an AV-sum
$(\bigoplus_{g\in G}gC,(e^{C}_{g})_{g\in G})$ for every object $C$ of $\bC$.  
But if $\kappa \colon \bC\to \bC'$ is  a morphism in $\Fun(BG,\nCcat_{\ndeg,\eadd,\omega\add})$, then it preserves AV-sums and for every object $C$ of $\bC$ we have a unique multiplier unitary 
$u_{C} \colon \bigoplus_{g\in G}gC\to \bigoplus_{g\in G}g\kappa(C)$ such that
$u_{C}e^{C}_{g}=e^{\kappa(C)}_{g}$ for every $g$ in $G$.
These unitaries induce a unitary filler of the square of $\Ccat$-valued functors
$$\xymatrix{ \mathbf{\tilde{\bar{C}}}^{\mathrm{ctr}}_{\lf}((-)_{min,max}) \rtimes_{r}G\ar[r]^-{\phi_{C}}\ar[d]  &\mathbf{\bar{C}}^{\mathrm{ctr}}_{\lf} ((-)_{min,max}\otimes G_{can,min}) \ar[d]  \\
\mathbf{\tilde{\bar{C'}}}^{\mathrm{ctr}}_{\lf}((-)_{min,max}) \rtimes_{r}G \ar[r]^-{\phi_{C'}} &\mathbf{\bar{C'}}^{\mathrm{ctr}}_{\lf} ((-)_{min,max}\otimes G_{can,min}) } $$
whose vertical maps are induced by $\kappa$. We therefore get an  equivalence
of functors from $\Fun(BG,\nCcat_{\ndeg,\eadd,\omega\add})$ to $\Fun(G\Set,\Ccat_{2,1})$.
Since $K^{\Ccat}$ factorizes over the localization $\Ccat\to \Ccat_{2,1}$  at unitary equivalences,
after applying $K^{\Ccat}$, restricting along $G\Orb\subseteq G\Set$, and using Definitions \ref{qroeigjoqergerqfewewfqewfeqwf} and \ref{qrogijeqoifefewfefewqffe} we get an equivalence 
\begin{equation}\label{gewrggrgreferwf}
 K^{\Ccat}( \mathbf{\tilde{\bar{C}}}^{\mathrm{ctr}}_{\lf}((-)_{min,max}) \rtimes_{r}G)\stackrel{\simeq}{\to} K^{\Ccat}(\mathbf{\bar{C}}^{\mathrm{ctr}}_{\lf} ((-)_{min,max}\otimes G_{can,min}))\simeq K\bC^{G}
\end{equation}
which is natural for $\bC$ in $\Fun(BG,\nCcat_{\ndeg,\eadd,\omega\add})$.

 We have  a natural transformation 
 \begin{equation}\label{vsdfvfdsvq3fr}
 v \colon \bC^{u}\otimes_{\max} \C[-]\to \bC^{u}[-]\,,
\end{equation} 
{see \eqref{ewrfewrfwefefewfefwefe} for $\C[-]$,}
 of functors from $G\Set$ to $\Fun(BG,\Ccat)$. Its 
 component on $X$ in ${G}\Set$ is the  functor 
 $$v_{X} \colon \bC^{u}\otimes_{\max} \C[X]\to  \bC^{u}[{X}]\, ,$$
  which 
 sends the object
 $(C,y)$ in $ \bC^{u}\otimes_{\max} \C[X]$ to the object $(C,(p^{y})_{x\in X})$
 with
 $$p_{x}^{y}:=\left\{\begin{array}{cc} \id_{C}&x=y,\\0&x\not=y, \end{array} \right.$$
 and which acts  by $a\otimes z\mapsto za$ on morphisms.
The functor $v_{X}$ is a Morita equivalence: It is fully faithful, and  every object of
$ \bC^{u}[{X}]$ is isomorphic to a finite sum of objects in the image of $v_{X}$.
Since $K^{\Ccat}$ is Morita invariant and $-\rtimes_{r}G$ preserves Morita equivalences by \cite[Prop.\ 16.11]{cank}, after restriction along $G\Orb\subseteq G\Set$ we get a natural transformation 
of functors 
\begin{equation}\label{dsfvdsfvs34rfefs} \hat K^{G}_{\bC^{u}}\simeq K^{\Ccat}( (\bC^{u}\otimes_{\max} \C[-])\rtimes_{r}G)\simeq  K^{\Ccat}( \bC^{u}[-]\rtimes_{r}G)
\end{equation} from $G\Orb$ to $\Sp$
 where we use Definition \ref{egiueheiugheiugwegerggwfewr} in order to see  the first equivalence.
 Since the transformation \eqref{vsdfvfdsvq3fr} is clearly natural for $\bC$ in $\Fun(BG,\nCcat_{\ndeg,\eadd,\omega\add})$, so is 
 \eqref{dsfvdsfvs34rfefs}.

Combining \eqref{dsfvdsfvs34rfefs}, \eqref{gewrggrgreferwf} and \eqref{svsdfvsdfvsfvfdv}
we get the equivalence asserted in the proposition.
\end{proof}

\begin{prop}\label{weiojgwoerferww}\mbox{}
If {$\cF \subseteq \Fin$},  then have a commutative square \begin{equation}\label{adfadfadfsafdsf}
\xymatrix{\Sigma RK^{G,\an}_{(\bC^{u})^{(G)}}(E_{\cF}G^{\cw})\ar[rrr]^-{\Sigma \mu^{\Kasp}_{(\bC^{u})^{(G)},\cF}}\ar@{-}[d]^{\simeq}&&&\Sigma  \KK(\C,(\bC^{u})^{(G)}\rtimes_{r} G)\\ RK^{G,\An}_{\bC}( E_{\cF}G^{\cw}) \ar[rrr]_{\Ass^{\an}_{\bC,\cF}}&&&{\Sigma}\KK(\C,{\bC}^{(G)}_{\std}\rtimes_{r} G)\ar[u]_{{\simeq}}}
\end{equation}
which is natural in $\bC$ in $\Fun(BG,\nCcat_{\ndeg,\eadd,\omega\add})$.
\end{prop}
\begin{proof}
We start with the construction of the square \eqref{adfadfadfsafdsf}.  
Its  left vertical morphism will be induced by a zig-zag and therefore does not have a preferred direction. We expand the square into the following commutative diagram:
\begin{align}\label{qewfqwefqwdqd}
\\
\mathclap{
\xymatrix{\ar[d]_{\simeq}\Sigma RK^{G,\an}_{(\bC^{u})^{G}}(E_{\cF}G^{\cw})\ar@/^1.5cm/[rrr]^{\Sigma \mu^{\Kasp}_{(\bC^{u})^{(G)},\cF}}\ar[rr]^{{\Sigma}\mu^{\Kasp}_{(\bC^{u})^{(G),\cF,\max}}}  && \Sigma \KK(\C,(\bC^{u})^{(G)}\rtimes G)\ar[r]\ar[d]^{\simeq}&\Sigma  \KK(\C,(\bC^{u})^{(G)}\rtimes_{r} G)\ar[d]_{\simeq}   \\ 
\Sigma RK^{G,\an}_{\bC^{(G)}_{\std,+}}(E_{\cF}G^{\cw}) \ar[rr]^{\Sigma \mu^{\Kasp}_{\bC^{(G)}_{\std,+},\cF,\max}} && \Sigma \KK(\C,\bC^{(G)}_{\std,+}\rtimes G)\ar[r]& \Sigma \KK(\C,\bC^{(G)}_{\std,+}\rtimes_{r} G)\\ 
\Sigma   RK^{G,\an}_{{\bC}^{(G)}_{\std}}(E_{\cF}G^{\cw}) \ar[u]^{\simeq}\ar[rr]^{\Sigma \mu^{\Kasp}_{{\bC}^{(G)}_{\std},\cF,\max}} & &\ar[u] \Sigma \KK(\C,{\bC}^{(G)}_{\std}\rtimes  G) \ar[r]&\ar[u]^{{\simeq}} \Sigma \KK(\C,{\bC}^{(G)}_{\std}\rtimes_{r} G) \ar@{=}[d] \\ 
RK^{G,\an}_{\bQ^{(G)}_{\std}}(E_{\cF}G^{\cw}) \ar@{=}[d]\ar[rr]^{\mu^{\Kasp}_{\bQ^{(G)}_{\std},\cF,\max}}\ar[u]^{\simeq} &&\KK(\C,\bQ^{(G)}_{\std}\rtimes G) \ar[r]\ar[u]^{\simeq}& {\Sigma}\KK(\C,{\bC}^{(G)}_{\std}\rtimes_{r} G) \\ 
RK^{G,\An}_{\bC}( E_{\cF}G^{\cw})\ar[rrru]^{\simeq}_{\Ass^{\an}_{\bC,\cF}} &&& }
}
\notag
\end{align}

The  two upper rows of vertical maps are induced by the zig-zag  $$(\bC^{u})^{(G)} \to \bC^{(G)}_{\std,+}\leftarrow   \bC^{(G)}_{\std}$$
(see \eqref{aefefdfafadfdffsa}), where the first map    is a weak Morita equivalence and the second is a split relative Morita equivalence. 
 We use (see below for details) that  the functors $RK^{G,\an}_{-}(E_{\cF}G^{\cw})$  and $\KK(\C,-\rtimes_{r} G)$ {send}  weak Morita equivalences and split relative Morita equivalences to  equivalences.
\begin{enumerate}
\item Recall that $RK^{G,\an}_{\bD}(E_{\cF}G^{\cw})\cong\colim_{W\subseteq E_{\cF}G^{\cw}} K^{G,\an}_{\bD}(W)$, where the   colimit runs over the filtered poset of  $G$-finite $G$-CW subcomplexes of $E_{\cF}G$. For fixed $W$ the functor $\bD\mapsto K^{G,\an}_{\bD}(W)$ sends   relative Morita equivalences 
to equivalences by Lemma \ref{eroigjweogregwgefwe}.\ref{qroigwoeigwjereferwfwerf3}.
Its sends weak Morita equivalences to equivalences by \cite[Thm.\ 1.32.3]{KKG}.
 \item Since we have the equivalence  $\KK(\C,-\rtimes G)\simeq \KK(\C,-)\circ (-\rtimes G)\circ \kkG_{\Ccat}$ of functors from $\Fun(BG,\nCcat)$ to $\Sp$,
 the functor $ \KK(\C,-\rtimes G)$ sends weak Morita  equivalences to equivalences since already $\kkG_{\Ccat}$ does so by \cite[Thm.\ 1.32.3]{KKG}. Hence the middle upper vertical arrow is an equivalence. 
One could also show that the other vertical arrow in this column is an equivalence, but since this is not needed in our argument we will not go through the details here.
 \item Since $\KK(\C,-\rtimes_{r} G)\simeq \KK(\C,-)\circ (-\rtimes _{r}G)\circ \kkG_{\Ccat}$, as in the previous point, 
the functor $\KK(\C,-\rtimes_{r} G)$ sends weak Morita equivalences to equivalences. Since
$-\rtimes_{r} G$ preserves Morita equivalences by \cite[Prop.\ 16.11]{cank} and $ \KK(\C,-)=K^{\Ccat}$ sends Morita equivalences to equivalences by \cite[Prop.\ 16.18]{cank} we see that 
$\KK(\C,-\rtimes_{r} G)$ sends Morita equivalences to equivalences.
In order to see that it also sends split relative  Morita equivalences to equivalences we apply
$-\rtimes_{r} G$ to the diagram \eqref{csdioucsadoicuoasdcdsacasdcasc}. In view of the existence of splits for $p$ and $q$,
exactness of the horizontal sequences is preserved. Because $-\rtimes_{r} G$ preserves Morita equivalences the resulting diagram shows that 
 $\phi\rtimes_{r}G \colon \bD\rtimes_{r}G\to \bE\rtimes_{r}G$ is a relative Morita equivalence.
Since $\KK(\C,-)=K^{\Ccat}$ is a Morita invariant homological functor, 
it sends relative Morita equivalences to equivalences by \cite[Prop.\ 17.4]{cank}.
\end{enumerate}

%
%

%

 The   two upper right squares
 are provided  by the    natural transformation $-\rtimes G\to -\rtimes_{r}G$.
The two lower left vertical arrows are induced by 
the  boundary map  of the fibre sequence 
 associated to the 
exact sequence  $0\to {\bC}^{(G)}_{\std}\to  {\bM\bC}^{(G)}_{\std}\to  \bQ^{(G)}_{\std}\to 0$ in $\Fun(BG,\nCcat)$, see the proof of Proposition \ref{wegojeogrregregwergwegre}. This connecting map is an equivalence
  since $ \bM\bC^{(G)}_{\std} $ is flasque.  The three  left squares commute by the 
 naturality of the Kasparov assembly map with respect to the coefficients in $\KKG$. 
 The upper triangle and the lower triangle reflect the Definitions \ref{wergoijowergerrrfrfrfrfeggwgw} and \ref{wergoijowergerreggwgw}  of
 $\mu_{(\bC^{u})^{(G)},\cF}^{\Kasp}$ and $\Ass^{\an}_{\bC,\cF}$. 
\end{proof}

Note that the statement of Theorem \ref{wtkogwegerfwerf} depends on an object $\kkG(A)$ in $\KK^{G}_{\sepa}$ for a separable $G$-$C^{*}$-algebra $A$.
In the proof we want to relate the Kasparov assembly map $\mu^{\Kasp}_{A,\Fin}$ with the Davis-Lück assembly map
$\mu^{\DL}_{A,\Fin}$ by comparing them with the analytic assembly maps $\Ass^{\an}_{\bC,\Fin}$ and
$\Ass^{h}_{\bC,\Fin}$, respectively, for a suitably choice of $G$-$C^{*}$-category $\bC$ and invoking Theorem \ref{wtoiguwegwergergregwe}. 
If $A$ is a unital separable $G$-$C^{*}$-algebra, then
we could take $\bC=\Hilb_{c}(A)$.  But not every
class in $\KK^{G}_{\sepa}$ is represented by a unital $G$-$C^{*}$-algebra.
But every class is  a fibre of a morphism between classes of unital algebras algebras.  Indeed, if a class is represented by a $G$-$C^{*}$-algebra  $A$, then it is equivalent to the fibre of
  $\kkG(A^{+})\to \kkG(\C)$. In order to apply this we must model the unitalization map by a  suitable 
essential functor between associated effectively additive $G$-$C^{*}$-categories. This is the contents of the following proposition.

Let $A$ be in $\Fun(BG,\nCalg)$ and consider the split  unitalization sequence
$$0\to A\to A^{+}\xrightarrow{p} \C\to 0$$
whose split will be denoted by 
 $e \colon \C\to A^{+}$.

\begin{prop}\label{weiorgwergrwefw}
There exists the following data:
\begin{enumerate}
\item $\bC_{+}$, $\bC_{\C}$ in $\Fun(BG,\nCcat_{\ndeg,\eadd,\omega\add})$,
\item $q \colon \bC_{+}\to \bC_{\C}$ in $\Fun(BG,\nCcat_{\ndeg,\eadd,\omega\add})$,
\item $s \colon \bC_{\C}\to \bC_{+}$ in $\Fun(BG,\nCcat_{\ndeg,\eadd,\omega\add})$,
\item   $i \colon A^{+}\to (\bC_{+}^{u})^{(G)}$   and  $j \colon \C\to (\bC_{\C}^{u})^{(G)}$ in   $\Fun(BG,\Ccat)$,
\end{enumerate}
 with the following properties:
\begin{enumerate}
\item  \label{gijerwigogwerffwef1} The squares $$\xymatrix{A^{+}\ar[rr]^{p}\ar[d]^{i} && \C \ar[d]^{j} \\ 
(\bC_{+}^{u})^{(G)}\ar[rr]^-{(q^{u})^{(G)}} && (\bC_{\C}^{u})^{(G)}} \qquad \mbox{and}\qquad
 \xymatrix{A^{+} \ar[d]^{i} && \ar[ll]_{e}\C \ar[d]^{j} \\ (\bC^{u}_{+})^{(G)}  && (\bC_{\C}^{u})^{(G)}\ar[ll]_-{(s^{u})^{(G)}}} $$
commute.
\item \label{gijerwigogwerffwef}$G$ weakly fixes the objects of $\bC^{u}_{+}$ and $\bC^{u}_{\C}$, see Definition \ref{werkgerwgreferfwref}.
 \item $i$ and $j$ are  Morita equivalences.
  \item  $q$ is a quotient   and $q\circ s=\id_{\bC_{\C}}$.
\end{enumerate}
\end{prop}
\begin{proof}
%
%
%
%

We let $\widehat{\bA^{+}}$ be the full subcategory of $\Hilb_{c}(A^{+})$ on the objects which are isomorphic to
$\hat A^{+}$, see Example \ref{wteklgwrtegwergfwrefrfwferfer}. Since the object $\hat A^{+}$ has an extension $(\hat A^{+},\kappa)$ in $((\widehat{\bA^{+}})^{u})^{(G)}$  we have unitary isomorphisms $\kappa_{g} \colon \hat A^{+}\to g\hat A^{+}$  in $\Hilb_{c}(A^{+})$ for all $g$ in $G$. It follows that
$\widehat{\bA^{+}}$ is $G$-invariant and 
 therefore  inherits a $G$-action from  $\Hilb_{c}(A^{+})$.
Furthermore, we  have $\widehat{\bA^{+}}=(\widehat{\bA^{+}})^{u}$ and 
 $G$ weakly fixes the objects of $(\widehat{\bA^{+}})^{u}$.

 We set $$\bC_{+} \coloneqq \widehat{\bA^{+}} \otimes_{\max} \Hilb_{c}(\C)$$ with the $G$-action induced from the first factor. 
 We furthermore let $\bF$ be the $G$-$C^{*}$-category with the same objects as  $\widehat{\bA^{+}}$ but morphism spaces isomorphic to $\C$ between any two objects. We have a canonical projection
 $q' \colon \widehat{\bA^{+}}\to \bF$ involving $p$ and a split $s' \colon \bF\to \widehat{\bA^{+}}$ involving the units of $A^{+}$. We set $$\bC_{\C} \coloneqq \bF\otimes_{\max} \Hilb_{c}(\C)\, .$$
 Then we have a quotient projection $q \coloneqq q'\otimes \id_{\Hilb_{c}(\C)} \colon \bC_{+}\to  \bC_{\C}$ and the split functor
 $s \coloneqq s'\otimes \id_{\Hilb_{c}(\C)} \colon \bC_{\C}\to \bC_{+}$ such that $q\circ s=\id_{\bC_{\C}}$.
 Because of this equality the condition that $q$ is a quotient simply means that it is bijective on objects.
 
 We define $j \colon \C\to (\bC_{\C}^{u})^{(G)}$ using the object $((\hat A^{+},\C), \kappa\otimes \id_{\C})$  and the canonical identification
 $\End_{(\bC^{u}_{\C})^{(G)}}(((\hat A^{+},\C), \kappa\otimes \id_{\C}))\cong \C$.
 We further define
 $i \colon A^{+}\to (\bC^{u}_{+})^{(G)}$ using the object
 $((\hat A^{+},\C), \kappa\otimes \id_{\C})$ and the canonical $G$-equivariant identification $\End_{(\bC_{+}^{u})^{(G)}}(((\hat A^{+},\C), \kappa\otimes \id_{\C}))
 \cong A^{+}$. 
 Then the two squares commute.

If we forget the $G$-action, then  $\bC_{+}$ is isomorphic to $A^{+}\otimes_{\max}\Hilb_{c}(\C)$. We can conclude that $\bC_{+}$ admits all AV-sums and is therefore effectively additive.
 A similar reasoning applies to $\bC_{\C}$.


The functor $q$ is full  and hence non-degenerate.
The split $s' \colon \bF\to \widehat{\bA^{+}}$ is unital and hence also non-degenerate.
This implies that $s$ is non-degenerate.

In order to show that $i$ is a Morita equivalence we note that  {any} object in $(\bC_{+}^{u})^{(G)}$ is unitarily isomorphic to an object
$((\hat A,H),\kappa\otimes \id_{H})$ for some finite-dimensional Hilbert space $H$. It is therefore
unitarily isomorphic to a finite sum of copies of $i(A^{+})$. {The same reasoning applies to show that $j$ is a Morita equivalence.}
%
%
%
\end{proof}

We now finish the proof of the Theorem \ref{wtkogwegerfwerf}.
\newcommand{\Intrs}{\mathrm{intrs}}
The statement of the theorem depends on an object $A$ of $\KKGs$ which is assumed to belong to $\ci$.
We can choose an object of  $\Fun(BG,\nCalg_{\sepa})$ which realizes $A$ in $\KKGs$ upon applying $\kkGs$. So from now on $A$ denotes this $G$-$C^{*}$-algebra.

We apply Proposition \ref{weiorgwergrwefw} to $A$ in order to get the asserted data.
For any functor $F$ from
$\Fun(BG,\nCcat_{\ndeg,\eadd,\omega\add})$ to an additive category
we get a decomposition $$F(\bC_{+})\simeq F^{\Intrs}\oplus F(\bC_{\C})\, ,$$ where the projection to and inclusion of the second summand are given by $F(q)$ and $F(s)$.
We call $F^{\Intrs}$ the interesting summand. A natural transformation  $f \colon F\to F'$ of functors
induces a map $f^{\Intrs} \colon F^{\Intrs}\to F^{\prime,\Intrs}$ between the interesting summands.
We call $f^{\Intrs}$ the interesting summand of  $f$.
Finally, a natural equivalence $f\simeq  f'$ between  natural transformations   
induces a natural equivalence $f^{\Intrs}\simeq f^{\prime,\Intrs}$ between the interesting summands.
We now have the following facts:
\begin{enumerate}
\item The interesting summand of $\mu^{\Kasp}_{(\bC^{u}_{+})^{(G)},\Fin}$ is equivalent to the interesting summand of
$\Ass^{\an}_{\bC_{+} ,\Fin}$  by  Proposition \ref{weiojgwoerferww}.
\item  By Theorem \ref{wtoiguwegwergergregwe} the  interesting summand of $\Ass^{\an}_{\bC_{+},\Fin}$ is  an equivalence if and only if
the interesting summand of 
$\Ass^{h}_{\bC_{+},\Fin}$ is an equivalence.
\item The interesting summand of 
$\Ass^{h}_{\bC_{+},\Fin}$  is equivalent to the interesting summand of $\mu^{\DL}_{\bC_{+}^{u},\Fin}$
 by  Proposition \ref{weroigjoewrgergergfewrf}.
\item 
The interesting summand of $\mu^{\DL}_{(\bC_{+}^{u})^{(G)},\Fin}$ is
equivalent to the interesting summand of $\mu^{\DL}_{\bC_{+}^{u},\Fin}$ by Lemma \ref{wekgowerfrefwerf}. Here we use Property \ref{gijerwigogwerffwef} of the data from Proposition \ref{weiorgwergrwefw}.
\item  We note that the Davis--Lück assembly map
$\mu^{\DL}_{-,\Fin}$ depends functorially on an object of $\KKG$. 
The pair of morphisms $p \colon A^{+}\to \C$ and $e \colon \C\to A^{+}$ provides a decomposition
$\kkG(A^{+})\simeq \kkG(A)\oplus \kkG(\C)$. The commutative squares
in  Property \ref{gijerwigogwerffwef1}
of the data from Proposition \ref{weiorgwergrwefw} provide a decomposition of the transformation  
$\mu^{\DL}_{i,\Fin}$ into a sum $(\mu^{\DL}_{i,\Fin})^{\Intrs}\oplus \mu^{\DL}_{j,\Fin}$.
Since $i$ is a   Morita equivalence and the transformation between the  Davis--Lück assembly maps depends on $\hat K^{G}_{i}$, by  Corollary \ref{wegiwrogerfrweferfwerf}.\ref{wthgjiwoegerferferfwef2} the  transformations  $ \mu^{\DL}_{i,\Fin}$ and hence $(\mu^{\DL}_{i,\Fin})^{\Intrs}$ are equivalences.
We conclude that the interesting  summand  of  $\mu^{\DL}_{(\bC_{+}^{u})^{(G)},\Fin}$ is equivalent to 
$\mu^{\DL}_{A,\Fin}$.  
\item By a completely analogous argument the interesting  summand  of  $\mu^{\Kasp}_{(\bC_{+})^{(G)},\Fin}$ is equivalent to  $\mu^{\Kasp}_{A,\Fin}$. Here  we use 
 that the domain and target $RK^{G,\an}_{-}(E_{\Fin}G^{\cw})$ and 
 $\KK(\C,-\rtimes_{r}G)$ of $\mu^{\Kasp}_{-,\Fin}$ considered as functors on $\Fun(BG,\nCcat)$
via $\kkGA$ send Morita equivalences to equivalences. 
For $\KK(\C,-\rtimes_{r}G)$ this has been observed above in the proof of  Corollary \ref{wegiwrogerfrweferfwerf}.\ref{wthgjiwoegerferferfwef2}. For the other functor we use  the formula
$$RK^{G,an}_{-}(E_{\Fin}G^{\cw})\simeq \colim_{W\subseteq E_{\Fin}G^{\cw}}\KK^{G}(C_{0}(W),-)\, ,$$ where $W$ runs over the $G$-finite subcomplexes of $ E_{\Fin}G^{\cw}$, 
and Lemma \ref{eroigjweogregwgefwe}.\ref{qroigwoeigwjereferwfwerf3} saying that
$\KK^{G}(C_{0}(W),-)$ sends Morita equivalences to equivalences for every $W$.
 \end{enumerate}
By a combination of these facts we see that $\mu^{\Kasp}_{A,\Fin}$ is an equivalence if and only if $\mu^{\DL}_{A,\Fin}$ is an equivalence. Under the assumption that $\kkGs(A)$ belongs to $\ci$ we know that 
$\mu^{\DL}_{A,\Fin}$ is an equivalence by  Lemma \ref{egiowgwgeregegwref}.
\end{proof}

\bibliographystyle{alpha}
\bibliography{forschung}

\end{document}